\documentclass{amsart}

\usepackage{time}

\usepackage[margin=1.5in]{geometry}

\usepackage{amscd}
\usepackage{amsmath}
\usepackage{amssymb}
\usepackage{amsthm}
\usepackage{arydshln}
\usepackage{bbm}
\usepackage{bm}
\usepackage{colonequals}
\usepackage{color}
\usepackage{enumitem}
\usepackage{latexsym}
\usepackage{mathrsfs}
\usepackage{mathtools}
\usepackage{stmaryrd}
\usepackage{textcomp}
\usepackage{tikz-cd}
\usepackage{url}
\usepackage{varwidth}
\usepackage[all]{xy}
\usepackage{yhmath}

\usepackage[colorlinks=false,hidelinks]{hyperref}
\usetikzlibrary{decorations.pathmorphing}

\newcommand{\ab}{\mathrm{ab}}
\newcommand{\abs}{\mathrm{abs}}
\newcommand{\ad}{\mathrm{ad}}

\newcommand{\ccl}{\mathrm{ccl}}

\newcommand{\der}{\mathrm{der}}

\newcommand{\evrs}{\mathrm{evrs}}
\newcommand{\evsrs}{\mathrm{evsrs}}
\newcommand{\geom}{\mathrm{geom}}

\newcommand{\nevrs}{\mathrm{nevrs}}

\newcommand{\nvreg}{\mathrm{nvreg}}
\newcommand{\pr}{\mathrm{pr}}
\newcommand{\ram}{\mathrm{ram}}

\newcommand{\red}{\mathrm{red}}

\newcommand{\res}{\mathrm{res}}
\newcommand{\rs}{\mathrm{rs}}

\renewcommand{\sc}{\mathrm{sc}}

\renewcommand{\ss}{\mathrm{ss}}

\newcommand{\sym}{\mathrm{sym}}

\newcommand{\unip}{\mathrm{unip}}

\newcommand{\ur}{\mathrm{ur}}

\newcommand{\FKS}{\mathrm{FKS}}

\newcommand{\Jord}{\mathfrak{Jord}}

\newcommand{\Yu}{\mathrm{Yu}}

\newcommand{\ul}{\underline}
\newcommand{\ol}{\overline}

\newcommand{\cO}{\mathcal{O}}
\newcommand{\mcO}{\mathcal{O}}
\newcommand{\mfp}{\mathfrak{p}}

\newcommand{\bfT}{\mathbf{T}}
\newcommand{\x}{\mathbf{x}}

\newcommand{\bbT}{\mathbb{T}}

\newcommand{\cT}{\mathcal{T}}

\newcommand{\bbU}{\mathbb{U}}

\newcommand{\Z}{\mathbb{Z}}
\newcommand{\F}{\mathbb{F}}
\newcommand{\FF}{\mathbb{F}}
\newcommand{\bbF}{\mathbb{F}}

\newcommand{\cP}{\mathcal{P}}

\newcommand{\R}{\mathbb{R}}

\newcommand{\Q}{\mathbb{Q}}
\newcommand{\bbQ}{\mathbb{Q}}
\newcommand{\QQ}{\mathbb{Q}}
\newcommand{\C}{\mathbb{C}}
\newcommand{\bbC}{\mathbb{C}}
\newcommand{\Gm}{\mathbb{G}_{\mathrm{m}}}
\newcommand{\bbS}{\mathbb{S}}

\newcommand{\bG}{\mathbb{G}}
\newcommand{\bbG}{\mathbb{G}}
\newcommand{\bfG}{\mathbf{G}}

\newcommand{\bbZ}{\mathbb{Z}}
\newcommand{\J}{\mathbf{J}}
\newcommand{\bfS}{\mathbf{S}}

\newcommand{\bfZ}{\mathbf{Z}}

\newcommand{\cB}{\mathcal{B}}
\newcommand{\mcB}{\mathcal{B}}
\newcommand{\cR}{\mathcal{R}}

\newcommand{\mfG}{\mathfrak{G}}

\newcommand{\II}{\mathrm{II}}

\DeclareMathOperator{\cInd}{c-Ind}
\DeclareMathOperator{\depth}{depth}

\DeclareMathOperator{\id}{id}

\DeclareMathOperator{\val}{val}

\DeclareMathOperator{\Cent}{Cent}

\DeclareMathOperator{\Frob}{Frob}
\DeclareMathOperator{\Gal}{Gal}

\DeclareMathOperator{\Ind}{Ind}
\DeclareMathOperator{\Int}{Int}
\DeclareMathOperator{\Irr}{Irr}

\DeclareMathOperator{\Ker}{Ker}
\DeclareMathOperator{\Lie}{Lie}
\DeclareMathOperator{\LLC}{LLC}

\DeclareMathOperator{\Nr}{Nr}

\DeclareMathOperator{\Res}{Res}

\DeclareMathOperator{\Stab}{Stab}

\DeclareMathOperator{\Tr}{Tr}

\DeclareMathOperator{\GL}{GL}

\DeclareMathOperator{\PGL}{PGL}
\DeclareMathOperator{\SL}{SL}
\DeclareMathOperator{\SO}{SO}
\DeclareMathOperator{\Sp}{Sp}

\newcommand{\from}{\colon}

\theoremstyle{plain}
\newtheorem{thm}{Theorem}[section]
\newtheorem{theorem}[thm]{Theorem}
\newtheorem*{thm*}{Theorem}
\newtheorem{prop}[thm]{Proposition}
\newtheorem{proposition}[thm]{Proposition}
\newtheorem{lem}[thm]{Lemma}
\newtheorem{lemma}[thm]{Lemma}
\newtheorem{cor}[thm]{Corollary}

\theoremstyle{definition}
\newtheorem{defn}[thm]{Definition}
\newtheorem{definition}[thm]{Definition}

\theoremstyle{remark}
\newtheorem{rem}[thm]{Remark}
\newtheorem*{claim*}{Claim}
\newtheorem{remark}[thm]{Remark}
\newtheorem{example}[thm]{Example}
\newtheorem{notation}[thm]{Notation}

\theoremstyle{theorem}
\newtheorem{displaytheorem}{Theorem}

\newtheorem{displaytheoremprime}{Theorem}

\makeatletter
\newcommand{\dashover}[2][\mathop]{#1{\mathpalette\df@over{{\dashfill}{#2}}}}
\newcommand{\fillover}[2][\mathop]{#1{\mathpalette\df@over{{\solidfill}{#2}}}}
\newcommand{\df@over}[2]{\df@@over#1#2}
\newcommand\df@@over[3]{%
  \vbox{
    \offinterlineskip
    \ialign{##\cr
      #2{#1}\cr
      \noalign{\kern1pt}
      $\m@th#1#3$\cr
    }
  }%
}
\newcommand{\dashfill}[1]{%
  \kern-.5pt
  \xleaders\hbox{\kern.5pt\vrule height.4pt width \dash@width{#1}\kern.5pt}\hfill
  \kern-.5pt
}
\newcommand{\dash@width}[1]{%
  \ifx#1\displaystyle
    2pt
  \else
    \ifx#1\textstyle
      1.5pt
    \else
      \ifx#1\scriptstyle
        1.25pt
      \else
        \ifx#1\scriptscriptstyle
          1pt
        \fi
      \fi
    \fi
  \fi
}
\newcommand{\solidfill}[1]{\leaders\hrule\hfill}
\makeatother

\SelectTips{cm}{11}

\title{Characterization of supercuspidal representations and very regular elements}

\author{Charlotte Chan}
\address{Department of Mathematics, University of Michigan, 2074 East Hall, 530 Church Street, Ann Arbor, MI 48105, USA.}
\email{charchan@umich.edu}

\author{Masao Oi}
\address{Department of Mathematics (Hakubi center), Kyoto University, Kitashirakawa, Oiwake-cho, Sakyo-ku, Kyoto 606-8502, Japan.}
\email{masaooi@math.kyoto-u.ac.jp}

\begin{document}

\begin{abstract}
    We prove that regular supercuspidal representations of $p$-adic groups are uniquely determined by their character values on very regular elements---a special class of regular semisimple elements on which character formulae are very simple---provided that this locus is sufficiently large. As a consequence, we resolve a question of Kaletha by giving a description of Kaletha's $L$-packets of regular supercuspidal representations which mirrors Langlands' construction for real groups following Harish-Chandra's characterization theorem for discrete series representations. Our techniques additionally characterize supercuspidal representations in general, giving $p$-adic analogues of results of Lusztig on reductive groups over finite fields. In particular, we establish an easy, non-cohomological characterization of unipotent supercuspidal representations when the residue field of the base field is sufficiently large.
\end{abstract}

\maketitle

\section{Introduction} \label{sec:intro}

The subject of this paper is to examine the following question: How much do you need to know about the character of a supercuspidal representation of a $p$-adic group in order to determine it? In the 1990s, Henniart proved that certain classes of supercuspidal representations of $\GL_n$ (namely, those associated to an unramified torus, or those associated to a totally tamely ramified torus if $n$ is prime) can be recognized by their characters on \textit{very regular elements} (\cite{Hen92,Hen93}). This can be thought of as a (special case of) a $p$-adic analogue of Harish-Chandra's characterization of discrete series representations of real groups via their characters on compact regular semisimple elements. 
It begs the question: Can one establish Henniart's theorem for general $p$-adic groups? In the last thirty years, this central question has been asked repeatedly \cite{AS09,Kal19} and in this paper we give what we believe to be a ``sharp'' answer to this question. Especially given the subtle and difficult nature of giving general character formulas for supercuspidal representations \cite{AS09,DS18,Spi18,Spi21,FKS21}, our focus is on finding a special domain of regular semisimple elements which is simultaneously small enough that the character formulas are extremely simple and large enough that it may determine certain invariants of supercuspidal representations. 

One of the main theorems of this paper is a characterization theorem for Kaletha's $L$-packets of regular supercuspidal representations \cite{Kal19}. Let $\bfS$ be a tame elliptic maximal torus of a connected reductive group $\bfG$ over a non-archimedean local field $F$ and let $\theta$ be a regular character of $\bfS(F)$. From this data, Kaletha constructs a Langlands parameter $\phi_\theta$ and proposes an associated $L$-packet $\Pi^{\bfG}_{\phi_{\theta}}$ of supercuspidal representations of $\bfG(F)$. In this paper, we establish a description of this $L$-packet which mirrors Langlands' original construction  of $L$-packets of discrete series representations for real reductive groups \cite{Lan89}:

\begin{displaytheorem}\label{thm:L-packet char}
Assume there are sufficiently many very regular elements in $\bfS(F)$.
Then $\Pi^{\bfG}_{\phi_{\theta}}$ exactly consists of the irreducible supercuspidal representations $\pi_j$ whose character is
    \begin{equation}\label{eq:vreg character}
        \Theta_{\pi_j}(\gamma) = c \cdot \sum_{w \in W_G(\bfT_{\gamma},j(\bfS))} \Delta({}^w \gamma) \cdot \theta_j({}^w \gamma) \qquad \text{for all very regular $\gamma \in \bfG(F)$,}
    \end{equation}
    where $(j,\theta_j)$ varies over the rational conjugacy classes in the stable conjugacy class of $(\bfS \hookrightarrow \bfG,\theta)$.
    Here, $c \in \bbC^1$ is some constant independent of $\gamma$, the index set is the Weyl group with respect to the connected centralizer $\bfT_{\gamma}$ of $\gamma$ and $\bfS$, and $\Delta$ is a certain explicit function (the ``transfer factor'').
\end{displaytheorem}

Kaletha's original construction of $\Pi_{\phi_\theta}^\bfG$ is as follows.
To the pair $(\bfS, \theta)$, Kaletha  associates a $\bfG$-equivalence class of Yu-data $\Psi = (\vec \bfG, \vec \phi, \vec r, \x, \rho_0)$ in the sense of Hakim--Murnaghan \cite{HM08} and obtains, by applying Yu's construction \cite{Yu01}, an irreducible supercuspidal representation $\pi^{\Yu}_{(\bfS,\theta)}$.
Then Kaletha declares $\Pi_{\phi_\theta}^\bfG$ to be the set of all $\pi^{\Yu}_{(j(\bfS),\theta_j \varepsilon_j)}$, where $(j,\theta_j)$ varies over (a choice of representatives of) the rational conjugacy classes in the stable conjugacy class of $(\bfS \hookrightarrow \bfG,\theta)$ and where $\varepsilon_j$ is an extremely delicately defined quadratic character which depends both on $j$ and $\theta$ (see \cite[p.\ 1153-1154]{Kal19}).

Under the additional assumption that $\theta$ satisfies a strong genericity condition, the supercuspidal $\pi^{\Yu}_{(\bfS,\theta)}$ was constructed by Adler \cite{Adl98}, predating Yu.
These supercuspidals are called toral.
For $L$-packets of toral supercuspidals corresponding to unramified tori, the necessity of the quadratic twist $\varepsilon_j$ of $(j,\theta_j)$ was established by DeBacker--Spice \cite{DS18}.
In \cite{CO21}, the authors of the present paper proved that under a mild assumption on the size of the residue field of $F$, for these $(j,\theta_j)$, the parahoric Deligne--Lusztig induction of \cite{CI21-RT} gives rise to an irreducible supercuspidal representation $\pi_{(j(\bfS),\theta_j)}^{\geom}$ and that $\pi_{(j(\bfS),\theta_j)}^{\geom} \cong \pi^{\Yu}_{(j(\bfS),\theta_j \varepsilon_j)}$. 
This comparison was established by considering very regular elements, though the exact methods were different and for the most part somewhat simpler. 
It is worth noting that very regularity has geometric significance in the unramified setting (see \cite[Theorem 1.2]{CI21-RT}), and that the notion of very regularity in Theorem \ref{thm:L-packet char} generalizes all previous usages of this terminology \cite{Hen92,Hen93,CI21-RT,CO21}.

Recently, Fintzen--Kaletha--Spice \cite{FKS21} proved that there exists a twist of Yu's construction such that the resulting supercuspidal representation $\pi_{(j(\bfS),\theta_j)}^{\FKS}$ is isomorphic to $\pi^{\Yu}_{(j(\bfS),\theta_j \varepsilon_j)}$ in Kaletha's more general setting, thereby removing the need to externally twist. Throughout this paper, we will use Fintzen--Kaletha--Spice's renormalization of Yu's construction.
We also note that by Fintzen \cite{Fin21-Ann}, any supercuspidal representation is obtained by Yu's construction when $p$ does not divide the order of the absolute Weyl group.
In this paper, we always assume this condition on $p$.

The key to Theorem \ref{thm:L-packet char} is a characterization theorem: 

\begin{displaytheoremprime}[Theorem \ref{thm:regsc-characterization}]\label{thm:intro regsc char}
    If $\bfS(F)$ has sufficiently many very regular elements, then $\pi_{(j(\bfS),\theta_j)}^{\FKS}$ is the unique irreducible supercuspidal representation satisfying \eqref{eq:vreg character}.
\end{displaytheoremprime}

Moreover, the constant $c$ in \eqref{eq:vreg character} can remain unspecified in the sense that if $\pi,\pi'$ are two supercuspidal representations satisfying the character formula \eqref{eq:vreg character} for two constants $c,c'\in\C^{1}$, then necessarily $c = c'$ and $\pi \cong \pi'$. (It also turns out that $c \in \{\pm 1\}$.) 

In \cite[Section 9]{CO21}, we proved Theorem \ref{thm:intro regsc char} under several additional assumptions: If $\bfS$ is \textit{unramified} and $\theta$ is \textit{toral}, then there exists a unique \textit{regular} supercuspidal representation $\pi$ of $\bfG(F)$ whose character is given by \eqref{eq:vreg character}. Theorem \ref{thm:intro regsc char} relaxes each of these three italicized conditions---unramified being replaced by tame, toral being replaced by regular, and regular supercuspidal replaced by supercuspidal. 

Let us explain how the three theorems in Section \ref{subsec:intro structure} of this introduction contribute to the relaxation of these three italicized conditions. 
To prove that a representation is distinguished amongst all supercuspidals (as opposed to amongst all regular supercuspidals), we need to establish the character formula of an arbitrary supercuspidal representation on very regular elements (Part \ref{part:formula} of this paper; see also Theorem \ref{thm:intro character formula} in this introduction).
We then use this to reduce the characterization of supercuspidal representations to its ``depth zero part'' (Part \ref{part:characterization} of this paper; see also Theorem \ref{thm:intro characterization} in this introduction). When $\bfS$ is unramified, this depth zero part is captured by a connected reductive group $\bbG^\circ$ over a finite field $\F_{q}$, and so relaxing the torality condition on $\theta$ in this unramified setting can be reduced to a characterization theorem of irreducible representations of $\bbG^\circ(\FF_q)$. Such a theorem was already established by Lusztig \cite{Lus20} several decades ago. However, if $\bfS$ does not split over an unramified extension of $F$, then we need to work in a more general setting and establish Lusztig's characterization theorem for representations of $\bbG(\FF_q)$, where $\bbG$ is a possibly disconnected group scheme over $\F_{q}$ which may not even be of finite type. We establish the basic representation theory of $\bbG(\FF_q)$ (\`a la Deligne and Lusztig \cite{DL76}) and prove Lusztig's theorem in this more general context in Part \ref{part:finite} (for a refinement of this theorem for regular characters, see Theorem \ref{thm:intro GG} in this introduction).

As demonstrated in the $\GL_n$ setting by Henniart \cite{Hen92,Hen93} and others \cite{BW13,Cha20,CI21-MA}, a characterization result like Theorem \ref{thm:intro regsc char} can be used to not only compare different constructions of supercuspidal representations but characterize interesting correspondences. 
To serve as a proof of concept for applications of Theorem \ref{thm:intro regsc char}, we prove (Theorem \ref{thm:LJLC-SO}) a new instance of the local Jacquet--Langlands correspondence: when $\bfG^{\ast}= \SO_{2n+1}$ and $\bfG$ is an inner form of $\bfG^{\ast}$, then the Jacquet--Langlands transfer of a depth zero regular supercuspidal representation $\pi_{(\bfS^{\ast},\theta^{\ast})}^{\FKS}$ of $\bfG^{\ast}(F)$ is $\pi_{(\bfS,\theta)}^{\FKS}$, provided that the size of the residue field of $F$ is sufficiently large. From Theorem \ref{thm:intro regsc char}, it is easy to characterize the local Jacquet--Langlands correspondence assuming that a regular supercuspidal $L$-packet and its transfer are singletons (Theorem \ref{thm:LJLC}); the subtlety in establishing Theorem \ref{thm:LJLC-SO} then lies in proving that the $L$-packet (in the sense of Arthur) of a depth zero regular supercuspidal of $\bfG^{\ast}$ is a singleton (Proposition \ref{prop:SO singleton}). We strongly expect that the approach of Section \ref{sec:LJLC} works in far greater generality; we have chosen to establish only the depth zero $\SO_{2n+1}$ case to demonstrate the application potential of our work.

We finish this part of our introduction by moving beyond the regular supercuspidal setting. As we have alluded to above, our methods in proving Theorem \ref{thm:intro regsc char} turn out to additionally apply even when $\theta$ is \text{not} assumed to be regular. Without the regularity assumption, one cannot hope to characterize supercuspidals individually (indeed, the analogous statement for connected reductive groups over finite fields is not true), but we are able to characterize certain families of supercuspidal representations by only their Harish-Chandra characters on very regular elements. In general, as one would expect, the formulae will not be quite as simplistic as in \eqref{eq:vreg character}, but it is not so far off (see Theorem \ref{thm:intro character formula} in tandem with Proposition \ref{prop:rho ss} for the character of the depth zero part). We will describe these generalizations and our methods of proof in Section \ref{subsec:intro structure}. Before doing this, let us explain the specialization of our general result to the setting diametrically opposite to the regular supercuspidal case: unipotent supercuspidal representations.

\begin{displaytheorem}[Theorem \ref{thm:p-adic unipotent}]\label{thm:intro unipotent}
    If the size of the residue field of $F$ is sufficiently large relative to the absolute rank of $\bfG$, 
    then an irreducible supercuspidal representation $\pi$ of $\bfG(F)$ is unipotent if and only if the following two conditions hold:
    \begin{enumerate}[label=(\roman*)]
        \item $\Theta_\pi|_{S_{0,\evrs}}$ is constant for every maximally unramified elliptic maximal torus $\bfS$, and
        \item $\Theta_\pi|_{S_{0,\evrs}} \neq 0$ for some maximally unramified elliptic maximal torus $\bfS$.
    \end{enumerate}
    Here, $S_{0,\evrs}$ denotes the set of very regular elements of the parahoric subgroup of $\bfS(F)$.
\end{displaytheorem}

Theorem \ref{thm:intro unipotent} is the supercuspidal analogue of Lusztig's unipotent characterization result for connected reductive groups over finite fields \cite{Lus78,Lus20} and answers a question of DeBacker.

\subsection{Structure of the paper: theorems and techniques}\label{subsec:intro structure}

With the applications to Theorems \ref{thm:L-packet char} and \ref{thm:intro unipotent} in mind, we now describe the general results and set-up of our paper. Our paper is structured in three parts, each essentially culminating in one general result. Let $\Psi = (\vec \bfG, \vec \phi, \vec r, \x, \rho_0)$ be a Yu-datum and let $\pi_\Psi^{\FKS}$ be the associated supercuspidal representation of $\bfG(F)$.
(See Section \ref{sec:tamesc} for details of these notions.)

In Part \ref{part:finite}, we study a class of algebraic groups $\bbG$ over $\F_{q}$ which is large enough to contain the reductive quotient of the (integral model of the) stabilizer of an arbitrary vertex in the reduced Bruhat--Tits building of a connected reductive group over $F$. 
More precisely, this is the class of algebraic groups satisfying the following three conditions: 
\begin{enumerate}[label=(\arabic*)]
    \item the group $\pi_0(\bbG)$ of connected components is a finitely generated abelian group, 
    \item the identity component $\bbG^\circ$ is reductive, and 
    \item there exists a central subgroup $\bbZ_{\bbG}$ of $\bbG$ such that $\bbG^\circ \bbZ_{\bbG}$ has finite index in $\bbG$.
\end{enumerate}
Following Kaletha \cite{Kal19-sc}, the choice of the central subgroup $\bbZ_{\bbG}$ allows us to extend classical Deligne--Lusztig induction $R_{\bbS^\circ}^{\bbG^\circ}$ for a maximal torus $\bbS^\circ \hookrightarrow \bbG^\circ$ to a functor $R_\bbS^\bbG$, where $\bbS \colonequals \bbS^\circ \bbZ_{\bbG}$ now plays the role of a ``maximal torus'' in $\bbG$.
We use this then to establish the representation theory of $\bbG(\FF_q)$ \`a la Deligne and Lusztig, including character formulas and Green function results; we feel this section may be of independent interest.
This is essential for us as it provides the foundation for our characterization theorem for representations of $\bbG(\FF_q)$.

Before stating these results, let us give some motivation starting from the setting of connected reductive groups over finite fields. Deligne and Lusztig \cite{DL76} proved that for any connected reductive group $\bbG^\circ$ over $\FF_q$, one has a well-defined map
\[
E \from \{\text{irred.\ rep'ns of $\bbG^\circ(\FF_q)$}\}
\to
\biggl\{
(\bbS^{\circ},\theta)
\ \bigg\vert\,
\begin{array}{l}
\text{$\bbS^\circ$: an $\F_{q}$-rational maximal torus of $\bbG^{\circ}$} \\
\text{$\theta$: a character of $\bbS(\F_{q})$}
\end{array}
\biggl\}\big/{\sim},
\]
where the symbol $\sim$ on the right-hand side denotes the geometric conjugacy.
This map is defined cohomologically: to an irreducible representation $\rho$, one associates the geometric conjugacy class containing any pair $(\bbS^\circ,\theta)$ for which $\rho$ appears in the cohomology of the Deligne--Lusztig variety associated to $\bbS^\circ \hookrightarrow \bbG^\circ$ with coefficients in the corresponding local system defined by $\theta$. In \cite{Lus20}, Lusztig proved that $E$ can be realized non-cohomologically: if $q$ is sufficiently large relative to the rank of $\bbG^{\circ}$, then in fact $E(\rho)$ is determined by the character values of $\rho$ at regular semisimple elements, for which the character formula is very simple. 
As a corollary of this, Lusztig established the following simple and elementary characterization of unipotent representations of $\bbG^\circ(\FF_q)$ for $q \gg 0$ (\cite[6]{Lus20}):
An irreducible representation $\rho$ of $\bbG^\circ(\FF_q)$ is unipotent if and only if $\Theta_\rho|_{\bbS^\circ(\FF_q)_{\rs}}$ is constant for every $\F_{q}$-rational maximal torus $\bbS^\circ$ of $\bbG^\circ$.
(Theorem \ref{thm:intro unipotent} is the analogue of this for supercuspidal representations of $p$-adic groups.)

We establish the map $E$ for $\bbG$ and show that Lusztig's techniques for establishing a non-cohomological description of $E$ can be extended to our more general setting (Theorem \ref{thm:intro GG}(b), Theorem \ref{thm:unique-Z-rho}). 
We note to the reader that in the body of this paper, we actually work with a refinement of this map $E$ (see Section \ref{subsec:Lusztig-E} for the definition of the refinement $\tilde E$). This refinement is handled already in the Lusztig's proofs in the classical case. Although establishing the representation theory of this larger class of groups $\bbG(\FF_q)$ requires some technical adjustments to the classical proofs in the setting of $\bbG^\circ(\FF_q)$, the main subtlety which we deal with in Part \ref{part:finite} is something separate altogether, which we now explain.

Our motivation for developing the theory and results in Part \ref{part:finite} is to study $\pi_\Psi^{\FKS}$. For this, we take $\bbG$ to be the reductive quotient of the stabilizer of $\bar \x$ in $\bfG^0(F)$ ($\bfG^0$ is the first group in the chain $\vec \bfG$) and choose $\bbZ_{\bbG}$ to be the image of the center of $\bfG^0$. Our $p$-adic characterization theorems then require us to determine $E(\rho_0)$ from the character values of $\rho_0$ on a \textit{subset} $\bbG(\FF_q)_\evrs$ of the regular semisimple locus $\bbG(\FF_q)_\rs$ controlled by the notion of (elliptic) very regularity in $\bfG(F)$. This introduces two subtle points: first, the set $\bbG(\FF_q)_{\evrs}$ could be infinite, so one must be careful with the meaning of there being enough such elements; and second, because $\bbG(\FF_q)_{\evrs}$ is cut out by a regularity notion coming from $\bfG$ (not $\bfG^0$!), the locus $\bbG(\FF_q)_{\evrs}$ may not be stable under multiplication by $\bbZ_{\bbG}(\FF_q)$! To handle these issues, we must introduce another central subgroup (which we call $\bbZ_{\bbG}^\star$; see Sections \ref{subsec:char-finite-gen-pos} and \ref{subsec:ur-vreg} for details).

In light of this, we work in the following generalized setting: let $\bbG(\FF_q)_\bullet \subset \bbG(\FF_q)_\rs$ be a conjugation-invariant subset which is invariant under multiplication by a finite-index subgroup of $\bbZ_{\bbG}(\FF_q)$. 

\begin{displaytheorem}\label{thm:intro GG}
    \mbox{}
    \begin{enumerate}[label=(\alph*)]
        \item (Theorem \ref{thm:Henniart})
Assume that the following inequality $(\mathfrak{H}_\bullet)$ holds (see Section \ref{subsec:char-finite-gen-pos}):
\[
\frac{|[\bbS]^{\star}|}{|[\bbS]^{\star}_{\circ}|}
=
\frac{|[\bbS]^{\star}|}{|[\bbS]^{\star}\smallsetminus[\bbS]^{\star}_{\bullet}|}
>
2\cdot|W_{\bbG(\F_{q})}(\bbS)|.
\]
        If $\theta$ is a character in general position of a torus $\bbS(\FF_q)$, then there exists a unique  irreducible representation $\rho$ of $\bbG(\FF_q)$ such that for some constant $c \in \bbC^1$,
        \begin{equation*}
            \Theta_{\rho}(\gamma) = c \cdot \sum_{w \in W_{\bbG(\FF_q)}(\bbG_{\gamma}^\circ, \bbS^\circ)} \theta({}^w \gamma) \qquad \text{for all $\gamma \in \bbG(\FF_q)_\bullet$.}
        \end{equation*}
    
        \item (Theorem \ref{thm:unique-Z-rho})
Assume that the following inequality $(\mathfrak{L}_{\bullet})$ holds (see Section \ref{subsec:Lusztig-E}):
\[
\frac{|[\bbS]^{\star}|}{|[\bbS]^{\star}_{\circ}|}
=
\frac{|[\bbS]^{\star}|}{|[\bbS]^{\star} \smallsetminus [\bbS]^{\star}_{\bullet}|}
>
2^{2|W_{\bbG}|\cdot i(\bbS)-1}.
\]
        Then $E$ can be defined non-cohomologically via character values on $\bbG(\FF_q)_\bullet$.
    \end{enumerate}
\end{displaytheorem}

Here, the statement of (a) is a special case of (b) with a weaker assumption on the size of $\bbG(\FF_q)_\bullet$. We note that in the case that $\bullet = \rs$ and $\bbG = \bbG^\circ$, Theorem \ref{thm:intro GG}(b) is exactly Lusztig's result and is proved using Lusztig's techniques, after the representation theory of $\bbG$ is sufficiently developed. Theorem \ref{thm:intro GG}(a), on the other hand, is proved by completely different methods to Theorem \ref{thm:intro GG}(b). Furthermore, even in the case that $\bullet = \rs$ and $\bbG = \bbG^\circ$, this result is new: the required bound \eqref{ineq:Henniart-bullet} is much weaker than \eqref{ineq:Lusztig-bullet}. Taking Theorem \ref{thm:intro GG}(a) in the elliptic very regular setting $\bullet = \evrs$ is the ``depth zero'' input used to prove Theorem \ref{thm:intro regsc char}.

In Part \ref{part:formula}, we establish a simple formula for the character values at elliptic very regular elements of a supercuspidal representation $\pi_{\Psi}^{\FKS}$ associated to an arbitrary Yu-datum $\Psi$. We remind the reader that we use the twist of Yu's construction defined by Fintzen, Kaletha, and Spice \cite{FKS21}. 

In existing work on character formulae on supercuspidal representations, the aim is generally to provide, for as wide a class of supercuspidals as possible, a description of the character value at an arbitrary regular semisimple element.
In \cite{AS09} and \cite{DS18}, Adler, DeBacker, and Spice achieve this for Yu-data satisfying a certain compactness condition.
More recently, using a result of Spice \cite{Spi18, Spi21}, Fintzen, Kaletha, and Spice \cite{FKS21} give a character formula for regular supercuspidal representations, under the assumption that $F$ has characteristic zero and sufficiently large residual characteristic.
By contrast, our aim has swapped quantifiers: we wish to provide, for only elliptic very regular elements, character values for an arbitrary supercuspidal representation.

Our method is a slight generalization of that of \cite{Kal19} which is based on the work of Adler--DeBacker--Spice \cite{AS09,DS18} and used for establishing a character formula at ``shallow'' elements.
Recall that the representation $\pi^{\FKS}_{\Psi}$ is given by the compact induction $\cInd_{K}^{\bfG(F)}\rho^{\FKS}_{\Psi}$ of an irreducible representation $\rho^{\FKS}_{\Psi}$ of a certain open compact-mod-center subgroup $K$.
Via Harish-Chandra's integration formula, the computation of $\Theta_{\pi^{\FKS}_{\Psi}}(\gamma)$ for a given element $\gamma$ basically comes down to computing the index set of the integration formula (i.e., the locus where the integrand does not vanish) and the character $\Theta_{\rho^{\FKS}_{\Psi}}$ of $\rho^{\FKS}_{\Psi}$.
The point of Kaletha's argument is that the index set is drastically simplified if $\gamma$ is a \textit{shallow} element.
In this paper, we introduce the notion of \textit{very regularity} for semisimple elements, which simultaneously generalizes the notion of shallowness and notions of very regularity which have appeared in specialized contexts \cite{Hen92,Hen93,CI21-RT,CO21}.
Once we pin down the correct definition of the very regularity, we can obtain the character formula in the same manner as in \cite{Kal19}:

\begin{displaytheorem}[Theorem \ref{thm:CF}]\label{thm:intro character formula}
    For any elliptic very regular element $\gamma \in \bfG(F)$,
    \begin{equation}\label{eq:intro char formula}
        \Theta_{\pi_{\Psi}^{\FKS}}(\gamma) = c \cdot \sum_{g} \Theta_{\rho_0}({}^g \gamma) \cdot \Delta({}^g \gamma) \cdot \phi_{\geq 0}({}^g \gamma),
    \end{equation}
    where the sum ranges over a set of elements which only depends on depth zero data in $\Psi$ and $\phi_{\geq0}$ is the product of all characters in $\vec{\phi}$.
    In particular, if $\pi_\Psi^{\FKS} = \pi_{(j(\bfS),\theta_j)}$ for $(j,\theta_j)$ as in Theorem \ref{thm:L-packet char}, then \eqref{eq:intro char formula} simplifies to \eqref{eq:vreg character}.
\end{displaytheorem}

Examining character formulae such as \cite[Theorem 7.1]{AS09}, we see that the shape of the character formula of a supercuspidal representation at an arbitrary regular semisimple element greatly depends on the relationship between the ``genericity depths'' of the supercuspidal and the ``regularity depths'' of the semisimple element. From the perspective of elliptic very regular elements, we have the following ``sharpness'' observation: if $\gamma$ is a regular semisimple element which is not very regular, then there simultaneously exists a supercuspidal representation $\pi$ whose character value at $\gamma$ involves orbital integrals and also a supercuspidal representation $\pi'$ whose character value does not. The locus of very regular elements can perhaps be viewed as the largest subset of regular semisimple elements such that the supercuspidal character formula is uniform across all Yu-data. 

The simplicity of the formula \eqref{eq:intro char formula} is crucial for us. In Part \ref{part:characterization}, we use Theorem \ref{thm:intro character formula} to reduce the problem of characterizing supercuspidal representations to characterizing its depth zero parts, which then allows us to use the results of Part \ref{part:finite} to conclude. To this end, we introduce the notion of a \textit{clipped Yu-datum} $\dashover{\Psi}$, which is obtained from the Yu-datum $\Psi$ by simply excising the depth zero representation $\rho_0$ (Definition \ref{defn:clip}). We prove in Section \ref{sec:clipped p-adic} that if two Yu-data $\Psi, \Psi'$ are such that $\Theta_{\pi_\Psi^{\FKS}}(\gamma) = \Theta_{\pi_{\Psi'}^{\FKS}}(\gamma)$ for every elliptic very regular element $\gamma \in \bfG(F)$, then $\dashover{\Psi} = \dashover{\Psi}'$ and the character values of the depth zero parts $\rho_0, \rho_0'$ agree on very regular elements. In Section \ref{sec:sc characterization}, we use characterization results in the $\FF_q$ setting of Part \ref{part:finite} to obtain Theorems \ref{thm:intro regsc char}, \ref{thm:intro unipotent}, and the following most general form of our characterization theorem for supercuspidal representations of $\bfG(F)$:

\begin{displaytheorem}[Theorem \ref{thm:p-adic E}]\label{thm:intro characterization}
    Assume that there are sufficiently many very regular elements. Let $\Psi,\Psi'$ be any two Yu-data such that
    \begin{equation*}
        \pi_\Psi^{\FKS}(\gamma) = \pi_{\Psi'}^{\FKS}(\gamma) \qquad \text{for every very regular element $\gamma\in\bfG(F)$.}
    \end{equation*}
    Then $\dashover{\Psi} = \dashover{\Psi}'$ and $E(\rho_0) = E(\rho_0')$, where $\rho_0, \rho_0'$ the depth zero parts of $\Psi, \Psi'$.
\end{displaytheorem}

We remark that while Theorem \ref{thm:intro characterization} applies to arbitrary Yu-data, which obviously includes both the class of regular supercuspidal representations and the class of unipotent supercuspidal representations, this theorem does not subsume either Theorems \ref{thm:intro regsc char} nor \ref{thm:intro unipotent}. In the case of regular supercuspidal representations, this is because we may apply the depth zero characterization Theorem \ref{thm:intro GG}(a), which gives a far less restrictive bound on how many very regular elements constitutes ``sufficiently many.'' In the case of unipotent supercuspidal representations, we are able to guarantee that as long as $q$ is sufficiently large relative to the absolute rank of $\bfG$, then we have a characterization of the class of unipotent supercuspidal representations. That this assumption does \textit{not} imply that there are sufficiently many very regular elements in $\bfS(F)$ for an arbitrary tame elliptic maximal torus is a remarkably subtle point that we have so far not discussed in this introduction. 

\subsection{Failures}\label{subsec:intro failures}

The orthogonality relation of the elliptic inner product (\cite[Theorem 3]{Clo91}) implies that any irreducible supercuspidal (or even discrete series) representation is determined uniquely by its Harish-Chandra character on all elliptic regular semisimple elements. 
On the other hand, general character formulas on such elements are complicated and unwieldy. 
Hence, as explained at the start of this introduction, our goal in the present paper was to find a special domain of elliptic regular semisimple elements large enough to characterize supercuspidal representations and small enough so that character formulae could be easily recognizable.

In the $\GL_n$ setting, for example, Henniart's characterization for regular supercuspidal representations attached to the unramified elliptic maximal torus $\bfS$ has proven to be applicable in a wide range of contexts.
The prototypical setting for such an application is when one has two vastly different constructions of supercuspidal representations and wishes to understand the relationship between them.
Character formulae arising from different constructions are in general expected not to be comparable, which calls for characterization results like Henniart's result, allowing one to \textit{only} have to check character formulae on some very special locus.

We explained after Theorem \ref{thm:intro character formula} that there is a sense in which our consideration of very regular elements is sharp: the character formulae of two supercuspidals can look vastly different at any element outside this locus. But the ``sharpness'' of the domain of very regular elements is also the source of the failures of Theorems \ref{thm:L-packet char}, \ref{thm:intro regsc char}, and \ref{thm:intro characterization}.

An essential assumption for these characterization theorems is that there are \textit{sufficiently many very regular elements} in $\bfS(F)$, and the harsh reality is that there exist groups $\bfG$ with tame elliptic maximal tori $\bfS \subset \bfG$ containing \textit{no very regular elements}.
Furthermore, one need not work very hard to find such a $\bfG$ and such a $\bfS$: Let $\bfG = \SL_n$ and let $\bfS$ be \textit{any} tame elliptic maximal torus which is not unramified (see Section \ref{subsubsec:Henn-SL-PGL}). 
For $\GL_n$, it is often---but still not always!---the case that a given tame elliptic maximal torus will have sufficiently many very regular elements, but this is surprisingly subtle as well (see Section \ref{subsec:Henn-GL} for explicit criteria in this setting). 

It is not clear to us how to write down a necessary and sufficient criterion for $\bfS \subset \bfG$ that would guarantee that $\bfS(F)$ has enough very regular elements. Even when $\bfG = \GL_n$, such a criterion is not so immediate. We provide these calculations in Section \ref{subsec:Henn-GL}. 

For the $\GL_n$ case, it is worth noting that every tame supercuspidal representation is regular, and the finer Theorem \ref{thm:intro regsc char} applies here. Specializing Theorem \ref{thm:intro regsc char} to Henniart's two cases---when $\bfS$ is unramified, or when $n$ is prime and $\bfS$ is totally ramified---recovers Henniart's characterization results. We lament, on the other hand, that our efforts fail to extend Henniart's $\GL_n$ characterization results to the setting of arbitrary $\bfS$.

The potential failure of $\bfS(F)$ to have sufficiently many very regular elements is a subtlety related to the ramification of $\bfS$. This allows us to leave the reader with one redeeming fact: we at least have that if the size of the residue field of $F$ is sufficiently large, then any \textit{maximally unramified} elliptic maximal torus will have sufficiently many very regular elements (Lemma \ref{lem:strong Henniart sat max unram}). This makes possible the unipotent supercuspidal criterion Theorem \ref{thm:intro unipotent} for arbitrary $\bfG$.

\medbreak
\noindent{\bfseries Acknowledgments.}\quad
We heartily thank George Lusztig for communicating his proofs in the finite-field setting \cite{Lus20} after the first author gave a talk on Theorem \ref{thm:intro GG}(a) in the MIT Lie Groups seminar in September 2020. George's techniques completely generalize to the more general setting we require in the present paper,  yielding Theorem \ref{thm:intro GG}(b) and allowing us to then establish the supercuspidal analogues of these results.
We also thank Stephen DeBacker for asking us about obtaining a criterion for distinguishing unipotent supercuspidal representations. 
We are grateful to Jessica Fintzen and Tasho Kaletha for their helpful comments on an earlier draft of this paper.
The first author was partially supported by NSF grant DMS-2101837 and a Sloan Research Fellowship. 
The second author was partially supported by JSPS KAKENHI Grant Number 20K14287.

\newpage

\tableofcontents

\newpage

\section{Notations and assumptions}\label{sec:notations}

Let $F$ be a non-archimedean local field with finite residue field $\cO_F/\mfp_F \cong \FF_q$ of prime characteristic $p$, where we write $\mcO_{F}$ and $\mfp_{F}$ for the ring of its integers and the maximal ideal, respectively.
We let $F^{\ur}$ denote the maximal unramified extension of $F$.
We write $\Gamma_{F}$ for the absolute Galois group of $F$.
We fix a non-trivial additive character $\psi$ of $F$.

For an algebraic variety $\J$ over $F$, we denote the set of its $F$-valued points by $J$.
When $\J$ is an algebraic group, we write $\bfZ_{\J}$ for its center.

Let us assume that $\J$ is a connected reductive group over $F$.
We follow the notation around Bruhat--Tits theory used by \cite{AS08, AS09, DS18}.
(See, for example, \cite[Section 3.1]{AS08} for  details.)
Especially, $\mcB(\J,F)$ (resp.\ $\mcB^{\red}(\J,F)$) denotes the enlarged (resp.\ reduced) Bruhat--Tits building of $\J$ over $F$.
For a point $\x\in\mcB(\J,F)=\mcB^{\red}(\J,F)\times X_{\ast}(\bfZ_{\J})_{\R}$, we write $\bar{\x}$ for the image of $\x$ in $\mcB^{\red}(\J,F)$, and $J_{\bar{\x}}$ for the stabilizer of $\bar{\x}$ in $J$.
We define $\widetilde{\R}$ to be the set $\R\sqcup\{r+\mid r\in\R\}\sqcup\{\infty\}$ with a natural order.
Then for any $r\in\widetilde{\R}_{\geq0}$ we can consider the $r$-th Moy--Prasad filtration $J_{\x,r}$ of $J$ with respect to the point $\x$.
For any $r,s\in\widetilde{\R}_{\geq0}$ satisfying $r<s$, we write $J_{\x,r:s}$ for the quotient $J_{\x,r}/J_{\x,s}$.

Suppose that $\bfT$ is a tamely ramified maximal torus of $\J$.
By fixing a $T$-equivariant embedding of $\mcB(\bfT,F)$ into $\mcB(\J,F)$, we may regard $\mcB(\bfT,F)$ as a subset of $\mcB(\J,F)$.
Then, for any point $\x\in\mcB(\J,F)$, the property that ``$\x$ belongs to the image of $\mcB(\bfT,F)$'' does not depend on the choice of such an embedding (see the second paragraph of \cite[Section 3]{FKS21} for details).
For any point $\x\in\mcB(\J,F)$ which belongs to $\mcB(\bfT,F)$, we have $T_{\mathrm{b}}\subset G_{\x}$, where $T_{\mathrm{b}}$ denotes the maximal bounded subgroup of $T$.
When $\bfT$ is elliptic in $\J$, the image of $\mcB(\bfT,F)$ in $\mcB^{\red}(\J,F)$ consists of only one point.
When the image of a point $\x\in\mcB(\J,F)$ in $\mcB^{\red}(\J,F)$ coincides with this point (or, equivalently, $\x$ belongs to $\mcB(\bfT,F)$), we say that $\x$ is associated with $\bfT$.
Note that, in this case, we have $T \subset J_{\bar{\x}}$.

In this paper, for any element $x$ of a group $H$, we let ${}^{x}(-)$ or $(-)^{x}$ denote the conjugation by $y$.
For example, for any $x,y\in H$, we put ${}^{x}y\colonequals xyx^{-1}$ and ${y}^{x}\colonequals x^{-1}yx$; for any $x\in H$ and a subset $H'\subset H$, we put ${}^{x}H'\colonequals xH'x^{-1}$ and $H'^{x}\colonequals x^{-1}H'x$; for any representation $\rho$ of a subgroup $H'$ of $H$ and $x\in H$, we define a representation ${}^{x}\rho$ of ${}^{x}H'$ (resp.\ $\rho^{x}$ of $H'{}^{x}$) by ${}^{x}\rho({}^{x}y)\colonequals \rho(y)$ (resp.\ $\rho^{x}(y^{x})\colonequals \rho(y)$).
For any group $H$ and its subgroups $H_{1}$ and $H_{2}$, we put
\[
N_{H}(H_{1},H_{2})\colonequals \{n\in H \mid {}^{n}H_{1}\subset H_{2}\}.
\]
When $H_{1}=H_{2}$, we shortly write $N_{H}(H_{1})\colonequals N_{H}(H_{1},H_{2})$.

\medbreak
\noindent{\bfseries Assumptions on $F$ and $\bfG$.}\quad
Let $\bfG$ be a tamely ramified connected reductive group over $F$.
We always assume that $p$ does not divide the order of the absolute Weyl group of $\bfG$.
Note that this assumption implies that
\begin{itemize}
\item
$p$ is odd (as long as $\bfG$ is not a torus), 
\item
$p$ is not bad for $\bfG$ (in the sense of \cite[Definition A.5]{AS08}), and 
\item
$p \nmid |\pi_1(\bfG_{\der})|$ and $p \nmid |\pi_1(\widehat \bfG_{\der})|$.
\end{itemize}

\newpage

\part{Characters of Deligne--Lusztig representations}\label{part:finite}

\section{Deligne--Lusztig representations of certain disconnected groups}\label{sec:finite}

In this section, we study the representation theory of the class of algebraic groups $\bbG$ over $\FF_q$ which contain a finite-index subgroup which is connected reductive modulo its center (Section \ref{subsec:disconn}). We define an extended Jordan decomposition (Section \ref{subsec:extended_jordan}) to replace the classical notion of Jordan decomposition and use this to establish the representation theory of $\bbG(\FF_q)$ \`a la Deligne and Lusztig (Sections \ref{subsec:DLCF}, \ref{subsec:scalar product}, \ref{subsec:semisimple}).

\subsection{Disconnected setting}\label{subsec:disconn}

Let $\bbG$ be a smooth group scheme defined over $\F_q$ such that its identity component $\bbG^\circ$ is a connected reductive group.
We additionally assume that the group of connected components $\pi_0(\bbG)$ is a finitely generated abelian group and that there exists a central subgroup $\bbZ_\bbG$ of $\bbG$ such that the index $[\bbG : \bbG^\circ \bbZ_{\bbG}]$ is finite. Note that this is exactly the setting of \cite[Section 2.1]{Kal19-sc} (see Assumption 2.1.1 of \textit{op.\ cit.}).

We will see that the choice of $\bbZ_\bbG$ allows us to mirror the structure theory of representations of $\bbG^\circ(\FF_q)$ in this more general context. To this end, we put $\bbG' \colonequals \bbG^\circ \bbZ_\bbG$ and $\bbZ_{\bbG^\circ} \colonequals \bbZ_\bbG \cap \bbG^\circ$. For any maximal torus $\bbS^\circ$ of $\bbG^\circ$, we put $\bbS \colonequals \bbS^\circ \bbZ_\bbG \subset \bbG'$; note that $\bbG' = \bbG^\circ \bbS$.

\begin{rem}\label{rem:rational-pts}
Note that $\bbG'(\F_{q})=\bbG^{\circ}(\F_{q})\cdot\bbS(\F_{q})$.
Indeed, since $\bbG'$ is equal to $\bbG^{\circ}\cdot\bbS$ and the intersection $\bbG^{\circ}\cap\bbS\,(=\bbS^{\circ})$ is connected, Lang's theorem implies that the equality holds also $\F_{q}$-rationally.
On the other hand, $\bbG'(\F_{q})$ might not equal $\bbG^{\circ}(\F_{q})\cdot\bbZ_{\bbG}(\F_{q})$ since the intersection $\bbG^{\circ}\cap\bbZ_{\bbG}$ might not be connected in general.
\end{rem}

\begin{rem}\label{rem:examples}
The following examples are of particular interest to us.
\begin{enumerate}
\item
Let $\bbG$ be a connected reductive group over $\F_{q}$.
Then $\bbG^{\circ}=\bbG$ and $\bbZ_{\bbG}\colonequals \{1\}$ satisfy the above assumptions.
\item
Let $\bfG^{0}$ be a tamely ramified connected reductive group over $F$.
We take a point $\x$ of $\mcB(\bfG^{0},F)$ whose image $\bar{\x}$ in $\mcB^{\red}(\bfG^{0},F)$ is a vertex.
Then smooth group schemes $\bbG$ and $\bbZ_{\bbG}$ over $\F_{q}$ associated with the groups $G^{0}_{\bar{\x}}$ and $Z_{\bfG^{0}}$ satisfy the above assumptions (see Section \ref{subsec:sc classes})
\end{enumerate}
\end{rem}

We follow \cite[Section 2.6]{Kal19-sc}. Mirroring the connected case \cite{DL76}, define
\begin{equation*}
    Y_{\bbU}^{\bbG} \colonequals \{g \in \bbG/\bbU \mid g^{-1} \sigma(g) \in \bbU \cdot \sigma(\bbU)\},
\end{equation*}
where $\sigma$ denotes a Frobenius map $\bbG \to \bbG$ with respect to the $\FF_q$-rational structure on $\bbG$, and where $\bbU$ is the unipotent radical of a Borel $\overline \FF_q$-subgroup of $\bbG^\circ$ containing $\bbS^\circ$. Then $\bbG(\FF_q)$ and $\bbS(\FF_q)$ act on $Y_{\bbU}^{\bbG}$ by left and right multiplication, respectively.
For any character $\theta\colon\bbS(\FF_q)\rightarrow\C^{\times}$, choosing a prime number $\ell\neq p$, define the virtual $\bbG(\FF_q)$-representation
\begin{equation*}
    R_{\bbS}^{\bbG}(\theta) \colonequals \sum_i H_c^i(Y_{\bbU}^{\bbG}, \overline \QQ_\ell)_\theta,
\end{equation*}
where $H_c^i(Y_{\bbU}^{\bbG}, \overline \QQ_\ell)_{\theta}$ denotes the subspace wherein $\bbS(\FF_q)$ acts by multiplication by the character $\theta$, which is regarded as a $\overline{\Q}_{\ell}$-valued character by fixing an isomorphism $\C\cong\overline{\Q}_{\ell}$.

Let $\theta$ be a character of $\bbS(\FF_q)$ and  write $\theta^\circ$ for the restriction $\theta|_{\bbS^\circ(\FF_q)}$ of $\theta$ to $\bbS^\circ(\FF_q)$. Our analysis in this section is built upon the relationship between $R_{\bbS}^{\bbG}(\theta)$, $R_{\bbS}^{\bbG'}(\theta)$, and $R_{\bbS^\circ}^{\bbG^\circ}(\theta^\circ)$. The following lemma describes this algebraically and follows from \cite[Lemma 2.6.1, Remark 2.6.5]{Kal19-sc}.

\begin{lem}\label{lem:alg desc} \mbox{}
    \begin{enumerate}
        \item $R_{\bbS}^{\bbG}(\theta) \cong \Ind_{\bbG'(\FF_q)}^{\bbG(\FF_q)}(R_{\bbS}^{\bbG'}(\theta))$.
        \item $R_{\bbS}^{\bbG'}(\theta)$ is an extension of $R_{\bbS^\circ}^{\bbG^\circ}(\theta^\circ)$ to $\bbG'(\FF_q)$ on which $\bbZ_\bbG(\FF_q)$ acts via $\theta|_{\bbZ_{\bbG}(\FF_q)}$.
    \end{enumerate}
\end{lem}

Note that while the variety $Y_{\bbU}^{\bbG}$ depends on the choice of $\bbU$, the virtual representation $R_{\bbS}^{\bbG}(\theta)$ does not. This follows from connected case \cite[Theorem 6.8]{DL76} and Lemma \ref{lem:alg desc}.
We regard $R_{\bbS}^{\bbG}(\theta)$, $R_{\bbS}^{\bbG'}(\theta)$, and $R_{\bbS^{\circ}}^{\bbG^{\circ}}(\theta^{\circ})$ as complex virtual representations through the fixed isomorphism $\C\cong\overline{\Q}_{\ell}$; note that these representations do not depend on the choice of an isomorphism $\C\cong\overline{\Q}_{\ell}$ by \cite[Proposition 3.3]{DL76} and Lemma \ref{lem:alg desc}.
Pictorially, we have:
\begin{equation*}    
    \begin{tikzcd}
        \bbG^\circ \arrow[r, phantom, "\subset"] & \bbG' \colonequals \bbG^\circ \bbZ_{\bbG} \arrow[r, phantom, "\subset"] & \bbG &
        R_{\bbS^\circ}^{\bbG^\circ}(\theta^\circ) \ar[rightsquigarrow]{r}[above]{\text{extend}} & R_{\bbS}^{\bbG'}(\theta) \ar[rightsquigarrow]{r}[above]{\Ind} & R_{\bbS}^{\bbG}(\theta) \\
        \bbS^\circ \arrow[u, phantom, sloped, "\subset"] \arrow[r, phantom, "\subset"] & \bbS \colonequals \bbS^\circ \bbZ_{\bbG} \arrow[u, phantom, sloped, "\subset"] &&
        \theta^\circ \ar[rightsquigarrow]{u}[left]{Y_{\bbU}^{\bbG^\circ}} & \theta \arrow[rightsquigarrow]{l}[above]{\text{restrict}} \arrow[rightsquigarrow]{u}[left]{Y_{\bbU}^{\bbG'}} \arrow[rightsquigarrow]{ur}[right]{\quad Y_{\bbU}^{\bbG}}
    \end{tikzcd}
\end{equation*}

We will work out several foundational results about the representations of $\bbG(\FF_q)$, guided largely by the techniques of \cite{DL76} (especially in Section \ref{sec:finite}) and \cite{Lus20} (especially in Section \ref{subsec:Lusztig-E}). For the most part, we will not work with $R_{\bbS}^{\bbG}(\FF_q)$ geometrically, and instead use Lemma \ref{lem:alg desc} to bootstrap what we need from the geometry of the connected case. 

We now collect some notation here that we will frequently use.

\begin{notation}\label{not:finite}
    \mbox{}
    \begin{enumerate}[label=(\alph*)]
        \item Set $[\bbG^\circ] \colonequals \bbG^\circ(\FF_q)/\bbZ_{\bbG^\circ}(\FF_q)$, $[\bbG'] \colonequals \bbG'(\FF_q)/\bbZ_{\bbG}(\FF_q)$, $[\bbG] \colonequals \bbG(\FF_q)/\bbZ_{\bbG}(\FF_q)$. We have then a sequence of finite groups 
        \begin{equation*}
            [\bbG^\circ] \subset [\bbG'] \subset [\bbG].
        \end{equation*} 
        For any $\FF_q$-rational maximal torus $\bbS^\circ$ of $\bbG^\circ$, set $[\bbS] \colonequals \bbS(\FF_q)/\bbZ_{\bbG}(\FF_q)$.
        \item For an $\FF_q$-rational maximal torus $\bbS^\circ$ of $\bbG^\circ$, let $\bbS(\FF_q)^\wedge$ denote the set of all characters of $\bbS(\FF_q)$. For any character $\omega$ of $\bbZ_{\bbG}(\FF_q)$, let $\bbS(\FF_q)^\wedge_\omega$ denote the subset of $\bbS(\FF_q)^\wedge$ consisting of $\theta$ for which $\theta|_{\bbZ_{\bbG}(\FF_q)} = \omega$. 
%        Let $\bbS(\FF_q)^\wedge_\unitary$ denote the union of all $\bbS(\FF_q)^\wedge_\omega$ such that $\omega$ is unitary.
        \item Let $\tilde \cT$ denote the set of pairs $(\bbS, \theta)$ where $\bbS=\bbS^{\circ}\bbZ_{\bbG}$ for an $\FF_q$-rational maximal torus $\bbS^{\circ}$ of $\bbG^\circ$ and $\theta \in \bbS(\FF_q)^\wedge$. Let $\tilde \cT_\omega$ denote the subset of $\tilde \cT$ consisting of $(\bbS, \theta)$ where $\theta \in \bbS(\FF_q)^\wedge_\omega$. 
%        Let $\tilde \cT_\unitary$ denote the union of all $\tilde \cT_\omega$ such that $\omega$ is unitary.
        \item Let $\Irr(\bbG(\FF_q))$ denote the set of isomorphism classes of irreducible representations of $\bbG(\FF_q)$. Given a character $\omega$ of $\bbZ_{\bbG}(\FF_q)$, let $\Irr(\bbG(\FF_q))_\omega$ denote the subset consisting of irreducible representations wherein $\bbZ_{\bbG}(\FF_q)$ acts by $\omega$. 
%        Let $\Irr(\bbG(\FF_q))_\unitary$ denote the union of all $\Irr(\bbG(\FF_q))_\omega$ such that $\omega$ is unitary.
        \item Let $\cR(\bbG(\FF_q)) = \bbZ[\Irr(\bbG(\FF_q))]$, and analogously $\cR(\bbG(\FF_q))_\omega = \bbZ[\Irr(\bbG(\FF_q))_\omega]$.
%         and $\cR(\bbG(\FF_q))_\unitary = \bbZ[\Irr(\bbG(\FF_q))_\unitary]$.
        When $\rho\in\cR(\bbG(\FF_q))$ belongs to $\cR(\bbG(\FF_q))_\omega$, we say that ``$\omega$ is the $\bbZ_{\bbG}$-central character of $\rho$''.
        \item For a unitary character $\omega$ of $\bbZ_{\bbG}(\FF_q)$,  let $\bbC[\bbG(\FF_q)]_\omega$ denote the set of all functions $f \from \bbG(\FF_q) \to \bbC$ such that $f(gz) = \omega(z) \cdot f(g)$ for all $g \in \bbG(\FF_q)$ and $z \in \bbZ_{\bbG}(\FF_q)$. We may endow $\bbC[\bbG(\FF_q)]_\omega$ with an inner product: For any two functions $f_1, f_2 \in \bbC[\bbG(\FF_q)]_\omega$, we put
        \begin{equation*}
            \langle f_1, f_2 \rangle \colonequals \frac{1}{|[\bbG]|} \sum_{g\in[\bbG]}f_1(g) \cdot \overline{f_2(g)}.
        \end{equation*}
        Note that this is well-defined since the summand $f_1(g) \cdot \overline{f_2(g)}$ is independent of the choice of a representative of the $\bbZ_{\bbG}(\FF_q)$-coset $g$ since $f_1, f_2 \in \bbC[\bbG(\FF_q)]_\omega$ and $\omega$ is unitary.

        For any two representations $\rho_1, \rho_2 \in \cR(\bbG(\FF_q))_\omega$, we set 
        \begin{equation*}
            \langle \rho_1, \rho_2 \rangle \colonequals \langle \Theta_{\rho_1}, \Theta_{\rho_2} \rangle,
        \end{equation*}
        where $\Theta_{\rho_i}$ denotes the character of $\rho_i$.
        We define this inner product also in the case where $\omega$ is not necessarily unitary by using the following Lemma \ref{lem:twist}; by choosing a character $\chi\colon\bbG(\F_{q})\rightarrow\C^{\times}$ such that $\omega\otimes(\chi|_{\bbZ_{\bbG}(\F_{q})})$ is unitary, we put $\langle \rho_1, \rho_2 \rangle \colonequals \langle \rho_{1}\otimes\chi,\rho_{2}\otimes\chi\rangle$ (the right-hand side does not depend on the choice of $\chi$).
        Furthermore, we extend this to an inner product on $\cR(\bbG(\FF_q))$ linearly: for distinct characters $\omega_1, \omega_2$ and any virtual representations $\rho_1 \in \cR(\bbG(\FF_q))_{\omega_1}, \rho_2 \in \cR(\bbG(\FF_q))_{\omega_2}$, we set $\langle \rho_1, \rho_2 \rangle = 0$.
        Later, we will also need a truncated inner product; we define this in Section \ref{subsec:scalar product}.
    \end{enumerate}
\end{notation}

\begin{lem}\label{lem:twist}
Let $\omega$ be a character of $\bbZ_{\bbG}(\F_{q})$.
Then there exists a character $\chi\colon\bbG(\F_{q})\rightarrow\C^{\times}$ such that $\omega\otimes(\chi|_{\bbZ_{\bbG}(\F_{q})})$ is unitary.
In particular, for any irreducible representation $\rho$ of $\bbG(\F_{q})$, there exists a character $\chi\colon\bbG(\F_{q})\rightarrow\C^{\times}$ such that $\rho\otimes\chi$ has a unitary $\bbZ_{\bbG}$-central character.
\end{lem}

\begin{proof}
Recall that, by assumptions made in the beginning of Section \ref{subsec:disconn},
\begin{enumerate}
\item
the group of connected components $\pi_{0}(\bbG)$ is a finitely generated abelian group, and
\item
the index $[\bbG: \bbG']$ is finite.
\end{enumerate}
By (2), $\bbZ_{\bbG}/\bbZ_{\bbG^{\circ}}$ is a finite index subgroup of $\bbG/\bbG^{\circ}=\pi_{0}(\bbG)$.
Hence (1) implies that $\bbZ_{\bbG}/\bbZ_{\bbG^{\circ}}$ is a finitely generated abelian group, hence so is $\bbZ_{\bbG}(\F_{q})/\bbZ_{\bbG^{\circ}}(\F_{q})$.
Thus $\bbZ_{\bbG}(\F_{q})$ is also a finitely generated abelian group since $\bbZ_{\bbG^{\circ}}(\F_{q})$ is a finite group.
Then it can be easily checked that there exists a character $\chi$ of $\bbZ_{\bbG}(\F_{q})/\bbZ_{\bbG^{\circ}}(\F_{q})$ such that $\omega\otimes\chi$ is unitary.
Let us take such a character $\chi$.
%Proof: Since $\bbZ_{\bbG}(\F_{q})$ is a finitely generated abelian group, we may write $\bbZ_{\bbG}(\F_{q})\cong(\mathrm{torsion})\oplus\Z^{\oplus r}$. Then $\bbZ_{\bbG^{\circ}}(\F_{q})$ must be contained in the torsion part as it is finite. In particular, $\bbZ_{\bbG}(\F_{q})/\bbZ_{\bbG^{\circ}}(\F_{q})$ has $\Z^{\oplus r}$ as its quotient. Thus it is enough to take $\chi$ to be the inverse of the free part of $\omega$.
Since $\bbG(\F_{q})/\bbG^{\circ}(\F_{q})\subset (\bbG/\bbG^{\circ})(\F_{q})$, we also see that $\bbG(\F_{q})/\bbG^{\circ}(\F_{q})$ is a finitely generated abelian group.
Thus we can extend $\chi$ from $\bbZ_{\bbG}(\F_{q})/\bbZ_{\bbG^{\circ}}(\F_{q})$ to $\bbG(\F_{q})/\bbG^{\circ}(\F_{q})$.
%The elementary divisor theory is useful for justifying this.
If we again write $\chi$ for this extended character, then $\chi$ satisfies the desired condition.
\end{proof}

\subsection{Extended Jordan decomposition}\label{subsec:extended_jordan}

We introduce a variant of the Jordan decomposition in the disconnected group $\bbG'$ following \cite{Kal19}.
(We note that we work in a more general set-up than in \cite{Kal19}, but the arguments of \textit{op.\ cit.}\ we mention here extend to our more general setting without issue.)
This will be necessary to state the character formula for the representation $R_{\bbS}^{\bbG}(\theta)$.

\begin{defn}[extended Jordan decomposition]
Let $g'$ be an element of $\bbG'(\F_{q})$.
An \textit{extended Jordan decomposition of $g'$} is a tuple 
\[
(g,t,\dot{t},z,s,u)
\] consisting of the following objects:
\begin{itemize}
\item
$g\in\bbG^{\circ}(\F_{q})$ and $t\in\bbS(\F_{q})$ are elements such that $g'=gt$;
\item
$\dot{t}\in\bbS^{\circ}$ and $z\in\bbZ_{\bbG}$ are elements such that $t=\dot{t}z$;
\item
$g'z^{-1} = g\dot{t}=su$ is the Jordan decomposition in $\bbG^{\circ}$, i.e., $s\in\bbG^{\circ}$ is a semisimple element and $u\in\bbG^{\circ}$ is a unipotent element such that $g\dot{t}=su=us$.
\end{itemize}
\end{defn}

\begin{lem}\label{lem:Jordan-exist}
For any $g'\in\bbG'(\F_{q})$, we can always find an extended Jordan decomposition of $g'$.
\end{lem}

\begin{proof}
This is explained in the paragraph before \cite[Proposition 3.4.24]{Kal19}.
For the sake of completeness, we explain it.
Since we have $\bbG'(\F_{q})=\bbG^{\circ}(\F_{q})\cdot\bbS(\F_{q})$ (see Remark \ref{rem:rational-pts}), we may find $g\in\bbG^{\circ}(\F_{q})$ and $t\in\bbS(\F_{q})$ such that $g'=gt$.
The existence of $\dot{t}$ and $z$ simply follows from that $\bbS=\bbS^{\circ}\cdot\bbZ_{\bbG}$ (note that this equality does not necessarily hold $\F_{q}$-rationally, thus $\dot{t}$ and $z$ are not necessarily $\F_{q}$-rational).
By taking the (usual) Jordan decomposition $g\dot{t}=su$ of $g\dot{t}\in\bbG^{\circ}$, we obtain an extended Jordan decomposition $(g,t,\dot{t},z,s,u)$ of $g'$.
\end{proof}

We put $\overline{\bbG}\colonequals \bbG'/\bbZ_{\bbG}\cong\bbG^{\circ}/\bbZ_{\bbG^{\circ}}$ and $\overline{\bbS}\colonequals \bbS/\bbZ_{\bbG}\cong\bbS^{\circ}/\bbZ_{\bbG^{\circ}}$.

\begin{lem}\label{lem:Jordan}
Let $g'\in\bbG'(\F_{q})$ and $(g,t,\dot{t},z,s,u)$ an extended Jordan decomposition of $g'$.
\begin{enumerate}
\item
The image of $s$ in $\overline{\bbG}$ is $\F_{q}$-rational.
In particular, the connected centralizer $\bbG_{s}^{\circ}$ of $s$ in $\bbG^{\circ}$ has an $\F_{q}$-rational structure.
\item
The unipotent part $u$ belongs to $\bbG_{s}^{\circ}(\F_{q})$.
\item
For any $x\in\bbG(\F_{q})$, the elements ${}^{x}s\cdot z$ and ${}^{x}s\cdot\dot{t}^{-1}$ are $\F_{q}$-rational.
\end{enumerate}
\end{lem}

\begin{proof}
See the paragraph before \cite[Proposition 3.4.24]{Kal19} for the proofs of (1) and (2).
By (1), there exists an element $z'\in \bbZ_{\bbG^{\circ}}$ such that $\Frob(s)=z's$.
As $su=g\dot{t}$ and $g$ is $\F_{q}$-rational, we also have
\[
g\Frob(\dot{t})
=\Frob(g\dot{t})
=\Frob(su)
=z'su
=z'g\dot{t}
\]
(in the third equality, we used (2)).
Hence we get $\Frob(\dot{t})=z'\dot{t}$.
Thus we can see that, for any $x\in\bbG(\F_{q})$, ${}^{x}s\cdot\dot{t}^{-1}$ is fixed by the Frobenius.
We can check that ${}^{x}s\cdot z$ is fixed by the Frobenius in a similar way.
\end{proof}

\begin{lem}\label{lem:Jordan-bij}
The association $g'\mapsto (s,u)$ defines an injective map
\[
[\bbG']
\hookrightarrow
\{(s,u)\mid s\in\overline{\bbG}(\F_{q})_{\ss},\, u\in\bbG^{\circ}_{s}(\F_{q})_{\unip}\},
\]
where
\begin{itemize}
\item
$s$ (resp.\ $u$) is the semisimple (resp.\ unipotent) part of an(y) extended Jordan decomposition of $g'$,
\item
$\overline{\bbG}(\F_{q})_{\ss}$ denotes the set of semisimple (in the usual sense; note that $\overline{\bbG}$ is an algebraic group by the assumption on $\bbZ_{\bbG}$) elements of $\overline{\bbG}(\F_{q})$, and
\item
$\bbG^{\circ}_{s}(\F_{q})_{\unip}$ denotes the set of unipotent elements of $\bbG_{s}^{\circ}(\F_{q})$.
\end{itemize}
\end{lem}

\begin{proof}
We first check that the association $g'\mapsto (s,u)$ gives a well-defined map
\[
\bbG'(\F_{q})
\rightarrow
\{(s,u)\mid s\in\overline{\bbG}(\F_{q})_{\ss},\, u\in\bbG^{\circ}_{s}(\F_{q})_{\unip}\}.
\]
Let $(g,t,\dot{t},z,s,u)$ be an extended Jordan decomposition of $g'$.
By Lemma \ref{lem:Jordan} (1) and (2), $s$ belongs to $\overline{\bbG}(\F_{q})_{\ss}$ and $u$ belongs to $\bbG_{s}^{\circ}(\F_{q})_{\unip}$.
Thus our task is to show that $(s,u)$ is independent of the choice of an extended Jordan decomposition.
Let us take another extended Jordan decomposition $(\ul{g},\ul{t},\ul{\dot{t}},\ul{z},\ul{s},\ul{u})$ of $g'$.
Then
\[
su
=g\dot{t}
=gtz^{-1}
=g'z^{-1}
=\ul{g}\ul{t}z^{-1}
=\ul{g}\ul{\dot{t}}\ul{z}z^{-1}
=\ul{s}\ul{u}\ul{z}z^{-1}.
\]
Thus the uniqueness of the (usual) Jordan decomposition implies that $s=\ul{s}\ul{z}z^{-1}$ and $u=\ul{u}$.

We next investigate the fibers of the map.
Let $g'_{1}$ and $g'_{2}$ be elements of $\bbG'(\F_{q})$ with extended Jordan decomposition $(g_{1},t_{1},\dot{t}_{1},z_{1},s_{1},u_{1})$ and $(g_{2},t_{2},\dot{t}_{2},z_{2},s_{2},u_{2})$, respectively, such that $s_{1}=s_{2}z$ (with $z\in\bbZ_{\bbG^{\circ}}$) and $u_{1}=u_{2}$.
Then we have
\[
g_{1}'
=g_{1}t_{1}
=g_{1}\dot{t}_{1}z_{1}
=s_{1}u_{1}z_{1}
=s_{2}u_{2}zz_{1}
=g_{2}\dot{t}_{2}zz_{1}
=g_{2}t_{2}zz_{1}z_{2}^{-1}
=g'_{2}zz_{1}z_{2}^{-1}.
\]
As $g'_{1}$ and $g'_{2}$ are $\F_{q}$-rational, so is $zz_{1}z_{2}^{-1}$.
This completes the proof.
\end{proof}

We write $\Jord$ for the map of Lemma \ref{lem:Jordan-bij}:
\[
\Jord\colon
[\bbG']
\hookrightarrow
\{(s,u)\mid s\in\overline{\bbG}(\F_{q})_{\ss},\, u\in\bbG^{\circ}_{s}(\F_{q})_{\unip}\}.
\]
Note that this map might not be surjective.
We define a subset $\overline{\bbG}(\F_{q})_{\ast}\subset\overline{\bbG}(\F_{q})_{\ss}$ to be the image of the map
\[
[\bbG']
\hookrightarrow
\{(s,u)\mid s\in\overline{\bbG}(\F_{q})_{\ss},\, u\in\bbG^{\circ}_{s}(\F_{q})_{\unip}\}
\xrightarrow{\pr_{\ss}}
\overline{\bbG}(\F_{q})_{\ss},
\]
where $\pr_{\ss}$ denotes the first projection.
Then, $\Jord$ gives a bijection between the sets $[\bbG']$ and $\{(s,u)\mid s\in\overline{\bbG}(\F_{q})_{\ast},\, u\in\bbG^{\circ}_{s}(\F_{q})_{\unip}\}$:
\[
\Jord\colon
[\bbG']
\xrightarrow{1:1}
\{(s,u)\mid s\in\overline{\bbG}(\F_{q})_{\ast},\, u\in\bbG^{\circ}_{s}(\F_{q})_{\unip}\}
;\quad g'\mapsto (s,u).
\]

We sometimes abuse notation and write $\Jord$ to also mean the composition of the natural surjection $\bbG'(\FF_q) \to [\bbG']$ with $\Jord.$

\begin{lem}\label{lem:Jordan-conj}
The map $\Jord$ is $\bbG(\F_{q})$-conjugation equivariant.
More precisely, if we have $\Jord(g')=(s,u)$ for $g'\in\bbG'(\F_{q})$, then we have $\Jord({}^{y}g')=({}^{y}s,{}^{y}u)$ for any $y\in\bbG(\F_{q})$.
\end{lem}

\begin{proof}
We take an extended Jordan decomposition $(g,t,\dot{t},z,s,u)$ of $g'$.
For any $y\in\bbG(\F_{q})$, we take an extended Jordan decomposition $(g_{y},t_{y},\dot{t}_{y},z_{y},s_{y},u_{y})$ of ${}^{y}g'$.
By the definition of an extended Jordan decomposition, we have
\[
s_{y}u_{y}
=g_{y}\dot{t}_{y}
=g_{y}t_{y}z_{y}^{-1}
={}^{y}g'z_{y}^{-1}
={}^{y}(gt)z_{y}^{-1}
={}^{y}(g\dot{t}z)z_{y}^{-1}
={}^{y}(suz)z_{y}^{-1}
={}^{y}(sz)z_{y}^{-1}\cdot{}^{y}u.
\]
Hence the uniqueness of the Jordan decomposition implies that $s_{y}z_{y}={}^{y}(sz)$ and $u_{y}={}^{y}u$.
Thus we get the assertion.
\end{proof}

\subsection{Character formula for Deligne--Lusztig induction}\label{subsec:DLCF}

Fix $(\bbS, \theta) \in \tilde \cT$. 
We first recall Kaletha's character formula for the representation $R_{\bbS}^{\bbG'}(\theta)$, which generalizes the Deligne--Lusztig character formula in the connected setting (\cite[Theorem 4.2]{DL76}).

\begin{prop}[{\cite[Proposition 3.4.24]{Kal19}}]\label{prop:Kaletha-CF}
    For $g'\in\bbG'(\F_{q})$ and any extended Jordan decomposition $(g,t,\dot{t},z,s,u)$ for $g'$, we have
\[
\Theta_{R_{\bbS}^{\bbG'}(\theta)}(g')
=\frac{|[\bbG^{\circ}]|}{|[\bbG']|}
\cdot\frac{|\bbZ_{\bbG^{\circ}}(\F_{q})|}{|\bbG^{\circ}_{s}(\F_{q})|}
\sum_{\begin{subarray}{c}x\in [\bbG'] \\ {}^{x}s\in\bbS^{\circ}\end{subarray}}
Q^{\bbG^{\circ}_{s}}_{\bbS^{x\circ}}(u)\cdot \theta({}^{x}sz),
\]
where $Q^{\bbG^{\circ}_{s}}_{\bbS^{x\circ}}$ is the Green function with respect to an $\F_{q}$-rational maximal torus $\bbS^{x\circ}$ of $\bbG^{\circ}_{s}$ (\cite[Definition 4.1]{DL76}).
Note that ${}^{x}sz$ is an $\F_{q}$-rational element of $\bbS$ by Lemma \ref{lem:Jordan} (3), hence $\theta({}^{x}sz)$ in the summand makes sense.
\end{prop}

\begin{proof}
The statement of \cite[Proposition 3.4.24]{Kal19} is that
\begin{align*}
\Theta_{R_{\bbS}^{\bbG'}(\theta)}(g')
&=
\frac{1}{|\bbG^{\circ}_{s}(\F_{q})|}
\sum_{\begin{subarray}{c}x\in \bbG^{\circ}(\F_{q}) \\ {}^{x}s\in\bbS^{\circ}\end{subarray}}
Q^{\bbG^{\circ}_{s}}_{\bbS^{x\circ}}(u)\cdot \theta^{\circ}({}^{x}s\cdot\dot{t}^{-1})\theta(t)\\
&=
\frac{1}{|\bbG^{\circ}_{s}(\F_{q})|}
\sum_{\begin{subarray}{c}x\in \bbG^{\circ}(\F_{q}) \\ {}^{x}s\in\bbS^{\circ}\end{subarray}}
Q^{\bbG^{\circ}_{s}}_{\bbS^{x\circ}}(u)\cdot \theta({}^{x}sz).
\end{align*}
Note that, although the character $\theta^{\circ}$ is supposed to be regular in \cite{Kal19}, Kaletha's proof of \cite[Proposition 3.4.24]{Kal19} works in general.
Only the point to care is that the Deligne--Lusztig virtual representation is concentrated in a single degree if $\theta^{\circ}$ is regular.
Thus, in \cite[Proposition 3.4.24]{Kal19}, Kaletha works with a genuine representation obtained by twisting via a sign ``$(-1)^{r_{G}-r_{S}}$'' coming from the degree.
Since we work with the virtual representation $R^{\bbG'}_{\bbS}(\theta)$ itself, such a sign does not appear in the above formula.

For any $x\in\bbG'(\F_{q})$, ${}^{x}s\in\bbS^{\circ}$ if and only if ${}^{yx}s\in\bbS^{\circ}$ for any $y\in\bbS(\F_{q})$.
Moreover, in this case, we have ${}^{yx}s={}^{x}s$ and $\bbS^{x\circ}=\bbS^{yx\circ}$.
By combining this trivial observation with the equality $\bbG'(\F_{q})=\bbS(\F_{q})\bbG^{\circ}(\F_{q})$, we see that
\begin{align*}
\frac{1}{|\bbG^{\circ}_{s}(\F_{q})|}
\sum_{\begin{subarray}{c}x\in \bbG^{\circ}(\F_{q}) \\ {}^{x}s\in\bbS^{\circ}\end{subarray}}
Q^{\bbG^{\circ}_{s}}_{\bbS^{x\circ}}(u)\cdot \theta({}^{x}sz)
&=\frac{|\bbZ_{\bbG^{\circ}}(\F_{q})|}{|\bbG^{\circ}_{s}(\F_{q})|}
\sum_{\begin{subarray}{c}x\in [\bbG^{\circ}] \\ {}^{x}s\in\bbS^{\circ}\end{subarray}}
Q^{\bbG^{\circ}_{s}}_{\bbS^{x\circ}}(u)\cdot \theta({}^{x}sz)\\
&=\frac{|\bbZ_{\bbG^{\circ}}(\F_{q})|}{|\bbG^{\circ}_{s}(\F_{q})|}
\cdot\frac{|[\bbG^{\circ}]|}{|[\bbG']|}
\sum_{\begin{subarray}{c}x\in [\bbG'] \\ {}^{x}s\in\bbS^{\circ}\end{subarray}}
Q^{\bbG^{\circ}_{s}}_{\bbS^{x\circ}}(u)\cdot \theta({}^{x}sz).\qedhere
\end{align*}
\end{proof}

\begin{cor}\label{cor:Kaletha-CF}
For $g\in\bbG(\F_{q})\smallsetminus\bbG'(\F_{q})$, we have $\Theta_{R_{\bbS}^{\bbG}(\theta)}(g)=0$.
For $g'\in\bbG'(\F_{q})$ and any extended Jordan decomposition $(g,t,\dot{t},z,s,u)$ of $g'$, we have
\[
\Theta_{R_{\bbS}^{\bbG}(\theta)}(g')
=
\frac{|[\bbG^{\circ}]|}{|[\bbG']|}
\cdot\frac{|\bbZ_{\bbG^{\circ}}(\F_{q})|}{|\bbG^{\circ}_{s}(\F_{q})|}
\sum_{\begin{subarray}{c}x\in [\bbG] \\ {}^{x}s\in\bbS^{\circ}\end{subarray}}
Q^{\bbG^{\circ}_{s}}_{\bbS^{x\circ}}(u)\cdot \theta({}^{x}sz).
\]
\end{cor}

\begin{proof}
Since $R_{\bbS}^{\bbG}(\theta)=\Ind_{\bbG'(\F_{q})}^{\bbG(\F_{q})}(R_{\bbS}^{\bbG'}(\theta))$ and $\bbG'(\F_{q})$ is a normal subgroup of $\bbG(\F_{q})$, we have 
\[
\Theta_{R_{\bbS}^{\bbG}(\theta)}(g')
=
\sum_{y\in \bbG'(\F_{q})\backslash\bbG(\F_{q})}\dot{\Theta}_{R_{\bbS}^{\bbG'}(\theta)}({}^{y}g')
\]
by the Frobenius formula, where $\dot{\Theta}_{R_{\bbS}^{\bbG'}(\theta)}$ is the zero extension of $\Theta_{R_{\bbS}^{\bbG'}(\theta)}$ from $\bbG'(\F_{q})$ to $\bbG(\F_{q})$.
(Note that the index set is finite by the assumption that $[\bbG:\bbG']$ is finite.)
Thus we get the first assertion.

To check the second assertion, we assume that $g'$ belongs to $\bbG'(\F_{q})$.
We take an extended Jordan decomposition $(g,t,\dot{t},z,s,u)$ of $g'$.
For any $y\in\bbG(\F_{q})$, we take an extended Jordan decomposition $(g_{y},t_{y},\dot{t}_{y},z_{y},s_{y},u_{y})$ of ${}^{y}g'$.
Then, by Proposition \ref{prop:Kaletha-CF}, we have
\[
\dot{\Theta}_{R_{\bbS}^{\bbG'}(\theta)}({}^{y}g')
=
\Theta_{R_{\bbS}^{\bbG'}(\theta)}({}^{y}g')
=
\frac{|[\bbG^{\circ}]|}{|[\bbG']|}
\cdot\frac{|\bbZ_{\bbG^{\circ}}(\F_{q})|}{|\bbG^{\circ}_{s_{y}}(\F_{q})|}
\sum_{\begin{subarray}{c}x\in [\bbG'] \\ {}^{x}s_{y}\in\bbS^{\circ}\end{subarray}}
Q^{\bbG^{\circ}_{s_{y}}}_{\bbS^{x\circ}}(u_{y})\cdot \theta({}^{x}s_{y}z_{y}).
\]
By Lemma \ref{lem:Jordan-conj}, we have ${}^{y}s=s_{y}$ in $\overline{\bbG}$ and ${}^{y}u=u_{y}$.
In particular,
\begin{itemize}
\item
the condition ${}^{x}s_{y}\in\bbS^{\circ}$ is equivalent to the condition that ${}^{xy}s\in\bbS^{\circ}$,
\item
$\bbG_{s_{y}}^{\circ}$ is equal to $\bbG_{{}^{y}s}^{\circ}={}^{y}\bbG_{s}^{\circ}$, and
\item
$Q^{\bbG^{\circ}_{s_{y}}}_{\bbS^{x\circ}}(u_{y})=Q^{{}^{y}\bbG^{\circ}_{s}}_{\bbS^{x\circ}}({}^{y}u)=Q^{\bbG^{\circ}_{s}}_{\bbS^{xy\circ}}(u)$.
\end{itemize}
Therefore, by also noting that $s_{y}z_{y}={}^{y}sz$ (see the proof of Lemma \ref{lem:Jordan-conj}), we have
\begin{align*}
\sum_{y\in \bbG'(\F_{q})\backslash\bbG(\F_{q})}\dot{\Theta}_{R_{\bbS}^{\bbG'}(\theta)}({}^{y}g')
&=
\sum_{y\in \bbG'(\F_{q})\backslash\bbG(\F_{q})}
\frac{|[\bbG^{\circ}]|}{|[\bbG']|}
\cdot\frac{|\bbZ_{\bbG^{\circ}}(\F_{q})|}{|\bbG^{\circ}_{s_{y}}(\F_{q})|}
\sum_{\begin{subarray}{c}x\in [\bbG'] \\ {}^{x}s_{y}\in\bbS^{\circ}\end{subarray}}
Q^{\bbG^{\circ}_{s_{y}}}_{\bbS^{x\circ}}(u_{y})\cdot \theta({}^{x}s_{y}z_{y})\\
&=
\frac{|[\bbG^{\circ}]|}{|[\bbG']|}
\cdot\frac{|\bbZ_{\bbG^{\circ}}(\F_{q})|}{|\bbG^{\circ}_{s}(\F_{q})|}
\sum_{y\in \bbG'(\F_{q})\backslash\bbG(\F_{q})}
\sum_{\begin{subarray}{c}x\in [\bbG'] \\ {}^{xy}s\in\bbS^{\circ}\end{subarray}}
Q^{\bbG^{\circ}_{s}}_{\bbS^{xy\circ}}(u)\cdot \theta({}^{xy}sz)\\
&=
\frac{|[\bbG^{\circ}]|}{|[\bbG']|}
\cdot\frac{|\bbZ_{\bbG^{\circ}}(\F_{q})|}{|\bbG^{\circ}_{s}(\F_{q})|}
\sum_{\begin{subarray}{c}x\in [\bbG] \\ {}^{x}s\in\bbS^{\circ}\end{subarray}}
Q^{\bbG^{\circ}_{s}}_{\bbS^{x\circ}}(u)\cdot \theta({}^{x}sz).\qedhere
\end{align*}
\end{proof}

Corollary \ref{cor:Kaletha-CF} simplifies further in the special case that $g' \in \bbG'(\FF_q)$ is semisimple or regular semisimple in the sense of the following definition.

\begin{definition}\label{def:G' semisimple}
Let $g' \in \bbG'(\FF_q)$ be an element with an extended Jordan decomposition $(g,t,\dot{t},z,s,u)$.
We say that $g'$ is \textit{semisimple} if the unipotent part $u$ is trivial.
When $g'$ is semisimple, we say that $g'$ is \textit{regular semisimple} (resp.\ \textit{elliptic regular semisimple}) if $s$ is regular semisimple (resp.\ elliptic regular semisimple) in $\bbG^{\circ}$.
\end{definition}

Note that, for any semisimple $g'\in\bbG'(\F_q)$ with an extended Jordan decomposition $(g,t,\dot{t},z,s,u)$, we have $\bbG_{g'}^{\circ}=\bbG_{s}^{\circ}$.
Also note that $g' \in \bbG'(\FF_q)$ is regular semisimple if and only if the connected centralizer $\bbG_{g'}^\circ$ of $g'$ in $\bbG^\circ$ is an $\FF_q$-rational maximal torus of $\bbG^\circ$.
We let $\bbG'(\FF_q)_{\ss}$ denote the subset of semisimple elements of $\bbG'(\F_{q})$.

\begin{cor}\label{cor:Kaletha-CF-ss}
For any semisimple element $g' \in \bbG'(\FF_q)_{\ss}$, we have
    \begin{equation*}
        \Theta_{R_{\bbS}^{\bbG}(\theta)}(g') = \frac{(-1)^{r(\bbG_{g'}^\circ) - r(\bbS^{\circ})}}{|\bbG_{g'}^\circ(\FF_q)|_p \cdot |[\bbS]|} \sum_{\substack{x \in [\bbG] \\ {}^x g' \in \bbS(\FF_q)}} \theta({}^x g'),
    \end{equation*}
    where $r(\bbG^{\circ}_{g'})$ (resp.\ $r(\bbS^{\circ})$) denotes the split rank of $\bbG^{\circ}_{g'}$ (resp.\ $\bbS^{\circ}$) and $|\bbG_{g'}^\circ(\FF_q)|_p$ denotes the largest power of $p$ which divides $|\bbG_{g'}^\circ(\FF_q)|$.
\end{cor}

\begin{proof}
Let $g' \in \bbG'(\FF_q)$ be a semisimple element with an extended Jordan decomposition $(g,t,\dot{t},z,s,1)$.
Then Corollary \ref{cor:Kaletha-CF} gives
\[
\Theta_{R_{\bbS}^{\bbG}(\theta)}(g')
=
\frac{|[\bbG^{\circ}]|}{|[\bbG']|}
\cdot\frac{|\bbZ_{\bbG^{\circ}}(\F_{q})|}{|\bbG^{\circ}_{g'}(\F_{q})|}
\sum_{\begin{subarray}{c}x\in [\bbG] \\ {}^{x}s\in\bbS^{\circ}\end{subarray}}
Q^{\bbG^{\circ}_{g'}}_{\bbS^{x\circ}}(1)\cdot \theta({}^{x}g').
\]

    Let $x \in \bbG(\FF_q)$ be arbitrary. By Deligne--Lusztig's formula for the value of Green functions at the identity \cite[Theorem 7.1]{DL76}, we have
    \begin{equation*}
        Q_{\bbS^{x \circ}}^{\bbG_{g'}^\circ}(1) = (-1)^{r(\bbG_{g'}^\circ) - r(\bbS^{\circ})} \cdot \frac{|\bbG_{g'}^\circ(\FF_q)|}{\mathrm{St}_{\bbG_{g'}^{\circ}}(1)\cdot |\bbS^{\circ}(\FF_q)|}
    \end{equation*}
    where we use that $r(\bbS^{\circ}) = r(\bbS^{x\circ})$ and $|\bbS^\circ(\FF_q)| = |\bbS^{x \circ}(\FF_q)|$.
Here, $\mathrm{St}_{\bbG_{g'}^{\circ}}$ denotes the Steinberg character of the group $\bbG_{g'}^{\circ}$.
According to, for example, \cite[Proposition 6.4.4]{Car85}, $\mathrm{St}_{\bbG_{g'}^{\circ}}(1)$ is given by $|\bbG_{g'}^\circ(\FF_q)|_p$.
Hence we get
\[
\Theta_{R_{\bbS}^{\bbG}(\theta)}(g')
=
\frac{|[\bbG^{\circ}]|}{|[\bbG']|}
\cdot\frac{(-1)^{r(\bbG_{g'}^\circ) - r(\bbS^{\circ})}}{|\bbG_{g'}^\circ(\FF_q)|_p\cdot |[\bbS^{\circ}]|}
\sum_{\begin{subarray}{c}x\in [\bbG] \\ {}^{x}s\in\bbS^{\circ}\end{subarray}}
\theta({}^{x}g')
\]

Note that ${}^{x}s\in\bbS^{\circ}$ if and only if ${}^{x}g'\in\bbS(\F_{q})$.
Thus, to deduce the desired result, we only need to show
    \begin{equation*}
        \frac{|[\bbG^\circ]|}{|[\bbG']|} \cdot \frac{1}{|[\bbS^\circ]|} = \frac{1}{|[\bbS]|}
    \end{equation*}
But this holds since the natural map $[\bbS]/[\bbS^\circ] \to [\bbG']/[\bbG^\circ]$ is bijective.
\end{proof}

Let us observe that the above formula takes an even simpler form when $g'\in\bbG'(\F_{q})$ is regular semisimple.
For any regular semisimple element $g'\in\bbG'(\F_{q})$, we let $\bbS_{g'}^{\circ}$ denote the connected centralizer of $g'$ in $\bbG^{\circ}$, which is an $\F_{q}$-rational maximal torus of $\bbG^{\circ}$, and put
\begin{align*}
W_{\bbG(\F_{q})}(\bbS^{\circ}_{g'},\bbS^{\circ})
&\colonequals \bbS(\F_{q})\backslash N_{\bbG(\F_{q})}(\bbS^{\circ}_{g'},\bbS^{\circ})\\
&=\bbS(\F_{q})\backslash\{n\in\bbG(\F_{q}) \mid {}^{n}\bbS^{\circ}_{g'}=\bbS^{\circ} \}.
\end{align*}
If we put $\bbS_{g'}\colonequals \bbS_{g'}^{\circ}\bbZ_{\bbG}\subset\bbG'$ and define $W_{\bbG(\F_{q})}(\bbS_{g'},\bbS)$ in a similar way, then obviously we have $N_{\bbG(\F_{q})}(\bbS_{g'}^{\circ},\bbS^{\circ})=N_{\bbG(\F_{q})}(\bbS_{g'},\bbS)$ and $W_{\bbG(\F_{q})}(\bbS_{g'}^{\circ},\bbS^{\circ})=W_{\bbG(\F_{q})}(\bbS_{g'},\bbS)$.
We remark that any regular semisimple element $g'\in\bbG'(\F_{q})$ belongs to $\bbS_{g'}$.

\begin{cor}\label{cor:Kaletha-CF-reg}
For any regular semisimple element $g'\in\bbG'(\F_{q})_{\ss}$, we have
\[
\Theta_{R_{\bbS}^{\bbG}(\theta)}(g')
=
\sum_{n\in W_{\bbG(\F_{q})}(\bbS^{\circ}_{g'},\bbS^{\circ})}\theta({}^{n}g').
\]
\end{cor}

\begin{proof}
By Corollary \ref{cor:Kaletha-CF-ss}, we have
\[
\Theta_{R_{\bbS}^{\bbG}(\theta)}(g') = \frac{(-1)^{r(\bbG_{g'}^\circ) - r(\bbS^{\circ})}}{|\bbG_{g'}^\circ(\FF_q)|_p \cdot |[\bbS]|} \sum_{\substack{x \in [\bbG] \\ {}^x g' \in \bbS(\FF_q)}} \theta({}^x g').
\]
Since $g'$ is regular semisimple, for any $x\in\bbG(\F_{q})$, ${}^{x}g'\in\bbS(\F_{q})$ if and only if ${}^{x}\bbS^{\circ}_{g'}=\bbS^{\circ}$.
Moreover, whenever such an element $x$ exists, $\bbG^{\circ}_{g'} \,(=\bbS_{g'}^{\circ})$ is $\F_{q}$-rationally isomorphic to $\bbS^{\circ}$.
In particular, we have $r(\bbG_{g'}^{\circ})=r(\bbS^{\circ})$.
Furthermore, since $\bbG^{\circ}_{g'}$ is a torus, $|\bbG^{\circ}_{g'}(\F_{q})|$ is prime-to-$p$, hence $|\bbG^{\circ}_{g'}(\F_{q})|_{p}$ equals $1$.
Thus we get the desired formula.
\end{proof}

\subsection{Scalar product formula}\label{subsec:scalar product}

Let $\omega$ be a unitary character of $\bbZ_{\bbG}(\FF_q)$. Recall the inner product on $\bbC[\bbG(\FF_q)]_\omega$ defined in Notation \ref{not:finite}.

We will need a truncated inner product, which we define now. 
Let $\overline{\bbG}(\F_{q})_{\bullet}$ be a subset of $\overline{\bbG}(\F_{q})_{\ast}$ which is invariant under $\bbG(\F_{q})$-conjugacy.
We put 
\[
[\bbG']_{\bullet}
\colonequals 
(\pr_{\ss}\circ\Jord)^{-1}(\overline{\bbG}(\F_{q})_{\bullet}),
\]
where $\pr_{\ss}$ denotes the first projection $(s,u)\mapsto s$.
We have:
\begin{equation*}
    \begin{tikzcd}
        {[\bbG']} \ar{r}[above]{\Jord} \ar{r}[below]{1:1} & \{(s,u) \mid s \in \overline{\bbG}(\FF_q)_*, \, u \in \bbG_s^\circ(\FF_q)_{\unip}\} \\
        {[\bbG']_\bullet} \ar[hookrightarrow]{u} \ar{r}[above]{\Jord} \ar{r}[below]{1:1} & \{(s,u) \mid s \in \overline{\bbG}(\FF_q)_\bullet, \, u \in \bbG_s^\circ(\FF_q)_{\unip}\} \ar[hookrightarrow]{u}
    \end{tikzcd}
\end{equation*}

We introduce a truncated inner product over $[\bbG']_{\bullet}$ as follows:
\[
\langle f_{1},f_{2}\rangle_{\bullet}
\colonequals 
\frac{1}{|[\bbG]|}
\sum_{g\in [\bbG']_{\bullet}} f_1(g) \cdot \overline{f_2(g)},
\]
where $f_{1},f_{2}\in\bbC[\bbG(\FF_q)]_\omega$.

\begin{prop}\label{prop:innerprod}
For any $(\bbS_1, \theta_1),(\bbS_2,\theta_2) \in \tilde \cT_\omega$, we have
\[
\langle R_{\bbS_1}^{\bbG}(\theta_1),R_{\bbS_2}^{\bbG}(\theta_2)\rangle_{\bullet}
=
\frac{1}{|[\bbS_{1}]|\cdot|[\bbS_{2}]|}
\sum_{\begin{subarray}{c}s\in\overline{\bbS}_{1}(\F_{q})_{\bullet} \\ n\in N_{\bbG(\F_{q})}(\bbS_{1},\bbS_{2})/\bbZ_{\bbG}(\F_{q})\end{subarray}}
\theta_{1}(\dot{s})\overline{\theta_{2}({}^{n}\dot{s})},
\]
where $\overline{\bbS}_{1}(\F_{q})_{\bullet}$ denotes $\overline{\bbS}_{1}(\F_{q})\cap\overline{\bbG}(\F_{q})_{\bullet}$ and $\dot{s}$ is any representative of $s\in\overline{\bbS}_{1}(\F_{q})_{\bullet}$ in $\bbS_{1}(\F_{q})$.
(Note that such $\dot{s}$ can be taken since $\overline{\bbG}(\F_{q})_{\bullet}\subset\overline{\bbG}(\F_{q})_{\ast}$ and also that the summands are independent of the choice of $\dot{s}$ since $\theta_{1}|_{\bbZ_{\bbG}(\F_{q})}=\theta_{2}|_{\bbZ_{\bbG}(\F_{q})}=\omega$.)
\end{prop}

\begin{proof}
For each (representative of) $g'\in[\bbG']_{\bullet}$, we fix an extended Jordan decomposition $(g,t,\dot{t},z,s,u)$.
Then, by Corollary \ref{cor:Kaletha-CF}, 
\[
\Theta_{R_{\bbS_i}^{\bbG}(\theta_i)}(g')
=
\frac{|[\bbG^{\circ}]|}{|[\bbG']|}
\cdot\frac{|\bbZ_{\bbG^{\circ}}(\F_{q})|}{|\bbG^{\circ}_{s}(\F_{q})|}
\sum_{\begin{subarray}{c}x\in [\bbG] \\ {}^{x}s\in\bbS_{i}^{\circ}\end{subarray}}
Q^{\bbG^{\circ}_{s}}_{\bbS_{i}^{x\circ}}(u)\cdot \theta_{i}({}^{x}sz)
\]
for each $i=1,2$.
Hence $\langle R_{\bbS_1}^{\bbG}(\theta_1),R_{\bbS_2}^{\bbG}(\theta_2)\rangle_{\bullet}$ is given by
\[
\frac{1}{|[\bbG]|}
\cdot\frac{|[\bbG^{\circ}]|^{2}}{|[\bbG']|^{2}}
\sum_{g'\in[\bbG']_{\bullet}}
\frac{|\bbZ_{\bbG^{\circ}}(\F_{q})|^{2}}{|\bbG_{s}^{\circ}(\F_{q})|^{2}}
\sum_{\begin{subarray}{c}x_{1}\in [\bbG] \\ {}^{x_{1}}s\in\bbS_{1}^{\circ}\end{subarray}}
Q^{\bbG^{\circ}_{s}}_{\bbS_{1}^{x_{1}\circ}}(u)\cdot \theta_{1}({}^{x_{1}}sz)
\sum_{\begin{subarray}{c}x_{2}\in [\bbG] \\ {}^{x_{2}}s\in\bbS_{2}^{\circ}\end{subarray}}
\overline{Q^{\bbG^{\circ}_{s}}_{\bbS_{2}^{x_{2}\circ}}(u)}\cdot\overline{\theta_{2}({}^{x_{2}}sz)}.
\]
Since $[\bbG']_{\bullet}$ is bijective to $\{(s,u)\mid s\in\overline{\bbG}(\F_{q})_{\bullet},\, u\in\bbG^{\circ}_{s}(\F_{q})_{\unip}\}$ under the map $\Jord$, we see that this equals
\[
\frac{|\bbZ_{\bbG^{\circ}}(\F_{q})|^{2}}{|[\bbG]|}
\cdot\frac{|[\bbG^{\circ}]|^{2}}{|[\bbG']|^{2}}
\sum_{s\in\overline{\bbG}(\F_{q})_{\bullet}}
\frac{1}{|\bbG_{s}^{\circ}(\F_{q})|^{2}}
\sum_{\begin{subarray}{c}x_{1},x_{2}\in[\bbG]\\ {}^{x_{i}}s\in\overline{\bbS}_{i}\end{subarray}}
\theta_{1}({}^{x_{1}}\dot{s})\overline{\theta_{2}({}^{x_{2}}\dot{s})}
\sum_{u\in\bbG_{s}^{\circ}(\F_{q})_{\unip}}
Q_{\bbS_{1}^{x_{1}\circ}}^{\bbG_{s}^{\circ}}(u)\overline{Q_{\bbS_{2}^{x_{2}\circ}}^{\bbG_{s}^{\circ}}(u)}.
\]
Here $\dot{s}$ is any representative of $s\in\overline{\bbG}(\F_{q})_{\bullet}$ in $\bbG(\F_{q})$.
We recall the orthogonality relation for Green functions of $\bbG_{s}^{\circ}$ (\cite[Theorem 6.9]{DL76}):
\[
\frac{1}{|\bbG_{s}^{\circ}(\F_{q})|}
\sum_{u\in\bbG_{s}^{\circ}(\F_{q})_{\unip}}
Q_{\bbS_{1}^{x_{1}\circ}}^{\bbG_{s}^{\circ}}(u)\cdot\overline{Q_{\bbS_{2}^{x_{2}\circ}}^{\bbG_{s}^{\circ}}(u)}
=
\frac{|N_{\bbG_{s}^{\circ}(\F_{q})}(\bbS_{1}^{x_{1}\circ},\bbS_{2}^{x_{2}\circ})|}{|\bbS_{1}^{x_{1}\circ}(\F_{q})|\cdot|\bbS_{2}^{x_{2}\circ}(\F_{q})|}.
\]
By noting that $|\bbS_{i}^{x_{i}\circ}(\F_{q})|=|\bbS_{i}^{\circ}(\F_{q})|$ and $|[\bbG^{\circ}]|/|[\bbG']|=|[\bbS_{i}^{\circ}]|/|[\bbS_{i}]|$, we get
\[
\langle R_{\bbS_1}^{\bbG}(\theta_1),R_{\bbS_2}^{\bbG}(\theta_2)\rangle_{\bullet}
=
\frac{1}{|[\bbG]|\cdot|[\bbS_{1}]|\cdot|[\bbS_{2}]|}
\sum_{s\in\overline{\bbG}(\F_{q})_{\bullet}}
\frac{1}{|\bbG_{s}^{\circ}(\F_{q})|}
\sum_{\begin{subarray}{c}x_{1},x_{2}\in[\bbG]\\ {}^{x_{i}}s\in\overline{\bbS}_{i} \\ n_{1}\in N_{\bbG_{s}^{\circ}(\F_{q})}(\bbS_{1}^{x_{1}\circ},\bbS_{2}^{x_{2}\circ})\end{subarray}}
\theta_{1}({}^{x_{1}}\dot{s})\overline{\theta_{2}({}^{x_{2}}\dot{s})}.
\]

Here, as in the proof of \cite[Lemma 6.10]{DL76}, we note that the set 
\[
\{(x_{1},x_{2},n_{1})\in[\bbG]\times[\bbG]\times N_{\bbG_{s}^{\circ}(\F_{q})}(\bbS_{1}^{x_{1}\circ},\bbS_{2}^{x_{2}\circ})
\mid
{}^{x_{1}}s\in\overline{\bbS}_{1}, {}^{x_{2}}s\in\overline{\bbS}_{2}\}
\]
is bijective to the set
\[
\{(x_{1},n,n_{1})\in [\bbG]\times (N_{\bbG(\F_{q})}(\bbS_{1},\bbS_{2})/\bbZ_{\bbG}(\F_{q}))\times\bbG_{s}^{\circ}(\F_{q})
\mid
{}^{x_{1}}s\in\overline{\bbS}_{1}\}
\]
by the map
\[
(x_{1},x_{2},n_{1})\mapsto (x_{1},x_{2}n_{1}x_{1}^{-1},n_{1}).
\]
(note that $N_{\bbG(\F_{q})}(\bbS_{1}^{\circ},\bbS_{2}^{\circ})=N_{\bbG(\F_{q})}(\bbS_{1},\bbS_{2})$).
Hence
\begin{align*}
\sum_{\begin{subarray}{c}x_{1},x_{2}\in[\bbG]\\ {}^{x_{i}}s\in\overline{\bbS}_{i} \\ n_{1}\in N_{\bbG_{s}^{\circ}(\F_{q})}(\bbS_{1}^{x_{1}\circ},\bbS_{2}^{x_{2}\circ})\end{subarray}}
\theta_{1}({}^{x_{1}}\dot{s})\overline{\theta_{2}({}^{x_{2}}\dot{s})}
&=
\sum_{\begin{subarray}{c}x_{1}\in[\bbG];\, {}^{x_{1}}s\in\overline{\bbS}_{1} \\ n\in N_{\bbG(\F_{q})}(\bbS_{1},\bbS_{2})/\bbZ_{\bbG}(\F_{q})\\ n_{1}\in \bbG_{s}^{\circ}(\F_{q})\end{subarray}}
\theta_{1}({}^{x_{1}}\dot{s})\overline{\theta_{2}({}^{nx_{1}n_{1}^{-1}}\dot{s})}\\
&=
|\bbG_{s}^{\circ}(\F_{q})|
\sum_{\begin{subarray}{c}x_{1}\in[\bbG];\, {}^{x_{1}}s\in\overline{\bbS}_{1} \\ n\in N_{\bbG(\F_{q})}(\bbS_{1},\bbS_{2})/\bbZ_{\bbG}(\F_{q})\end{subarray}}
\theta_{1}({}^{x_{1}}\dot{s})\overline{\theta_{2}({}^{nx_{1}}\dot{s})}.
\end{align*}
Therefore we have
\[
\langle R_{\bbS_1}^{\bbG}(\theta_1),R_{\bbS_2}^{\bbG}(\theta_2)\rangle_{\bullet}
=
\frac{1}{|[\bbG]|\cdot|[\bbS_{1}]|\cdot|[\bbS_{2}]|}
\sum_{s\in\overline{\bbG}(\F_{q})_{\bullet}}
\sum_{\begin{subarray}{c}x_{1}\in[\bbG];\, {}^{x_{1}}s\in\overline{\bbS}_{1} \\ n\in N_{\bbG(\F_{q})}(\bbS_{1},\bbS_{2})/\bbZ_{\bbG}(\F_{q})\end{subarray}}
\theta_{1}({}^{x_{1}}\dot{s})\overline{\theta_{2}({}^{nx_{1}}\dot{s})}.
\]
Finally, we note that the association $(s,x_{1})\mapsto {}^{x_{1}}s$ gives a well-defined map
\[
\{(s,x_{1})\in\overline{\bbG}(\F_{q})_{\bullet}\times[\bbG]\mid {}^{x_{1}}s\in\overline{\bbS}_{1}\}
\rightarrow
\overline{\bbS}_{1}(\F_{q})_{\bullet}.
\]
Furthermore, this map is surjective and the order of each fiber equals $|[\bbG]|$ since the set $\overline{\bbG}(\F_{q})_{\bullet}$ is invariant under $\bbG(\F_{q})$-conjugation.
Therefore we get
\[
\langle R_{\bbS_1}^{\bbG}(\theta_1),R_{\bbS_2}^{\bbG}(\theta_2)\rangle_{\bullet}
=
\frac{1}{|[\bbS_{1}]|\cdot|[\bbS_{2}]|}
\sum_{\begin{subarray}{c}s\in\overline{\bbS}_{1}(\F_{q})_{\bullet} \\ n\in N_{\bbG(\F_{q})}(\bbS_{1},\bbS_{2})/\bbZ_{\bbG}(\F_{q})\end{subarray}}
\theta_{1}(\dot{s})\overline{\theta_{2}({}^{n}\dot{s})}.\qedhere
\]
\end{proof}

\begin{defn}[in general position]\label{defn:in-gen-pos}
We say that a character $\theta$ of $\bbS(\F_{q})$ is \textit{in general position} if the stabilizer of $\theta$ in $W_{\bbG(\FF_q)}(\bbS)$ is trivial.
\end{defn}

\begin{cor}\label{cor:scalar-prod}
    For any $(\bbS_1,\theta_1),(\bbS_2,\theta_2) \in \tilde \cT$, we have
    \begin{equation*}
        \langle R_{\bbS_1}^{\bbG}(\theta_1), R_{\bbS_2}^{\bbG}(\theta_2) \rangle = |\{w \in W_{\bbG(\FF_q)}(\bbS_1, \bbS_2) \mid \theta_{1} = \theta_{2}^{w}\}|.
    \end{equation*} 
    In particular, $R^{\bbG}_{\bbS}(\theta)$ is irreducible up to sign if $\theta$ is in general position.
\end{cor}

\begin{proof}
By definition, if $\theta_1|_{\bbZ_{\bbG}(\FF_q)} \neq \theta_2|_{\bbZ_{\bbG}(\FF_q)}$, then both sides of the desired identity vanish. Hence we may assume both $\theta_1, \theta_2$ restrict to the same character $\omega$ on $\bbZ_{\bbG}(\FF_q)$.
Furthermore, by noting that $R_{\bbS}^{\bbG}(\theta)\otimes\chi\cong R_{\bbS}^{\bbG}(\theta\otimes\chi)$ for any character $\chi\colon\bbG(\F_{q})\rightarrow\C^{\times}$ (this can be checked by looking at the character formula of $R_{\bbS}^{\bbG}(\theta)$; Corollary \ref{cor:Kaletha-CF}), we may assume that $\omega$ is unitary by Lemma \ref{lem:twist}.
Then we may apply Proposition \ref{prop:innerprod}. Since we take $\bullet=\ast$ in this setting, $\overline{\bbS}_{1}(\F_{q})_{\bullet}$ is given by the image of $\bbS_{1}(\F_{q})$ in $\overline{\bbS}_{1}(\F_{q})$.
Hence
\[
\langle R_{\bbS_1}^{\bbG}(\theta_1),R_{\bbS_2}^{\bbG}(\theta_2)\rangle
=\langle R_{\bbS_1}^{\bbG}(\theta_1),R_{\bbS_2}^{\bbG}(\theta_2)\rangle_{\ast}
=
\frac{1}{|[\bbS_{1}]|\cdot|[\bbS_{2}]|}
\sum_{\begin{subarray}{c}s\in[\bbS_{1}] \\ n\in N_{\bbG(\F_{q})}(\bbS_{1},\bbS_{2})/\bbZ_{\bbG}(\F_{q})\end{subarray}}
\theta_{1}(s)\overline{\theta_{2}({}^{n}s)}.
\]
When $N_{\bbG(\F_{q})}(\bbS_{1},\bbS_{2})$ is empty, this equals zero.
Let us consider the case where $N_{\bbG(\F_{q})}(\bbS_{1},\bbS_{2})$ is nonempty.
In this case, by taking conjugation, we may assume that $\bbS_{1}=\bbS_{2}$.
Then we have
\[
\langle R_{\bbS_1}^{\bbG}(\theta_1),R_{\bbS_1}^{\bbG}(\theta_2)\rangle
=
\frac{1}{|[\bbS_{1}]|^{2}}
\sum_{\begin{subarray}{c}s\in[\bbS_{1}] \\ n\in N_{\bbG(\F_{q})}(\bbS_{1})/\bbZ_{\bbG}(\F_{q})\end{subarray}}
\theta_{1}(s)\overline{\theta_{2}({}^{n}s)}.
\]
As we have
\[
\sum_{s\in[\bbS_{1}]}
\theta_{1}(s)\overline{\theta_{2}({}^{n}s)}
=
\begin{cases}
0&\text{if $\theta_{1}\neq\theta_{2}^{n}$,}\\
|[\bbS_{1}]|&\text{if $\theta_{1}=\theta_{2}^{n}$,}
\end{cases}
\]
we get
\begin{align*}
\langle R_{\bbS_1}^{\bbG}(\theta_1),R_{\bbS_1}^{\bbG}(\theta_2)\rangle
&=
\frac{1}{|[\bbS_{1}]|}
\cdot|\{n\in N_{\bbG(\F_{q})}(\bbS_{1})/\bbZ_{\bbG}(\F_{q})\mid \theta_{1}=\theta_{2}^{n}\}|\\
&=
|\{w\in W_{\bbG(\F_{q})}(\bbS_{1})\mid \theta_{1}=\theta_{2}^{w}\}|.
\end{align*}
\end{proof}

\subsection{Computations on semisimple elements and exhaustion}\label{subsec:semisimple}

In this section, we reproduce the results of \cite[Section 7]{DL76} in our context. The proofs follow exactly the same strategy as \textit{op.\ cit.}, with modifications needed due to the fact that $\bbG(\FF_q)$ is a possibly infinite group. For example, compare our $\mu_s,\mu_s'$ in Section \ref{subsubsec:proof of rho ss} below with the $\mu$ and $\mu'$ after \cite[Proposition 7.5]{DL76}.

\begin{prop}\label{prop:rho ss}
    For any $\rho \in \cR(\bbG(\FF_q))$ and any semisimple element $s \in \bbG'(\FF_q)$, 
    \begin{equation*}
        \Theta_\rho(s) = \frac{1}{|\bbG_{s}^{\circ}(\F_{q})|_p} \sum_{\bbS^\circ \subset \bbG_{s}^\circ} \sum_{\theta \in \bbS(\FF_q)^\wedge} \theta(s) \cdot (-1)^{r(\bbG_s^\circ) - r(\bbS^\circ)} \cdot \langle \rho, R_{\bbS}^{\bbG}(\theta) \rangle,
    \end{equation*}
    where the first sum is over $\F_{q}$-rational maximal tori of $\bbG_{s}^{\circ}$.
    In particular, if $s \in \bbG'(\FF_q)$ is regular semisimple, then it is contained in a unique torus $\bbS$ and
    \begin{equation*}
        \Theta_\rho(s) = \sum_{\theta \in \bbS(\FF_q)^\wedge} \theta(s) \cdot \langle \rho, R_{\bbS}^{\bbG}(\theta) \rangle
    \end{equation*}
    Here, note that the sum over $\theta \in \bbS(\FF_q)^\wedge$ is in fact finite since $\langle \rho, R_{\bbS}^{\bbG}(\theta) \rangle$ vanishes unless $\theta|_{\bbZ_{\bbG}(\F_{q})}$ coincides with the $\bbZ_{\bbG}$-central character of an irreducible constituent of $\rho$ (recall that $\bbS(\F_{q})/\bbZ_{\bbG}(\F_{q})$ is finite).
\end{prop}

Before proving Proposition \ref{prop:rho ss}, we note that the Proposition \ref{prop:rho ss} in the case $s = 1$ gives a dimension formula: for any $\rho \in \cR(\bbG(\FF_q))$,
\begin{equation*}
    \dim \rho = \frac{1}{|\bbG^\circ(\FF_q)|_p} \sum_{\bbS^\circ \subset \bbG^\circ} \sum_{\theta \in \bbS(\FF_q)^\wedge} (-1)^{r(\bbG^\circ) - r(\bbS^\circ)} \cdot \langle \rho, R_\bbS^\bbG(\theta) \rangle.
\end{equation*}
It immediately follows that:

\begin{cor}\label{cor:exhaustion}
    For any irreducible $\rho \in \cR(\bbG(\FF_q))$, there exists $(\bbS,\theta)\in\tilde{\cT}$ such that $\rho$ is a subrepresentation of $R_{\bbS}^{\bbG}(\theta)$.
\end{cor}

\subsubsection{Proof of Proposition \ref{prop:rho ss}}\label{subsubsec:proof of rho ss}

Observe first that to prove the proposition, it suffices to prove that for any $\omega$, $\rho \in \cR(\bbG(\FF_q))_\omega$, and any semisimple element $s \in \bbG'(\FF_q)$, the following equation holds:
\begin{equation}\label{eq:omega ss}
    \Theta_\rho(s) = \frac{1}{|\bbG_{s}^{\circ}(\F_{q})|_p} \sum_{\bbS^\circ \subset \bbG_{s}^\circ} \sum_{\theta \in \bbS(\FF_q)^\wedge_\omega} \theta(s) \cdot (-1)^{r(\bbG_s^\circ) - r(\bbS^\circ)} \cdot \langle \rho, R_{\bbS}^{\bbG}(\theta) \rangle.
\end{equation}
Moreover, by Lemma \ref{lem:twist}, we may assume that $\omega$ is unitary.

Let $s \in \bbG'(\FF_q)$ be semisimple. For $g \in \bbG(\FF_q)$, define
\begin{align*}
    \mu_s(g) &\colonequals \frac{1}{|\bbG_s^\circ(\FF_q)|_p} \sum_{\bbS^\circ \subset \bbG_s^\circ} \sum_{\theta \in \bbS(\FF_q)^\wedge_\omega} \theta(s)^{-1} (-1)^{r(\bbG_s^\circ) - r(\bbS^\circ)} \Theta_{R_{\bbS}^{\bbG}(\theta)}(g) \\
    \mu_s'(g) &\colonequals \sum_{\substack{x \in [\bbG]\\ g^{-1} xsx^{-1} \in \bbZ_{\bbG}(\FF_q)}} \omega(g^{-1}xsx^{-1})^{-1}.
\end{align*}
It is easy to check that $\mu_s, \mu_s' \in \bbC[\bbG(\FF_q)]_\omega$.

\begin{prop}\label{prop:mu-mu'}
We have $\mu_s = \mu_s'$.
\end{prop}

\begin{proof}
    We calculate $\langle \mu_s', \mu_s' \rangle$, $\langle \mu_s, \mu_s' \rangle$, and $\langle \mu_s, \mu_s \rangle$:

    We have:
    \begin{align*}
        \langle \mu_s', \mu_s' \rangle
        &= \frac{1}{|[\bbG]|} \sum_{g \in [\bbG]} \sum_{\substack{x \in [\bbG] \\ g^{-1}xsx^{-1} \in \bbZ_{\bbG}(\FF_q)}} \sum_{\substack{x' \in [\bbG] \\ g^{-1} x'sx'{}^{-1} \in \bbZ_{\bbG}(\FF_q)}} \omega(g^{-1} xsx^{-1})^{-1} \cdot \omega(g^{-1} x'sx'{}^{-1}).
    \end{align*}
For any fixed $x\in[\bbG]$, only an element $g$ of the $\bbZ_{\bbG}$-coset of $xsx^{-1}$ can satisfy the condition $g^{-1}xsx^{-1} \in \bbZ_{\bbG}(\FF_q)$.
Thus we have
\[
\langle \mu_s', \mu_s' \rangle
=
\frac{1}{|[\bbG]|}
\sum_{x\in[\bbG]}
\sum_{\begin{subarray}{c}x'\in[\bbG] \\ xs^{-1}x^{-1}x'sx'{}^{-1}\in\bbZ_{\bbG}(\F_{q})\end{subarray}}
\omega(xs^{-1}x^{-1} x'sx'{}^{-1}).
\]
Note that, as $xs^{-1}x^{-1}x'sx'{}^{-1}\in\bbZ_{\bbG}(\F_{q})$, we have $xs^{-1}x^{-1} x'sx'{}^{-1}=s^{-1}x^{-1}x'sx'{}^{-1}x$.
Thus, by putting $y\colonequals x'^{-1}x$, we get
\begin{align}\label{eq:mu'mu'}
\langle \mu_s', \mu_s' \rangle
=
\frac{1}{|[\bbG]|}
\sum_{x\in[\bbG]}
\sum_{\begin{subarray}{c}y\in[\bbG] \\ s^{-1}y^{-1}sy\in\bbZ_{\bbG}(\F_{q})\end{subarray}}
\omega(s^{-1}y^{-1}sy)=
\sum_{\begin{subarray}{c}y\in[\bbG] \\ s^{-1}y^{-1}sy\in\bbZ_{\bbG}(\F_{q})\end{subarray}}
\omega(s^{-1}y^{-1}sy).
\end{align}

    We have:
    \begin{align*}
        \langle \mu_s, \mu_s' \rangle 
        &= \frac{1}{|[\bbG]| \cdot |\bbG_s^\circ(\FF_q)|_p} \sum_{g \in [\bbG]} \sum_{(\bbS^\circ,\theta,x)} \theta(s)^{-1} (-1)^{r(\bbG_s^\circ) - r(\bbS^\circ)} \Theta_{R_{\bbS}^{\bbG}(\theta)}(g) \omega(g^{-1} xsx^{-1})
    \end{align*}
    where the sum ranges over $(\bbS^\circ,\theta,x)$ where $\bbS^\circ$ is an $\F_{q}$-rational maximal torus of  $\bbG_{s}^\circ$ and $\theta \in \bbS(\FF_q)^\wedge_\omega$ and $x \in [\bbG]$ such that $g^{-1} xsx^{-1} \in \bbZ_{\bbG}(\FF_q)$. If $g^{-1} xsx^{-1} = z$, then 
    \begin{equation*}
        \Theta_{R_{\bbS}^{\bbG}(\theta)}(g) \omega(g^{-1} xsx^{-1}) = \Theta_{R_{\bbS}^{\bbG}(\theta)}(xsx^{-1} z^{-1}) \omega(z) = \Theta_{R_{\bbS}^{\bbG}(\theta)}(xsx^{-1}) = \Theta_{R_{\bbS}^{\bbG}(\theta)}(s).
    \end{equation*}
    Hence we obtain
    \begin{align*}
        \langle \mu_s, \mu_s' \rangle 
        &= \frac{1}{|[\bbG]| \cdot |\bbG_s^\circ(\FF_q)|_p} \sum_{g \in [\ccl(s)]} \sum_{(\bbS^\circ, \theta,x)} \theta(s)^{-1} (-1)^{r(\bbG_s^\circ) - r(\bbS^\circ)} \Theta_{R_{\bbS}^{\bbG}(\theta)}(s) \\
        &= \frac{1}{|\bbG_s^\circ(\FF_q)|_p} \sum_{(\bbS^\circ, \theta)} \theta(s)^{-1} (-1)^{r(\bbG_s^\circ) - r(\bbS^\circ)} \Theta_{R_{\bbS}^{\bbG}(\theta)}(s),
    \end{align*}
    where we use that there are $|[\ccl(s)]|$ choices for $g\in[\bbG]$ ($[\ccl(s)]$ denotes the image of the $\bbG(\F_{q})$-conjugacy class of $s$ in $[\bbG]$), $|[\bbG_s(\FF_q)]|$ choices for $x$, and $|[\bbG]| = |[\ccl(s)]| \cdot |[\bbG_s(\FF_q)]|$. By Corollary \ref{cor:Kaletha-CF-ss}, we have 
    \begin{equation*}
        \Theta_{R_{\bbS}^{\bbG}(\theta)}(s) = \frac{(-1)^{r(\bbG_s^\circ) - r(\bbS^\circ)}}{|\bbG_s^\circ(\FF_q)|_p \cdot |[\bbS]|} \sum_{\substack{x \in [\bbG] \\ x^{-1}sx \in \bbS(\FF_q)}} \theta(x^{-1} s x).
    \end{equation*}
    Therefore
    \begin{align}
    	\nonumber
        \langle \mu_s, \mu_s' \rangle 
        &= \frac{1}{|\bbG_s^\circ(\FF_q)|_p^2} \sum_{(\bbS^\circ, \theta)} \frac{1}{|[\bbS]|} \sum_{\substack{x \in [\bbG] \\ x^{-1} s x \in \bbS(\FF_q)}} \theta(s^{-1}x^{-1} s x) \\
        \nonumber
        &= \frac{1}{|\bbG_s^\circ(\FF_q)|_p^2} \sum_{\bbS^\circ} \frac{|\bbS(\FF_q)^\wedge_\omega|}{|[\bbS]|} \sum_{\substack{x \in [\bbG] \\ s^{-1} x^{-1} s x \in \bbZ_{\bbG}(\FF_q)}} \omega(s^{-1}x^{-1} s x) \\
        \label{eq:mumu'}
        &= \sum_{\substack{x \in [\bbG] \\ s^{-1} x^{-1} s x \in \bbZ_{\bbG}(\FF_q)}} \omega(s^{-1}x^{-1} s x),
    \end{align}
where
\begin{itemize}
\item
the second equality holds since $\sum_{\theta\in\bbS(\F_{q})^{\wedge}_{\omega}}\theta(s^{-1}x^{-1}sx)$ equals $0$ if $s^{-1}x^{-1}sx\notin\bbZ_{\bbG}(\F_{q})$ and equals $|\bbS(\F_{q})^{\wedge}_{\omega}|\cdot\omega(s^{-1}x^{-1}sx)$ if $s^{-1}x^{-1}sx\in\bbZ_{\bbG}(\F_{q})$, 
\item
the last equality holds since $|\bbS(\FF_q)^\wedge_\omega| = |[\bbS]|$ and the number of $\F_{q}$-rational maximal tori $\bbS^\circ$ in $\bbG_s^\circ$ is exactly $|\bbG_s^\circ(\FF_q)|_p^2$ (see \cite[Theorem 3.4.1]{Car85}).
\end{itemize}

    We have
    \begin{equation*}
        \langle \mu_s, \mu_s \rangle 
        = \frac{1}{|\bbG_s^\circ(\FF_q)|_p^2} \sum_{(\bbS^\circ,\theta), (\bbS'{}^\circ,\theta')} \theta(s)^{-1} \theta'(s) (-1)^{r(\bbS^\circ) + r(\bbS'{}^\circ)} \langle R_{\bbS}^{\bbG}(\theta), R_{\bbS'}^{\bbG}(\theta') \rangle,
    \end{equation*}
    where $(\bbS^\circ,\theta)$ ranges over all pairs consisting of an $\F_{q}$-rational maximal torus $\bbS^\circ$ of $\bbG_s^\circ$ and a character $\theta \in \bbS(\FF_q)^\wedge_\omega$, and similarly for $(\bbS'{}^\circ, \theta')$. 
By Corollary \ref{cor:scalar-prod}, $r(\bbS^\circ)=r(\bbS^{\prime\circ})$ whenever $\langle R_{\bbS}^{\bbG}(\theta), R_{\bbS'}^{\bbG}(\theta') \rangle\neq0$.
Hence, by Proposition \ref{prop:innerprod}, we get
\[
\langle \mu_s, \mu_s \rangle 
=
\frac{1}{|\bbG_s^\circ(\FF_q)|_p^2} \sum_{(\bbS, \theta), (\bbS'{}^\circ,\theta')} \theta(s)^{-1} \theta'(s) \sum_{\substack{t \in [\bbS] \\ n \in N_{\bbG(\FF_q)}(\bbS,\bbS')/\bbZ_{\bbG}(\FF_q)}} \frac{\theta(t) \cdot \theta'^{n}(t)^{-1}}{|[\bbS]| \cdot |[\bbS']|}.
\]
By noting that $|[\bbS]|=|[\bbS']|$ whenever $N_{\bbG(\FF_q)}(\bbS,\bbS')\neq\varnothing$ and that the sum of $\theta(t) \cdot \theta'^{n}(t)^{-1}$ over $t\in[\bbS]$ is not zero only if $\theta'^{n}=\theta$ (and is given by $|[\bbS]|$ in this case), we get
\[
\langle \mu_s, \mu_s \rangle 
=
\frac{1}{|\bbG_s^\circ(\FF_q)|_p^2}
\sum_{\begin{subarray}{c}\bbS^{\circ},\bbS^{\prime\circ} \\ n\in N_{\bbG(\F_{q})}(\bbS,\bbS')/\bbZ_{\bbG}(\F_{q}) \end{subarray}}
\sum_{\theta\in\bbS(\FF_q)^\wedge_\omega}
\theta(s)^{-1}\theta(n^{-1}sn)\cdot\frac{1}{|[\bbS]|}.
\]
Note that we have a bijection
\[
\{g\in\bbG(\F_{q}) \mid g^{-1}sg\in\bbS\}/N_{\bbG(\F_{q})}(\bbS)
\xrightarrow{1:1}
\{\bbS^{\prime\circ}\subset\bbG_{s}^{\circ} \mid N_{\bbG(\F_{q})}(\bbS,\bbS')\neq\varnothing\}
\]
given by $g\mapsto {}^{g}\bbS^{\circ}$ and that $N_{\bbG(\F_{q})}(\bbS,\bbS')=gN_{\bbG(\F_{q})}(\bbS)$ when $\bbS^{\prime\circ}={}^{g}\bbS^{\circ}$.
Hence we get
\begin{align}
\nonumber
\langle \mu_s, \mu_s \rangle
&=\frac{1}{|\bbG_s^\circ(\FF_q)|_p^2}
\sum_{\bbS^{\circ}} \frac{1}{|[\bbS]|}
\sum_{\begin{subarray}{c}g\in \bbG(\F_{q})/N_{\bbG(\F_{q})}(\bbS) \\ g^{-1}sg\in \bbS\end{subarray}}
\sum_{n\in gN_{\bbG(\F_{q})}(\bbS)/\bbZ_{\bbG}(\F_{q})}
\sum_{\theta\in\bbS(\FF_q)^\wedge_\omega} \theta(s^{-1}n^{-1}sn)\\
\nonumber
&=\frac{1}{|\bbG_s^\circ(\FF_q)|_p^2}
\sum_{\bbS^{\circ}} \frac{1}{|[\bbS]|}
\sum_{\begin{subarray}{c} n\in [\bbG] \\ n^{-1}sn\in \bbS\end{subarray}}
\sum_{\theta\in\bbS(\FF_q)^\wedge_\omega} \theta(s^{-1}n^{-1}sn)\\
\label{eq:mumu}
&=  \sum_{\substack{x \in [\bbG] \\ s^{-1} x^{-1} s x \in \bbZ_{\bbG}(\FF_q)}} \omega(s^{-1} x^{-1} s x),
\end{align}
where the last equality holds by exactly the reasons in the bullet points following \eqref{eq:mumu'}.    

Combining Equations \eqref{eq:mu'mu'}, \eqref{eq:mumu'}, \eqref{eq:mumu}, we have the desired equalities $\langle \mu_s', \mu_s' \rangle = \langle \mu_s, \mu_s' \rangle = \langle \mu_s, \mu_s \rangle$, which implies $\langle \mu_s - \mu_s', \mu_s - \mu_s' \rangle = 0$, so that $\mu_s = \mu_s'$.
\end{proof}

We are now ready to prove \eqref{eq:omega ss}. We have
    \begin{equation*}
        \langle \Theta_\rho, \mu_s' \rangle = \frac{1}{|[\bbG]|} \sum_{g \in [\bbG]} \sum_{\substack{x \in [\bbG] \\ g^{-1} xsx^{-1} \in \bbZ_{\bbG}(\FF_q)}} \Theta_\rho(g) \omega(g^{-1} xsx^{-1}).
    \end{equation*}
    For any $x\in[\bbG]$, there exists a unique $\bbZ_{\bbG}(\FF_q)$-coset of $\dot g \in \bbG(\FF_q)$ such that $\dot g^{-1} x s x^{-1} \in \bbZ_{\bbG}(\FF_q)$.
    Let $\dot x \in \bbG(\FF_q)$ be any lift of $x$ and write $\dot g^{-1} \dot xs \dot x^{-1} = z \in \bbZ_{\bbG}(\FF_q)$. Then 
    \begin{equation*}
    \Theta_\rho(\dot g) \omega(\dot g^{-1} \dot xs \dot x^{-1}) = \Theta_\rho(\dot xs\dot x^{-1} z^{-1}) \omega(z) = \Theta_\rho(\dot xs \dot x^{-1}) = \Theta_\rho(s).
    \end{equation*}
    We therefore have
    \begin{equation*}
        \langle \Theta_\rho, \mu_s' \rangle = \frac{1}{|[\bbG]|} \sum_{x \in [\bbG]} \Theta_\rho(s) = \Theta_\rho(s).
    \end{equation*}
    The desired formula \eqref{eq:omega ss} for $\Theta_\rho(s)$ now holds by applying Proposition \ref{prop:mu-mu'}.

\section{Characterization theorem at finite field level}\label{sec:char finite}

We retain the set-up and notation of Section \ref{sec:finite} and use the result on the representation theory of $\bbG(\FF_q)$ established there to address the question:
\begin{center}
    Is an irreducible representation of $\bbG(\FF_q)$ determined by its character on \\ regular semisimple elements (in the sense of Definition \ref{def:G' semisimple})? 
\end{center}
More generally, we study this question for an arbitrary conjugation-invariant subset $\bbG(\FF_q)_\bullet$ of the regular semisimple locus. In Section \ref{subsec:char-finite-gen-pos}, we will see that for a fixed maximal torus $\bbS^\circ \subset \bbG^\circ$, if $\theta$ is in general position and $\bbS(\FF_q)_\bullet$ is sufficiently large in the sense of \eqref{ineq:Henniart-bullet}, then $R_\bbS^\bbG(\theta)$ is uniquely determined by its character on $\bbG(\FF_q)_\bullet$ (Theorem \ref{thm:Henniart}), which we already know is given by a remarkably simple formula (Corollary \ref{cor:Kaletha-CF-reg}). In Section \ref{subsec:Lusztig-E}, we prove that if $\bbG(\FF_q)_\bullet$ satisfies the stronger inequality \eqref{ineq:Lusztig-bullet}, then in fact Lusztig's map $E$ (and its refinement $\tilde E$) can be defined purely from the elementary data of character values on $\bbG(\FF_q)_\bullet$.

\subsection{Characterization theorem at finite level for $\theta$ in general position}\label{subsec:char-finite-gen-pos}

As in Section \ref{subsec:scalar product}, let $\overline{\bbG}(\FF_q)_{\bullet}$ be any subset of $\overline{\bbG}(\F_{q})_{\ast}$ invariant under $\bbG(\FF_q)$-conjugacy. 
We furthermore assume that $\overline{\bbG}(\FF_q)_{\bullet}$ is contained in the regular semisimple locus $\overline{\bbG}(\FF_q)_{\rs}$ of $\overline{\bbG}(\FF_q)$.

Suppose that we have a subset $\bbG'(\F_{q})_{\bullet}$ of $\bbG'(\F_{q})$ whose image in $[\bbG']$ is equal to $[\bbG']_{\bullet}$ and suppose furthermore that there exists a finite-index subgroup $\bbZ_\bbG^\star$ of $\bbZ_\bbG$ such that $\bbG'(\bbF_q)_\bullet$ is invariant under $\bbZ_\bbG^\star(\FF_q)$-translation.
Note here that $\bbG'(\FF_q)_{\bullet}$ may not be the full preimage of $[\bbG']_\bullet$ under $\bbG'(\FF_q) \twoheadrightarrow [\bbG']$! This allows for the possibility that $\bbG'(\FF_q)_\bullet$ may not be invariant under $\bbZ_\bbG(\FF_q)$-translation, hence the necessity for the introduction of $\bbZ_\bbG^\star$.

\[
\begin{tikzcd}
\bbG'(\F_{q}) \ar[twoheadrightarrow]{r} & {[\bbG']} \ar{r}[above]{\Jord} \ar{r}[below]{1:1} & \{(s,u) \mid s \in \overline{\bbG}(\FF_q)_*, \, u \in \bbG_s^\circ(\FF_q)_{\unip}\} \\
\bbG'(\F_{q})_{\bullet} \ar[hookrightarrow]{u} \ar[twoheadrightarrow]{r} & {[\bbG']_\bullet} \ar[hookrightarrow]{u} \ar{r}[above]{\Jord} \ar{r}[below]{1:1} & \{(s,u) \mid s \in \overline{\bbG}(\FF_q)_\bullet, \, u \in \bbG_s^\circ(\FF_q)_{\unip}\} \ar[hookrightarrow]{u}
\end{tikzcd}
\]

We put $\overline{\bbG}(\F_{q})_{\circ}\colonequals \overline{\bbG}(\F_{q})_{\ast}\smallsetminus\overline{\bbG}(\F_{q})_{\bullet}$, $\bbG'(\F_{q})_{\circ}\colonequals \bbG'(\F_{q})\smallsetminus\bbG'(\F_{q})_{\bullet}$, and $\bbG(\F_{q})_{\circ}\colonequals \bbG(\F_{q})\smallsetminus\bbG'(\F_{q})_{\bullet}$.
We also put $[\bbG']_{\circ}\colonequals [\bbG']\smallsetminus[\bbG']_{\bullet}$ and $[\bbG]_{\circ}\colonequals [\bbG]\smallsetminus[\bbG']_{\bullet}$.
We also introduce the quotients of $\bbG(\F_{q})$, $\bbG'(\F_{q})_{?}$ (for $?\in\{\bullet,\circ\}$), and $\bbG(\F_{q})_{\circ}$  by $\bbZ^{\star}_{\bbG}(\F_{q})$:
\begin{itemize}
\item
$[\bbG]^{\star}\colonequals \bbG(\F_{q})/\bbZ^{\star}_{\bbG}(\F_{q})$,
\item 
$[\bbG']^{\star}_{?}\colonequals \bbG'(\F_{q})_{?}/\bbZ^{\star}_{\bbG}(\F_{q})$, and 
\item 
$[\bbG]^{\star}_{\circ}\colonequals \bbG(\F_{q})_{\circ}/\bbZ^{\star}_{\bbG}(\F_{q})$.
\end{itemize}

For any $\F_{q}$-rational maximal torus $\bbS^{\circ}$ of $\bbG^{\circ}$, we put $\bbS(\F_{q})_{?}\colonequals \bbS(\F_{q})\cap\bbG'(\F_{q})_{?}$ for $?\in\{\bullet,\circ\}$.
We also put
\begin{itemize}
\item
$[\bbS]^{\star}\colonequals \bbS(\F_{q})/\bbZ^{\star}_{\bbG}(\F_{q})$, and
\item
$[\bbS]^{\star}_{?}\colonequals \bbS(\F_{q})_{?}/\bbZ^{\star}_{\bG}(\F_{q})$.
\end{itemize}

\begin{rem}\label{rem:examples2}
In the cases of Remark \ref{rem:examples}, we take a finite index subgroup $\bbZ^{\star}_{\bbG}$ of $\bbZ_{\bbG}$ as follows:
\begin{enumerate}
\item
We simply take $\bbZ^{\star}_{\bbG}\colonequals \{1\}$.
\item
Let suppose that $\bfG^{0}$ is a tame twisted Levi subgroup of a connected reductive group $\bfG$ over $F$ such that $\bfZ_{\bfG^{0}}/\bfZ_{\bfG}$ is anisotropic.
Then we take $\bbZ^{\star}_{\bbG}$ to be the image of $Z_{\bfG}$ (see Section \ref{subsec:ur-vreg}).
\end{enumerate}
\end{rem}

For a unitary character $\omega^{\star}$ of $\bbZ^{\star}_\bbG(\FF_q)$, we define an inner product $\langle-,-\rangle^{\star}$ on the space $\C[\bbG(\F_{q})]_{\omega^{\star}}$ by
\[
\langle f_{1},f_{2} \rangle^{\star}
\colonequals
\frac{1}{|[\bbG]^{\star}|} \sum_{g\in[\bbG]^{\star}}f_1(g) \cdot \overline{f_2(g)},
\]
for any $f_{1},f_{2}\in\C[\bbG(\F_{q})]_{\omega^{\star}}$.
Then, in the same way as $\langle-,-\rangle$, we can define the inner product $\langle \rho_{1}, \rho_{2} \rangle^{\star}$ for any two representations $\rho_1, \rho_2 \in \cR(\bbG(\FF_q))$ and also its truncated versions $\langle \rho_{1}, \rho_{2} \rangle^{\star}_{\bullet}$ and $\langle \rho_{1}, \rho_{2} \rangle^{\star}_{\circ}$.

\begin{prop}\label{prop:innerprod2}
For any $(\bbS_1, \theta_1)\in\tilde{\cT}_{\omega_{1}}$ and $(\bbS_2,\theta_2) \in\tilde{\cT}_{\omega_{2}}$ such that $\omega_{1}|_{\bbZ^{\star}_\bbG(\FF_q)}=\omega_{2}|_{\bbZ^{\star}_\bbG(\FF_q)}$, we have
\[
\langle R_{\bbS_1}^{\bbG}(\theta_1),R_{\bbS_2}^{\bbG}(\theta_2)\rangle^{\star}_{\circ}
=
\frac{1}{|[\bbS_{1}]^{\star}|\cdot|[\bbS_{2}]^{\star}|}
\sum_{\begin{subarray}{c}s\in[\bbS_{1}]^{\star}_{\circ} \\ n\in N_{\bbG(\F_{q})}(\bbS_{1},\bbS_{2})/\bbZ^{\star}_{\bbG}(\F_{q})\end{subarray}}
\theta_{1}(s)\overline{\theta_{2}({}^{n}s)}.
\]
\end{prop}

\begin{proof}
A similar proof to Proposition \ref{prop:innerprod} works, but we need a minor modification as we explain in the following.

For each (representative of) $g^{\star}\in[\bbG']^{\star}_{\circ}$, we fix an extended Jordan decomposition $(g,t,\dot{t},z,s,u)$.
Then, by Corollary \ref{cor:Kaletha-CF}, 
\begin{align*}
\Theta_{R_{\bbS_i}^{\bbG}(\theta_i)}(g^{\star})
&=
\frac{|[\bbG^{\circ}]|}{|[\bbG']|}
\cdot\frac{|\bbZ_{\bbG^{\circ}}(\F_{q})|}{|\bbG^{\circ}_{s}(\F_{q})|}
\sum_{\begin{subarray}{c}x\in [\bbG] \\ {}^{x}s\in\bbS_{i}^{\circ}\end{subarray}}
Q^{\bbG^{\circ}_{s}}_{\bbS_{i}^{x\circ}}(u)\cdot \theta_{i}({}^{x}sz)\\
&=
\frac{|[\bbG^{\circ}]|}{|[\bbG']|}
\cdot\frac{|\bbZ_{\bbG^{\circ}}(\F_{q})|}{|\bbG^{\circ}_{s}(\F_{q})|}
\cdot\frac{|\bbZ^{\star}_{\bbG}(\F_{q})|}{|\bbZ_{\bbG}(\F_{q})|}
\sum_{\begin{subarray}{c}x\in [\bbG]^{\star} \\ {}^{x}s\in\bbS_{i}^{\circ}\end{subarray}}
Q^{\bbG^{\circ}_{s}}_{\bbS_{i}^{x\circ}}(u)\cdot \theta_{i}({}^{x}sz)
\end{align*}
for each $i=1,2$.
Hence $\langle R_{\bbS_1}^{\bbG}(\theta_1),R_{\bbS_2}^{\bbG}(\theta_2)\rangle^{\star}_{\circ}$ is given by
\begin{multline*}
\frac{|\bbZ_{\bbG^{\circ}}(\F_{q})|^{2}}{|[\bbG]^{\star}|}
\cdot\frac{|[\bbG^{\circ}]|^{2}}{|[\bbG']|^{2}}
\cdot\frac{|\bbZ^{\star}_{\bbG}(\F_{q})|^{2}}{|\bbZ_{\bbG}(\F_{q})|^{2}}
\sum_{g^{\star}\in[\bbG']^{\star}_{\circ}}
\frac{1}{|\bbG_{s}^{\circ}(\F_{q})|^{2}}\\
\cdot\sum_{\begin{subarray}{c}x_{1}\in [\bbG]^{\star} \\ {}^{x_{1}}s\in\bbS_{1}^{\circ}\end{subarray}}
Q^{\bbG^{\circ}_{s}}_{\bbS_{1}^{x_{1}\circ}}(u)\cdot \theta_{1}({}^{x_{1}}sz)
\sum_{\begin{subarray}{c}x_{2}\in [\bbG]^{\star} \\ {}^{x_{2}}s\in\bbS_{2}^{\circ}\end{subarray}}
\overline{Q^{\bbG^{\circ}_{s}}_{\bbS_{2}^{x_{2}\circ}}(u)}\cdot\overline{\theta_{2}({}^{x_{2}}sz)}.
\end{multline*}

We consider the natural quotient map
\[
\nu\colon [\bbG']^{\star}_{\circ}\twoheadrightarrow[\bbG']_{\circ}.
\]
For any $(s,u)$ such that $s \in \overline{\bbG}(\FF_q)_{\circ}$ and $u\in \bbG_s^\circ(\FF_q)_{\unip}$, we put $g'_{(s,u)}\colonequals \Jord^{-1}(s,u)\in[\bbG']_{\circ}$.
We fix an element $g^{\star}_{(s,u)}\in \nu^{-1}(g'_{(s,u)})$ for each $(s,u)$, hence we have $\nu^{-1}(g'_{(s,u)})\subset g^{\star}_{(s,u)}\bbZ_{\bbG}(\F_{q})/\bbZ_{\bbG}^{\star}(\F_{q})$.
Let $Z_{(s,u)}\subset\bbZ_{\bbG}(\F_{q})/\bbZ_{\bbG}^{\star}(\F_{q})$ be the subset satisfying $\nu^{-1}(g'_{(s,u)})=g^{\star}_{(s,u)}Z_{(s,u)}$.
Then we have
\begin{align}\label{eq:stratification}
[\bbG']^{\star}_{\circ}
=
\bigsqcup_{s\in\overline{\bbG}(\F_{q})_{\circ}}\bigsqcup_{u\in\bbG_{s}(\F_{q})_{\unip}}g^{\star}_{(s,u)}Z_{(s,u)}.
\end{align}
Note that it can happen $Z_{(s,u)}\subsetneq\bbZ_{\bbG}(\F_{q})/\bbZ_{\bbG}^{\star}(\F_{q})$ only when $s$ is regular semisimple.
Indeed, if $Z_{(s,u)}\subsetneq\bbZ_{\bbG}(\F_{q})/\bbZ_{\bbG}^{\star}(\F_{q})$, then there exists a $z\in\bbZ_{\bbG}(\F_{q})$ satisfying $g^{\star}_{(s,u)}z\in [\bbG']^{\star}_{\bullet}$, which implies that $s$ is regular semisimple and $u=1$.
Hence, we have $Z_{(s,u)}=\bbZ_{\bbG}(\F_{q})/\bbZ_{\bbG}^{\star}(\F_{q})$ whenever $u\neq1$.
In other words, $Z_{(s,u)}$ depends only on $s$.
Let us simply write $Z_{s}$ for $Z_{(s,u)}$.

For each semisimple element $s\in\overline{\bbG}(\F_{q})_{\circ}$, we put $\dot{s}:=g^{\star}_{(s,1)}\in[\bbG']^{\star}_{\circ}$.
By using \eqref{eq:stratification}, we get
\begin{multline*}
\sum_{g'\in[\bbG']^{\star}_{\circ}}
\frac{1}{|\bbG_{s}^{\circ}(\F_{q})|^{2}}
\sum_{\begin{subarray}{c}x_{1}\in [\bbG]^{\star} \\ {}^{x_{1}}s\in\bbS_{1}^{\circ}\end{subarray}}
Q^{\bbG^{\circ}_{s}}_{\bbS_{1}^{x_{1}\circ}}(u)\cdot \theta_{1}({}^{x_{1}}sz)
\sum_{\begin{subarray}{c}x_{2}\in [\bbG]^{\star} \\ {}^{x_{2}}s\in\bbS_{2}^{\circ}\end{subarray}}
\overline{Q^{\bbG^{\circ}_{s}}_{\bbS_{2}^{x_{2}\circ}}(u)}\cdot\overline{\theta_{2}({}^{x_{2}}sz)}\\
=
\sum_{s\in\overline{\bbG}(\F_{q})_{\circ}}
\frac{1}{|\bbG_{s}^{\circ}(\F_{q})|^{2}}
\sum_{t\in Z_{s}}
\sum_{\begin{subarray}{c}x_{1},x_{2}\in[\bbG]^{\star}\\ {}^{x_{i}}s\in\overline{\bbS}_{i}\end{subarray}}
\theta_{1}({}^{x_{1}}\dot{s}t)\overline{\theta_{2}({}^{x_{2}}\dot{s}t)}
\sum_{u\in\bbG_{s}^{\circ}(\F_{q})_{\unip}}
Q_{\bbS_{1}^{x_{1}\circ}}^{\bbG_{s}^{\circ}}(u)\overline{Q_{\bbS_{2}^{x_{2}\circ}}^{\bbG_{s}^{\circ}}(u)}.
\end{multline*}
Hence, by using the orthogonality relation for Green functions of $\bbG_{s}^{\circ}$ (\cite[Theorem 6.9]{DL76}),
$\langle R_{\bbS_1}^{\bbG}(\theta_1),R_{\bbS_2}^{\bbG}(\theta_2)\rangle^{\star}_{\circ}$ equals
\begin{multline*}
\frac{|\bbZ_{\bbG^{\circ}}(\F_{q})|^{2}}{|[\bbG]^{\star}|}
\cdot\frac{|[\bbG^{\circ}]|^{2}}{|[\bbG']|^{2}}
\cdot\frac{|\bbZ^{\star}_{\bbG}(\F_{q})|^{2}}{|\bbZ_{\bbG}(\F_{q})|^{2}}
\sum_{s\in\overline{\bbG}(\F_{q})_{\circ}}
\frac{1}{|\bbG_{s}^{\circ}(\F_{q})|}\\
\cdot\sum_{t\in Z_{s}}
\sum_{\begin{subarray}{c}x_{1},x_{2}\in[\bbG]^{\star}\\ {}^{x_{i}}s\in\overline{\bbS}_{i}\end{subarray}}
\theta_{1}({}^{x_{1}}\dot{s}t)\overline{\theta_{2}({}^{x_{2}}\dot{s}t)}
\cdot\frac{|N_{\bbG_{s}^{\circ}(\F_{q})}(\bbS_{1}^{x_{1}\circ},\bbS_{2}^{x_{2}\circ})|}{|\bbS_{1}^{x_{1}\circ}(\F_{q})|\cdot|\bbS_{2}^{x_{2}\circ}(\F_{q})|}.
\end{multline*}
By noting that $|\bbS_{i}^{x_{i}\circ}(\F_{q})|=|\bbS_{i}^{\circ}(\F_{q})|$, $|[\bbG^{\circ}]|/|[\bbG']|=|[\bbS_{i}^{\circ}]|/|[\bbS_{i}]|$, and $|[\bbS_{i}]|/|[\bbS_{i}]^{\star}|=|\bbZ^{\star}_{\bbG}(\F_{q})|/|\bbZ_{\bbG}(\F_{q})|$, we see that this equals
\[
\frac{1}{|[\bbG]^{\star}|}
\cdot\frac{1}{|[\bbS_{1}]^{\star}|\cdot|[\bbS_{2}]^{\star}|}
\sum_{s\in\overline{\bbG}(\F_{q})_{\circ}}
\frac{1}{|\bbG_{s}^{\circ}(\F_{q})|}
\sum_{t\in Z_{s}}
\sum_{\begin{subarray}{c}x_{1},x_{2}\in[\bbG]^{\star}\\ {}^{x_{i}}s\in\overline{\bbS}_{i} \\ n_{1}\in N_{\bbG_{s}^{\circ}(\F_{q})}(\bbS_{1}^{x_{1}\circ},\bbS_{2}^{x_{2}\circ})\end{subarray}}
\theta_{1}({}^{x_{1}}\dot{s}t)\overline{\theta_{2}({}^{x_{2}}\dot{s}t)}.
\]
As in the proof of \cite[Lemma 6.10]{DL76}, we note that the set 
\[
\{(x_{1},x_{2},n_{1})\in[\bbG]^{\star}\times[\bbG]^{\star}\times N_{\bbG_{s}^{\circ}(\F_{q})}(\bbS_{1}^{x_{1}\circ},\bbS_{2}^{x_{2}\circ})
\mid
{}^{x_{1}}s\in\overline{\bbS}_{1}, {}^{x_{2}}s\in\overline{\bbS}_{2}\}
\]
is bijective to the set
\[
\{(x_{1},n,n_{1})\in [\bbG]^{\star}\times (N_{\bbG(\F_{q})}(\bbS_{1},\bbS_{2})/\bbZ^{\star}_{\bbG}(\F_{q}))\times\bbG_{s}^{\circ}(\F_{q})
\mid
{}^{x_{1}}s\in\overline{\bbS}_{1}\}
\]
by the map $(x_{1},x_{2},n_{1})\mapsto (x_{1},x_{2}n_{1}x_{1}^{-1},n_{1})$.
Hence
\begin{align*}
\sum_{\begin{subarray}{c}x_{1},x_{2}\in[\bbG]^{\star}\\ {}^{x_{i}}s\in\overline{\bbS}_{i} \\ n_{1}\in N_{\bbG_{s}^{\circ}(\F_{q})}(\bbS_{1}^{x_{1}\circ},\bbS_{2}^{x_{2}\circ})\end{subarray}}
\theta_{1}({}^{x_{1}}\dot{s}t)\overline{\theta_{2}({}^{x_{2}}\dot{s}t)}
&=
\sum_{\begin{subarray}{c}x_{1}\in[\bbG]^{\star};\, {}^{x_{1}}s\in\overline{\bbS}_{1} \\ n\in N_{\bbG(\F_{q})}(\bbS_{1},\bbS_{2})/\bbZ^{\star}_{\bbG}(\F_{q})\\ n_{1}\in \bbG_{s}^{\circ}(\F_{q})\end{subarray}}
\theta_{1}({}^{x_{1}}\dot{s}t)\overline{\theta_{2}({}^{nx_{1}n_{1}^{-1}}\dot{s}t)}\\
&=
|\bbG_{s}^{\circ}(\F_{q})|
\sum_{\begin{subarray}{c}x_{1}\in[\bbG]^{\star};\, {}^{x_{1}}s\in\overline{\bbS}_{1} \\ n\in N_{\bbG(\F_{q})}(\bbS_{1},\bbS_{2})/\bbZ^{\star}_{\bbG}(\F_{q})\end{subarray}}
\theta_{1}({}^{x_{1}}\dot{s}t)\overline{\theta_{2}({}^{nx_{1}}\dot{s}t)}.
\end{align*}
Therefore $\langle R_{\bbS_1}^{\bbG}(\theta_1),R_{\bbS_2}^{\bbG}(\theta_2)\rangle_{\circ}$ equals 
\[
\frac{1}{|[\bbG]^{\star}|\cdot|[\bbS_{1}]^{\star}|\cdot|[\bbS_{2}]^{\star}|}
\sum_{s\in\overline{\bbG}(\F_{q})_{\circ}}
\sum_{t\in Z_{s}}
\sum_{\begin{subarray}{c}x_{1}\in[\bbG]^{\star};\, {}^{x_{1}}s\in\overline{\bbS}_{1} \\ n\in N_{\bbG(\F_{q})}(\bbS_{1},\bbS_{2})/\bbZ^{\star}_{\bbG}(\F_{q})\end{subarray}}
\theta_{1}({}^{x_{1}}\dot{s}t)\overline{\theta_{2}({}^{nx_{1}}\dot{s}t)}.
\]
Finally, we note that the association $(s,t,x_{1})\mapsto {}^{x_{1}}\dot{s}t$ gives a well-defined map
\[
\{(s,t,x_{1}) \mid s\in\overline{\bbG}(\F_{q})_{\circ}, t\in Z_{s}, x_{1}\in[\bbG]^{\star}\mid {}^{x_{1}}s\in\overline{\bbS}_{1}\}
\rightarrow
[\bbS_{1}]^{\star}_{\circ}.
\]
Furthermore, this map is surjective and the order of each fiber equals $|[\bbG]^{\star}|$ since the set $\overline{\bbG}(\F_{q})_{\circ}$ is invariant under $\bbG(\F_{q})$-conjugation.
Therefore we get
\[
\langle R_{\bbS_1}^{\bbG}(\theta_1),R_{\bbS_2}^{\bbG}(\theta_2)\rangle^{\star}_{\circ}
=
\frac{1}{|[\bbS_{1}]^{\star}|\cdot|[\bbS_{2}]^{\star}|}
\sum_{\begin{subarray}{c}s\in[\bbS_{1}]^{\star}_{\circ} \\ n\in N_{\bbG(\F_{q})}(\bbS_{1},\bbS_{2})/\bbZ^{\star}_{\bbG}(\F_{q})\end{subarray}}
\theta_{1}(s)\overline{\theta_{2}({}^{n}s)}.\qedhere
\]
\end{proof}

In this section, we consider $(\bbS,\theta) \in \tilde \cT$ such that $\theta$ is in general position.
Recall from Corollary \ref{cor:scalar-prod} that then $R_\bbS^{\bbG}(\theta)$ is irreducible.

We consider the following inequality:
\[\label{ineq:Henniart-bullet}
\frac{|[\bbS]^{\star}|}{|[\bbS]^{\star}_{\circ}|}
=
\frac{|[\bbS]^{\star}|}{|[\bbS]^{\star}\smallsetminus[\bbS]^{\star}_{\bullet}|}
>
2\cdot|W_{\bbG(\F_{q})}(\bbS)|.
\tag{$\mathfrak{H}_\bullet$}
\]

\begin{lem}\label{lem:Henniart}
Let $(\bbS,\theta) \in \tilde \cT$ be such that $\theta$ is in general position.
If \eqref{ineq:Henniart-bullet} is satisfied, then $\Theta_{R_{\bbS}^{\bbG}(\theta)}(g)\neq 0$ for some  $g\in \bbS(\F_{q})_{\bullet}$.
\end{lem}

\begin{proof}
By Lemma \ref{lem:twist}, we may suppose that $\theta|_{\bbZ_{\bbG}(\F_{q})}$ is unitary.
For notational convenience, write $R_\theta \colonequals R_{\bbS}^{\bbG}(\theta)$.
We obviously have
\[
\langle R_{\theta},R_{\theta}\rangle^{\star}
=
\langle R_{\theta},R_{\theta}\rangle^{\star}_{\bullet}
+
\langle R_{\theta},R_{\theta}\rangle^{\star}_{\circ}.
\]
Recall that $R_{\theta}$ is supported on $\bbG'(\F_{q})$ (Corollary \ref{cor:Kaletha-CF}), hence we may replace the index set $[\bbG]^{\star}_{\circ}$ of the sum in $\langle R_{\theta},R_{\theta}\rangle^{\star}_{\circ}$ with $[\bbG']^{\star}_{\circ}$.
By Proposition \ref{prop:innerprod2}, we get
\[
\langle R_{\theta},R_{\theta}\rangle^{\star}_{\circ}
=
\frac{1}{|[\bbS]^{\star}|^{2}}
\sum_{\begin{subarray}{c}s\in[\bbS]^{\star}_{\circ} \\ n\in N_{\bbG(\F_{q})}(\bbS)/\bbZ^{\star}_{\bbG}(\F_{q})\end{subarray}}
\theta(s)\cdot\overline{\theta({}^{n}s)}.
\]
As the character $\theta$ is $\C^{1}$-valued, the triangle inequality and \eqref{ineq:Henniart-bullet} imply that
\begin{align*}
\langle R_{\theta},R_{\theta}\rangle^{\star}_{\circ}
&\leq
\frac{1}{|[\bbS]^{\star}|^{2}}
\cdot|N_{\bbG(\F_{q})}(\bbS)/\bbZ^{\star}_{\bbG}(\F_{q})|
\cdot|[\bbS]^{\star}_{\circ}|\\
&=
|W_{\bbG(\F_{q})}(\bbS)|
\cdot\frac{|[\bbS]^{\star}_{\circ}|}{|[\bbS]^{\star}|}
<\frac{1}{2}.
\end{align*}
Since we have $\langle R_{\theta},R_{\theta}\rangle=1$ by the irreducibility of $R_{\theta}$, this implies that $\langle R_{\theta},R_{\theta}\rangle^{\star}_{\bullet}\neq0$.
Hence there exists an element $g\in\bbG'(\F_{q})_{\bullet}$ satisfying $\Theta_{R_{\theta}}(g)\neq 0$.

Let $g\in\bbG'(\F_{q})_{\bullet}$ be such an element.
Then, since $\overline{\bbG}(\F_{q})_{\bullet}\subset\overline{\bbG}(\F_{q})_{\rs}$, $g$ is regular semisimple.
Hence Corollary \ref{cor:Kaletha-CF-ss} implies that 
\[
\Theta_{R_{\theta}}(g)
=
\frac{(-1)^{r(\bbG_{g}^\circ) - r(\bbS^{\circ})}}{|\bbG_{g}^\circ(\FF_q)|_p \cdot |[\bbS]|} \sum_{\substack{x \in [\bbG] \\ {}^x g \in \bbS(\FF_q)}} \theta({}^x g).
\]
In particular, there must exist an element ${}^{x}g$ of $\bbS(\F_{q})$ which is conjugate (by $x\in\bbG(\F_{q})$) to $g$.
Replacing $g$ with ${}^{x}g$, we get an element satisfying the desired condition.
\end{proof}

\begin{thm}\label{thm:Henniart}
Let $(\bbS,\theta) \in \tilde \cT$ be such that $\theta$ is in general position.
Assume that \eqref{ineq:Henniart-bullet} is satisfied. Then there exists a unique finite-dimensional irreducible representation $\rho$ of $\bbG(\F_{q})$ such that there exists a constant $c \in \bbC^{1}$ for which
\[
\Theta_{\rho}(g)
=
c\cdot\Theta_{R_{\bbS}^{\bbG}(\theta)}(g)
\]
for any $g\in \bbG'(\F_{q})_{\bullet}$.
Moreover, $c = \varepsilon$ and $\rho = \varepsilon R_{\bbS}^{\bbG}(\theta)$, where $\varepsilon$ is the sign such that $\varepsilon R_{\bbS}^{\bbG}(\theta)$ is a genuine representation.
\end{thm}

\begin{proof}
For notational convenience, in this proof, we write $R_\theta \colonequals R_{\bbS}^{\bbG}(\theta)$.
We may suppose that $\theta|_{\bbZ_{\bbG}(\F_{q})}$ is unitary by Lemma \ref{lem:twist}.
Let $\rho$ be an irreducible representation of $\bbG(\F_{q})$ satisfying the assumption on $\Theta_{\rho}$ as in the statement. 

We first note that the assumption implies that $\rho$ and $R_{\theta}$ have the same $\bbZ^{\star}_{\bbG}$-central character.
Indeed, by Lemma \ref{lem:Henniart}, there exists an element $g\in\bbS(\F_{q})_{\bullet}$ such that $\Theta_{R_{\theta}}(g)\neq0$.
By the assumption on $\bbZ^{\star}_{\bbG}$, we have $zg\in\bbS(\F_{q})_{\bullet}$ for any $z\in\bbZ^{\star}_{\bbG}(\F_{q})$.
Thus the assumption on $\Theta_{\rho}$ implies that $\Theta_{\rho}(g)=c\cdot\Theta_{R_{\theta}}(g)\neq0$ and $\Theta_{\rho}(zg)=c\cdot\Theta_{R_{\theta}}(zg)\neq0$.
Hence $\rho$ and $R_{\theta}$ have the same $\bbZ^{\star}_{\bbG}$-central character.
In particular, this implies that the central character of $\rho$ is unitary (recall that $\bbZ^{\star}_{\bbG}$ is of finite index in $\bbZ_{\bbG}$).

Thus, as both $\rho$ and $\varepsilon R_{\theta}$ are irreducible, it suffices to show that
\[
\langle \rho, R_{\theta}\rangle^{\star}\neq0.
\]
By the definition of the truncated inner products, we have
\begin{align*}
\langle R_{\theta},R_{\theta}\rangle^{\star}
&=
\langle R_{\theta},R_{\theta}\rangle^{\star}_{\bullet}
+
\langle R_{\theta},R_{\theta}\rangle^{\star}_{\circ}\quad\text{and}\\
\langle \rho,\rho\rangle^{\star}
&=
\langle \rho,\rho\rangle^{\star}_{\bullet}
+
\langle \rho,\rho\rangle^{\star}_{\circ}.
\end{align*}
The assumption on $\Theta_{\rho}$ implies that $\langle R_{\theta},R_{\theta}\rangle^{\star}_{\bullet}=\langle \rho,\rho\rangle^{\star}_{\bullet}$.
Thus we get $\langle R_{\theta},R_{\theta}\rangle^{\star}_{\circ}=\langle \rho,\rho\rangle^{\star}_{\circ}$.
We put
\[
X_{\bullet}\colonequals \langle R_{\theta},R_{\theta}\rangle^{\star}_{\bullet}=\langle \rho,\rho\rangle^{\star}_{\bullet}
\quad\text{and}\quad
X_{\circ}\colonequals \langle R_{\theta},R_{\theta}\rangle^{\star}_{\circ}=\langle \rho,\rho\rangle^{\star}_{\circ}.
\]
(note that $X_{\bullet}$ and $X_{\circ}$ are non-negative numbers satisfying that $X_{\bullet}+X_{\circ}=1$).

Again by the assumption on $\Theta_{\rho}$, we have
\[
\langle \rho,R_{\theta}\rangle^{\star}
= \langle \rho,R_{\theta}\rangle^{\star}_{\bullet}+\langle \rho,R_{\theta}\rangle^{\star}_{\circ}\\
= c X_{\bullet}+\langle \rho,R_{\theta}\rangle^{\star}_{\circ}.
\]
On the other hand, by the Cauchy--Schwarz inequality, we have
\[
|\langle\rho,R_{\theta}\rangle^{\star}_{\circ}|
\leq
\langle\rho,\rho\rangle^{\star\frac{1}{2}}_{\circ}\cdot\langle R_{\theta},R_{\theta}\rangle^{\star\frac{1}{2}}_{\circ}
=
X_{\circ}.
\]
Therefore, if we have $X_{\circ}<X_{\bullet}$, then $\langle \rho,R_{\theta}\rangle^{\star}$ is necessarily non-zero.
As we have $X_{\bullet}+X_{\circ}=1$, the inequality $X_{\circ}<X_{\bullet}$ holds if and only if the inequality $X_{\circ}<\frac{1}{2}$ holds.
This follows from Proposition \ref{prop:innerprod2} and the assumption \eqref{ineq:Henniart-bullet} as in the proof of Lemma \ref{lem:Henniart}.
\end{proof}

Let us focus on the special case where $\bbG=\bbG^{\circ}$ is a connected reductive group over $\F_{q}$ and $\bbZ_{\bbG}\colonequals \{1\}$ (hence $\bbS=\bbS^{\circ}$).
In this case, we have $\overline{\bbG}=\bbG/\bbZ_{\bbG}=\bbG$ and the map $\Jord$ is nothing but the usual Jordan decomposition map
\[
\Jord\colon
\bbG(\F_{q})
\xrightarrow{1:1}
\{(s,u)\mid s\in\bbG(\F_{q})_{\ss},\, u\in\bbG^{\circ}_{s}(\F_{q})_{\unip}\}
;\quad g\mapsto (s,u)
\]
(hence $\bbG(\F_{q})_{\ast}=\bbG(\F_{q})_{\ss}$).
By taking a subset $\bbG(\F_{q})_{\bullet}$ to be the regular semisimple locus  $\bbG(\F_{q})_{\rs}$, we obtain the following from Theorem \ref{thm:Henniart}:

\begin{cor}\label{cor:Henniart}
Let $(\bbS,\theta) \in \tilde \cT$ be such that $\theta$ is in general position.
Assume that the following inequality is satisfied:
\[\label{ineq:Henniart-rs}
\frac{|\bbS(\F_{q})|}{|\bbS(\F_{q}) \smallsetminus \bbS(\FF_q)_{\rs}|} > 2\cdot|W_{\bbG(\F_{q})}(\bbS)|. 
\tag{$\mathfrak{H}_\rs$}
\]
Then there exists a unique finite-dimensional irreducible representation $\rho$ of $\bbG(\F_{q})$ such that there exists a constant $c \in \bbC^{1}$ for which
\[
\Theta_{\rho}(g)
=
c\cdot\Theta_{R_{\bbS}^{\bbG}(\theta)}(g)
\]
for any $g\in \bbG(\F_{q})_{\rs}$.
Moreover, $c$ equals $\varepsilon$ and $\rho$ equals $\varepsilon R_{\bbS}^{\bbG}(\theta)$, where $\varepsilon$ is a sign such that $\varepsilon R_{\bbS}^{\bbG}(\theta)$ is a genuine representation.
\end{cor}

The inequality \eqref{ineq:Henniart-rs} can be explicated as long as $\bbG$ and $\bbS$ are given explicitly.
See Section \ref{subsec:Henn-ur} for an explicit computation in some particular cases where $\bbG$ is a split simple group and $\bbS$ is an elliptic maximal torus of Coxeter type.

\subsection{Lusztig's map $E$ and a refinement}\label{subsec:Lusztig-E}

In this section, we extend Lusztig's results \cite{Lus20} for connected reductive groups to our slightly more general setting $\bbG$. Because of the foundational results on the representation theory of $\bbG(\FF_q)$ provided in Section \ref{sec:finite}, the proofs in this section work out to be direct extensions of Lusztig's ideas. These results will play an important role in our characterization theorems for supercuspidal representations of $p$-adic groups (Section \ref{sec:sc characterization}).

For each $\rho \in \cR(\bbG(\FF_q))$, let
\begin{align*}
    \tilde Z_\rho &\colonequals \{(\bbS, \theta, n) \in \tilde \cT \times (\bbZ \smallsetminus \{0\}) \mid n = \langle \rho, R_{\bbS}^{\bbG}(\theta) \rangle\}.
\end{align*}
Define
\begin{equation*}
    \tilde E \from \Irr(\bbG(\FF_q)) \to \cP(\tilde \cT \times \bbZ); \quad \rho \mapsto \tilde Z_\rho,
\end{equation*}
where $\cP(A)$ denote the power set of $A$.

\subsubsection{Lusztig's map $E$}

\begin{defn}[geometric conjugacy]\label{defn:geom-conj}
Let $(\bbS_{1},\theta_{1}), (\bbS_{2},\theta_{2})\in\tilde{\cT}$.
We say that $(\bbS_{1},\theta_{1})$ and $(\bbS_{2},\theta_{2})$ are \textit{geometrically conjugate in $\bbG$} if $\theta_{1}^{\circ}$ and $\theta_{2}^{\circ}$ are geometrically conjugate in $\bbG$ in the sense of Deligne--Lusztig \cite[Definition 5.5]{DL76}; i.e., there exists a finite extension $\F_{q^{n}}$ of $\F_{q}$ such that the character $\theta_{1}^{\circ}\circ\Nr_{\F_{q^{n}}/\F_{q}}$ of $\bbS_{1}^{\circ}(\F_{q^{n}})$ and the character $\theta_{2}^{\circ}\circ\Nr_{\F_{q^{n}}/\F_{q}}$ of $\bbS_{2}^{\circ}(\F_{q^{n}})$ are $\bbG(\F_{q^{n}})$-conjugate.
\end{defn}

Let us describe a coarser version of the map $\tilde E$.
We first establish the following lemma, which follows easily from a classical result of Deligne--Lusztig \cite[Theorem 6.2]{DL76}.

\begin{lemma}\label{lem:not-geom-conj}
If $(\bbS_{1},\theta_{1}),(\bbS_{2},\theta_{2}) \in \tilde \cT$ are not geometrically conjugate in $\bbG$, then no element of $\Irr(\bbG(\FF_q))$ can occur in both virtual representations $R_{\bbS_{1}}^{\bbG}(\theta_{1})$ and $R_{\bbS_{2}}^{\bbG}(\theta_{2})$.
\end{lemma}

\begin{proof}
Let $\rho \in \Irr(\bbG(\FF_q))$ be such that $\langle \rho, R_{\bbS_{1}}^{\bbG}(\theta_{1}) \rangle$ and $\langle \rho, R_{\bbS_{2}}^{\bbG}(\theta_{2})\rangle$ are both nonzero.
Since $R_{\bbS_{1}}^{\bbG}(\theta_{1}) = \Ind_{\bbG'(\FF_q)}^{\bbG(\FF_q)}(R_{\bbS_{1}}^{\bbG'}(\theta_{1}))$ by Lemma \ref{lem:alg desc}, then by Frobenius reciprocity we know that $\langle \rho|_{\bbG'(\FF_q)}, R_{\bbS_{1}}^{\bbG'}(\theta) \rangle$ is also nonzero, and similarly for $\langle \rho|_{\bbG'(\FF_q)}, R_{\bbS_{2}}^{\bbG'}(\theta_{2}) \rangle$. 
Let $\rho' \in \Irr(\bbG'(\FF_q))$ be an irreducible constituent of $\rho|_{\bbG'(\FF_q)}$ such that $\langle \rho', R_{\bbS_{1}}^{\bbG'}(\theta_{1}) \rangle \neq 0$. 
Since $\bbG'(\FF_q)$ is a normal subgroup of $\bbG(\FF_q)$, then we know all the irreducible $\bbG'(\FF_q)$-subrepresentations of $\rho|_{\bbG'(\FF_q)}$ are $\bbG(\FF_q)$-conjugate.
In particular, we know that there exists an element $g \in \bbG(\FF_q)$ such that $\langle \rho'{}^{g^{-1}}, R_{\bbS_{2}}^{\bbG'}(\theta_{2}) \rangle \neq 0$.
This implies that $\langle \rho', R_{\bbS_{2}^{g}}^{\bbG'}(\theta_{2}^g) \rangle \neq 0$.  

Recall from Lemma \ref{lem:alg desc} that $R_{\bbS_1}^{\bbG'}(\theta_1)$ is an extension of the $\bbG^\circ(\FF_q)$-representation $R_{\bbS_1^\circ}^{\bbG^\circ}(\theta_1^\circ)$. 
In fact, as explained in \cite[Remark 2.6.5]{Kal19-sc}, one may define actions of $\bbS(\FF_q) \times \bbG'(\FF_q)$ on $Y_{\bbU_1}^{\bbG^\circ}$ so that for each $i$, $H_c^i(Y_{\bbU_1}^{\bbG^\circ}, \overline \QQ_\ell) \cong H_c^i(Y_{\bbU_1}^{\bbG'}, \overline \QQ_\ell)$ as representations of $\bbS(\FF_q) \times \bbG'(\FF_q)$.
(Here, $\bbU_1$ is the unipotent radical of a Borel subgroup of $\bbG^{\circ}$ containing $\bbS^{\circ}$.)
Hence the condition that $\langle \rho', R_{\bbS_1}^{\bbG'}(\theta_1) \rangle \neq 0$ implies that there exists an $i \in \bbZ_{\geq 0}$ such that $\langle \rho', H_c^i(Y_{\bbU_1}^{\bbG^\circ}, \overline \QQ_\ell)[\theta_1^\circ]\rangle \neq 0$. 
This furthermore implies that $\langle \rho'|_{\bbG^\circ(\FF_q)}, H_c^i(Y_{\bbU_1}^{\bbG^\circ}, \overline \QQ_\ell)[\theta_1^\circ]\rangle \neq 0$.
Similarly, there also exists $j \in \bbZ_{\geq 0}$ such that $\langle \rho'|_{\bbG^\circ(\FF_q)}, H_c^i(Y_{\bbU_2^{g}}^{\bbG^\circ}, \overline \QQ_\ell)[\theta_2^{g\circ}]\rangle \neq 0$.

Now we utilize \cite[Corollary 6.3]{DL76}; since $H_{c}^{i}(Y_{\bbU_{1}}^{\bbG^\circ},\overline{\Q}_{\ell})[\theta_{1}^{\circ}]$ and $H_{c}^{j}(Y_{\bbU_{2}^{g}}^{\bbG^\circ},\overline{\Q}_{\ell})[\theta_{2}^{g\circ}]$ contain the same irreducible representation (any constituent of $\rho'|_{\bbG^{\circ}(\F_{q})}$), we see that $\theta_{1}^{\circ}$ and $\theta_{2}^{g\circ}$ are geometrically conjugate in $\bbG^{\circ}$.
Hence, $\theta_{1}^{\circ}$ and $\theta_{2}^{\circ}$ are geometrically conjugate in $\bbG$.
(Note that here we used a statement slightly stronger than \cite[Corollary 6.3]{DL76} because \cite[Corollary 6.3]{DL76} is stated for alternating sums $R_{\bbS_{1}^{\circ}}^{\bbG^{\circ}}(\theta_{1}^{\circ})$ and $R_{\bbS_{2}^{g\circ}}^{\bbG^{\circ}}(\theta_{2}^{g\circ})$.
However, the proof of \cite[Theorem 6.2]{DL76} in fact shows the aforementioned stronger version for individual degrees.)
\end{proof}

Denote by $\sim$ the equivalence relation on $\tilde \cT$ obtained by geometric conjugation. What Lemma \ref{lem:not-geom-conj} implies is that for any $\rho \in \Irr(\bbG(\FF_q))$, the set
\begin{equation*}
    Z_\rho \colonequals \{(\bbS,\theta) \in \tilde \cT \mid \langle \rho, R_{\bbS}^{\bbG}(\theta) \rangle \neq 0\} \subset \tilde \cT
\end{equation*}
is contained in a \textit{single} equivalence class of $\tilde \cT$. Hence we have a well-defined map
\begin{equation*}
    E \from \Irr(\bbG(\FF_q)) \to \tilde \cT/{\sim}; \qquad \rho \mapsto Z_\rho.
\end{equation*}
Note that $E$ can be obtained from $\tilde E$ by forgetting the $\bbZ$-component and descending to $\tilde \cT/{\sim}$.

\begin{definition}\label{def:cuspidal}
We say $\rho\in\Irr(\bbG(\FF_q))$ is \textit{cuspidal} if the restriction $\rho|_{\bbG^\circ(\FF_q)}$ contains a cuspidal representation of $\bbG^\circ(\FF_q)$.
For any $\rho'\in\Irr(\bbG'(\FF_q))$, we define the cuspidality of $\rho'$ in the same way.
\end{definition}

\begin{rem}\label{rem:cuspidal}
Note that, as $\bbG^{\circ}(\F_q)$ is normal in $\bbG(\F_q)$, any two irreducible subrepresentations of $\rho|_{\bbG^{\circ}(\F_q)}$ are $\bbG(\F_{q})$-conjugate.
In particular, since the $\bbG(\F_{q})$-conjugation preserves the cuspidality of representations of $\bbG^{\circ}(\F_q)$, $\rho\in\Irr(\bbG(\FF_q))$ is cuspidal if and only if any irreducible constituent of the restriction $\rho|_{\bbG^\circ(\F_q)}$ is a cuspidal representation of $\bbG^\circ(\FF_q)$.
The same is true for the cuspidality of representations of $\bbG'(\F_{q})$.
\end{rem}

\begin{lem}\label{lem:cuspidal}
If $\rho\in\Irr(\bbG(\FF_q))$ is cuspidal, then $\langle\rho,R_{\bbS}^{\bbG}(\theta)\rangle=0$ for all $(\bbS,\theta) \in \tilde{\cT}$ such that $\bbS^\circ$ is not elliptic in $\bbG^\circ$.
\end{lem}

\begin{proof}
Let $\rho\in\Irr(\bbG(\FF_q))$ be a cuspidal representation.
Suppose that $(\bbS,\theta) \in \tilde{\cT}$ satisfies $\langle\rho,R_{\bbS}^{\bbG}(\theta)\rangle\neq0$.
Let us show that $\bbS^{\circ}$ is elliptic in $\bbG^{\circ}$.
By the same arguments and the same notations as in the proof of \ref{lem:not-geom-conj}, we see that there exists an irreducible cuspidal representation $\rho'\in\Irr(\bbG'(\FF_q))$ which is an irreducible constituent of $\rho|_{\bbG'(\FF_q)}$ and satisfies $\langle \rho'|_{\bbG^{\circ}(\F_{q})}, H_{c}^{i}(Y^{\bbG^{\circ}}_{\bbU},\overline{\Q}_{\ell})[\theta^{\circ}] \rangle \neq 0$.
Thus there exists an irreducible cuspidal representation $\rho^{\circ}$ (any irreducible constituent of $\rho'|_{\bbG^{\circ}(\F_{q})}$) satisfying $\langle \rho^{\circ}, H_{c}^{i}(Y^{\bbG^{\circ}}_{\bbU},\overline{\Q}_{\ell})[\theta^{\circ}] \rangle \neq 0$.

Now suppose that $\bbS^{\circ}$ is not elliptic in $\bbG^{\circ}$ for the sake of contradiction.
Then we can find a proper parabolic subgroup $\mathbb{P}^{\circ}$ of $\bbG^{\circ}$ with Levi subgroup $\mathbb{M}^{\circ}$ satisfying $\bbS^{\circ}\subset\mathbb{M}^{\circ}$.
By \cite[Proposition 8.2]{DL76}, we have $H_{c}^{i}(Y^{\bbG^{\circ}}_{\bbU},\overline{\Q}_{\ell})[\theta^{\circ}]\cong \Ind_{\mathbb{P}^{\circ}(\F_{q})}^{\bbG^{\circ}(\F_{q})}H_{c}^{i}(Y^{\mathbb{M}^{\circ}}_{\bbU\cap\mathbb{M}^{\circ}},\overline{\Q}_{\ell})[\theta^{\circ}]$.
(Here we give a similar remark to the one given in the proof of Lemma \ref{lem:not-geom-conj}; although \cite[Proposition 8.2]{DL76} is stated for the alternating sums $R_{\bbS^{\circ}}^{\bbG^{\circ}}(\theta^{\circ})$ and $R_{\bbS^{\circ}}^{\mathbb{M}^{\circ}}(\theta^{\circ})$, we can check that the same is true for individual degrees by looking at the proof of \cite[Proposition 8.2]{DL76}.)
Thus we get $\langle \rho^{\circ}, \Ind_{\mathbb{P}^{\circ}(\F_{q})}^{\bbG^{\circ}(\F_{q})}H_{c}^{i}(Y^{\mathbb{M}^{\circ}}_{\bbU\cap\mathbb{M}^{\circ}},\overline{\Q}_{\ell})[\theta^{\circ}] \rangle \neq 0$.
However, this cannot be true since $\rho^{\circ}$ is cuspidal.
\end{proof}

\subsubsection{A non-cohomological definition of $\tilde E$}

In this section, we work under the same setting as Section \ref{subsec:char-finite-gen-pos}.
In particular, we have a subset $\overline{\bbG}(\FF_q)_\bullet$ of $\overline{\bbG}(\FF_q)_{\rs}$.
The main result of this section is that if $\bbG'(\FF_q)_\bullet$ is sufficiently large, then $\Theta_\rho|_{\bbG'(\FF_q)_\bullet}$ determines $\tilde Z_\rho$.

For any $\F_{q}$-rational maximal torus $\bbS^{\circ}$ of $\bbG^{\circ}$, we consider the following variant of the inequality \eqref{ineq:Henniart-bullet}:
\[\label{ineq:Lusztig-bullet}
\frac{|[\bbS]^{\star}|}{|[\bbS]^{\star}_{\circ}|}
=
\frac{|[\bbS]^{\star}|}{|[\bbS]^{\star} \smallsetminus [\bbS]^{\star}_{\bullet}|}
>
2^{2|W_{\bbG}|\cdot i(\bbS)-1},
\tag{$\mathfrak{L}_\bullet$}
\]
where $W_{\bbG}\colonequals N_{\bbG}(\bbS)/\bbS$ denotes the absolute Weyl group of $\bbG$ (note that its order $|W_{\bbG}|$ does not depend on the choice of a maximal torus $\bbS^{\circ}$) and we put $i(\bbS):=[\bbS(\F_{q}):\bbS^{\circ}(\F_{q})\bbZ_{\bbG}^{\star}(\F_{q})]$.

Recall that, for any $\rho \in \cR(\bbG(\FF_q))$, we put
\[
\tilde{Z}_{\rho}
\colonequals
\{(\bbS, \theta, n) \in \tilde{\cT}\times(\bbZ \smallsetminus \{0\}) \mid n = \langle \rho, R_{\bbS}^{\bbG}(\theta) \rangle\}.
\]
When $\tilde{Z}$ is a subset of $\tilde{\cT}\times(\bbZ \smallsetminus \{0\})$, for any $\F_{q}$-rational maximal torus $\bbS^{\circ}$ of $\bbG^{\circ}$, we define $\tilde{Z}_{\bbS}$ to be the subset of elements of $\tilde{Z}$ whose maximal torus is given by $\bbS$.

In this section we will prove:

\begin{theorem}\label{thm:unique-Z-rho}
Assume that \eqref{ineq:Lusztig-bullet} holds for every $\F_{q}$-rational maximal torus $\bbS^{\circ}$ of $\bbG^{\circ}$.
Then $\tilde Z_\rho$ is the unique set of triples $(\bbS,\theta,n) \in \tilde \cT \times (\bbZ \smallsetminus \{0\})$ such that
\begin{itemize}
\item
all characters $\theta$ appearing in $\tilde{Z}_{\rho,\bbS}$ are pairwise distinct for any $\bbS$,
\item
$|\tilde{Z}_{\rho,\bbS}|\leq |W_{\bbG}|\cdot i(\bbS)$ for any $\bbS$, and
\item
for any $s \in \bbG'(\FF_q)_\bullet$,
\[
\Theta_{\rho}(s)
=
\sum_{\substack{(\bbS,\theta,n) \in \tilde Z_\rho \\ s \in \bbS}} n \cdot \theta(s).
\]
\end{itemize}
\end{theorem}

Hence Theorem \ref{thm:unique-Z-rho} gives a non-cohomological description of the map $\tilde E$. Note that Theorem \ref{thm:unique-Z-rho} implies the following:

\begin{cor}\label{cor:tilde-Z}
Assume that \eqref{ineq:Lusztig-bullet} holds for every $\F_{q}$-rational maximal torus $\bbS^{\circ}$ of $\bbG^{\circ}$.
For any $\rho,\rho' \in \cR(\bbG(\FF_q))$,
    \begin{equation*}
        \Theta_\rho(g) = \Theta_{\rho'}(g) \qquad \text{for all $g \in \bbG'(\FF_q)_{\bullet}$,}
    \end{equation*}
    if and only if
    \begin{equation*}
        \tilde Z_\rho = \tilde Z_{\rho'}, \qquad \text{i.e., $\tilde E(\rho) = \tilde E(\rho').$}
    \end{equation*}
\end{cor}

We follow Lusztig's strategy in the connected case, see \cite{Lus20}.
We will need a mild generalization of \cite[Lemma 8.1]{Lus90}, which is a refinement of Dedekind's theorem.

\begin{lemma}\label{lem:Lusztig-Dedekind}
Let $\Gamma$ be a group with a central subgroup $Z$ such that $\Gamma/Z$ is finite.
Let $\psi_1, \ldots, \psi_n \from \Gamma \to K^\times$ be distinct homomorphisms to the multiplicative group of a field $K$.
Let $\Gamma_{\bullet}$ be a subset of $\Gamma$ stable under multiplication by $Z$ such that 
    \begin{equation*}
        |\Gamma/Z| > 2^{n-1} \cdot |(\Gamma/Z) \smallsetminus (\Gamma_{\bullet}/Z)|.
    \end{equation*} 
    Then the restrictions of $\psi_1, \ldots, \psi_n$ to $\Gamma_{\bullet}$ are linearly independent as functions $\Gamma_{\bullet} \to K$.
\end{lemma}

\begin{proof}   
    This proof is completely identical to the proof of \cite[Lemma 8.1]{Lus90}. Note that the case $n = 1$ is trivial. 
        
    We proceed by induction on $n$. Suppose for a contradiction that we have a linear relation $a_i \psi_1(x) + \cdots + a_n \psi_n(x) = 0$ for all $x \in \Gamma_{\bullet}$, with the $a_i \in K$, not all zero. We may assume $a_1 \neq 0$. Since $\psi_1$ is valued in $K^\times$, we may also assume $a_2 \neq 0$. Since $\psi_1 \neq \psi_2$ by assumption, we can choose a $\gamma \in \Gamma$ such that $\psi_1(\gamma) \neq \psi_2(\gamma)$. 

    We now have $\sum_{i=1}^n a_i \psi_i(x) \psi_1(\gamma) = 0$ for all $x \in \Gamma_{\bullet}$ and $\sum_{i=1}^n a_i \psi_i(x\gamma) = 0$ for all $x \in \Gamma_{\bullet} \gamma^{-1}$. This implies that
    \begin{equation*}
        \sum_{i=2}^n a_i(\psi_1(\gamma) - \psi_i(\gamma)) \psi_i(x) = 0 \qquad \text{for all $x \in \Gamma_{\bullet} \cap \Gamma_{\bullet} \gamma^{-1}$}.
    \end{equation*}
    Since $a_2(\psi_1(\gamma) - \psi_2(\gamma)) \neq 0$ by assumption, this is a nontrivial linear relation for $\psi_2, \ldots \psi_n$ on $\Gamma_{\bullet} \cap \Gamma_{\bullet} \gamma^{-1}$. Since both $\Gamma_{\bullet},\Gamma_{\bullet}\gamma^{-1}$ are stable under multiplication by $Z$, so must their intersection. We have
    \begin{equation*}
        \frac{|\Gamma/Z|}{|(\Gamma/Z) \smallsetminus ((\Gamma_{\bullet} \cap \Gamma_{\bullet} \gamma^{-1})/Z)|}
        \geq \frac{|\Gamma/Z|}{2|(\Gamma/Z) \smallsetminus (\Gamma_{\bullet}/Z)|} > 2^{n-2},
    \end{equation*}
    which contradicts the induction hypothesis.
\end{proof}

We are now ready to prove Theorem \ref{thm:unique-Z-rho}.

\begin{proof}[Proof of Theorem \ref{thm:unique-Z-rho}]
We follow Lusztig's proof in the connected case \cite{Lus20}.
Let $\rho \in \cR(\bbG(\FF_q))$.
By Proposition \ref{prop:rho ss}, for any $s \in \bbG'(\F_{q})_{\bullet}$, we have
\[
\Theta_{\rho}(s)
=
\sum_{\substack{(\bbS,\theta,n) \in \tilde Z_\rho \\ s \in \bbS}} n \cdot \theta(s).
\]
If we let $\tilde{Z}_{\rho,\bbS}=\{(\bbS,\theta_{1},n_{1}),\ldots,(\bbS,\theta_{r},n_{r})\}$ for any $\F_{q}$-rational maximal torus $\bbS^{\circ}$ of $\bbG^{\circ}$, then $\theta_{i}$'s are pairwise distinct by definition. 
Furthermore, we have $r \leq |W_{\bbG}|\cdot i(\bbS)$. 
Indeed, by Lemma \ref{lem:not-geom-conj}, there exists an $n \geq 1$ such that the $\bbS^{\circ}(\FF_{q^n})$-characters $\theta_1^{\circ} \circ \Nr_{\F_{q^{n}}/\F_{q}}, \theta_2^{\circ} \circ \Nr_{\F_{q^{n}}/\F_{q}}, \ldots, \theta_r^{\circ} \circ \Nr_{\F_{q^{n}}/\F_{q}}$ are all contained in a single orbit under the action of the normalizer of $\bbS$ in $\bbG(\FF_{q^n})$. 
Since the norm map $\Nr_{\F_{q^{n}}/\F_{q}}\colon\bbS^{\circ}(\F_{q^{n}})\rightarrow\bbS^{\circ}(\F_{q})$ is surjective, we know that there are at most $|W_{\bbG}|$ possibilities of $\theta^{\circ}$ for $(\bbS,\theta,n)\in \tilde{Z}_{\rho,\bbS}$.
Moreover, for any $(\bbS,\theta,n)\in \tilde{Z}_{\rho,\bbS}$, the $\bbZ_{\bbG}^{\star}$-central character of $R_{\bbS}^{\bbG}(\theta)$ must be the same as that of $\rho$.
Thus all the restrictions $\theta_{i}|_{\bbZ_{\bbG}^{\star}(\F_{q})}$ are the same.
Hence there are at most $|W_{\bbG}|$ possibilities of $\theta|_{\bbS^{\circ}(\F_{q})\bbZ_{\bbG}^{\star}(\F_{q})}$ for $(\bbS,\theta,n)\in \tilde{Z}_{\rho,\bbS}$, which implies $r \leq |W_{\bbG}|\cdot i(\bbS)$.

Now suppose that there exists another subset $\tilde{Z}'_{\rho}\subset \cT\times (\bbZ \smallsetminus \{0\})$ satisfying the same condition as $\tilde{Z}_{\rho}$.
Our task is to show that $\tilde{Z}'_{\rho}=\tilde{Z}_{\rho}$.
It is enough to show that $\tilde{Z}_{\rho,\bbS}=\tilde{Z}'_{\rho,\bbS}$ by fixing any $\F_{q}$-rational maximal torus $\bbS^{\circ}$ of $\bbG^{\circ}$.
Let us write $\tilde{Z}'_{\rho,\bbS}=\{(\bbS,\theta'_{1},n'_{1}),\ldots,(\bbS,\theta'_{m},n'_{m})\}$.
Then $\theta'_{i}$'s are pairwise distinct and $m \leq |W_{\bbG}|\cdot i(\bbS)$.
Moreover, for any $s\in\bbS(\F_{q})_{\bullet}$,
\[
\Theta_{\rho}(s)
=
n'_{1}\theta'_{1}(s)+\cdots+n'_{m}\theta'_{m}(s),
\]
hence we have
\[
n_{1}\theta_{1}(s)+\cdots+n_{r}\theta_{r}(s)
=
n'_{1}\theta'_{1}(s)+\cdots+n'_{m}\theta'_{m}(s).
\]
Let us suppose that $\{(n_{1},\theta_{1}),\ldots,(n_{r},\theta_{r})\}\neq\{(n'_{1},\theta'_{1}),\ldots,(n'_{m},\theta'_{m})\}$ for a contradiction.
Then, the above equality gives a linear dependence between $\phi_{1},\ldots,\phi_{l}$ viewed as functions on $\bbS(\F_{q})_{\bullet}$, where $\{\phi_{1},\ldots,\phi_{l}\}$ is a nonempty subset of $\{\theta_{1},\ldots,\theta_{r}\}\cup\{\theta'_{1},\ldots,\theta'_{m}\}$.
On the other hand, since \eqref{ineq:Lusztig-bullet} holds by assumption, this contradicts Lemma \ref{lem:Lusztig-Dedekind} applied to the setting $\Gamma = \bbS(\FF_q)$, $Z = \bbZ^{\star}_{\bbG}(\FF_q)$, $\Gamma_{\bullet} = \bbS(\FF_q)_\bullet$, and $n = 2|W_{\bbG}|\cdot i(\bbS)$ (since $l\leq m+r \leq 2|W_{\bbG}|\cdot i(\bbS)$). 
(See the beginning of Section \ref{subsec:DLCF} for the definition of $\bbZ^{\star}_{\bbG}$.)
Indeed, by Lemma \ref{lem:Lusztig-Dedekind}, we must have
\[
|[\bbS]^{\star}|
\leq
2^{l-1}\cdot
|[\bbS]^{\star}\smallsetminus[\bbS]^{\star}_{\bullet}|
\leq
2^{2|W_{\bbG}|\cdot i(\bbS)-1}\cdot
|[\bbS]^{\star}\smallsetminus[\bbS]^{\star}_{\bullet}|
\]
since $\{\phi_{1},\ldots,\phi_{l}\}$ are linear independent.
This contradicts to \eqref{ineq:Lusztig-bullet}.
\end{proof}

The proof of Theorem \ref{thm:unique-Z-rho} can be used to prove several interesting corollaries.

\begin{cor}\label{cor:unique-Z-rho-cuspidal}
Assume that \eqref{ineq:Lusztig-bullet} holds for every $\F_{q}$-rational maximal torus $\bbS^{\circ}$ of $\bbG^{\circ}$.
If $\rho$ is cuspidal, then $\tilde Z_\rho$ is the unique set of triples $(\bbS,\theta,n) \in \tilde \cT \times (\bbZ \smallsetminus \{0\})$ whose $\bbS^{\circ}$ is elliptic and satisfying the same assumptions as in Theorem \ref{thm:unique-Z-rho}.
\end{cor}

\begin{proof}
    By Lemma \ref{lem:cuspidal}, we know that $\langle \rho, R_\bbS^\bbG(\theta) \rangle = 0$ for all $(\bbS, \theta) \in \tilde \cT$ where $\bbS^\circ \subset \bbG^\circ$ is not elliptic. Hence to determine $\tilde Z_\rho$, it remains only to apply the proof of Theorem \ref{thm:unique-Z-rho} to determine $\tilde{Z}_{\rho,\bbS}$ for $\bbS^\circ \subset \bbG^\circ$ elliptic.
\end{proof}

\begin{defn}[unipotent representation]\label{defn:unip}
We say that $\rho \in \Irr(\bbG(\FF_q))$ is \textit{unipotent} if $\langle \rho, R_\bbS^\bbG(\theta) \rangle \neq 0$ for some $(\bbS, \theta) \in \tilde \cT$ satisfying $\theta^{\circ}=\mathbbm{1}$.
\end{defn}

\begin{rem}\label{rem:unipotent}
When $\rho\in\Irr(\bbG(\FF_q))$ is unipotent, the restriction $\rho|_{\bbG^{\circ}(\F_{q})}$ contains an irreducible constituent $\rho^{\circ}$ which is unipotent in the sense of Deligne--Lusztig, i.e., $\langle\rho^{\circ},R_{\bbS^{\circ}}^{\bbG^{\circ}}(\mathbbm{1})\rangle\neq0$ for some $\F_{q}$-rational maximal torus $\bbS^{\circ}$ of $\bbG^{\circ}$.
Indeed, let us suppose that $\rho\in\Irr(\bbG(\FF_q))$ satisfies $\langle\rho,R_{\bbS}^{\bbG}(\theta)\rangle\neq0$ for $(\bbS,\theta)\in\tilde\cT$ satisfying $\theta^{\circ}=\mathbbm{1}$.
Then, by the same argument as in the proof of Lemma \ref{lem:cuspidal}, there is an irreducible constituent $\rho^{\circ}$ of $\rho|_{\bbG^{\circ}(\F_{q})}$ satisfying $\langle\rho^{\circ}, H^{i}_{c}(X^{\bbG^{\circ}}_{\bbS^{\circ}},\overline{\Q}_{\ell})[\mathbbm{1}]\rangle\neq0$.
By the exhaustion theorem of Deligne--Lusztig (\cite[Corollary 7.7]{DL76}), there exists an $\F_{q}$-rational maximal torus $\bbS^{\prime\circ}$ and a character $\theta^{\prime\circ}$ of $\bbS^{\prime\circ}(\F_{q})$ satisfying $\langle\rho^{\circ}, R_{\bbS^{\prime\circ}}^{\bbG^{\circ}}(\theta^{\prime\circ})\rangle\neq0$.
Again by the same argument as in the proof of Lemma \ref{lem:not-geom-conj}, we see that $(\bbS^{\prime\circ},\theta^{\prime\circ})$ must be geometrically conjugate to $(\bbS^{\circ},\mathbbm{1})$ in $\bbG^{\circ}$; in particular, $\theta^{\prime\circ}=\mathbbm{1}$.
Thus $\rho^{\circ}$ is unipotent.
\end{rem}

\begin{cor}\label{cor:ff unipotent}
Assume that \eqref{ineq:Lusztig-bullet} holds for every $\F_{q}$-rational maximal torus $\bbS^{\circ}$ of $\bbG^{\circ}$.
Then $\rho \in \Irr(\bbG(\FF_q))$ is unipotent if and only if $\Theta_\rho|_{\bbS^{\circ}(\FF_q)_\bullet}$ is constant for every maximal torus $\bbS^\circ \subset \bbG^\circ$.
\end{cor}

\begin{proof}
    By Lemma \ref{lem:not-geom-conj}, we know that if $\rho$ is a unipotent representation, then $\tilde{Z}_{\rho}$ must only consist of pairs of the form $(\bbS,\theta,n)$ satisfying $\theta^{\circ}=\mathbbm{1}$. 
    Hence by Proposition \ref{prop:rho ss}, if $\rho \in \Irr(\bbG(\FF_q))$ is unipotent, then $\Theta_\rho|_{\bbS^{\circ}(\FF_q)_\bullet}$ is constant for every maximal torus $\bbS^\circ \subset \bbG^\circ$.

    Now let $\rho \in \Irr(\bbG(\FF_q))$ be such that $\Theta_\rho|_{\bbS^{\circ}(\FF_q)_\bullet}$ is constant for every maximal torus $\bbS^\circ \subset \bbG^\circ$. By Corollary \ref{cor:exhaustion}, we know that $\tilde{Z}_{\rho}$ is nonempty; let $\tilde{Z}_{\rho,\bbS}\neq\varnothing$. Then by Proposition \ref{prop:rho ss}, we have
    \begin{equation*}
        \Theta_\rho(s) = \sum_{(\bbS,\theta,n) \in \tilde{Z}_{\rho,\bbS}} \theta(s) \cdot n= c \qquad \text{for all $s \in \bbS^{\circ}(\FF_q)_\bullet$}.
    \end{equation*}
    Of course $c = c \cdot \mathbbm{1}(s)$ for all $s \in \bbS^{\circ}(\FF_q)_\bullet$.
We let $\tilde{Z}_{\rho,\bbS}=\{(\bbS,\theta_{1},n_{1}),\ldots(\bbS,\theta_{r},n_{r})\}$ and apply Lemma to $\Gamma=\bbS^{\circ}(\F_{q})$, $\Gamma_{\bullet}=\bbS^{\circ}(\F_{q})_{\bullet}$, $Z=\{1\}$ and $n=r+1$.
For this, let us check that the inequality $|\bbS^{\circ}(\F_{q})|>2^{r}\cdot|\bbS^{\circ}(\F_{q})_{\circ}|$ is satisfied.
By putting $\bbZ_{\bbG}^{\star\circ}(\F_{q}):=\bbZ_{\bbG}^{\star}(\F_{q})\cap\bbS^{\circ}(\F_{q})$, we have
\[
\frac{|\bbS^{\circ}(\F_{q})|}{|\bbS^{\circ}(\F_{q})_{\circ}|}
=
\frac{|\bbS^{\circ}(\F_{q})/\bbZ_{\bbG}^{\star\circ}(\F_{q})|}{|\bbS^{\circ}(\F_{q})_{\circ}/\bbZ_{\bbG}^{\star\circ}(\F_{q})|}
\geq
\frac{i(\bbS)^{-1}\cdot|[\bbS]^{\star}|}{|[\bbS]^{\star}_{\circ}|}
>
2^{2|W_{\bbG}|\cdot i(\bbS)-1}\cdot i(\bbS)^{-1}.
\]
Here, 
\begin{itemize}
\item
in the first inequality, we used $i(\bbS)\cdot|\bbS^{\circ}(\F_{q})/\bbZ_{\bbG}^{\star\circ}(\F_{q})|=|[\bbS]^{\star}|$ for the numerator and $|\bbS^{\circ}(\F_{q})_{\circ}/\bbZ_{\bbG}^{\star\circ}(\F_{q})|\leq |[\bbS]^{\star}_{\circ}|$ for the denominator, and
\item
in the second inequality, we used \eqref{ineq:Lusztig-bullet}.
\end{itemize}
We note that 
\[
2^{2|W_{\bbG}|\cdot i(\bbS)-1}\cdot i(\bbS)^{-1}
=2^{|W_{\bbG}|\cdot i(\bbS)}\cdot 2^{|W_{\bbG}|\cdot i(\bbS)-1}\cdot i(\bbS)^{-1}.
\]
By the proof of Theorem \ref{thm:unique-Z-rho}, we have $r\leq|W_{\bbG}\cdot i(\bbS)|$.
Moreover, it can be easily checked that $2^{|W_{\bbG}|\cdot i(\bbS)-1}\cdot i(\bbS)^{-1}\geq1$.
Hence $2^{2|W_{\bbG}|\cdot i(\bbS)-1}\cdot i(\bbS)^{-1}\geq 2^{r}$, which implies that $|\bbS^{\circ}(\F_{q})|>2^{r}\cdot|\bbS^{\circ}(\F_{q})_{\circ}|$ as desired.
Therefore, by Lemma \ref{lem:Lusztig-Dedekind}, any $r+1$ distinct characters of $\bbS^{\circ}(\F_{q})$ are linearly independent on $\bbS^{\circ}(\F_{q})_{\bullet}$.
This implies that we there necessarily exists $\theta_{i}$ such that $\theta_{i}^{\circ}=\mathbbm{1}$ (note that then any $\theta_{i}$ satisfies $\theta_{i}^{\circ}=\mathbbm{1}$ by Lemma \ref{lem:not-geom-conj}).
Hence $\rho$ is unipotent.
\end{proof}

\begin{cor}\label{cor:ff unipotent cuspidal}
Assume that \eqref{ineq:Lusztig-bullet} holds for every $\F_{q}$-rational maximal torus $\bbS^{\circ}$ of $\bbG^{\circ}$.
Then a cuspidal representation $\rho \in \Irr(\bbG(\FF_q))$ is unipotent if and only if $\Theta_\rho|_{\bbS^{\circ}(\FF_q)_\bullet}$ is constant for every elliptic maximal torus $\bbS^\circ \subset \bbG^\circ$.
\end{cor}

\begin{proof}
    Let $\rho \in \Irr(\bbG(\FF_q))$ be a cuspidal representation such that $\Theta_\rho|_{\bbS^{\circ}(\FF_q)_\bullet}$ is constant for every elliptic maximal torus $\bbS^\circ \subset \bbG^\circ$.  Since $\rho$ is cuspidal, by Lemma \ref{lem:cuspidal} we have that if $(\bbS,\theta,n) \in \tilde{Z}_\rho$, then $\bbS$ is elliptic. Since we have assumed that \eqref{ineq:Lusztig-bullet} holds for every elliptic maximal torus $\bbS^\circ \subset \bbG^\circ$, by the proof of Corollary \ref{cor:ff unipotent} and Theorem \ref{thm:unique-Z-rho}, we see that if $\bbS^\circ \subset \bbG^\circ$ is elliptic and $(\bbS,\theta,n) \in \tilde{Z}_\rho$, then necessarily $\theta^{\circ} = \mathbbm{1}$. Hence $\rho$ is unipotent.

    As in Corollary \ref{cor:ff unipotent}, the converse immediately follows from Lemma \ref{lem:not-geom-conj} and \ref{prop:rho ss}.
\end{proof}

\begin{remark}
    We remark that if we take $\rho = R_{\bbS}^\bbG(\theta)$ for a character $\theta$ in general position, then Corollary \ref{cor:tilde-Z} recovers Corollary \ref{cor:Henniart}, but under much stronger conditions on $q$. Indeed, for Corollary \ref{cor:Henniart} to hold, we require $q\gg0$ so that there are sufficiently many $\FF_q$-rational regular elements inside the torus $\bbS$ \textit{only}; in contrast, for Corollary \ref{cor:tilde-Z}, we require $q \gg0$ so that there are sufficiently many $\FF_q$-rational regular elements inside \textit{every} maximal torus $\bbS^\circ \subset \bbG^\circ$.
\end{remark}

\newpage

\part{Characters of supercuspidal representations}\label{part:formula}

\section{Tame supercuspidal representations}\label{sec:tamesc}

\subsection{Yu's construction of tame supercuspidal representations}\label{subsec:tamesc}

We first recall the notion of a Yu-datum, which is needed to produce a tame supercuspidal representation.
(Here we follow the convention of \cite[Section 3.1]{HM08}.)

\begin{defn}[Yu-datum]\label{defn:Yu-datum}
A \textit{Yu-datum} is a quintuple $\Psi=(\vec{\bfG},\vec{\phi},\vec{r},\x,\rho_{0})$ consisting of the following objects:
\begin{itemize}
\item
$\vec{\bfG}$ is a sequence $\bfG^{0}\subsetneq\bfG^{1}\subsetneq\cdots\subsetneq\bfG^{d}=\bfG$ of tame twisted Levi subgroups (i.e., each $\bfG^{i}$ is an $F$-rational subgroup of $\bfG$ which becomes a Levi subgroup of $\bfG$ over a tamely ramified extension of $F$) such that $\bfZ_{\bfG^{0}}/\bfZ_{\bfG}$ is anisotropic,
\item
$\x$ is a point of $\mcB(\bfG^{0},F)$ whose image $\bar{\x}$ in $\mcB^{\red}(\bfG^{0},F)$ is a vertex,
\item
$\vec{r}$ is a sequence $0\leq r_{0}<\cdots<r_{d-1}\leq r_{d}$ of real numbers such that $0<r_{0}$ when $d>0$,
\item
$\vec{\phi}$ is a sequence $(\phi_{0},\ldots,\phi_{d})$ of characters $\phi_{i}$ of $G^{i}$ satisfying
\begin{itemize}
\item
for $0\leq i<d$, $\phi_{i}$ is $\bfG^{i+1}$-generic of depth $r_{i}$ at $\x$, and
\item
for $i=d$, 
$
\begin{cases}
\depth_{\x}(\phi_{d})=r_{d}& \text{if $r_{d-1}<r_{d}$,}\\
\phi_{d}=\mathbbm{1} & \text{if $r_{d-1}=r_{d}$,}
\end{cases}
$
\end{itemize}
\item
$\rho_{0}$ is an irreducible representation of $G^{0}_{\bar{\x}}$ whose restriction to $G^{0}_{\x,0}$ contains the inflation of a cuspidal representation of the quotient $G^{0}_{\x,0:0+}$.
\end{itemize}
\end{defn}

\begin{rem}\label{rem:Yu-datum}
We note that $\mcB^{\red}(\bfG^{0},F)$ can be regarded as a subset of $\mcB^{\red}(\bfG^{i},F)$ for any $0\leq i\leq d$ thanks to the assumption on $\bfZ_{\bfG^{0}}/\bfZ_{\bfG}$ (see \cite[Remark 3.4]{Yu01}).
In particular, we may regard $\bar{\x}\in\mcB^{\red}(\bfG^{0},F)$ as a point of $\mcB^{\red}(\bfG^{i},F)$ for any $0\leq i\leq d$.
\end{rem}

For our convenience, we also introduce a ``clipped'' version of Yu-data as follows:
\begin{defn}[clipped Yu-datum]\label{defn:clip}
We call a tuple $(\vec{\bfG},\vec{\phi},\vec{r},\x)$ consisting of the objects as in Definition \ref{defn:Yu-datum} (except for $\rho_{0}$) a \textit{clipped Yu-datum}.
For any clipped Yu-datum $(\vec{\bfG},\vec{\phi},\vec{r},\x)$, we put $\phi_{\geq0}\colonequals \prod_{i=0}^{d}\phi_{i}|_{G^{0}}$.
For any Yu-datum $\Psi=(\vec{\bfG},\vec{\phi},\vec{r},\x,\rho_{0})$, we put $\dashover{\Psi}\colonequals(\vec{\bfG},\vec{\phi},\vec{r},\x)$.
\end{defn}

In \cite{Yu01}, Yu associated a supercuspidal representation $\pi_{\Psi}$ to each Yu-datum $\Psi$ as follows.
We first put $(s_{0},\ldots,s_{d})\colonequals(\frac{r_{0}}{2},\ldots,\frac{r_{d}}{2})$ and define the subgroups $K^{i}$, $J^{i}$, and $J^{i}_{+}$ of $G$ for $1\leq i \leq d$ by
\[
K^{i}
\colonequals 
G^{0}_{\bar{\x}}(G^{0},\ldots,G^{i})_{\x,(0+,s_{0},\ldots,s_{i-1})},
\]
\[
J^{i}
\colonequals 
(G^{i-1},G^{i})_{\x,(r_{i-1},s_{i-1})},
\]
\[
J^{i}_{+}
\colonequals 
(G^{i-1},G^{i})_{\x,(r_{i-1},s_{i-1}+)},
\]
where the right-hand sides denote the subgroups associated to pairs of a tame twisted Levi sequence and an admissible sequence (see \cite[Sections 1 and 2]{Yu01}).
Note that we have $K^{i+1}=K^{i}J^{i+1}$.
For $i=0$, we put $K^{0}\colonequals G^{0}_{\bar{\x}}$.
Then we construct a representation $\rho_{i+1}$ of $K^{i+1}$ from $\rho_{i}$ of $K^{i}$ inductively in the following manner.
By investigating the quotient $J^{i}/J^{i}_{+}$ (which has a symplectic structure derived from the character $\phi_{i-1}$), we obtain a finite Heisenberg group as a quotient of the group $J^{i}$.
Then, as a consequence of the Stone--von Neumann theorem (together with the liftability of an associated projective representation to a linear representation), we obtain a Heisenberg--Weil representation $\tilde{\phi_{i}}$ of the semi-direct product $G^{i}_{\bar{\x}}\ltimes J^{i+1}$.
The tensor representation 
\[
(\tilde{\phi_{i}}|_{K^{i}\ltimes J^{i+1}})
\otimes 
\bigl((\rho_{i}\otimes\phi_{i}|_{K^{i}})\ltimes\mathbbm{1}\bigr)
\]
of $K^{i}\ltimes J^{i+1}$ descends to $K^{i}J^{i+1}=K^{i+1}$ (factors through the canonical map $K^{i}\ltimes J^{i+1}\twoheadrightarrow K^{i}J^{i+1}$), and we define the representation $\rho_{i+1}$ of $K^{i+1}$ to be the descended one.
By putting $\rho_{\Psi}^{\Yu}\colonequals \rho_{d}\otimes\phi_{d}$, we define
\[
\pi_{\Psi}^{\Yu}\colonequals \cInd_{K^{d}}^{G}\rho_{\Psi}^{\Yu}.
\]
This representation is irreducible \cite{Yu01, Fin21-Compos} and hence supercuspidal.
The irreducible supercuspidal representations of $G$ obtained from Yu-data in this way are called \textit{tame supercuspidal representations}.

This procedure gives a map from the set of Yu-data to the set of equivalence classes of tame supercuspidal representations.
The fibers of this map are described by the notion of the $\bfG$-equivalence introduced by Hakim--Murnaghan.

\begin{defn}[$\bfG$-equivalence, {\cite[Definition 6.3 (see also Lemma 6.5)]{HM08}}]
Let $\Psi=(\vec{\bfG},\vec{\phi},\vec{r},\x,\rho_{0})$ and $\Psi'=(\vec{\bfG}',\vec{\phi}',\vec{r}',\x',\rho'_{0})$ be Yu-data.
We say that $\Psi$ and $\Psi'$ are \textit{$\bfG$-equivalent} if $\Psi'$ can be obtained from $\Psi$ by a finite sequence of refactorizations, $G$-conjugations, and elementary transformations, which are explained in the following:
\begin{enumerate}
\item
$\Psi'$ is said to be a \textit{refactorization} of $\Psi$ if $(\vec{\bfG}',\bar{\x}')=(\vec{\bfG},\bar{\x})$ and the following conditions are satisfied:
% (\cite[Definition 4.19]{HM08}):
\begin{itemize}
\item[(F0)]
If $\phi_{d}=\mathbbm{1}$, then $\phi'_{d}=\mathbbm{1}$;
\item[(F1)]
We define a character $\chi_{i}\colon G^{i}\rightarrow\C^{\times}$ by $\chi_{i}(g)\colonequals \prod_{j=i}^{d}\phi_{j}(g)\phi'_{j}(g)^{-1}$.
Then the depth of $\chi_{i}$ is at most $r_{i-1}$ for any $0\leq i\leq d$ (we put $r_{-1}\colonequals 0$).
\item[(F2)]
We have $\rho'_{0}=\rho_{0}\otimes\chi_{0}$.
\end{itemize}
(Note that the conditions (F0)--(F2) automatically implies that $\vec{r}'=\vec{r}$.)
\item
$\Psi'$ is said to be a \textit{$G$-conjugation} of $\Psi$ if $(\vec{\bfG}',\vec{\phi}',\vec{r}',\bar{\x}',\rho'_{0})=({}^{g}\vec{\bfG},{}^{g}\vec{\phi},\vec{r},g\bar{\x},{}^{g}\rho_{0})$ for some $g\in G$;
\item
$\Psi'$ is said to be an \textit{elementary transformation} of $\Psi$ if $\dashover{\Psi}'=\dashover{\Psi}$ and $\rho'_{0}\cong\rho_{0}$.
% (\cite[Definition 6.2]{HM08}).
\end{enumerate}
Similarly, for clipped Yu-data $\dashover\Psi$ and $\dashover{\Psi}'$, we say that $\dashover\Psi$ and $\dashover{\Psi}'$ are \textit{$\bfG$-equivalent} if $\dashover{\Psi}'$ can be obtained from $\dashover\Psi$ by a finite sequence of refactorizations and $G$-conjugations, where the refactorization is defined only by the first two conditions (F0) and (F1).
\end{defn}

\begin{thm}[{\cite[Theorem 6.6]{HM08}}]\label{thm:HM}
For any Yu-data $\Psi$ and $\Psi'$, the associated tame supercuspidal representations $\pi_{\Psi}^{\Yu}$ and $\pi_{\Psi'}^{\Yu}$ are equivalent if and only if $\Psi$ and $\Psi'$ are $\bfG$-equivalent.
\end{thm}

We recall the exhaustion result of tame supercuspidal representations.
In \cite{Kim07}, Kim proved that when $p$ is sufficiently large, any supercuspidal representation of $G$ is in fact tame supercuspidal.
Recently, Fintzen obtained this exhaustion result under a better assumption on $p$ via a different method:
\begin{thm}[{\cite[Theorem 8.1]{Fin21-Ann}}]\label{thm:Fintzen}
Let $W_{\bfG}$ be the absolute Weyl group of $\bfG$.
When $p\nmid|W_{\bfG}|$, any supercuspidal representation of $G$ is tame supercuspidal.
\end{thm}

In summary, we have an injective map from the set of $\bfG$-equivalence classes of Yu-data to the set of equivalence classes of irreducible supercuspidal representations of $G$, which is surjective if $p\nmid|W_{\bfG}|$:
\begin{equation*}
    \begin{tikzcd}[column sep=1in]
        & \{\text{irred.\ s.c.\ rep'ns of $G$}\}/{\sim}\\
        \{\text{Yu-data}\}/\text{$\bfG$-eq.} \ar{r}[above]{1:1} \ar{r}[below]{\text{Yu's construction}} & \{\text{tame s.c.\ rep'ns of $G$}\}/{\sim}\ar[hookrightarrow]{u}[right]{\quad \text{equal if $p\nmid|W_{\bfG}|$}}
    \end{tikzcd}
\end{equation*}

%\subsection{Modified construction of Fintzen--Kaletha--Spice}\label{subsec:FKS}
We finally recall a modified version of the construction of tame supercuspidal representations proposed by Fintzen--Kaletha--Spice recently \cite{FKS21}.
The key in their construction is a sign character $\epsilon_{\Psi}\colon K^{d}\rightarrow\C^{\times}$ associated to a Yu-datum $\Psi$ (see \cite[Definition 4.1.10]{FKS21} and also \cite[15 page]{FKS21})\footnote{The character $\epsilon_{\Psi}$ is denoted by $\epsilon$ in \cite{FKS21}. We use the symbol $\epsilon_{\Psi}$ to emphasize its dependence on $\Psi$.}.
For a Yu-datum $\Psi$, they defined an irreducible supercuspidal representation $\pi_{\Psi}^{\FKS}$ to be the compact induction of $\rho_{\Psi}^{\Yu}\otimes\epsilon_{\Psi}$.
Let us write $\pi_{\Psi}^{\FKS}$ for this representation:
\[
\pi_{\Psi}^{\FKS}\colonequals \cInd_{K^{d}}^{G}(\rho_{\Psi}^{\Yu}\otimes\epsilon_{\Psi}).
\]

We give a few more comments about the sign character $\epsilon_{\Psi}$.
Suppose that a Yu-datum $\Psi$ is given by $(\vec{\bfG},\vec{\phi},\vec{r},\x,\rho_{0})$.
Then the restriction $\epsilon_{\Psi}|_{G^{0}_{\bar{\x}}}$ of $\epsilon_{\Psi}$ to $G^{0}_{\bar{\x}}$ is given by the product
\[
\epsilon_{\Psi}|_{G^{0}_{\bar{\x}}}
=\prod_{i=1}^{d}\epsilon_{\bar{\x}}^{G^{i}/G^{i-1}}|_{G^{0}_{\bar{\x}}},
\]
where $\epsilon_{\bar{\x}}^{G^{i}/G^{i-1}}\colon G^{i-1}_{\bar{\x}}\rightarrow\{\pm1\}$ is a sign character satisfying the following condition (see \cite[Theorem 3.4 and Lemma 3.5]{FKS21}):
%Each $\epsilon_{\bar{\x}}^{G^{i}/G^{i-1}}|_{G^{0}_{\bar{\x}}}$ factors through the natural quotient map $G^{0}_{\bar{\x}}\rightarrow G^{0}_{\ad,\bar{\x}}$, where $\bfG^{0}_{\ad}$ denotes the adjoint group of $\bfG^{0}$.
%MO: This kind of explanation is needed only when we treat non-elliptic maximal torus.
Let $\bfS$ be a tame elliptic maximal torus of $\bfG^{0}$ such that $\x$ belongs to $\mcB(\bfS,F)$ (hence we have $S\subset G^{0}_{\bar{\x}}$).
%Hence, if we let $\bfS_{\ad}$ be the image of $\bfS$ in $\bfG_{\ad}$, then we have $S_{\ad}\subset G^{0}_{\ad,\bar{\x}}$.
%The pull back of $\epsilon_{\bar{\x}}^{G^{i}/G^{i-1}}|_{G^{0}_{\bar{\x}}}$ (regarded as a character of $G^{0}_{\ad,\bar{\x}}$ by (1)) to $S$ via $S\rightarrow S_{\ad}\rightarrow G^{0}_{\ad,\bar{\x}}$ is given by 
Then $\epsilon_{\bar{\x}}^{G^{i}/G^{i-1}}|_{S}$ is given by
\[
\epsilon_{\sharp,\x}^{G^{i}/G^{i-1}}
\cdot\epsilon_{\flat}^{G^{i}/G^{i-1}}
\cdot\epsilon_{f}^{G^{i}/G^{i-1}}.
\]
Here, $\epsilon_{\sharp,\x}^{G^{i}/G^{i-1}}$, $\epsilon_{\flat}^{G^{i}/G^{i-1}}$, and $\epsilon_{f}^{G^{i}/G^{i-1}}$ are characters of $S$ determined by the root systems $R(\bfG^{i},\bfS)$ and $R(\bfG^{i-1},\bfS)$; see \cite[Definition 3.1]{FKS21} for the details.
We remark that $\epsilon_{\sharp,\x}^{G^{i}/G^{i-1}}$ and $\epsilon_{f}^{G^{i}/G^{i-1}}$ are nothing but the quantities $\epsilon^{G^{i}/G^{i-1},\ram}$ and $\epsilon_{f,\ram}^{G^{i}/G^{i-1}}$ introduced in \cite{Kal19}\footnote{Although $\epsilon^{G^{i}/G^{i-1},\ram}$ and $\epsilon_{f,\ram}^{G^{i}/G^{i-1}}$ are simply denoted by $\epsilon^{\ram}$ and $\epsilon_{f,\ram}$ in \cite{Kal19}, we put $G^{i}/G^{i-1}$ on their exponents in order to emphasize that they are defined for each successive pair $(G^{i},G^{i-1})$.}, which appear naturally in the character formula of tame supercuspidal representations (see \cite[Remark 3.3]{FKS21}).
In summary, we have the following.
\begin{prop}\label{prop:epsilon}
Let $\bfS$ be a tame elliptic maximal torus of $\bfG^{0}$ such that $\x$ belongs to $\mcB(\bfS,F)$.
We define the characters $\epsilon^{\ram}_{\Psi,\bfS}$, $\epsilon_{\Psi,\bfS,\flat}$, and $\epsilon_{\Psi,\bfS,f,\ram}$ by
\[
\epsilon_{\Psi,\bfS}^{\ram}\colonequals \prod_{i=1}^{d}\epsilon^{G^{i}/G^{i-1},\ram},\quad
\epsilon_{\Psi,\bfS,\flat}\colonequals \prod_{i=1}^{d}\epsilon_{\flat}^{G^{i}/G^{i-1}},\quad
\epsilon_{\Psi,\bfS,f,\ram}\colonequals \prod_{i=1}^{d}\epsilon_{f,\ram}^{G^{i}/G^{i-1}}.
\]
Then we have 
\[
\epsilon_{\Psi}|_{S}
=
\epsilon_{\Psi,\bfS}^{\ram}
\cdot\epsilon_{\Psi,\bfS,\flat}
\cdot\epsilon_{\Psi,\bfS,f,\ram}.
\]
\end{prop}

\begin{rem}
Theorems \ref{thm:HM} and \ref{thm:Fintzen} again hold for the modified construction of Fintzen--Kaletha--Spice.
This can be checked as follows.
For any Yu-datum $\Psi=(\vec{\bfG},\vec{\phi},\vec{r},\x,\rho_{0})$, we put $\Psi'\colonequals(\vec{\bfG},\vec{\phi},\vec{r},\x,\rho_{0}\otimes(\epsilon_{\Psi}|_{G^{0}_{\bar{\x}}}))$.
Then, by noting that $\epsilon_{\Psi}$ is the character of $K^{d}$ such that $\epsilon_{\Psi}|_{J^{i}}$ is trivial for any $0<j\leq d$, we see that $\pi^{\FKS}_{\Psi}=\pi^{\Yu}_{\Psi'}$.
(Recall that $\pi^{\FKS}_{\Psi}=\cInd_{K^{d}}^{G}(\rho_{\Psi}^{\Yu}\otimes\epsilon_{\Psi})$ and $K^{d}=G^{0}_{\bar{\x}}J_{1}\cdots J_{d}$.)
Since $\epsilon_{\bar{\x}}^{G^{i}/G^{i-1}}$ is a sign character depending only on $\vec{r}$ and $\bar{\x}$, we have $\epsilon_{\Psi}=\epsilon_{\Psi'}$, which implies that $(\Psi')'=\Psi$ and also that $\Psi_{1}\cong\Psi_{2}$ if and only if $\Psi'_{1}\cong\Psi'_{2}$ for any Yu-data $\Psi_{1}$, $\Psi_{2}$.
Thus $\Psi\mapsto\Psi'$ gives an involutive automorphism on the set of Yu-data of $\bfG$ which is equivariant with respect to the $\bfG$-equivalence.
This implies Theorems \ref{thm:HM} and \ref{thm:Fintzen} for the modified construction.
\end{rem}

%\begin{rem}
%Strictly speaking, in \cite[Definition 3.1]{FKS21}, the characters $\epsilon_{\sharp,\x}^{G/M}$, $\epsilon_{\flat}^{G/M}$, and $\epsilon_{f}^{G/M}$ are defined only in the setting where $\bfG$ is of adjoint type.
%Thus, in fact, the sign character $\epsilon_{\x}^{G^{i}/G^{i-1}}$ is firstly defined on the adjoint group $G_{\ad,\bar{\x}}^{i-1}$ and then regarded as a character on $G_{\bar{\x}}^{i-1}$ by pulling back along the natural quotient map $G_{\bar{\x}}^{i-1}\rightarrow G_{\ad,\bar{\x}}^{i-1}$.
%Also, the factorization property $\epsilon_{\x}^{G^{i}/G^{i-1}}=\epsilon_{\sharp,\x}^{G^{i}/G^{i-1}}\cdot\epsilon_{\flat}^{G^{i}/G^{i-1}}\cdot\epsilon_{f}^{G^{i}/G^{i-1}}$, which characterizes $\epsilon_{\x}^{G^{i}/G^{i-1}}$, is stated only at the level of the adjoint group $G_{\ad,\bar{\x}}^{i-1}$ in \cite[Theorem 3.4]{FKS21}.
%However, we can define $\epsilon_{\sharp,\x}^{G/M}$, $\epsilon_{\flat}^{G/M}$, and $\epsilon_{f}^{G/M}$ in exactly the same way as in \cite[Definition 3.1]{FKS21} even when $\bfG$ is not of adjoint type.
%Then we can easily see that the factorization property of $\epsilon_{\x}^{G^{i}/G^{i-1}}$ holds on $G_{\bar{\x}}^{i-1}$ as well.
%\end{rem}

\subsection{Non-singularity and unipotency of supercuspidal representations}\label{subsec:sc classes}

Let $\Psi=(\vec{\bfG},\vec{\phi},\vec{r},\x,\rho_{0})$ be a Yu-datum and $\pi_{\Psi}^{\FKS}$ the associated tame supercuspidal representation.

Following \cite[Section 3]{Kal19-sc}, we introduce several smooth group schemes over $\F_{q}$ attached to the Yu-datum $\Psi$.
Let $\bbG$ be the reductive quotient of the special fiber of the unique smooth integral model $\mathcal{G}^{0}$ of $\bfG^{0}$ satisfying $\mathcal{G}^{0}(F^{\ur})=\bfG^{0}(F^{\ur})_{\bar{\x}}$.
Let $\bbG^{\circ}$ be the identity component of $\bbG$, which is a connected reductive group scheme over $\F_{q}$.
Note that then we have 
\[
\bbG(\F_{q})\cong G^{0}_{\bar{\x}}/G^{0}_{\x,0+}
\quad\text{and}\quad
\bbG^{\circ}(\F_{q})\cong G^{0}_{\x,0:0+}.
\]
We also define a group scheme $\bbZ_{\bbG}$ over $\F_{q}$ to be the subgroup of $\bbG$ whose $\overline \F_q$-points are given by the image of $\bfZ_{\bfG^{0}}(F^{\ur})$ in $\bfG^{0}(F^{\ur})_{\bar{\x}}/\bfG^{0}(F^{\ur})_{\x,0+}$.
Then the pair $(\bbG^{\circ},\bbZ_{\bbG})$ satisfies the assumptions as in the beginning of Section \ref{subsec:disconn} (see \cite[Section 3.2]{Kal19-sc} for details).

Recall that $\rho_{0}$ is an irreducible representation of $G^{0}_{\bar{\x}}$ which is trivial on $G^{0}_{\x,0+}$, hence can be regarded as an irreducible representation of $\bbG(\F_{q})$.
By Corollary \ref{cor:exhaustion}, there exists an $\F_{q}$-rational maximal torus $\bbS^{\circ}$ of $\bbG^{\circ}$ and a character $\phi_{-1}$ of $\bbS(\F_{q})$ such that $\rho_{0}$ is a subrepresentation of $\pm R_{\bbS}^{\bbG}(\phi_{-1})$, where $\bbS\colonequals \bbS^{\circ}\bbZ_{\bbG}$.
By Lemma \ref{lem:not-geom-conj}, such a pair $(\bbS,\phi_{-1})$ is determined by $\rho_{0}$ uniquely up to geometric conjugacy.
Furthermore, by Lemma \ref{lem:cuspidal}, the maximal torus $\bbS^\circ$ must be elliptic in $\bbG^\circ$.
Let us fix such $(\bbS,\phi_{-1})$ and put $\phi_{-1}^{\circ}\colonequals \phi_{-1}|_{\bbS^{\circ}(\F_{q})}$.

\begin{defn}\label{def:ss and unip}
\begin{enumerate}
\item
We say that $\Psi$ and $\pi_{\Psi}^{\FKS}$ are \textit{non-singular}\footnote{In \cite[Section 1.2]{Kal22-ICM}, it is suggested to use the terminology ``semisimple'' instead of ``non-singular''.} if $\phi_{-1}^{\circ}$ is non-singular in $\bbG^{\circ}$ in the sense of Deligne--Lusztig (\cite[Definition 5.15]{DL76}).
\item
We say that $\Psi$ and $\pi_{\Psi}^{\FKS}$ are \textit{unipotent} if $\vec{\bfG}=(\bfG^{0}=\bfG)$, $\vec{\phi}=(\phi_{0}=\mathbbm{1})$, $\vec{r}=(r_{0}=0)$, and $\phi_{-1}^{\circ}$ is trivial.
\end{enumerate}
Note that these notions are independent of the choice of a pair $(\bbS,\phi_{-1})$ since the non-singularity or the triviality of $\phi_{-1}^{\circ}$ is invariant in the geometric conjugacy class of $(\bbS,\phi_{-1})$.
\end{defn}

\subsection{Tame elliptic non-singular pairs}\label{subsec:TENSP}
We review the notion of tame elliptic non-singular pairs, which is needed in Kaletha's classification of non-singular supercuspidal representations.

Let $\bfG^{0}$ be a tamely ramified connected reductive group over $F$.
Let $\bfS$ be an elliptic maximally unramified (in the sense of \cite[Definition 3.4.2]{Kal19}) maximal torus of $\bfG^{0}$ with maximal unramified subtorus $\bfS'$.
We fix a finite unramified extension $F'$ of $F$ which splits $\bfS'$.
We write $\Nr_{F'/F}$ for the norm map $\bfS'(F')\rightarrow\bfS'(F)$.
We let $N_{\bfG^{0}}(\bfS)$ (resp.\ $N_{G^{0}}(\bfS)$) be the normalizer group of $\bfS$ in $\bfG^{0}$ (resp.\ $G^{0}$) and put $W_{\bfG^{0}}(\bfS)\colonequals N_{\bfG^{0}}(\bfS)/\bfS$ (resp.\ $W_{G^{0}}(\bfS)\colonequals N_{G^{0}}(\bfS)/S$).
Note that we have $\bbS(\F_{q})=S/S_{0+}$, $\bbS^{\circ}(\F_{q})=S_{0}/S_{0+}=S'_{0}/S'_{0+}$ and $W_{\bbG(\F_{q})}(\bbS)\cong W_{G^{0}}(\bfS)$ (see \cite[Lemma 3.2.2]{Kal19-sc}).

\begin{defn}[{\cite[Definition 3.4.16]{Kal19}, \cite[Definition 3.1.1]{Kal19-sc}}]\label{defn:reg-nonsing}
Let $\phi_{-1}\colon S\rightarrow\C^{\times}$ be a character.
\begin{enumerate}
\item
We say that the character $\phi_{-1}$ is \textit{extra regular} if the stabilizer of $\phi_{-1}|_{S_{0}}$ in $W_{\bfG^{0}}(\bfS)(F)$ is trivial.
\item
We say that the character $\phi_{-1}$ is \textit{regular} if the stabilizer of $\phi_{-1}|_{S_{0}}$ in $W_{G^{0}}(\bfS)$ is trivial.
\item
We say that the character $\phi_{-1}$ is \textit{$F$-non-singular} if the character 
\[
\phi_{-1}\circ \Nr_{F'/F}\circ\alpha^{\vee}_{\res}\colon F^{\prime\times}\rightarrow\C^{\times}
\]
is not trivial on the subgroup $\mcO_{F'}^{\times}$ for any $\alpha_{\res}\in R_{\res}(\bfG^{0},\bfS')$, where $R_{\res}(\bfG^{0},\bfS')$ denotes the set of restrictions of the absolute roots of $\bfS$ in $\bfG^{0}$ to $\bfS'$.
\item
We say that the character $\phi_{-1}$ is \textit{$k_{F}$-non-singular} if the character 
\[
\phi_{-1}\circ \Nr_{F'/F}\circ\bar{\alpha}^{\vee}\colon F^{\prime\times}\rightarrow\C^{\times}
\]
is not trivial on the subgroup $\mcO_{F'}^{\times}$ for any $\bar{\alpha}\in R(\bbG^{\circ},\bbS^{\circ})$ (note that $R(\bbG^{\circ},\bbS^{\circ})$ can be regarded as a subset of $R_{\res}(\bfG^{0},\bfS')$).
\end{enumerate}
\end{defn}

The relationship between the above regularities in the depth zero setting and the regularities in the finite field setting due to Deligne--Lusztig is summarized as follows.

\begin{prop}\label{prop:reg-nonsing}
Let $\phi_{-1}\colon S\rightarrow\C^{\times}$ be a depth zero character (i.e., trivial on $S_{0+}$, hence regarded as a character of $\bbS(\F_{q})=S/S_{0+}$).
\begin{enumerate}
\item
The character $\phi_{-1}$ is $k_{F}$-non-singular if and only if $\phi_{-1}^{\circ}$ is non-singular in the sense of Deligne--Lusztig, where $\phi_{-1}^{\circ}$ denotes the restriction of $\phi_{-1}$ to $S_{0}/S_{0+}=\bbS^{\circ}(\F_{q})$.
\item
If $\phi_{-1}$ is $F$-non-singular, then $\phi_{-1}$ is $k_{F}$-non-singular.
\item
If $\phi_{-1}$ is regular, then $\phi_{-1}$ is $F$-non-singular.
\item
The character $\phi_{-1}$ is regular if and only if $\phi_{-1}^{\circ}$ is not stabilized by any nontrivial element of $W_{\bbG(\F_{q})}(\bbS)$.
In particular, if $\phi_{-1}$ is regular, then $\phi_{-1}$ (resp.\ $\phi_{-1}^{\circ}$) is in general position in the sense of Definition \ref{defn:in-gen-pos} (resp.\ Deligne--Lusztig).
\item
If $\phi_{-1}^{\circ}$ is in general position, then $\phi_{-1}^{\circ}$ is non-singular.
\end{enumerate}
\[
\xymatrix{
\text{$\phi_{-1}$: $k_{F}$-non-singular} &&& \text{$\phi_{-1}^{\circ}$: non-singular} \ar@{<=>}_-{(1)}[lll] \\
\text{$\phi_{-1}$: $F$-non-singular} \ar@{=>}_-{(2)}[u] &&&\\
\text{$\phi_{-1}$: regular} \ar@{=>}_-{(3)}[u] \ar@{=>}^-{(4)}[rrr]&&& \text{$\phi_{-1}^{\circ}$: in general position} \ar@{=>}_-{(5)}[uu]
}
\]
\end{prop}

\begin{proof}
The assertion (1) is explained in \cite[Remark 3.1.2]{Kal19-sc}.
See \cite[Fact 3.1.4(1)]{Kal19-sc} and \cite[Fact 3.1.4(3)]{Kal19-sc} for the assertions (2) and (3), respectively.
The assertion (4) follows from the definition of the regularity of $\phi_{-1}$ and the isomorphism $W_{\bbG(\F_{q})}(\bbS)\cong W_{G^{0}}(\bfS)$ mentioned above.
The assertion (5) is nothing but \cite[Corollary 5.18]{DL76}.
\end{proof}

\begin{defn}[{\cite[Definition 3.4.1]{Kal19-sc}}]\label{defn:non-sing-pair}
We call a pair $(\bfS,\theta)$ of a tame elliptic maximal torus $\bfS$ of $\bfG$ and a character $\theta\colon S\rightarrow\C^{\times}$ a \textit{tame elliptic (extra) regular/$F$-non-singular/$k_{F}$-non-singular pair} if the following conditions are satisfied:
\begin{enumerate}
\item
Let $R_{0+}$ be the subset of $R(\bfS,\bfG)$ defined by
\[
R_{0+}\colonequals \{\alpha\in R(\bfS,\bfG) \mid \theta\circ\Nr_{E/F}\circ\alpha^{\vee}(E_{0+}^{\times})=1\},
\]
where $E$ is a tamely ramified extension of $F$ which splits $\bfS$.
Let $\bfG^{0}$ be the connected reductive subgroup of $\bfG$ with maximal torus $\bfS$ and root system $R_{0+}$.
Then $\bfS$ is maximally unramified in $\bfG^{0}$.
\item
The character $\theta$ is (extra) regular/$F$-non-singular/$k_{F}$-non-singular in $\bfG^{0}$ in the sense of Definition \ref{defn:reg-nonsing}.
\end{enumerate}
\end{defn}

\subsection{Kaletha's classification of non-singular supercuspidal representations}\label{subsec:sssc}
Now we recall Kaletha's classification of non-singular supercuspidal representations.

Let us first suppose that $\Psi=(\vec{\bfG},\vec{\phi},\vec{r},\x,\rho_{0})$ is a non-singular Yu-datum.
For any $F$-rational maximal torus $\bfT$ of $\bfG^{0}$, we write $\mathbb{T}$ for the special fiber of the ft-N\'eron model of $\bfT$.
We have the following correspondence between maximal tori over $p$-adic fields and those over finite fields according to DeBacker's classification \cite{DeB06}:

\begin{prop}[{\cite[Lemma 3.4.4]{Kal19}}]\label{prop:DeBacker-Kaletha}
If $\bfT$ is a maximally unramified elliptic maximal torus of $\bfG^{0}$ with associated point $\x$ (hence $T\subset G^{0}_{\bar{\x}}$), then $\bbT^{\circ}$ is an elliptic maximal torus of $\bbG^{\circ}$.
Conversely, any elliptic maximal torus of $\bbG^{\circ}$ arises in this way.
\end{prop}

Now we attach a tame elliptic $k_{F}$-non-singular pair $(\bfS,\theta)$ to $\Psi$ in the following manner.
As discussed in Section \ref{subsec:sc classes}, from the data $(\bfG^{0},\x,\rho_{0})$, we can construct a unique (up to geometric conjugacy) pair $(\bbS,\phi_{-1})$ whose $\bbS^{\circ}$ is elliptic in $\bbG^{\circ}$.
Thus, by Proposition \ref{prop:DeBacker-Kaletha}, there exists a maximally unramified elliptic maximal torus  $\bfS$ of $\bfG^{0}$ whose connected N\'eron model has $\bbS^{\circ}$ as its special fiber.
Since $\phi_{-1}^{\circ}\colonequals \phi_{-1}|_{\bbS^{\circ}(\F_{q})}$ is non-singular in the sense of Deligne--Lusztig by the definition of the non-singularity of $\Psi$, Proposition \ref{prop:reg-nonsing} (1) implies that $\phi_{-1}$ is a depth zero $k_{F}$-non-singular character of $S$ (with respect to $\bfS\subset\bfG^{0}$).
Then we can check that, by putting $\theta=\prod_{i=-1}^{d}\phi_{i}|_{S}$, the pair $(\bfS,\theta)$ is a tame elliptic $k_{F}$-non-singular pair.
Indeed, the group $\bfG^{0}$ of Definition \ref{defn:non-sing-pair} (1) is nothing but $\bfG^{0}$ of $\vec{\bfG}$ in this case, hence the condition (1) is satisfied.
Since $(\vec{\bfG},\vec{\phi})\colonequals ((\bfG^{-1}\colonequals \bfS\subset\bfG^{0}\subset\cdots\bfG^{d}),(\phi_{-1},\phi_{0},\ldots,\phi_{d}))$ gives a Howe factorization of $(\bfS,\theta)$ in the sense of Kaletha \cite[Definition 3.6.2]{Kal19}, we have $\phi_{-1}|_{S_{\sc,0}^{0}}=\theta|_{S_{\sc,0}^{0}}$ by \cite[Fact 3.6.4]{Kal19}.
Here, $S_{\sc,0}^{0}$ denotes the parahoric subgroup of the preimage $\bfS_{\sc}^{0}$ of $\bfS$ in the simply-connected cover $\bfG^{0}_{\sc}$ of the derived group of $\bfG^{0}$.
Thus, by noting that $(\bfS,\theta)$ (resp.\ $(\bfS,\phi_{-1})$) is $k_{F}$-non-singular in $\bfG^{0}$ if and only if so is $(\bfS_{\sc}^{0},\theta|_{S_{\sc}^{0}})$ (resp.\ $(\bfS_{\sc}^{0},\phi_{-1}|_{S_{\sc}^{0}})$) in $\bfG^{0}_{\sc}$, we see that the condition (2) is satisfied by $\theta$.

The converse of this procedure can be given as follows.
Let us suppose that $(\bfS,\theta)$ is a tame elliptic $k_{F}$-non-singular pair of $\bfG$.
Then, by using \cite[Proposition 3.6.7]{Kal19}, we get a Howe factorization $(\vec{\bfG},\vec{\phi})\colonequals ((\bfG^{-1}\colonequals \bfS\subset\bfG^{0}\subset\cdots\bfG^{d}),(\phi_{-1},\phi_{0},\ldots,\phi_{d}))$ of $(\bfS,\theta)$ (see \cite[Section 3.6]{Kal19} for details).
Note that, in particular, the sequence $(\phi_{-1},\ldots,\phi_{d})$ satisfies $\theta=\prod_{i=-1}^{d}\phi_{i}|_{S}$.
Since $\bfS$ is a maximally unramified elliptic maximal torus of $\bfG^{0}$, any  point $\x\in\mcB(\bfG^{0},F)$ associated to $\bfS$ maps to a vertex $\bar{\x}\in\mcB^{\red}(\bfG^{0},F)$ by \cite[Lemma 3.4.3]{Kal19}.
Moreover, again by the same argument as in the previous paragraph, we see that $\phi_{-1}$ is a $k_{F}$-non-singular character of $S$ in $\bfG^{0}$ since so is $\theta$.
We put $\vec{\phi}\colonequals (\phi_{0},\ldots,\phi_{d})$ and $\vec{r}\colonequals (0\leq r_{0}<\cdots<r_{d-1}\leq r_{d})$, where 
\[
r_{i}
\colonequals 
\begin{cases}
\depth_{\x}(\phi_{i}) & \text{if $0\leq i<d$,}\\
\depth_{\x}(\phi_{d}) & \text{if $i=d$ and $\phi_{d}\neq\mathbbm{1}$,}\\
\depth_{\x}(\phi_{d-1}) & \text{if $i=d$ and $\phi_{d}=\mathbbm{1}$.}
\end{cases}
\]
Then we get a clipped Yu-datum $(\vec{\bfG},\vec{\phi},\vec{r},\x)$ equipped with a $k_{F}$-non-singular character $\phi_{-1}$ of $S$ in $\bfG^{0}$.

As explained in \cite[Section 2.4]{Kal19-sc}, the $k_{F}$-non-singularity of $\phi_{-1}$ implies that $(-1)^{d(\bfS)}R_{\bbS^{\circ}}^{\bbG^{\circ}}(\phi_{-1}^{\circ})$ is a genuine representation of $\bbG^{\circ}(\F_{q})$, where $d(\bfS) $ is an explicit number determined by $\bfS$.
Accordingly, also $(-1)^{d(\bfS)}R_{\bbS}^{\bbG}(\phi_{-1})$ is a genuine representation of $\bbG(\F_{q})$.
Let 
\[
(-1)^{d(\bfS)}R_{\bbS}^{\bbG}(\phi_{-1})
=\bigoplus_{i=1}^{r} \rho_{i}^{\oplus m_{i}}
\]
be the irreducible decomposition of $(-1)^{d(\bfS)}R_{\bbS}^{\bbG}(\phi_{-1})$, where $m_{i}\in\Z_{>0}$ and $\rho_{i}$'s are pairwise non-isomorphic irreducible representations of $\bbG(\F_{q})$.
Then we get a non-singular Yu-datum $\Psi_{i}=(\vec{\bfG},\vec{\phi},\vec{r},\x,\rho_{i})$, hence an irreducible non-singular supercuspidal representation $\pi_{\Psi_{i}}^{\FKS}$.
We put
\[
\pi_{(\bfS,\theta)}^{\FKS}
\colonequals \bigoplus_{i=1}^{r} \pi_{\Psi_{i}}^{\FKS\oplus m_{i}}
\quad\text{and}\quad
[\pi_{(\bfS,\theta)}^{\FKS}]
\colonequals \{\pi_{\Psi_{i}}^{\FKS}\mid i=1,\ldots,r\}.
\]
Following Kaletha \cite{Kal19-sc}, we call $[\pi_{(\bfS,\theta)}^{\FKS}]$ the \textit{non-singular Deligne--Lusztig packet associated to $(\bfS,\theta)$}.

\begin{rem}\label{rem:parametrization}
We remark that if $\bbG$ is connected reductive, then all the multiplicities $m_i$ are equal to 1 \cite{Lus88}; in general, the integers $\{m_1,\ldots, m_r\}$ are exactly the dimensions of the irreducible representations $\tau$ of the stabilizer of $(\bbS,\theta)$ in $\bbG(\FF_q)$ whose restriction $\tau|_{\bbS(\FF_q)}$ is $\theta$-isotypic \cite{Kal19-sc}. 
This analysis of the irreducible constituents plays a crucial role in Kaletha's construction of the local Langlands correspondence for non-singular supercuspidal representations.
\end{rem}

\begin{prop}[{\cite[Corollary 3.4.7]{Kal19-sc}}]\label{prop:DL-packet}
\begin{enumerate}
\item
The isomorphism class of the representation $\pi_{(\bfS,\theta)}^{\FKS}$ depends only on $(\bfS,\theta)$.
In particular, it is independent of the choice of a Howe factorization of $(\bfS,\theta)$.
\item
For tame elliptic $k_{F}$-non-singular pairs $(\bfS,\theta)$ and $(\bfS',\theta')$, the associated non-singular Deligne--Lusztig packets $[\pi_{(\bfS,\theta)}^{\FKS}]$ and $[\pi_{(\bfS',\theta')}^{\FKS}]$ are equal or disjoint.
Moreover, they are equal if and only if $(\bfS,\theta)$ and $(\bfS',\theta')$ are $G$-conjugate.
\end{enumerate}
\end{prop}

By combining this proposition with the above construction of $(\bfS,\theta)$ from $\Psi$, we see that the set of equivalence classes of non-singular supercuspidal representations is divided into the disjoint union of non-singular Deligne--Lusztig packets labelled by $G$-conjugacy classes of tame elliptic non-singular pairs:
\[
\{\text{irred.\ non-singular s.c.\ rep'ns of $G$}\}/{\sim}
=
\bigsqcup_{\begin{subarray}{c} \text{$(\bfS,\theta)$: TENS pairs}\\ /\text{$G$-conj.}\end{subarray}}
[\pi_{(\bfS,\theta)}^{\FKS}]
\]

\begin{rem}\label{rem:regular}
When $(\bfS,\theta)$ is a tame elliptic regular pair, then $(-1)^{d(\bfS)}R_{\bbS}^{\bbG}(\phi_{-1})$ is irreducible since $\phi_{-1}$ is in general position.
Hence $\pi_{(\bfS,\theta)}^{\FKS}$ is irreducible; this is what is called a \textit{regular supercuspidal representation} (see \cite[Section 3]{Kal19} and also \cite[Section 3.2]{CO21} for the details).
We call a non-singular Yu-datum arising from a tame elliptic regular pair \textit{a regular Yu-datum}.
Therefore Proposition \ref{prop:DL-packet} asserts that, in particular, the equivalence classes of regular supercuspidal representations bijectively correspond to $G$-conjugacy classes of tame elliptic regular pairs.
This is nothing but \cite[Proposition 3.7.8]{Kal19}.
\end{rem}

\section{Characters of supercuspidal representations at very regular elements}\label{sec:character}

\subsection{Definition and properties of very regular elements}\label{subsec:vreg}

We first introduce the notion of tame very regularity, which generalizes the notion of unramified very regularity considered by Chan--Ivanov \cite{CI21-RT} (and also \cite{CO21}).

\begin{definition}[very regular elements]
We say that a regular semisimple element $\gamma \in G$ is \textit{tame very regular} if
\begin{itemize}
\item
the connected centralizer $\bfT_{\gamma}$ of $\gamma$ in $\bfG$ is a tamely ramified maximal torus, and
\item
 $\alpha(\gamma) \not\equiv 1$ (mod $\mathfrak{p}_{\ol{F}}$) for any root $\alpha$ of $\bfT_{\gamma}$ in $\bfG$.
 \end{itemize}
\end{definition}

\begin{example}
Let $\bfG = \GL_N$ and $\bfS$ the maximal torus of $\bfG$ corresponding to a tamely ramified extension $E$ of $F$ of  degree $N$.
Then a regular semisimple element $\gamma\in S=E^{\times}$ is very regular if and only if $\val_{E}(\gamma)$ is coprime to the ramification index of $E/F$, where $\val_{E}$ is the valuation of $E$ normalized so that $\val_{E}(E^{\times})=\Z$.
\end{example}

According to \cite[Corollary 2.6]{Fin21-IMRN}, any maximal torus of $\bfG$ is tamely ramified under the assumption that $p\nmid|W_{\bfG}|$.
Thus, since we always assume that $p\nmid|W_{\bfG}|$ in this paper, the first condition of the tame very regularity is always satisfied.
For this reason, we simply say that \textit{$\gamma$ is very regular} when $\gamma$ is a tame very regular element of $G$.

\begin{rem}\label{rem:shallow}
Let $\gamma\in G$ be a regular semisimple element with topological Jordan decomposition $\gamma=\gamma_{0}\cdot\gamma_{+}$ (see \cite{Spi08}).
Then $\gamma$ is very regular if and only if $\gamma_{0}$ is regular semisimple.
We note that, in \cite{Kal19}, a regular topologically semisimple (modulo $Z_{\bfG}$) element of $G$ is called a \textit{shallow} element (see \cite[Section 4.10]{Kal19}).
Thus, with our terminology, we may understand that a regular semisimple element of $G$ is shallow in the sense of Kaletha if and only $\gamma$ is a very regular element with trivial $\gamma_{+}$.
\end{rem}

When a very regular element $\gamma$ of $G$ is furthermore elliptic, it associates a unique point $\bar{\x}_{\gamma}$ of $\mathcal{B}^{\red}(\bfG,F)$.
For any lift $\x_{\gamma}\in\mathcal{B}(\bfG,F)$ of $\bar{\x}_{\gamma}$, we say that $\gamma$ is an \textit{elliptic very regular element with point $\x_{\gamma}$}.
When $\gamma$ is an elliptic very regular element with point $\x_{\gamma}$, $T_{\gamma}$ is contained in $G_{\bar{\x}_{\gamma}}$.
In particular, $\gamma$ is an element of $G_{\bar{\x}_{\gamma}}$.
In fact, the point $\bar{\x}_{\gamma}\in\mathcal{B}^{\red}(\bfG,F)$ is uniquely characterized by this property:

\begin{lem}\label{lem:point}
Let $\gamma\in G$ be an elliptic very regular element with point $\x_{\gamma}$.
If $\x\in\mathcal{B}(\bfG,F)$ is a point such that $\gamma\in G_{\bar{\x}}$, then $\bar{\x}=\bar{\x}_{\gamma}$.
\end{lem}

\begin{proof}
By definition, we have $\alpha(\gamma) \not\equiv 1 \pmod{\mathfrak{p}_{\ol{F}}}$ for any $\alpha \in R(\bfT_{\gamma}, \bfG)$, and so by \cite[Section 3.6]{Tit79}, the set of fixed points of $\gamma$ in $\cB^{\red}(\bfG, F^{\ur})$ is $\cB^{\red}(\bfT_\gamma, F^{\ur})=\{\bar{\x}_{\gamma}\}$.
Since $\gamma\in G_{\bar{\x}}$, $\gamma$ fixes $\bar{\x}\in\cB^{\red}(\bfG,F)$, hence $\bar{\x}=\bar{\x}_{\gamma}$.
\end{proof}

The following properties of very regular elements are investigated in \cite{CO21} in the unramified setting.
We can easily check that the same proofs work by using elliptic very regularity instead of unramified very regularity.
In the following, we explain only minor modifications necessary for the proof in our setting.

\begin{lem}\label{lem:CO21-4.7}
Let $K^{d}=G^{0}_{\bar{\x}}(G^{0},\ldots,G^{d})_{\x,(0+,s_{0},\ldots,s_{d-1})}$ be the group associated to a Yu-datum $\Psi=(\vec{\bfG},\vec{\phi},\vec{r},\x,\rho_{0})$.
Any elliptic very regular element $\gamma\in K^{d}$ is $K^{d}$-conjugate to an element of $G^{0}_{\bar{\x}}$.
\end{lem}

\begin{proof}
The proof is the same as \cite[Lemma 4.7]{CO21}.
Note that \cite[Lemma 4.3]{CO21}, which is necessary for establishing \cite[Lemma 4.7]{CO21}, can be proved by the same argument in the present setting.
We remark that, in the proof of \cite[Lemma 4.3]{CO21}, we need to assume that the condition ``$(\mathbf{Gd}^{G})$'' is satisfied for the maximal torus $\bfT_{\gamma}$.
This follows from our assumption that $p\nmid|W_{\bfG}|$ by Fintzen's result \cite{Fin21-IMRN} (see \cite[Remark 4.4 (1)]{CO21}).
\end{proof}

\begin{lemma}\label{lem:CO21-5.1-5.2}
Let $\bfS$ be a tame elliptic maximal torus of $\bfG$ with associated point $\x$.
Then the set $G_{\bar{\x},\evrs}$ of elliptic very regular elements in $G_{\bar{\x}}$ is stable under $G_{\x,0+}$-translation.
Moreover, for any $\gamma_{1},\gamma_{2}\in G_{\bar{\x},\evrs}$, the associated maximal tori $\bfT_{\gamma_{1}}$ and $\bfT_{\gamma_{2}}$ are $G_{\x,0+}$-conjugate.
\end{lemma}

\begin{proof}
The same arguments as in \cite[Lemma 5.1]{CO21} give the assertions.
Note that \cite[Lemma 5.1]{CO21} is stated only for $G_{\x,0+}$-translations of very regular elements of $S$, but the same proof works.\footnote{We take this opportunity to correct a small mistake in the proof of \cite[Lemma 5.2]{CO21}; so that the proof of \cite[Lemma 5.2]{CO21} makes sense, we have to state \cite[Lemma 5.2]{CO21} in a generalized form as in Lemma \ref{lem:CO21-5.1-5.2}}
\end{proof}

\subsection{Elliptic very regular elements over finite fields}\label{subsec:ur-vreg}

Suppose that
\begin{itemize}
\item
$\bfG$ is a tamely ramified connected reductive group over $F$,
\item
$\bfG^{0}$ is a tame twisted Levi subgroup of $\bfG$ such that $\bfZ_{\bfG^{0}}/\bfZ_{\bfG}$ is anisotropic, 
\item
$\bfS$ is a maximally unramified elliptic maximal torus of $\bfG^{0}$,
\item
$\x\in\mcB(\bfG^{0},F)$ is a point associated to $\bfS$ such that its image $\bar{\x}$ in $\mcB^{\red}(\bfG^{0},F)$ is a vertex (this condition is automatic by \cite[Lemma 3.4.3]{Kal19}).
\end{itemize}
We introduce the groups $\bbG\supset\bbG^{\circ}$, $\bbS\supset\bbS^{\circ}$, and $\bbZ_{\bbG}$ as in Section \ref{subsec:sc classes}.

We put $\bbG':=\bbG^{\circ}\cdot\bbZ_{\bbG}$.
Recall that we have $\bbG'(\F_{q})=\bbS(\F_{q})\bbG^{\circ}(\F_{q})$ (Remark \ref{rem:rational-pts}).
Thus the natural reduction map gives a map $SG^{0}_{\x,0}\twoheadrightarrow\bbS(\F_{q})\bbG^{\circ}(\F_{q})=\bbG'(\F_{q})\colon\gamma\mapsto\overline{\gamma}$.

\begin{prop}\label{prop:evreg-reduction}
Let $\gamma\in SG^{0}_{\x,0}$ be an elliptic very regular element.
Then the image $\overline{\gamma}$ of $\gamma$ in $\bbG'(\F_{q})$ is elliptic regular semisimple in the sense of Definition \ref{def:G' semisimple}.
Moreover, the connected centralizer $\bfT_{\gamma}$ of $\gamma$ in $\bfG^{0}$ is a maximally unramified elliptic maximal torus of $\bfG^{0}$.
\end{prop}

\begin{proof}
Let $\gamma=\gamma_{0}\cdot\gamma_{0+}$ be a topological Jordan decomposition of $\gamma$, i.e., $\gamma_{0}$ is a topologically semisimple (modulo $Z_{\bfG^{0}}$) element and $\gamma_{0+}$ is a topologically unipotent element such that $\gamma_{0}\gamma_{0+}=\gamma_{0+}\gamma_{0}$.
Then, by \cite[Lemma 3.4.22]{Kal19}, $\gamma_{0}$ belongs to $SG^{0}_{\x,0}$ and $\gamma_{0+}$ belongs to $G^{0}_{\x,0}$.
Furthermore, the image $\overline{\gamma}=\overline{\gamma_{0}}\cdot\overline{\gamma_{0+}}$ in $\bbG'(\F_{q})$ gives the Jordan decomposition of $\overline{\gamma}$ in the usual sense.
Let us first show that $\overline{\gamma_{0}}$ is regular semisimple (this implies that $\overline{\gamma}$ is regular semisimple in the sense of Definition \ref{def:G' semisimple}).

Since $\gamma_{0}$ is of finite prime-to-$p$ order (modulo $Z_{\bfG^{0}}$), the parahoric Lie subalgebra $\Lie\bfG^{0}(F^{\ur})_{\x,0}$ has the eigenspace decomposition with respect to the conjugate action of $\gamma_{0}$.
By the very regularity of $\gamma$, the eigenspace with eigenvalue $1$ is given by $\Lie\bfT_{\gamma}(F^{\ur})\cap\Lie\bfG^{0}(F^{\ur})_{\x,0}=\Lie\bfT_{\gamma}(F^{\ur})_{0}$ and the eigenvalue of any other eigenspace is given by a root of unity of finite prime-to-$p$ order (not equal to $1$).
Hence, as $\bbG^{\circ}(\overline{\F}_{q})=\bfG^{0}(F^{\ur})_{\x,0:0+}$, the same is true for the conjugate action of $\overline{\gamma_{0}}$ on the Lie algebra of $\bbG^{\circ}$.
In particular, the connected centralizer of $\overline{\gamma_{0}}$ in $\bbG^{\circ}$ is given by a torus of $\bbG^{\circ}$, hence a maximal torus of $\bbG^{\circ}$.
In other words, $\overline{\gamma_{0}}$ is regular semisimple in $\bbG^{\circ}$.

Let us show that $\bfT_{\gamma}$ is maximally unramified.
For this, since $\bfS$ is maximally unramified, it is enough to show that the ranks of the maximally unramified subtori of $\bfS$ and $\bfT_{\gamma}$ are equal.
Note that, for any torus $\bfT$ over $F$, the rank of its maximally unramified subtorus is the same as the rank of the reduction of its connected N\'eron model (see, e.g., \cite[Corollary B.7.12]{KP23}).
%\cite[(Proof of) Proposition 4.6.4]{BT84}
Thus it suffices to show that the ranks of $\bbS^{\circ}$ and $\mathbb{T}_{\gamma}^{\circ}$ coincide.
This follows from that both are maximal tori of $\bbG^{\circ}$.
(Note that $\mathbb{T}_{\gamma}^{\circ}$, which is the reduction of the connected N\'eron model of $\bfT_{\gamma}$, is nothing but the connected centralizer of $\overline{\gamma_{0}}$ in $\bbG^{\circ}$.)

Finally, the ellipticity of $\overline{\gamma}$ follows from Proposition \ref{prop:DeBacker-Kaletha}.
\end{proof}

\begin{defn}\label{defn:evrs-finite-fields}
We say that an element of $\bbG'(\F_{q})$ is \textit{elliptic very regular} if it is the image of an elliptic very regular element of $SG^{0}_{\x,0}$ under the reduction map.
\end{defn}

We define $\bbG'(\F_{q})_{\evrs}$ to be the set of elliptic very regular elements of $\bbG'(\F_{q})$.
We put $\overline{\bbG}(\F_{q})_{\evrs}$ to be the image of $\bbG'(\FF_q)_{\evrs}$ under the map $\pr_{\ss}\circ\Jord$ (Section \ref{subsec:extended_jordan}). 
Then, by Proposition \ref{prop:evreg-reduction}, $\overline{\bbG}(\F_{q})_{\evrs}$ is contained in $\overline{\bbG}(\F_{q})_{\rs}$ and stable under $\bbG(\F_{q})$-conjugation.
We put $\bbS(\F_{q})_{\evrs}\colonequals \bbS(\F_{q})\cap\bbG'(\F_{q})_{\evrs}$ and define $[\bbS]_{\evrs}$ to be the image of $\bbS(\F_{q})_{\evrs}$ under the map $\bbS(\F_{q})\twoheadrightarrow[\bbS]$.
We also put $\overline{\bbS}(\F_{q})_{\evrs}$ to be the image of $\bbS(\F_{q})_{\evrs}$ under the map $\bbS(\F_{q})\twoheadrightarrow\overline{\bbS}(\F_{q})$.

We caution that, in the above definition of $\bbG'(\F_{q})_{\evrs}$, the very regularity of elements of $SG^{0}_{\x,0}$ is considered in $\bfG$ (not $\bfG^{0}$).
The set of very regular elements of $SG^{0}_{\x,0}$ is stable under $Z_{\bfG}$-translation, but might not be stable under $Z_{\bfG^{0}}$-translation.
This means that $\bbG'(\F_{q})_{\evrs}$ might not be stable under the $\bbZ_{\bbG}(\F_{q})$-translation.
Based on this observation, we define a subgroup $\bbZ^{\star}_{\bbG}$ of $\bbZ_{\bbG}$ to be the reduction of $Z_{\bfG}$.
(More precisely, $\bbZ^{\star}_{\bbG}$ is the image of $\bfZ_{\bfG}(F^{\ur})$ in $\bfG^{0}(F^{\ur})_{\bar{\x}}/\bfG^{0}(F^{\ur})_{\x,0+}$.)
Then $\bbG'(\F_{q})_{\evrs}$ is stable under the $\bbZ^{\star}_{\bbG}(\F_{q})$-translation.
Furthermore, by the assumption that $\bfZ_{\bfG^{0}}/\bfZ_{\bfG}$ is anisotropic, $\bbZ^{\star}_{\bbG}(\F_{q})$ is of finite index in $\bbZ_{\bbG}(\F_{q})$.
Therefore, $(\overline{\bbG}(\F_{q})_{\evrs}, \bbG'(\F_{q})_{\evrs}, \bbZ^{\star}_{\bbG})$ introduced here satisfies the assumptions explained in the beginning of Section \ref{subsec:char-finite-gen-pos}.

\subsection{$a$-data, $\chi$-data, and $\Delta_{\II}^{\abs,\bfG}$}\label{subsec:tran}

In this subsection, we recall the notions of $a$-data and $\chi$-data, and the absolute transfer factor $\Delta_{\II}^{\abs,\bfG}$, which will be utilized to describe the characters of supercuspidal representations.

Let $\bfS$ be an $F$-rational maximal torus of $\bfG$.
Then we get the set $R(\bfG,\bfS)$ of absolute roots of $\bfS$ in $\bfG$ which has an action of the absolute Galois group $\Gamma_{F}$ of $F$.
For each $\alpha\in R(\bfG,\bfS)$, we put $\Gamma_{\alpha}$ (resp.\ $\Gamma_{\pm\alpha}$) to be the stabilizer of $\alpha$ (resp.\ $\{\pm\alpha\}$) in $\Gamma_{F}$.
Let $F_{\alpha}$ (resp.\ $F_{\pm\alpha}$) be the subfield of $\ol{F}$ fixed by $\Gamma_{\alpha}$ (resp.\ $\Gamma_{\pm\alpha}$): 
\[
F\subset F_{\pm\alpha}\subset F_{\alpha}
\quad
\longleftrightarrow
\quad
\Gamma_{F}\supset \Gamma_{\pm\alpha}\supset \Gamma_{\alpha}.
\]
\begin{itemize}
\item
When $F_{\alpha}=F_{\pm\alpha}$, we say $\alpha$ is an \textit{asymmetric} root.
\item
When $F_{\alpha}\supsetneq F_{\pm\alpha}$, we say $\alpha$ is a \textit{symmetric} root.
Note that, in this case, the extension $F_{\alpha}/F_{\pm\alpha}$ is necessarily quadratic.
Furthermore,
\begin{itemize}
\item
when $F_{\alpha}/F_{\pm\alpha}$ is unramified, we say $\alpha$ is \textit{symmetric unramified}, and
\item
when $F_{\alpha}/F_{\pm\alpha}$ is ramified, we say $\alpha$ is \textit{symmetric ramified}.
\end{itemize}
\end{itemize}

\begin{defn}[$a$-data]
A family $\{a_{\alpha}\}_{\alpha\in R(\bfG,\bfS)}$ of elements $a_{\alpha}\in F_{\alpha}^{\times}$ is called a \textit{set of $a$-data (with respect to $\bfS$)} if the following conditions are satisfied:
\begin{itemize}
\item
$a_{-\alpha}=a_{\alpha}^{-1}$ for any $\alpha\in R(\bfG,\bfS)$, and
\item
$a_{\sigma(\alpha)}=\sigma(a_{\alpha})$ for any $\alpha\in R(\bfG,\bfS)$ and $\sigma\in\Gamma_{F}$.
\end{itemize}
\end{defn}

\begin{defn}[$\chi$-data]
A family $\{\chi_{\alpha}\}_{\alpha\in R(\bfG,\bfS)}$ of characters $\chi_{\alpha}\colon F_{\alpha}^{\times}\rightarrow\C^{\times}$ is called a \textit{set of $\chi$-data (with respect to $\bfS$)} if the following conditions are satisfied:
\begin{itemize}
\item
$\chi_{-\alpha}=\chi_{\alpha}^{-1}$ for any $\alpha\in R(\bfG,\bfS)$,
\item
$\chi_{\sigma(\alpha)}=\chi_{\alpha}\circ\sigma^{-1}$ for any $\alpha\in R(\bfG,\bfS)$ and $\sigma\in\Gamma_{F}$, and
\item
the restriction of $\chi_{\alpha}$ to $F_{\pm\alpha}^{\times}$ is the nontrivial quadratic character corresponding to the quadratic extension $F_{\alpha}/F_{\pm\alpha}$ for any symmetric root $\alpha\in R(\bfG,\bfS)$.
\end{itemize}
\end{defn}

\begin{defn}[$\Delta_{\II}^{\abs,\bfG}$]
Let $a=\{a_{\alpha}\}_{\alpha}$ be a set of $a$-data and $\chi=\{\chi_{\alpha}\}_{\alpha}$ a set of $\chi$-data with respect to $\bfS$.
We define a function $\Delta_{\II,\bfS}^{\abs,\bfG}[a,\chi]\colon S\rightarrow\C^{\times}$ by
\[
\Delta_{\II,\bfS}^{\abs,\bfG}[a,\chi](\gamma)
\colonequals \prod_{\begin{subarray}{c}\Gamma_{F}\backslash R(\bfG,\bfS)\\ \alpha(\gamma)\neq1\end{subarray}} \chi_{\alpha}\biggl(\frac{\alpha(\gamma)-1}{a_{\alpha}}\biggr).
\]
\end{defn}

Of course, the function $\Delta_{\II,\bfS}^{\abs,\bfG}[a,\chi]$ on $S$ depends on the choices of a set of $a$-data and a set of $\chi$-data with respect to $\bfS$.
In \cite[Section 4.7]{Kal19}, Kaletha associated a set of $a$-data $a_{\Psi,\bfS}=\{a_{\Psi,\bfS,\alpha}\}_{\alpha}$ and a set of $\chi$-data $\chi'_{\Psi,\bfS}=\{\chi'_{\Psi,\bfS,\alpha}\}_{\alpha}$ to each Yu-datum $\Psi$ and a tamely ramified maximal torus $\bfS\subset\bfG^{0}$.
Then he described the character of the supercuspidal representation $\pi_{\Psi}^{\Yu}$ by using the function $\Delta_{\II,\bfS}^{\abs,\bfG}[a_{\Psi,\bfS},\chi'_{\Psi,\bfS}]$ when $\Psi$ is a regular Yu-datum.
(In \cite[Section 4.7]{Kal19}, the sets $a_{\Psi,\bfS}$ and $\chi'_{\Psi,\bfS}$ are simply written by $a$ and $\chi'$, respectively.)
On the other hand, in his more recent paper \cite{Kal19-sc}, Kaletha introduced another set of $\chi$-data $\chi''_{\Psi,\bfS}$ for a better understanding of the local Langlands correspondence for supercuspidal representations \cite[Section 3.5]{Kal19-sc}.

Let us recall the definitions of Kaletha's $a$-data $a_{\Psi,\bfS}$ and $\chi$-data $\chi'_{\Psi,\bfS}$, $\chi''_{\Psi,\bfS}$ (see \cite[Section 4.2]{FKS21} for the details).
%\cite[Section 4.7]{Kal19}, \cite[Section 3.5]{Kal19-sc}
Let $\Psi=(\vec{\bfG},\vec{\phi},\vec{r},\x,\rho_{0})$ be a Yu-datum.
(Note that the following construction depends only on the clipped part $\dashover{\Psi}$ of $\Psi$.)
Let $\bfS$ be a tamely ramified maximal torus of $\bfG^{0}$.

Since $\vec{\bfG}$ is a sequence $\bfG^{0}\subset\bfG^{1}\subset\cdots\subset\bfG^{d}=\bfG$ of tame twisted Levi subgroups, each $\alpha\in R(\bfG,\bfS)$ belongs to $R(\bfG^{i},\bfS)\smallsetminus R(\bfG^{i-1},\bfS)$ for a unique $0\leq i\leq d$ (we put $\bfG^{-1}\colonequals \bfS$).
When $i=0$, we simply put $a_{\Psi,\bfS,\alpha}\colonequals 1$.
Thus we suppose that $\alpha\in R(\bfG^{i},\bfS)\smallsetminus R(\bfG^{i-1},\bfS)$ for $0<i\leq d$ in the following.
Let $\bfG^{i-1}_{\sc}$ and $\bfS_{\sc}$ denote the preimages of $\bfG^{i-1}$ and $\bfS$ in the simply-connected cover of the derived group of $\bfG^{i}$, respectively.
Let $\bfG^{i-1}_{\sc,\ab}$ be the maximal abelian quotient of $\bfG^{i-1}_{\sc}$.
We fix a $\bfG^{i}$-generic element $X^{\ast}_{i-1}\in\Lie^{\ast}(\bfG^{i-1}_{\sc,\ab})(F)$ of depth $r_{i-1}$ which represents the character $\phi_{i-1}$, where $\Lie^{\ast}$ denotes the dual Lie algebra.\footnote{We caution that here we adopt a different convention from \cite{Yu01}. In \cite{Yu01}, an element of $\Lie^{\ast}(\bfZ^{i-1})(F)$ is used to represent the character $\phi_{i-1}$, where $\bfZ^{i-1}$ is the identity component of the center of $\bfG^{i-1}$. See \cite[Remark 4.1.3]{FKS21} for the details.}
We regard $X^{\ast}_{i-1}$ as an element of $\Lie^{\ast}(\bfS_{\sc})(F)$ via the injection $\Lie^{\ast}(\bfG^{i-1}_{\sc,\ab})(F)\hookrightarrow\Lie^{\ast}(\bfS_{\sc})(F)$ induced from the surjection $\bfS_{\sc}\twoheadrightarrow\bfG^{i-1}_{\sc,\ab}$.
We put $H_{\alpha}\colonequals d\alpha_{\sc}^{\vee}(1)\in\Lie(\bfS_{\sc})(F_{\alpha})$, where $\alpha_{\sc}$ is the element of $R(\bfG^{i-1}_{\sc},\bfS_{\sc})$ corresponding to $\alpha$ under the identification $R(\bfG^{i-1}_{\sc},\bfS_{\sc})\cong R(\bfG^{i-1},\bfS)$.
Then we define $a_{\alpha}\in F_{\alpha}$ by
\[
a_{\Psi,\bfS,\alpha}\colonequals \langle X^{\ast}_{i-1},H_{\alpha}\rangle.
\]

We next recall the definition of the $\chi$-data $\{\chi'_{\Psi,\bfS,\alpha}\}_{\alpha}$.
Let $\alpha\in R(\bfG,\bfS)$.
\begin{itemize}
\item
When $\alpha$ is asymmetric, we define $\chi'_{\Psi,\bfS,\alpha}\colon F_{\alpha}^{\times}\rightarrow\C^{\times}$ to be the trivial character.
\item
When $\alpha$ is symmetric unramified, we define $\chi'_{\Psi,\bfS,\alpha}\colon F_{\alpha}^{\times}\rightarrow\C^{\times}$ to be the unique unramified quadratic character.
\item
When $\alpha$ is symmetric ramified, we define $\chi'_{\Psi,\bfS,\alpha}\colon F_{\alpha}^{\times}\rightarrow\C^{\times}$ to be the unique character such that 
\begin{itemize}
\item
its restriction $\chi'_{\Psi,\bfS,\alpha}|_{\mcO_{F_{\alpha}}^{\times}}$ to $\mcO_{F_{\alpha}^{\times}}$ is the inflation of the nontrivial quadratic character of $k_{F_{\alpha}}^{\times}$, and
\item
we have $\chi'_{\Psi,\bfS,\alpha}(2a_{\alpha})=\lambda_{F_{\alpha}/F_{\pm\alpha}}(\psi\circ\Tr_{F_{\pm\alpha}/F})$, where $\lambda_{F_{\alpha}/F_{\pm\alpha}}(\psi\circ\Tr_{F_{\pm\alpha}/F})$ denotes the Langlands constant (recall that $\psi$ is the fixed nontrivial additive character of $F$).
\end{itemize}
\end{itemize}

We finally recall the definition of the $\chi$-data $\{\chi''_{\Psi,\bfS,\alpha}\}_{\alpha}$.
Let $\bfZ^{i-1}$ denotes the identity component of the center of $\bfG^{i-1}$.
For $\alpha\in R(\bfG^{i},\bfS)\smallsetminus R(\bfG^{i-1},\bfS)$, we write $\alpha_{0}$ for the restriction of $\alpha$ to $\bfZ^{i-1}$.
Let $F_{0}$ be the subfield of $\overline{F}$ fixed by $\Gamma_{\alpha_{0}}\colonequals \Stab_{\Gamma_{F}}(\alpha_{0})$.
\begin{itemize}
\item
When $\alpha_{0}$ is asymmetric, we define $\chi''_{\Psi,\bfS,\alpha_{0}}\colon F_{\alpha_{0}}^{\times}\rightarrow\C^{\times}$ to be the trivial character.
\item
When $\alpha_{0}$ is symmetric unramified, we define $\chi''_{\Psi,\bfS,\alpha_{0}}\colon F_{\alpha_{0}}^{\times}\rightarrow\C^{\times}$ to be the unique unramified quadratic character.
\item
When $\alpha$ is symmetric ramified, we define $\chi''_{\Psi,\bfS,\alpha_{0}}\colon F_{\alpha_{0}}^{\times}\rightarrow\C^{\times}$ to be the unique character such that 
\begin{itemize}
\item
its restriction $\chi''_{\Psi,\bfS,\alpha_{0}}|_{\mcO_{F_{\alpha_{0}}}^{\times}}$ to $\mcO_{F_{\alpha_{0}}^{\times}}$ is the inflation of the nontrivial quadratic character of $k_{F_{\alpha_{0}}}^{\times}$, and
\item
we have $\chi''_{\Psi,\bfS,\alpha_{0}}(\ell(\alpha^{\vee})a_{\alpha})=(-1)^{f_{\alpha_{0}}+1}\mfG_{k_{\alpha_{0}}}(\psi^{0}\circ\Tr_{k_{\alpha_{0}}/k_{F}})$.
\end{itemize}
\end{itemize}
Then we define $\chi''_{\Psi,\bfS,\alpha}\colon F_{\alpha}^{\times}\rightarrow\C^{\times}$ by $\chi''_{\Psi,\bfS,\alpha}\colonequals \chi''_{\Psi,\bfS,\alpha_{0}}\circ\Nr_{F_{\alpha}/F_{\alpha_{0}}}$.
Here we do not explain the definitions of the quantities appearing in the symmetric ramified case (see \cite[Section 4.2]{FKS21} for the details).
The only important thing for us is that both characters $\chi'_{\Psi,\bfS,\alpha}$ and $\chi''_{\Psi,\bfS,\alpha}$ are tamely ramified (i.e., trivial on $1+\mfp_{F_{\alpha}}$) for every $\alpha\in R(\bfG,\bfS)$.

The following properties of the function $\Delta_{\II,\bfS}^{\abs,\bfG}$ associated with Kaletha's $a$-data and $\chi$-data will be needed later.

\begin{lem}[{\cite[(Proof of) Lemma 4.2.7]{FKS21}}]\label{lem:tran-diff}
With the above notations, for any $\gamma\in S$, we have
\[
\frac{\Delta_{\II,\bfS}^{\abs,\bfG}[a_{\Psi,\bfS},\chi''_{\Psi,\bfS}](\gamma)}{\Delta_{\II,\bfS}^{\abs,\bfG^{0}}[a_{\Psi,\bfS},\chi''_{\Psi,\bfS}](\gamma)}
=\frac{\Delta_{\II,\bfS}^{\abs,\bfG}[a_{\Psi,\bfS},\chi'_{\Psi,\bfS}](\gamma)}{\Delta_{\II,\bfS}^{\abs,\bfG^{0}}[a_{\Psi,\bfS},\chi'_{\Psi,\bfS}](\gamma)}
\cdot\epsilon_{\Psi,\bfS,\flat}(\gamma),
\]
where
\begin{itemize}
\item
$\Delta_{\II,\bfS}^{\abs,\bfG^{0}}$ is a function defined in the same way as $\Delta_{\II,\bfS}^{\abs,\bfG}$ by the product over $\{\alpha\in\Gamma_{F}\backslash R(\bfG^{0},\bfS) \mid \alpha(\gamma)\neq1\}$, and
\item
 $\epsilon_{\Psi,\bfS,\flat}$ is the character of $S$ mentioned in Section \ref{subsec:tamesc}.
 \end{itemize}
\end{lem}

\begin{lem}\label{lem:tran-translation}
When $\gamma\in S$ is a very regular element of $G$, for any $\gamma_{+}\in S_{0+}$, we have
\[
\Delta_{\II,\bfS}^{\abs,\bfG}[a_{\Psi,\bfS},\chi''_{\Psi,\bfS}](\gamma\cdot\gamma_{+})
=\Delta_{\II,\bfS}^{\abs,\bfG}[a_{\Psi,\bfS},\chi''_{\Psi,\bfS}](\gamma).
\]
\end{lem}

\begin{proof}
Since $\gamma$ is very regular, $\alpha(\gamma)\not\equiv1 \pmod{\mfp_{\ol{F}}}$ for any $\alpha\in R(\bfG,\bfS)$.
On the other hand, since $\gamma_{+}$ belongs to $S_{0+}$, $\alpha(\gamma_{+})\equiv1 \pmod{\mfp_{\ol{F}}}$ for any $\alpha\in R(\bfG,\bfS)$.
Hence $\alpha(\gamma)\equiv\alpha(\gamma\cdot\gamma_{+})\not\equiv1 \pmod{\mfp_{\ol{F}}}$ for any $\alpha\in R(\bfG,\bfS)$.
By noting that $\chi''_{\Psi,\bfS,\alpha}$ is tamely ramified, this implies that
\[
\chi''_{\Psi,\bfS,\alpha}\biggl(\frac{\alpha(\gamma)-1}{a_{\Psi,\bfS,\alpha}}\biggr)
=
\chi''_{\Psi,\bfS,\alpha}\biggl(\frac{\alpha(\gamma\cdot\gamma_{+})-1}{a_{\Psi,\bfS,\alpha}}\biggr)
\]
for any $\alpha\in R(\bfG,\bfS)$.
Thus $\Delta_{\II,\bfS}^{\abs,\bfG}[a_{\Psi,\bfS},\chi''_{\Psi,\bfS}](\gamma)$ equals $\Delta_{\II,\bfS}^{\abs,\bfG}[a_{\Psi,\bfS},\chi''_{\Psi,\bfS}](\gamma\cdot\gamma_{+})$.
\end{proof}

\begin{lem}\label{lem:tran-conj}
The function $\Delta_{\II,\bfS}^{\abs,\bfG}[a_{\Psi,\bfS},\chi''_{\Psi,\bfS}]$ is $G^{0}$-conjugation invariant.
More precisely, for any $g\in G^{0}$, 
\[
\Delta_{\II,\bfS}^{\abs,\bfG}[a_{\Psi,\bfS},\chi''_{\Psi,\bfS}](\gamma)
=\Delta_{\II,{}^{g}\bfS}^{\abs,\bfG}[a_{\Psi,{}^{g}\bfS},\chi''_{\Psi,{}^{g}\bfS}]({}^{g}\gamma),
\]
where $a_{\Psi,{}^{g}\bfS}$ and $\chi''_{\Psi,{}^{g}\bfS}$ are Kaletha's $a$-data and $\chi$-data with respect to the tamely ramified maximal torus ${}^{g}\bfS\subset\bfG^{0}$.
\end{lem}

\begin{proof}
Let $g\in G^{0}$.
We let $\alpha\leftrightarrow{}^{g}\alpha$ denote the identification $R(\bfG,\bfS)\cong R(\bfG,{}^{g}\bfS)$ induced by the $g$-conjugation.
It is enough to check that 
\[
\chi''_{\Psi,\bfS,\alpha}\biggl(\frac{\alpha(\gamma)-1}{a_{\Psi,\bfS,\alpha}}\biggr)
=
\chi''_{\Psi,{}^{g}\bfS,{}^{g}\alpha}\biggl(\frac{{}^{g}\alpha({}^{g}\gamma)-1}{a_{\Psi,{}^{g}\bfS,{}^{g}\alpha}}\biggr)
\]
for any $\alpha\in R(\bfG,\bfS)$.
Note that ${}^{g}\alpha({}^{g}\gamma)=\alpha(\gamma)$.
Since $g$ is $F$-rational, the identification $R(\bfG,\bfS)\cong R(\bfG,{}^{g}\bfS)$ is $\Gamma_{F}$-equivariant.
Therefore, by the construction of $\chi''_{\Psi,\bfS,\alpha}$, it suffices to show that $a_{\Psi,{}^{g}\bfS,{}^{g}\alpha}\cdot a_{\Psi,\bfS,\alpha}^{-1}\in 1+\mfp_{F_{\alpha}}$.
When $\alpha$ belongs to $R(\bfG^{0},\bfS)$, this is obvious since $a_{\Psi,\bfS,\alpha}=a_{\Psi,{}^{g}\bfS,{}^{g}\alpha}=1$.
Let us suppose that $\alpha$ belongs to $R(\bfG^{i},\bfS)\smallsetminus R(\bfG^{i-1},\bfS)$ for $0<i\leq d$.
Recall that we put
\[
a_{\Psi,\bfS,\alpha}
=\langle X^{\ast}_{i-1},H_{\alpha}\rangle
\quad\text{and}\quad
a_{\Psi,{}^{g}\bfS,{}^{g}\alpha}
=\langle X^{\ast}_{i-1},H_{{}^{g}\alpha}\rangle,
\]
where $X^{\ast}_{i-1}$ is a(ny) $\bfG^{i}$-generic element of $\Lie^{\ast}(\bfG^{i-1}_{\sc,\ab})(F)$ of depth $r_{i-1}$ which represents the character $\phi_{i-1}$.
Note that $X^{\ast,g}_{i-1}\colonequals g^{-1}X^{\ast}_{i-1}g$ is a $\bfG^{i}$-generic element of $\Lie^{\ast}(\bfG^{i-1}_{\sc,\ab})(F)$ of depth $r_{i-1}$ which represents the character $\phi_{i-1}^{g}$.
Since $g$ belongs to $G^{0}$, we have $\phi_{i-1}^{g}=\phi_{i-1}$.
Hence both $X^{\ast}_{i-1}$ and $X^{\ast,g}_{i-1}$ represents the character $\phi_{i-1}$.
This implies that $X^{\ast,g}_{i-1}\in X^{\ast}_{i-1}\cdot(1+\mfp_{F_{\alpha}})$.
Thus we have
\[
a_{\Psi,{}^{g}\bfS,{}^{g}\alpha}
=\langle X^{\ast}_{i-1},H_{{}^{g}\alpha}\rangle
=\langle X^{\ast,g}_{i-1},H_{\alpha}\rangle
\in\langle X^{\ast}_{i-1},H_{\alpha}\rangle\cdot(1+\mfp_{F_{\alpha}})
=a_{\Psi,\bfS,\alpha}\cdot(1+\mfp_{F_{\alpha}}). \qedhere
\]
\end{proof}

When $\gamma\in G$ is a regular semisimple element contained in $G^{0}$, any maximal torus of $\bfG$ containing $\gamma$ is necessarily equal to $\bfT_{\gamma}$, which is contained in $\bfG^{0}$.
For this reason, we simply write
\begin{align*}
\Delta_{\II}^{\abs,\bfG}[a_{\Psi},\chi'_{\Psi}](\gamma)
&\colonequals \Delta_{\II,\bfT_{\gamma}}^{\abs,\bfG}[a_{\Psi,\bfT_{\gamma}},\chi'_{\Psi,\bfT_{\gamma}}](\gamma),
\quad\text{and}\\
\Delta_{\II}^{\abs,\bfG}[a_{\Psi},\chi''_{\Psi}](\gamma)
&\colonequals \Delta_{\II,\bfT_{\gamma}}^{\abs,\bfG}[a_{\Psi,\bfT_{\gamma}},\chi''_{\Psi,\bfT_{\gamma}}](\gamma)
\end{align*}
for any regular semisimple element $\gamma\in G$ contained in $G^{0}$.

\subsection{Character formula at very regular elements}\label{subsec:character}
Let $\pi^{\FKS}_{\Psi}$ be the supercuspidal representation associated with a Yu-datum $\Psi=(\vec{\bfG},\vec{r},\x,\rho_{0},\vec{\phi})$ via modified construction of Fintzen--Kaletha--Spice (see Section \ref{subsec:tamesc}).

We first show supplementary lemmas which will be needed in the proof of our character formula.

\begin{lem}\label{lem:Weyl}
Let $\gamma\in G$ be an elliptic very regular element.
\begin{enumerate}
\item
We have $\{g\in G \mid {}^{g}\gamma\in G^{0}_{\bar{\x}}\}=N_{G}(T_{\gamma},G^{0}_{\bar{\x}})$.
\item
For any $g\in N_{G}(T_{\gamma},G^{0}_{\bar{\x}})$, we have $N_{G}(T_{\gamma},G^{0}_{\bar{\x}})=N_{G_{\bar{\x}}}({}^{g}T_{\gamma},G^{0}_{\bar{\x}})g$.
\end{enumerate}
\end{lem}

\begin{proof}
\begin{enumerate}
\item
As $\gamma$ belongs to $T_{\gamma}$, the inclusion $\{g\in G \mid {}^{g}\gamma\in G^{0}_{\bar{\x}}\} \supset N_{G}(T_{\gamma},G^{0}_{\bar{\x}})$ is obvious.
Let us consider the converse inclusion.
Let $g\in G$ be an element satisfying ${}^{g}\gamma\in G^{0}_{\bar{\x}}$.
Then, in particular, we have ${}^{g}\gamma\in \bfG^{0}$.
As $\gamma$ is regular semisimple element in $\bfG$ and $\bfG^{0}$ is a tame twisted Levi subgroup of $\bfG$, this implies that ${}^{g}\bfT_{\gamma}\subset \bfG^{0}$, hence ${}^{g}T_{\gamma}\subset G^{0}$.
On the other hand, as ${}^{g}\gamma$ is elliptic very regular with point $g\x_{\gamma}$, Lemma \ref{lem:point} implies that $g\bar{\x}_{\gamma}=\bar{\x}$.
Hence we have ${}^{g}\bfT_{\gamma}\subset G_{\bar{\x}}$.
Thus we get ${}^{g}T_{\gamma}\subset G_{\bar{\x}}\cap G^{0}=G^{0}_{\bar{\x}}$.

\item
The inclusion $N_{G}(T_{\gamma},G^{0}_{\bar{\x}})\supset N_{G_{\bar{\x}}}({}^{g}T_{\gamma},G^{0}_{\bar{\x}})g$ is obvious.
Let us show the converse inclusion.
Let $g_{1}$ and $g_{2}$ be elements of $N_{G}(T_{\gamma},G^{0}_{\bar{\x}})$.
Since a point $g_{i}\x_{\gamma}$ is associated with ${}^{g_{i}}\gamma\in G_{\bar{\x}}$, we get $g_{i}\bar{\x}_{\gamma}=\bar{\x}$ ($i=1,2$) by Lemma \ref{lem:point}.
Hence $n\colonequals g_{1}g_{2}^{-1}$ belongs to $G_{\bar{\x}}$.
As ${}^{n}({}^{g_{2}}T_{\gamma})={}^{g_{1}}T_{\gamma}\subset G^{0}_{\bar{\x}}$, $n$ belongs to $N_{G_{\bar{\x}}}({}^{g_{2}}T_{\gamma},G^{0}_{\bar{\x}})$. \qedhere
\end{enumerate}
\end{proof}

\begin{lem}\label{lem:Weyl2}
Let $\gamma\in G$ be an elliptic very regular element.
Then we have $K^{d}\cap N_{G_{\bar{\x}}}(T_{\gamma},G^{0}_{\bar{\x}})=N_{G^{0}_{\bar{\x}}}(T_{\gamma},G^{0}_{\bar{\x}})$.
\end{lem}

\begin{proof}
This follows from the same argument as in the proof of \cite[Lemma 4.10]{CO21}.
We explain it for the sake of completeness.

Recall that $K^{d}=G^{0}_{\bar{\x}}(G^{0},\ldots,G^{d})_{\x,(0+,s_{0},\ldots,s_{d-1})}$.
The inclusion $K^{d}\cap N_{G_{\bar{\x}}}(T_{\gamma},G^{0}_{\bar{\x}})\supset N_{G^{0}_{\bar{\x}}}(T_{\gamma},G^{0}_{\bar{\x}})$ is obvious.
Let us show the converse.
Let $g\in K^{d}\cap N_{G_{\bar{\x}}}(T_{\gamma},G^{0}_{\bar{\x}})$.
We write $g=g^{0}k$ with elements $g^{0}\in G^{0}_{\bar{\x}}$ and $k\in(G^{0},\ldots,G^{d})_{\x,(0+,s_{0},\ldots,s_{d-1})}$.
Since $g$ satisfies ${}^{g}T_{\gamma}\subset G^{0}_{\bar{\x}}$, we have $g\bar{\x}_{\gamma}=\bar{\x}$ by Lemma \ref{lem:point}.
As $g\in G_{\bar{\x}}$, this implies that $\bar{\x}_{\gamma}=\bar{\x}$.
Moreover, as $g^{0}\in G^{0}_{\bar{\x}}$, we have ${}^{k}T_{\gamma}\subset G^{0}_{\bar{\x}}$.
Then, by applying \cite[Lemma 9.10]{AS08} to $(\bfG^{d-1},\bfG^{d})$, we get $k\in G^{d-1}_{\x,0+}T_{\gamma,0+}=G^{d-1}_{\x,0+}$.
Thus $k$ belongs to $(G^{0},\ldots,G^{d})_{\x,(0+,s_{0},\ldots,s_{d-1})}\cap G^{d-1}_{\x,0+}$, which equals to $(G^{0},\ldots,G^{d-1})_{\x,(0+,s_{0},\ldots,s_{d-2})}$.
Repeating this procedure for $(\bfG^{d-2},\bfG^{d-1}),\ldots(\bfG^{0},\bfG^{1})$, we eventually get $k\in G^{0}_{\x,0+}$.
Hence we obtain $g=g^{0}k\in G^{0}_{\bar{\x}}$.
\end{proof}

\begin{prop}\label{prop:CF1}
Let $\gamma \in G$ be an elliptic very regular element.
Then we have
\[
\Theta_{\pi^{\FKS}_{\Psi}}(\gamma)
=
\phi_{d}(\gamma)\sum_{g \in G^{0}_{\bar{\x}}\backslash N_{G}(T_{\gamma},G^{0}_{\bar{\x}})} \Theta_{\rho_{d}}({}^{g}\gamma)\cdot\epsilon_{\Psi}({}^{g}\gamma).
\]
\end{prop}

\begin{proof}

Recall that the irreducible supercuspidal representation $\pi^{\FKS}_{\Psi}$ is obtained by the compact induction of $\rho_{\Psi}^{\Yu}\otimes\epsilon_{\Psi}=(\rho_{d}\otimes\phi_{d})\otimes\epsilon_{\Psi}$ from $K^{d}$ to $G$ (see Section \ref{subsec:tamesc}).
Since $\gamma$ is an elliptic regular semisimple element, the Harish-Chandra integration formula (see \cite[page 94]{HC70} and also \cite[proof of Theorem 6.4]{AS09} for the validity in the positive characteristic case) gives 
\[
\Theta_{\pi^{\FKS}_{\Psi}}(\gamma)
= 
\frac{\deg\pi^{\FKS}_{\Psi}}{\dim\rho_{d}} \phi_d(\gamma)
\int_{G/Z_\bfG}\dot{\Theta}_{\rho_{d}}({}^{g} \gamma)\cdot\epsilon_{\Psi}({}^{g}\gamma) \, dg,
\]
where $\deg\pi^{\FKS}_{\Psi}$ is the formal degree of $\pi^{\FKS}_{\Psi}$ with respect to a fixed Haar measure $dg$ of $G/Z_{\bfG}$ and $\dot{\Theta}_{\rho_{d}}$ is the zero extension of the character $\Theta_{\rho_{d}}$ of $\rho_{d}$ from $K^{d}$ to $G$.
From this, we see that $\Theta_{\pi^{\FKS}_{\Psi}}(\gamma)=0$ if $g$ is not $G$-conjugate to any element of $K^{d}$, or equivalently (by Lemma \ref{lem:CO21-4.7}), $\gamma$ is not $G$-conjugate to any element of $G^{0}_{\bar{\x}}$.
Note that, in this case, the index set of the sum in the assertion is empty by Lemma \ref{lem:Weyl} (1), hence we get the assertion.

From now on, let $g\in G$ be an elliptic very regular element which is $G$-conjugate to an element of $G^{0}_{\bar{\x}}$.
Let us fix an element $g_{\gamma}\in G$ such that ${}^{g_{\gamma}}\gamma\in G^{0}_{\bar{\x}}$.
Note that then we have $g_{\gamma}\bar{\x}_{\gamma}=\bar{\x}$ by Lemma \ref{lem:point}.
We first argue that the function $g \mapsto \dot{\Theta}_{\rho_{d}}({}^g \gamma)\cdot\epsilon_{\Psi}({}^{g}\gamma)$ on $G/Z_\bfG$ is supported on $G_{\bar{\x}}g_{\gamma}/Z_\bfG$.
For any $g\in G$, the point $g\x_{\gamma}$ is associated with ${}^{g}\gamma$.
Hence, if ${}^{g}\gamma$ belongs to $K^{d}$, we have $g\bar{\x}_{\gamma}=\bar{\x}$ by Lemma \ref{lem:point}.
In other words, $gg_{\gamma}^{-1}$ necessarily belongs to the stabilizer subgroup $G_{\bar{\x}}$ of $\bar{\x}$.
In particular, unless $g\in G_{\bar{\x}}g_{\gamma}$,  we have ${}^g \gamma \notin K^{d}$ and $\dot{\Theta}_{\rho_{d}}({}^g \gamma) = 0$.
Therefore $g \mapsto \dot{\Theta}_{\rho_{d}}({}^g \gamma)\cdot\epsilon_{\Psi}({}^{g}\gamma)$ is supported on $G_{\bar{\x}}g_{\gamma}/Z_\bfG$.

We next note that $\rho_{d}\otimes\epsilon_{\Psi}$ is a representation of $K^{d}$, hence its character $\Theta_{\rho_{d}}\cdot\epsilon_{\Psi}$ is invariant under $K^{d}$-conjugation.
Then we can compute the integral as follows:
\begin{align*}
\int_{G/Z_\bfG} \dot{\Theta}_{\rho_{d}}({}^{g} \gamma)\cdot\epsilon_{\Psi}({}^{g}\gamma) \, dg
&=
\sum_{g' \in K^{d} \backslash G_{\bar{\x}}g_{\gamma}} \int_{K^{d} g'/Z_\bfG} \dot{\Theta}_{\rho_{d}}({}^{g} \gamma)\cdot\epsilon_{\Psi}({}^{g}\gamma) \, dg \\
&=\sum_{g' \in K^{d} \backslash G_{\bar{\x}}g_{\gamma}} \mathrm{meas}(K^{d} g'/Z_\bfG) \cdot \dot{\Theta}_{\rho_{d}}({}^{g'} \gamma)\cdot\epsilon_{\Psi}({}^{g'}\gamma) \\
&=\mathrm{meas}(K^{d}/Z_\bfG)\sum_{g \in K^{d} \backslash G_{\bar{\x}}g_{\gamma}}  \dot{\Theta}_{\rho_{d}}({}^{g} \gamma)\cdot\epsilon_{\Psi}({}^{g}\gamma).
\end{align*}
Since $\deg\pi^{\FKS}_{\Psi}=\dim\rho_{d}\cdot\mathrm{meas}(K^{d}/Z_\bfG)^{-1}$ (see, e.g., \cite[Theorem A.14]{BH96}),
\[
\Theta_{\pi^{\FKS}_{\Psi}}(\gamma)
=
\phi_{d}(\gamma) \sum_{g\in K^{d}\backslash G_{\bar{\x}}g_{\gamma}}\dot{\Theta}_{\rho_{d}}({}^{g}\gamma)\cdot\epsilon_{\Psi}({}^{g}\gamma).
\]

We finally rewrite the index set.
Let us put $\gamma'\colonequals {}^{g_{\gamma}}\gamma$.
Then we have
\[
\sum_{g\in K^{d}\backslash G_{\bar{\x}}g_{\gamma}}\dot{\Theta}_{\rho_{d}}({}^{g}\gamma)\cdot\epsilon_{\Psi}({}^{g}\gamma)
=
\sum_{g\in K^{d}\backslash G_{\bar{\x}}}\dot{\Theta}_{\rho_{d}}({}^{g}\gamma')\cdot\epsilon_{\Psi}({}^{g}\gamma').
\]
Whenever ${}^{g}\gamma'$ belongs to $K^{d}$, we may suppose that ${}^{g}\gamma'$ belongs to $G^{0}_{\bar{\x}}$ ($\subset K^{d}$) by replacing it with its $K^{d}$-conjugate element by Lemma \ref{lem:CO21-4.7}.
Hence
\[
\sum_{g\in K^{d}\backslash G_{\bar{\x}}}\dot{\Theta}_{\rho_{d}}({}^{g}\gamma')\cdot\epsilon_{\Psi}({}^{g}\gamma')
=
\sum_{\begin{subarray}{c}g \in K^{d}\backslash G_{\bar{\x}} \\  {}^{g}\gamma'\in G^{0}_{\bar{\x}}\end{subarray}}
\Theta_{\rho_{d}}({}^{g}\gamma')\cdot\epsilon_{\Psi}({}^{g}\gamma').
\]

Thus, by Lemma \ref{lem:Weyl} (1), the index set on the right-hand side can be rewritten as 
\[
(K^{d}\cap N_{G_{\bar{\x}}}(T_{\gamma'},G^{0}_{\bar{\x}}))\backslash N_{G_{\bar{\x}}}(T_{\gamma'},G^{0}_{\bar{\x}}).
\]
By Lemma \ref{lem:Weyl2}, $K^{d}\cap N_{G_{\bar{\x}}}(T_{\gamma'},G^{0}_{\bar{\x}})$ equals $N_{G^{0}_{\bar{\x}}}(T_{\gamma'},G^{0}_{\bar{\x}})=G^{0}_{\bar{\x}}$.
Therefore we get
\begin{align*}
\sum_{\begin{subarray}{c}g \in K^{d}\backslash G_{\bar{\x}} \\  {}^{g}\gamma'\in G^{0}_{\bar{\x}}\end{subarray}}
\Theta_{\rho_{d}}({}^{g}\gamma')\cdot\epsilon_{\Psi}({}^{g}\gamma')
&=
\sum_{g\in G^{0}_{\bar{\x}}\backslash N_{G_{\bar{\x}}}(T_{\gamma'},G^{0}_{\bar{\x}})}
\Theta_{\rho_{d}}({}^{g}\gamma')\cdot\epsilon_{\Psi}({}^{g}\gamma')\\
&=
\sum_{g\in G^{0}_{\bar{\x}}\backslash N_{G_{\bar{\x}}}(T_{\gamma'},G^{0}_{\bar{\x}})g_{\gamma}}
\Theta_{\rho_{d}}({}^{g}\gamma)\cdot\epsilon_{\Psi}({}^{g}\gamma)\\
&=
\sum_{g\in G^{0}_{\bar{\x}}\backslash N_{G}(T_{\gamma},G^{0}_{\bar{\x}})}
\Theta_{\rho_{d}}({}^{g}\gamma)\cdot\epsilon_{\Psi}({}^{g}\gamma).
\end{align*}
Here, we used Lemma \ref{lem:Weyl} (2) in the last equality.
\end{proof}

Let us investigate the summands in the formula in Proposition \ref{prop:CF1}.
We put $G^{0}_{\bar{\x},\evrs}$ to be the set of elliptic very regular elements of $G$ which lies in $G^{0}_{\bar{\x}}$.
(Thus note that $G^{0}_{\bar{\x},\evrs}$ depends on $\bfG$ since the very regularity is considered in $\bfG$.)

\begin{prop}\label{prop:CF2}
For any $\gamma\in G^{0}_{\bar{\x},\evrs}$, we have
\[
\phi_{d}(\gamma)\cdot\Theta_{\rho_{d}}(\gamma)\cdot\epsilon_{\Psi}(\gamma)
=
\Theta_{\rho_{0}}(\gamma)\cdot\frac{e(\bfG)}{e(\bfG^{0})}\cdot \varepsilon_{L}(\bfT_{\bfG^{\ast}}-\bfT_{\bfG^{0\ast}})\cdot\frac{\Delta_{\II}^{\abs,\bfG}[a_{\Psi},\chi''_{\Psi}](\gamma)}{\Delta_{\II}^{\abs,\bfG^{0}}[a_{\Psi},\chi''_{\Psi}](\gamma)}\cdot\phi_{\geq0}(\gamma),
\]
where 
\begin{itemize}
\item
$e(\bfG)$ (resp.\ $e(\bfG^{0})$) denotes the Kottwitz sign of $\bfG$ (resp.\ $\bfG^{0}$), 
\item
$\bfT_{\bfG^{\ast}}$ (resp.\ $\bfT_{\bfG^{0\ast}}$) denotes a minimal Levi subgroup of the quasi-split inner form of $\bfG$ (resp. $\bfG^{0}$),
\item
$\varepsilon_{L}(\bfT_{\bfG^{\ast}}-\bfT_{\bfG^{0\ast}})$ denotes the central value of the $\varepsilon$-factor of the virtual complex representation $X^{\ast}(\bfT_{\bfG^{\ast}})_{\C}-X^{\ast}(\bfT_{\bfG^{0\ast}})_{\C}$ of the absolute Galois group $\Gamma_{F}$.
\end{itemize}
\end{prop}

\begin{proof}
Since $\gamma$ belongs to $G^{0}_{\bar{\x}}$, we get
\[
\phi_{d}(\gamma)\Theta_{\rho_{d}}(\gamma)
=
\Theta_{\rho_{0}}(\gamma)
\prod_{i=0}^{d-1}\Theta_{\tilde{\phi}_{i}}(\gamma\ltimes1)
\prod_{i=0}^{d}\phi_{i}(\gamma)
\]
by the construction of the representation $\rho_{d}$.
The same argument as in the proof of \cite[Proposition 4.9]{CO21} shows that
\[
\Theta_{\tilde{\phi}_{i}}(\gamma\ltimes1)
=
\varepsilon_{\sym,\ram}(\pi',\gamma)\cdot\varepsilon^{\ram}(\pi',\gamma)\cdot\tilde{e}(\pi',\gamma).
\]
See the proof of \cite[Proposition 4.9]{CO21} for the details of the notations used here.
Then, by \cite[Corollary 4.7.6]{Kal19} (cf.\ the proof of \cite[Corollary 4.10.1]{Kal19}), we get
\[
\prod_{i=0}^{d-1}\Theta_{\tilde{\phi}_{i}}(\gamma\ltimes1)
=
\epsilon_{\Psi,\bfT_{\gamma},f,\ram}(\gamma)\cdot\frac{e(\bfG)}{e(\bfG^{0})}\cdot \varepsilon_{L}(\bfT_{\bfG^{\ast}}-\bfT_{\bfG^{0\ast}})
\cdot\frac{\Delta_{\II}^{\abs,\bfG}[a_{\Psi},\chi'_{\Psi}](\gamma)}{\Delta_{\II}^{\abs,\bfG^{0}}[a_{\Psi},\chi'_{\Psi}](\gamma)}
\cdot\epsilon_{\Psi,\bfT_{\gamma}}^{\ram}(\gamma)
\]
($J^{i}$ in Corollary 4.7.6 of [Kal19] equals $\bfT_{\gamma}$ by the very regularity).
Since we have
\[
\frac{\Delta_{\II}^{\abs,\bfG}[a_{\Psi},\chi''_{\Psi}](\gamma)}{\Delta_{\II}^{\abs,\bfG^{0}}[a_{\Psi},\chi''_{\Psi}](\gamma)}
=\frac{\Delta_{\II}^{\abs,\bfG}[a_{\Psi},\chi'_{\Psi}](\gamma)}{\Delta_{\II}^{\abs,\bfG^{0}}[a_{\Psi},\chi'_{\Psi}](\gamma)}
\cdot\epsilon_{\Psi,\bfT_{\gamma},\flat}(\gamma)
\]
by Lemma \ref{lem:tran-diff} and
\[
\epsilon_{\Psi}(\gamma)=\epsilon_{\Psi,\bfT_{\gamma}}^{\ram}(\gamma)\cdot\epsilon_{\Psi,\bfT_{\gamma},\flat}(\gamma)\cdot\epsilon_{\Psi,\bfT_{\gamma},f,\ram}(\gamma)
\]
by Proposition \ref{prop:epsilon}, we get the desired identity.
(Note that all of $\epsilon$'s are signs.)
\end{proof}

By Propositions \ref{prop:CF1} and \ref{prop:CF2}, we get the following.

\begin{thm}\label{thm:CF}
Let $\gamma \in G$ be an elliptic very regular element.
Then
\[
\Theta_{\pi^{\FKS}_{\Psi}}(\gamma)
=
\frac{e(\bfG)}{e(\bfG^{0})}\cdot \varepsilon_{L}(\bfT_{\bfG^{\ast}}-\bfT_{\bfG^{0\ast}})\cdot
\sum_{g \in G^{0}_{\bar{\x}}\backslash N_{G}(T_{\gamma},G^{0}_{\bar{\x}})}\Theta_{\rho_{0}}({}^{g}\gamma)\cdot\frac{\Delta_{\II}^{\abs,\bfG}[a_{\Psi},\chi''_{\Psi}]({}^{g}\gamma)}{\Delta_{\II}^{\abs,\bfG^{0}}[a_{\Psi},\chi''_{\Psi}]({}^{g}\gamma)}\cdot\phi_{\geq0}({}^{g}\gamma).
\]
\end{thm}

\begin{lem}\label{lem:max-ur}
Let $\gamma$ be an element of $G^{0}_{\bar{\x},\evrs}$ such that $\bfT_{\gamma}$ is maximally unramified in $\bfG^{0}$.
Then $\Delta_{\II}^{\abs,\bfG^{0}}[a_{\Psi},\chi''_{\Psi}](\gamma)=1$.
\end{lem}

\begin{proof}
The statement is proved in the final paragraph of the proof of \cite[Proposition 4.9.2]{Kal19} in the case where $\gamma$ is in particular ``shallow'' in the sense of Kaletha.
Hence, by Lemma \ref{lem:tran-translation} and Remark \ref{rem:shallow}, the statement also holds for any very regular element.
\end{proof}

Let us investigate how the formula of Theorem \ref{thm:CF} can be simplified when the Yu-datum $\Psi$ is regular.

\begin{cor}\label{cor:regsc}
Suppose that $\Psi$ is a regular Yu-datum which corresponds to a tame elliptic regular pair $(\bfS,\theta)$.
Let $\gamma \in G$ be an elliptic very regular element.
Then
\[
\Theta_{\pi^{\FKS}_{(\bfS,\theta)}}(\gamma)
=
e(\bfG)\cdot\varepsilon_{L}(\bfT_{\bfG^{\ast}}-\bfS)\cdot
\sum_{w\in W_{G}(T_{\gamma},S)}\Delta_{\II}^{\abs,\bfG}[a_{\Psi},\chi''_{\Psi}]({}^{w}\gamma)\cdot\theta({}^{w}\gamma).
\]
Note that, in particular, if $W_{G}(T_{\gamma},S)$ is empty (equivalently, $\gamma$ is not conjugate to any element of $S$), then $\Theta_{\pi^{\FKS}_{(\bfS,\theta)}}(\gamma)=0$.
\end{cor}

\begin{proof}
By Theorem \ref{thm:CF}, we have
\[
\Theta_{\pi^{\FKS}_{(\bfS,\theta)}}(\gamma)
=
\frac{e(\bfG)}{e(\bfG^{0})}\cdot \varepsilon_{L}(\bfT_{\bfG^{\ast}}-\bfT_{\bfG^{0\ast}})\cdot
\sum_{g \in G^{0}_{\bar{\x}}\backslash N_{G}(T_{\gamma},G^{0}_{\bar{\x}})}\Theta_{\rho_{0}}({}^{g}\gamma)\cdot\frac{\Delta_{\II}^{\abs,\bfG}[a_{\Psi},\chi''_{\Psi}]({}^{g}\gamma)}{\Delta_{\II}^{\abs,\bfG^{0}}[a_{\Psi},\chi''_{\Psi}]({}^{g}\gamma)}\cdot\phi_{\geq0}({}^{g}\gamma).
\]
Since $(\bfS,\theta)$ is regular, we have $\rho_0 \cong (-1)^{r(\bbS^{\circ})-r(\bbG^{\circ})} R_{\bbS}^{\bbG}(\phi_{-1})$, where $\bbG$ and $\bbS$ are as in Section \ref{subsec:sc classes} (see \cite[Section 3.2]{CO21}). 
As the character $R_{\bbS}^{\bbG}(\phi_{-1})$ is supported on $\bbG'(\F_{q})$ (Corollary \ref{cor:Kaletha-CF}), only the summand corresponding to $g$ such that ${}^{g}\gamma\in SG^{0}_{\x,0}$ contributes to the sum nontrivially.
Let us compute each summand by assuming that ${}^{g}\gamma$ belongs to $SG^{0}_{\x,0}$.
By Proposition \ref{prop:evreg-reduction}, the image of $\gamma$ in $\bbG'(\F_{q})$ is regular semisimple.
Hence, by Corollary \ref{cor:Kaletha-CF-reg}, we obtain
\begin{align*}
\Theta_{\rho_0}({}^{g}\gamma)
&=
(-1)^{r(\bbS^{\circ})-r(\bbG^{\circ})}\Theta_{R_{\bbS}^{\bbG}(\phi_{-1})}({}^{g}\gamma)\\
&=
(-1)^{r(\bbS^{\circ})-r(\bbG^{\circ})}\sum_{w\in W_{\bbG(\F_{q})}({}^{g}\bbT^{\circ}_{\gamma},\bbS^{\circ})}\phi_{-1}({}^{wg}\gamma)\\
&=
(-1)^{r(\bbS^{\circ})-r(\bbG^{\circ})}\sum_{w\in W_{G^{0}_{\bar{\x}}}({}^{g}T_{\gamma},S)}\phi_{-1}({}^{wg}\gamma).
\end{align*}
(See Lemma  \ref{lem:Weyl-reduction} for the last equality.)

As explained in the final paragraph of the proof of \cite[Proposition 4.9.2]{Kal19}, we have
\begin{align*}
(-1)^{r(\bbS^{\circ})-r(\bbG^{\circ})}
&=(-1)^{r(\bfG^{0})-r(\bfS)}\\
&=(-1)^{r(\bfG^{0})-r(\bfT_{\bfG^{0\ast}})}\cdot(-1)^{r(\bfT_{\bfG^{0\ast}})-r(\bfS)}\\
&=e(\bfG^{0})\cdot\varepsilon_{L}(\bfS-\bfT_{\bfG^{0\ast}}).
\end{align*}
Furthermore, since $\bfS$ is maximally unramified in $\bfG^{0}$, Lemma \ref{lem:max-ur} implies that $\Delta_{\II}^{\abs,\bfG^{0}}[a_{\Psi},\chi''_{\Psi}](\gamma)$ is trivial (as long as $\gamma$ is conjugate to an element of $S$).

Therefore we see that $\Theta_{\pi^{\FKS}_{(\bfS,\theta)}}(\gamma)$ is equal to
\[
e(\bfG)\cdot\varepsilon_{L}(\bfT_{\bfG^{\ast}}-\bfS)\cdot
\sum_{g \in G^{0}_{\bar{\x}}\backslash N_{G}(T_{\gamma},G^{0}_{\bar{\x}})}
\Delta_{\II}^{\abs,\bfG}[a_{\Psi},\chi''_{\Psi}]({}^{g}\gamma)\cdot\phi_{\geq0}({}^{g}\gamma)\cdot
\sum_{w\in W_{G^{0}_{\bar{\x}}}({}^{g}T_{\gamma},S)}\phi_{-1}({}^{wg}\gamma).
\]
Since $\Delta_{\II}^{\abs,\bfG}[a_{\Psi},\chi''_{\Psi}]$ and $\phi_{\geq0}$ are invariant under $G^{0}$-conjugation (see Lemma \ref{lem:tran-conj} for the former assertion), by also noting that $\theta=\phi_{-1}\cdot\phi_{\geq0}$, we get
\begin{align*}
&\sum_{g \in G^{0}_{\bar{\x}}\backslash N_{G}(T_{\gamma},G^{0}_{\bar{\x}})}
\Delta_{\II}^{\abs,\bfG}[a_{\Psi},\chi''_{\Psi}]({}^{g}\gamma)\cdot\phi_{\geq0}({}^{g}\gamma)\cdot
\sum_{w\in W_{G^{0}_{\bar{\x}}}({}^{g}T_{\gamma},S)}\phi_{-1}({}^{wg}\gamma)\\
&=
\sum_{g \in G^{0}_{\bar{\x}}\backslash N_{G}(T_{\gamma},G^{0}_{\bar{\x}})}
\sum_{w\in W_{G^{0}_{\bar{\x}}}({}^{g}T_{\gamma},S)}
\Delta_{\II}^{\abs,\bfG}[a_{\Psi},\chi''_{\Psi}]({}^{wg}\gamma)\cdot\phi_{\geq0}({}^{wg}\gamma)\cdot
\phi_{-1}({}^{wg}\gamma)\\
&=
\sum_{w\in W_{G}(T_{\gamma},S)}\Delta_{\II}^{\abs,\bfG}[a_{\Psi},\chi''_{\Psi}]({}^{w}\gamma)\cdot\theta({}^{w}\gamma).
\end{align*}
This completes the proof.
\end{proof}

\begin{lem}\label{lem:Weyl-reduction}
For any very regular element $\gamma\in SG^{0}_{\x,0}$, we have $W_{\bbG(\F_{q})}(\bbT^{\circ}_{\gamma},\bbS^{\circ})=W_{G^{0}_{\bar{\x}}}(T_{\gamma},S)$.
\end{lem}

\begin{proof}
We first note that $W_{G^{0}_{\bar{\x}}}(T_{\gamma},S)=W_{G^{0}_{\bar{\x}}}(\bfT_{\gamma},\bfS)$.
The natural reduction map $W_{G^{0}_{\bar{\x}}}(\bfT_{\gamma},\bfS)\rightarrow W_{\bbG(\F_{q})}(\bbT^{\circ}_{\gamma},\bbS^{\circ})$ is obviously well-defined and injective.
Let us show the surjectivity.
Let $\bar{n}$ be an element of $W_{\bbG(\F_{q})}(\bbT^{\circ}_{\gamma},\bbS^{\circ})$.
Then, as the map $G^{0}_{\bar{\x}}\twoheadrightarrow\bbG(\F_{q})$ is surjective, we can take an element $n\in G^{0}_{\bar{\x}}$ whose image in $\bbG(\F_{q})$ equals $\bar{n}$.
Since $\bbT_{\gamma}$ is the centralizer of $\bar{\gamma}$ and $\bar{\gamma}$ is regular semisimple (Proposition \ref{prop:evreg-reduction}), the condition $\bar{n}\in W_{\bbG(\F_{q})}(\bbT^{\circ}_{\gamma},\bbS^{\circ})$ is equivalent to the condition ${}^{\bar{n}}\bar{\gamma}\in\bbS$.
This is furthermore equivalent to that ${}^{n}\gamma\in SG^{0}_{\x,0+}$.
Then Lemma \ref{lem:CO21-5.1-5.2} implies that there exists a very regular element $\gamma'$ of $S$ and that ${}^{n}\bfT_{\gamma}$ and $\bfT_{\gamma'}=\bfS$ are $G^{0}_{\x,0+}$-conjugate.
If we let $n_{+}\in G^{0}_{\x,0+}$ be an element such that ${}^{n_{+}n}\bfT_{\gamma}=\bfS$, then $n_{+}n$ gives an element of $W_{G^{0}_{\bar{\x}}}(\bfT_{\gamma},\bfS)$ which maps to $\bar{n}\in W_{\bbG(\F_{q})}(\bbT^{\circ}_{\gamma},\bbS^{\circ})$.
\end{proof}

\newpage

\part{Supercuspidals distinguished by their characters on elliptic very regular elements}\label{part:characterization}

\section{Characterizing clipped Yu-data}\label{sec:clipped p-adic}

Recall from the previous part that on elliptic very regular elements, the character formula of any supercuspidal representation is extremely simple (Theorem \ref{thm:CF}). In this section, we prove that, up to equivalence, the clipped Yu-datum of any supercuspidal representation can be recovered from these character values (Proposition \ref{prop:clipped characterization}). Moreover, we prove that if two supercuspidal representations $\pi_\Psi^\FKS, \pi_{\Psi'}^\FKS$ have the same character values on the elliptic very regular locus, then their depth zero components $\rho_0, \rho_0'$ must also have matching character values on the image $\bbG'(\FF_q)_\evrs$ of the elliptic very regular locus (Proposition \ref{prop:reduction depth 0}).

\begin{prop}\label{prop:refacto}
Let $\dashover \Psi = (\vec{\bfG},\vec{\phi},\vec{r},\x)$ and $\dashover{\Psi}' = (\vec{\bfG}',\vec{\phi}',\vec{r}',\x')$ be clipped Yu-data such that $\bar{\x}=\bar{\x}'$.
Suppose that there exists a tame elliptic maximal torus $\bfS$ of $\bfG$ which is contained in both of $\bfG^{0}$ and $\bfG^{\prime0}$ such that $\phi_{\geq0}|_{S_{0+}}=\phi'_{\geq0}|_{S_{0+}}$.
Then $\dashover{\Psi}'$ is a refactorization of $\dashover \Psi$.
\end{prop}

\begin{proof}
According to \cite[Lemma 3.6.3]{Kal19}, the sequences $\vec{\bfG}$ and $\vec{r}$ are uniquely recovered from $\phi_{\geq0}|_{S_{0+}}$.
Hence the assumption that $\phi_{\geq0}|_{S_{0+}}=\phi'_{\geq0}|_{S_{0+}}$ implies that $(\vec{\bfG},\vec{r})=(\vec{\bfG}',\vec{r}')$.
The remaining part follows from the same argument as in \cite[Lemma 3.6.6]{Kal19}.
\end{proof}

\begin{prop}\label{prop:Yu8.3}
Let $(\vec{\bfG},\vec{\phi},\vec{r},\x)$ be a clipped Yu-datum.
Let $\bfS$ be a tame elliptic maximal torus of $\bfG$ contained in $\bfG^{0}$.
If an element $g\in N_{G_{\bar{\x}}}(\bfS,\bfG^{0})$ satisfies $\phi_{\geq0}^{g}|_{S_{0+}}=\phi_{\geq0}|_{S_{0+}}$, then $g$ belongs to $G^{0}_{\bar{\x}}$.
\end{prop}

\begin{proof}
As $\phi_{\geq0}$ is defined to be the product of $\phi_{0},\ldots,\phi_{d}$ and $\phi_{d}$ is a character of $G^{d}=G$, the assumption $\phi_{\geq0}^{g}|_{S_{0+}}=\phi_{\geq0}|_{S_{0+}}$ implies that $(\phi_{0}\cdots\phi_{d-1})^{g}|_{S_{0+}}=(\phi_{0}\cdots\phi_{d-1})|_{S_{0+}}$.

Note that $\bfZ^{d-1}\subset\bfS\subset\bfG^{d-1}$, where $\bfZ^{d-1}$ denotes the center of $\bfG^{d-1}$.
As explained in \cite[Section 8]{Yu01}, by noting the coadjoint actions of $\bfS$ and $\bfG^{d-1}$ on $\Lie^{\ast}(\bfG^{d-1})$, we have 
\[
\Lie^{\ast}(\bfZ^{d-1})
\subset\Lie^{\ast}(\bfS)
\subset\Lie^{\ast}(\bfG^{d-1}),
\]
which induces
\[
\Lie^{\ast}(\bfZ^{d-1})(F)_{-r_{d-1}}
\subset\Lie^{\ast}(\bfS)(F)_{-r_{d-1}}
\subset\Lie^{\ast}(\bfG^{d-1})(F)_{\x,-r_{d-1}}.
\]
Similarly, by the inclusions $\bfZ^{d-1,g}\subset\bfS\subset\bfG^{d-1,g}$, we get
\[
\Lie^{\ast}(\bfZ^{d-1,g})(F)_{-r_{d-1}}
\subset\Lie^{\ast}(\bfS)(F)_{-r_{d-1}}
\subset\Lie^{\ast}(\bfG^{d-1,g})(F)_{\x,-r_{d-1}}.
\]

By the $\bfG^{d}$-genericity assumption of $\phi_{d-1}$, there exists a $\bfG^{d}$-generic element $X_{d-1}^{\ast}\in\Lie^{\ast}(\bfZ^{d-1})(F)_{-r_{d-1}}$ of depth $r_{d-1}$ which realizes $\phi_{d-1}|_{G^{d-1}_{\x,r_{d-1}:r_{d-1}+}}$ (in the sense of \cite[Section 8]{Yu01}).
Then $X_{d-1}^{\ast,g}\in\Lie^{\ast}(\bfZ^{d-1,g})(F)_{-r_{d-1}}$ is a $\bfG^{d,g}$-generic element of depth $r_{d-1}$ which realizes $\phi^{g}_{d-1}|_{G^{d-1,g}_{\x,r_{d-1}:r_{d-1}+}}$.
Note that, as the depth of $\phi_{i}$ is smaller than $r_{d-1}$ for any $0\leq i\leq d-2$, 
\[
(\phi_{d-1}|_{G^{d-1}_{\x,r_{d-1}:r_{d-1}+}})|_{S_{r_{d-1}}}
=\phi_{d-1}|_{S_{r_{d-1}}}
=(\phi_{0}\cdots\phi_{d-1})|_{S_{r_{d-1}}}
\]
and
\[
(\phi_{d-1}^{g}|_{G^{d-1}_{\x,r_{d-1}:r_{d-1}+}})|_{S_{r_{d-1}}}
=\phi_{d-1}^{g}|_{S_{r_{d-1}}}
=(\phi_{0}\cdots\phi_{d-1})^{g}|_{S_{r_{d-1}}}.
\]
Thus the equality $(\phi_{0}\cdots\phi_{d-1})^{g}|_{S_{0+}}=(\phi_{0}\cdots\phi_{d-1})|_{S_{0+}}$ implies that the elements $X_{d-1}^{\ast}\in\Lie^{\ast}(\bfZ^{d-1})(F)_{-r_{d-1}}$ and $X_{d-1}^{\ast,g}\in\Lie^{\ast}(\bfZ^{d-1,g})(F)_{-r_{d-1}}$ are equal in $\Lie^{\ast}(\bfS)(F)_{-r_{d-1}:(-r_{d-1})+}$, hence in $\Lie^{\ast}(\bfG^{d-1})(F)_{\x,-r_{d-1}:(-r_{d-1})+}$.
Now we utilize \cite[Lemma 8,3]{Yu01} as follows.

We take a regular semisimple (in $\bfG^{d-1}$) element $Z^{\ast}\in \Lie^{\ast}(\bfS)(F)$.
Here, the existence of a regular semisimple element in $\Lie^{\ast}(\bfS)(F)$ follows from that the regular semisimple locus of $\Lie^{\ast}(\bfS)$ is Zariski dense in $\Lie^{\ast}(\bfS)$, hence unirational.
(Any $F$-rational torus is unirational, hence so is $\Lie^{\ast}(\bfS)$; see \cite[AG13.7 and 8.13]{Bor91} for the unirationality of a torus.)
By scaling $Z^{\ast}$, we may assume that $Z^{\ast}$ belongs to $\Lie^{\ast}(\bfS)(F)_{(-r_{d-1})+}$ and is sufficiently close to $0$ so that
\begin{itemize}
\item
$(Z^{\ast}){}^{g}$ belongs to $\Lie^{\ast}(\bfG^{d-1})(F)_{\x,(-r_{d-1})+}$ and
\item
$X_{d-1}^{\ast}+Z^{\ast}$ is regular semisimple (hence so is $(X_{d-1}^{\ast}+Z^{\ast}){}^{g}$).
\end{itemize}
Then the elements $Y_{1}^{\ast}\colonequals (X_{d-1}^{\ast}+Z^{\ast})^{g}$ and $Y_{2}^{\ast}\colonequals X_{d-1}^{\ast}+Z^{\ast}$ satisfy the assumptions of \cite[Lemma 8,3]{Yu01}, i.e., $Y_{1}^{\ast}$ and $Y_{2}^{\ast}$ are regular semisimple and satisfy $Y_{1}^{\ast}\equiv Y_{2}^{\ast}\equiv X_{d-1}^{\ast}$ modulo $\Lie^{\ast}(\bfG^{d-1})(F)_{\x,(-r_{d-1})+}$.
Hence, as ${}^{g}Y_{1}^{\ast}=Y_{2}^{\ast}$, \cite[Lemma 8,3]{Yu01} implies that $g\in G^{d-1}$.

Since $\phi_{d-1}$ is a character of $G^{d-1}$, the equality $(\phi_{0}\cdots\phi_{d-1})^{g}|_{S_{0+}}=(\phi_{0}\cdots\phi_{d-1})|_{S_{0+}}$ implies that $(\phi_{0}\cdots\phi_{d-2})^{g}|_{S_{0+}}=(\phi_{0}\cdots\phi_{d-2})|_{S_{0+}}$.
Hence, by applying the same argument to $(d-1,d-2)$, we get $g\in G^{d-2}$.
Repeating this procedure inductively, we finally get $g\in G^{0}$.
Thus $g\in G_{\bar{\x}}\cap G^{0}=G^{0}_{\bar{\x}}$.
\end{proof}

For any tame elliptic maximal torus $\bfS$ of $\bfG$, we put $S_{\evrs}$ to be the set of elliptic very regular elements of $S$.

\begin{proposition}\label{prop:clipped characterization}
Let $\Psi=(\vec{\bfG},\vec{\phi},\vec{r},\x,\rho_{0})$ and $\Psi'=(\vec{\bfG}',\vec{\phi}',\vec{r}',\x',\rho'_{0})$ be Yu-data.
Let $\bfS$ be a tame elliptic maximal torus of $\bfG^{0}$ with associated point $\x$.
Assume that $\Theta_{\rho_{0}}$ is not identically zero on $S_{\evrs}$.
    If there exists a constant $c \in \bbC^{1}$ for which
    \begin{equation}\label{eq:vreg for two}
        \Theta_{\pi_\Psi^{\FKS}}(\gamma) = c \cdot \Theta_{\pi_{\Psi'}^{\FKS}}(\gamma) \qquad \text{for any $\gamma \in S_{\evrs}$},
    \end{equation}
    then the clipped Yu-data $\dashover \Psi, \dashover{\Psi}'$ are $\bfG$-equivalent. 
\end{proposition}

\begin{proof}
    Applying Theorem \ref{thm:CF} to \eqref{eq:vreg for two}, for any $\gamma \in S_{\evrs}$ we have the identity  
    \begin{align*}
        \sum_{g \in G^{0}_{\bar{\x}}\backslash N_{G}(S,G^{0}_{\bar{\x}})}&\Theta_{\rho_{0}}({}^{g}\gamma)\cdot\frac{\Delta_{\II}^{\abs,\bfG}[a_{\Psi},\chi''_{\Psi}]({}^{g}\gamma)}{\Delta_{\II}^{\abs,\bfG^{0}}[a_{\Psi},\chi''_{\Psi}]({}^{g}\gamma)}\cdot\phi_{\geq0}({}^{g}\gamma) \\
        &=
        c \cdot c' \cdot
        \sum_{g \in G^{\prime 0}_{\bar{\x}'}\backslash N_{G}(S,G^{\prime 0}_{\bar{\x}'})}\Theta_{\rho_{0}'}({}^{g}\gamma)\cdot\frac{\Delta_{\II}^{\abs,\bfG}[a_{\Psi'},\chi''_{\Psi'}]({}^{g}\gamma)}{\Delta_{\II}^{\abs,\bfG^{\prime 0}}[a_{\Psi'},\chi''_{\Psi'}]({}^{g}\gamma)}\cdot\phi_{\geq0}'({}^{g}\gamma),
    \end{align*}
    where $c' \colonequals \frac{e(\bfG^0)}{e(\bfG'{}^{0})}\cdot \frac{\varepsilon_{L}(\bfT_{\bfG^{\ast}}-\bfT_{\bfG'{}^{0\ast}})}{\varepsilon_L(\bfT_{\bfG^*} - \bfT_{\bfG^{0\ast}})}$.
(Note that we have $N_{G}(S,G^{0}_{\bar{\x}})=N_{G_{\bar{\x}}}(S,G^{0}_{\bar{\x}})$ by Lemma \ref{lem:Weyl}.)
Let us take $\gamma\in S_{\evrs}$ such that $\Theta_{\rho_{0}}(\gamma)\neq0$ by assumption.
For any element $\gamma_{0+}\in S_{0+}$, the element $\gamma\gamma_{0+}$ is again very regular by Lemma \ref{lem:CO21-5.1-5.2}.
Since $\rho_0$ and $\rho_0'$ are of depth zero and $\Delta_{\II}^{\abs,\bfG}$ is invariant under $S_{0+}$-translation (Lemma \ref{lem:tran-translation}),
    \begin{align}\label{eq:0+ id}
        \sum_{g \in G^{0}_{\bar{\x}}\backslash N_{G_{\bar{\x}}}(S,G^{0}_{\bar{\x}})}&\Theta_{\rho_{0}}({}^{g}\gamma)\cdot\frac{\Delta_{\II}^{\abs,\bfG}[a_{\Psi},\chi''_{\Psi}]({}^{g}\gamma)}{\Delta_{\II}^{\abs,\bfG^{0}}[a_{\Psi},\chi''_{\Psi}]({}^{g}\gamma)}\cdot\phi_{\geq0}({}^{g}\gamma) \cdot \phi_{\geq 0}({}^{g} \gamma_{0+}) \\ \nonumber
        &=
        c \cdot c' \cdot
        \sum_{g \in G^{\prime 0}_{\bar{\x}'}\backslash N_{G}(S,G^{\prime 0}_{\bar{\x}'})}\Theta_{\rho_{0}'}({}^{g}\gamma)\cdot\frac{\Delta_{\II}^{\abs,\bfG}[a_{\Psi'},\chi''_{\Psi'}]({}^{g}\gamma)}{\Delta_{\II}^{\abs,\bfG^{\prime 0}}[a_{\Psi'},\chi''_{\Psi'}]({}^{g}\gamma)}\cdot\phi_{\geq0}'({}^{g}\gamma) \cdot \phi_{\geq 0}'({}^g \gamma_{0+}).
    \end{align}
    We regard this as an identity between linear combinations of the characters $\{\phi_{\geq 0}^{g}|_{S_{0+}}\}_{g \in G^{0}_{\bar{\x}}\backslash N_{G_{\bar{\x}}}(S,G^{0}_{\bar{\x}})}$ and $\{\phi_{\geq 0}^{\prime g}|_{S_{0+}}\}_{g \in G_{\bar{\x}'}^{\prime0} \backslash N_G(S,G_{\bar{\x}'}^{\prime0})}$ of $S_{0+}$. 
    We apply the inner product $\langle -, \phi_{\geq 0}|_{S_{0+}} \rangle_{S_{0+}}$ to the identity \eqref{eq:0+ id}. By Proposition \ref{prop:Yu8.3}, the stabilizer of $\phi_{\geq 0}|_{S_{0+}}$ under $N_{G_{\bar{\x}}}(S,G_{\bar \x}^0)$ is equal to $G_{\bar \x}^0$.
    Hence the left-hand side equals
    \[
    \Theta_{\rho_{0}}(\gamma)\cdot\frac{\Delta_{\II}^{\abs,\bfG}[a_{\Psi},\chi''_{\Psi}](\gamma)}{\Delta_{\II}^{\abs,\bfG^{0}}[a_{\Psi},\chi''_{\Psi}](\gamma)}\cdot\phi_{\geq0}(\gamma).
    \]
    As this is not equal to zero, we have now shown that the inner product of the left-hand side of \eqref{eq:0+ id} with $\phi_{\geq 0}|_{S_{0+}}$ is nonzero. 
    It follows then that the inner product of the right-hand side of \eqref{eq:0+ id} with $\phi_{\geq 0}|_{S_{0+}}$ is nonzero. 
    In particular, there exists an element $g \in G_{\bar{\x}'}^{\prime 0} \backslash N_G(S,G_{\bar{\x}'}^{\prime 0})$ such that $\phi_{\geq 0}^{\prime g}|_{S_{0+}} = \phi_{\geq 0}|_{S_{0+}}$. 
    By replacing the Yu-datum $\Psi'$ with its $g$-conjugation ${}^g \Psi'$, we may assume that this $g$ can be taken to be $1$.
    But now this implies that $S$ is contained also in $G_{\bar{\x}'}^{\prime 0}$.
    Since $\gamma\in S$ is an elliptic very regular element, Lemma \ref{lem:point} implies that $\bar{\x} = \bar{\x}'$.
Now Proposition \ref{prop:refacto} implies that the clipped Yu-datum $\dashover{\Psi}'$ is a refactorization of $\dashover{\Psi}$. 
\end{proof}

Recall that
\begin{itemize}
\item
$G^{0}_{\bar{\x},\evrs}$ is the set of elliptic very regular elements of $G^{0}_{\bar{\x}}$ (Section \ref{subsec:character}),
\item
$\bbG'(\F_{q})_{\evrs}$ is the set of reductions of elliptic very regular elements of $SG^{0}_{\x,0}$ (Section \ref{subsec:ur-vreg}).
\end{itemize}

\begin{proposition}\label{prop:reduction depth 0}
Let $\Psi=(\vec{\bfG},\vec{\phi},\vec{r},\x,\rho_{0})$ and $\Psi'=(\vec{\bfG},\vec{\phi},\vec{r},\x,\rho'_{0})$ be Yu-data (i.e., $\dashover{\Psi}=\dashover{\Psi}'$).
If there exists a constant $c \in \bbC^1$ for which
    \begin{equation}\label{eq:clipped id}
        \Theta_{\pi_\Psi^{\FKS}}(\gamma) = c \cdot \Theta_{\pi_{\Psi'}^{\FKS}}(\gamma) \qquad \text{for any $\gamma\in G^{0}_{\bar{\x},\evrs}$},
    \end{equation}
    then
    \begin{equation*}
        \Theta_{\rho_0}(\gamma) = c \cdot \Theta_{\rho_0'}(\gamma) \qquad \text{for any $\gamma \in \bbG'(\FF_q)_{\evrs}$}.
    \end{equation*}
\end{proposition}

\begin{proof}
Let $\gamma\in\bbG'(\F_{q})$.
We take an elliptic very regular element of $SG^{0}_{\bar{\x},0}$ whose reduction is $\gamma$ and again write $\gamma$ for it by abuse of notation.
    Combining \eqref{eq:clipped id} with Theorem \ref{thm:CF}, we have the identity
    \begin{multline*}
        \sum_{g \in G^{0}_{\bar{\x}}\backslash N_{G}(T_{\gamma},G^{0}_{\bar{\x}})}\Theta_{\rho_{0}}({}^{g}\gamma)\cdot\frac{\Delta_{\II}^{\abs,\bfG}[a_{\Psi},\chi''_{\Psi}]({}^{g}\gamma)}{\Delta_{\II}^{\abs,\bfG^{0}}[a_{\Psi},\chi''_{\Psi}]({}^{g}\gamma)}\cdot\phi_{\geq0}({}^{g}\gamma) \\
= c \cdot
        \sum_{g \in G^{0}_{\bar{\x}}\backslash N_{G}(T_{\gamma},G^{0}_{\bar{\x}})}\Theta_{\rho_{0}'}({}^{g}\gamma)\cdot\frac{\Delta_{\II}^{\abs,\bfG}[a_{\Psi'},\chi''_{\Psi'}]({}^{g}\gamma)}{\Delta_{\II}^{\abs,\bfG^{0}}[a_{\Psi'},\chi''_{\Psi'}]({}^{g}\gamma)}\cdot\phi_{\geq0}({}^{g}\gamma).
    \end{multline*}
Similarly to the proof of Proposition \ref{prop:clipped characterization}, we consider the $T_{\gamma,0+}$-translation of $\gamma$ and take the inner product $\langle -, \phi_{\geq 0}|_{T_{\gamma,0+}}\rangle_{T_{\gamma,0+}}$. Then the resulting left-hand side is given by 
    \begin{equation*}
        \sum_{g \in G^{0}_{\bar{\x}}\backslash N_{G}(T_{\gamma},G^{0}_{\bar{\x}})}\Theta_{\rho_{0}}({}^{g}\gamma)\cdot\frac{\Delta_{\II}^{\abs,\bfG}[a_{\Psi},\chi''_{\Psi}]({}^{g}\gamma)}{\Delta_{\II}^{\abs,\bfG^{0}}[a_{\Psi},\chi''_{\Psi}]({}^{g}\gamma)}\cdot\phi_{\geq0}({}^{g}\gamma) \cdot \langle \phi_{\geq 0}^g |_{T_{\gamma,0+}}, \phi_{\geq 0}|_{T_{\gamma,0+}} \rangle_{T_{\gamma,0+}}.
    \end{equation*}
    Since $N_G(T_\gamma,G_{\bar \x}^0) = N_{G_{\bar \x}}(T_\gamma,G_{\bar \x}^0)$ (Lemma \ref{lem:Weyl}), it follows from Proposition \ref{prop:Yu8.3} that if $g \in N_G(T_\gamma,G_{\bar \x}^0)$ is an element satisfying $\langle \phi_{\geq 0}^g|_{T_{\gamma,0+}}, \phi_{\geq 0}|_{T_{\gamma,0+}}\rangle_{T_{\gamma,0+}} \neq 0$, then necessarily $g \in G_{\bar \x}^0$. Applying the same argument to the right-hand side, we have now shown that for any $\gamma\in G^{0}_{\bar{\x},\evrs}$, 
    \begin{equation*}
        \Theta_{\rho_{0}}(\gamma)\cdot\frac{\Delta_{\II}^{\abs,\bfG}[a_{\Psi},\chi''_{\Psi}](\gamma)}{\Delta_{\II}^{\abs,\bfG^{0}}[a_{\Psi},\chi''_{\Psi}](\gamma)}\cdot\phi_{\geq0}(\gamma)
        = c \cdot \Theta_{\rho_{0}'}(\gamma)\cdot\frac{\Delta_{\II}^{\abs,\bfG}[a_{\Psi'},\chi''_{\Psi'}](\gamma)}{\Delta_{\II}^{\abs,\bfG^{0}}[a_{\Psi'},\chi''_{\Psi'}](\gamma)}\cdot\phi_{\geq0}(\gamma).
    \end{equation*}
Note that $\Delta_{\II}^{\abs,\bfG^{(0)}}[a_{\Psi},\chi''_{\Psi}](\gamma)=\Delta_{\II}^{\abs,\bfG^{(0)}}[a_{\Psi'},\chi''_{\Psi'}](\gamma)$ since Kaletha's $a$-data and $\chi$-data depend only on the clipped parts of Yu-data (see Section \ref{subsec:tran}).
Therefore, as $\frac{\Delta_{\II}^{\abs,\bfG}[a_{\Psi},\chi''_{\Psi}](\gamma)}{\Delta_{\II}^{\abs,\bfG^{0}}[a_{\Psi},\chi''_{\Psi}](\gamma)}\cdot\phi_{\geq0}(\gamma) \neq 0$, we get $\Theta_{\rho_0}(\gamma) = c \cdot \Theta_{\rho_0'}(\gamma)$.
\end{proof}

\section{Characterization on elliptic very regular elements}\label{sec:sc characterization}

In this section, we apply the results of Section \ref{sec:clipped p-adic} in various contexts to prove that \textit{if there are sufficiently many elliptic very regular elements}, then classes of supercuspidal representations are distinguished by only their character values on elliptic very regular elements. The essential line of reasoning is that Propositions \ref{prop:clipped characterization} and \ref{prop:reduction depth 0} allow us to reduce this kind of characterization problem to the depth zero setting, wherein we can apply the results of Section \ref{sec:char finite}.

\subsection{Characterizing regular supercuspidal representations}\label{subsec:regsc-characterization}

Let $\Psi=(\vec{\bfG},\vec{\phi},\vec{r},\x,\rho_{0})$ be a regular Yu-datum corresponding to a tame elliptic regular pair $(\bfS,\theta)$.
We introduce the groups $\bbG$, $\bbS$, $\bbZ_{\bbG}$, $\bbZ^{\star}_{\bbG}$ as in Section \ref{subsec:ur-vreg}.
Recall that we defined the set $\bbG'(\F_{q})_{\evrs}$ in Section \ref{subsec:ur-vreg}.
Let us put $\bbG'(\F_{q})_{\nevrs}\colonequals \bbG'(\F_{q})\smallsetminus\bbG'(\F_{q})_{\evrs}$ and $\bbG(\F_{q})_{\nevrs}\colonequals \bbG(\F_{q})\smallsetminus\bbG'(\F_{q})_{\evrs}$.
By taking $\bbG'(\F_{q})_{\bullet}$ and $\bbG'(\F_{q})_{\circ}$ in Section \ref{subsec:char-finite-gen-pos} to be $\bbG'(\F_{q})_{\evrs}$ and $\bbG'(\F_{q})_{\nevrs}$, respectively, we consider the inequality \eqref{ineq:Henniart-bullet}:
\[\label{ineq:Henniart-evrs}
\frac{|[\bbS]^{\star}|}{|[\bbS]^{\star}_{\nevrs}|} 
=
\frac{|[\bbS]^{\star}|}{|[\bbS]^{\star} \smallsetminus [\bbS]^{\star}_\evrs|} 
> 2\cdot|W_{\bbG(\F_{q})}(\bbS)|.
\tag{$\mathfrak{H}_{\evrs}$}
\]

\begin{theorem}\label{thm:regsc-characterization}
    Assume that \eqref{ineq:Henniart-evrs} is satisfied. 
    Then there exists a unique irreducible supercuspidal representation $\pi$ such that there exists a constant $c \in \bbC^1$ for which
    \begin{equation}\label{eq:vreg id}
        \Theta_{\pi}(\gamma)
        =
        c\cdot
        \sum_{w\in W_{G}(T_{\gamma},S)}\Delta_{\II}^{\abs,\bfG}[a_{\Psi},\chi''_{\Psi}]({}^{w}\gamma)\cdot\theta({}^{w}\gamma)
        \end{equation}
        for any $\gamma\in G^{0}_{\bar{\x},\evrs}$.
        Moreover, such $\pi$ is given by the regular supercuspidal representation $\pi_{(\bfS, \theta)}^{\FKS}$ associated with $(\bfS,\theta)$ via the modified construction of Fintzen--Kaletha--Spice.
\end{theorem}

The assumption \eqref{ineq:Henniart-evrs} appears as it is a sufficient condition for the nonvanishing assertion in the hypothesis of Proposition \ref{prop:clipped characterization} in the setting of regular supercuspidal representations (i.e., when the depth zero part $\rho_0$ of the Yu-datum is the Deligne--Lusztig induction of a character $\phi_{-1}$ in general position). The following lemma provides the reason for this sufficiency:

\begin{lemma}\label{lem:nonvan}
Let $\phi_{-1}$ be a regular depth zero character of $S$.
    If \eqref{ineq:Henniart-evrs} holds, then there exists an element $\gamma \in S_{\evrs}$ such that
    \begin{equation*}
    \sum_{v \in W_{G^{0}}(\bfS)} \phi_{-1}^{v}(\gamma) \neq 0.
    \end{equation*}
\end{lemma}
    
\begin{proof}
    Since we have $W_{G^{0}}(\bfS)=W_{\bbG}(\bbS)$ (\cite[Lemma 3.2.2]{Kal19-sc}), the averaged sum is nothing but the character of the disconnected Deligne--Lusztig virtual representation $R_{\bbS}^{\bbG}(\phi_{-1})$ at $\gamma\in \bbS(\F_{q})_{\evrs}$ by Corollary \ref{cor:Kaletha-CF-reg}.
Thus the claim follows from Lemma \ref{lem:Henniart}.
\end{proof}

We are now ready to prove the theorem.

\begin{proof}[Proof of Theorem \ref{thm:regsc-characterization}]
    We know by Corollary \ref{cor:regsc} that the regular supercuspidal representation $\pi_{(\bfS,\theta)}^{\FKS}$ satisfies \eqref{eq:vreg id}. We now prove that this is in fact the unique irreducible supercuspidal representation satisfying \eqref{eq:vreg id}.

    Let $\pi$ be an irreducible supercuspidal representation satisfying \eqref{eq:vreg id}.
    By Fintzen's exhaustion theorem (Theorem \ref{thm:Fintzen}), our assumption that $p\nmid|W_{\bfG}|$ implies that $\pi$ is a tame supercuspidal representation of $G$.
    Let $\Psi'=(\vec{\bfG}',\vec{\phi}',\vec{r}',\x',\rho'_{0})$ be a Yu-datum satisfying $\pi^{\FKS}_{\Psi'}\cong\pi$.
    By \eqref{eq:vreg id} and Corollary \ref{cor:regsc}, there exists a constant $c \in \bbC^1$ such that
    \begin{equation}\label{eq:regsc compare}
        \Theta_\pi(\gamma) = c \cdot \Theta_{\pi_{(\bfS,\theta)}^{\FKS}}(\gamma)
    \end{equation}
    for any $\gamma\in G^{0}_{\bar{\x},\evrs}$. In particular \eqref{eq:regsc compare} holds for all $\gamma \in S_{\evrs}$. By the assumption \eqref{ineq:Henniart-evrs} together Lemma \ref{lem:nonvan}, we know there exists an element $\gamma \in S_{\evrs}$ such that $\rho_0(\gamma) \neq 0$. Hence by Proposition \ref{prop:clipped characterization}, we have that the clipped Yu-data $\dashover{\Psi}$ and $\dashover{\Psi}'$ are $\bfG$-equivalent. Thus we may assume that $\Psi'=(\vec \bfG, \vec \phi, \vec r, \x, \rho_0')$. Furthermore, by Proposition \ref{prop:reduction depth 0}, we have
    \begin{equation}\label{eq:rho0 rho0'}
        \Theta_{\rho_0'}(\gamma) = c \cdot \Theta_{\rho_0}(\gamma), \qquad \text{for all $\gamma \in \bbG'(\FF_q)_\evrs$}.
    \end{equation}
    Since $\Psi$ is a regular Yu-datum, we know that $\rho_0 \cong (-1)^{r(\bbS^{\circ})-r(\bbG^{\circ})}R_{\bbS}^{\bbG}(\phi_{-1})$ for some character $\phi_{-1}$ of $\bbS(\FF_q)$ in general position, where $R_{\bbS}^{\bbG}(\phi_{-1})$ is the virtual $\bbG(\FF_q)$-representation defined in Section \ref{subsec:disconn}. Finally, Theorem \ref{thm:Henniart} implies that $c = (-1)^{r(\bbS^{\circ})-r(\bbG^{\circ})}$ and $\rho_0' \cong \rho_0$ under the assumption \eqref{ineq:Henniart-evrs}. This completes the proof.
\end{proof}

In the proof of Theorem \ref{thm:regsc-characterization}, the assumption \eqref{ineq:Henniart-evrs} is needed in two places: to ensure the nonvanishing assertion in the hypothesis of Proposition \ref{prop:clipped characterization} and to invoke Theorem \ref{thm:Henniart} after \eqref{eq:rho0 rho0'}.
Following \cite[Definition 3.7]{CO21}, we say that a tame elliptic regular pair $(\bfS,\theta)$ is \textit{toral} if for a(ny) Yu-datum $\Psi = (\vec \bfG, \vec \phi, \vec r, \x, \rho_0)$ corresponding to $(\bfS,\theta)$, we have that $\bfG^0 = \bfS$.
When $(\bfS,\theta)$ is toral, the Weyl group $W_{\bbG(\F_{q})}(\bbS)$ is trivial, hence the inequality \eqref{ineq:Henniart-evsrs} is given by
\[
\frac{|[\bbS]^{\star}|}{|[\bbS]^{\star}_{\nevrs}|}
=
\frac{|[\bbS]^{\star}|}{|[\bbS]^{\star} \smallsetminus [\bbS]^{\star}_{\evrs}|} > 2.
\]
Note that this inequality implies the following (see \cite[Corollary 5.5]{CO21}):
\[\label{ineq:toral-Henniart-evrs}
\text{$[\bbS]^{\star}_{\evrs}$ generates $[\bbS]^{\star}$ as a group.}
\tag{$\mathfrak{T}_{\evrs}$}
\]
The proof of Theorem \ref{thm:regsc-characterization} can be refined to yield a stronger result.

\begin{theorem}\label{thm:toral-characterization}
Assume that $(\bfS,\theta)$ is toral and that \eqref{ineq:toral-Henniart-evrs} is satisfied.
Then there exists a unique irreducible supercuspidal representation $\pi$ such that there exists a constant $c \in \bbC^1$ for which \eqref{eq:vreg id} holds for any $\gamma \in S_{\evrs}$. 
Moreover, such $\pi$ is given by the toral supercuspidal representation $\pi_{(\bfS,\theta)}^{\FKS}$.
\end{theorem}

\begin{proof}
The same proof as in Theorem \ref{thm:regsc-characterization} works with the modifications as follows.
We first note that $G^{0}_{\bar{\x},\evrs}=S_{\evrs}$ in this case.
The conclusion of Lemma \ref{lem:nonvan} obviously holds in this case since the Weyl group $W_{G^{0}}(\bfS)$ is trivial.
Hence, by using Propositions \ref{prop:clipped characterization} and \ref{prop:reduction depth 0}, we obtain \eqref{eq:rho0 rho0'}.
Note that $\bbG'(\F_{q})=\bbS(\F_{q})$, thus $\rho_{0}$ and $\rho'_{0}$ are one-dimensional characters.
Therefore the assumption \eqref{ineq:toral-Henniart-evrs} and \eqref{eq:rho0 rho0'} implies that $\rho_0'=\rho_0$ and $c = 1$.
\end{proof}

\begin{rem}\label{rem:toral-characterization}
It is worth noting that if $\bfS$ is a tamely and totally ramified (i.e., splits over a tamely and totally ramified extension of $F$) elliptic maximal torus of $\bfG$, then every tame elliptic regular pair $(\bfS,\theta)$ is automatically toral.
Indeed, since $\bfS$ is maximally unramified in $\bfG^{0}$, we have $\bfS=\Cent_{\bfG^{0}}(\bfS')$, where $\bfS'$ is the maximal unramified subtorus of $\bfS$ (\cite[Fact 3.4.1]{Kal19}).
As $\bfS'$ is trivial modulo center by the totally ramifiedness of $\bfS$, we have $\Cent_{\bfG^{0}}(\bfS')=\bfG^{0}$, thus $\bfS=\bfG^{0}$.
Therefore, when $\bfS$ is tamely and totally ramified, our characterization theorem for regular supercuspidals only requires the much weaker assumption \eqref{ineq:toral-Henniart-evrs}.
For example, if $\bfG = \GL_n$ and $\bfS$ is the totally ramified elliptic maximal torus, then we show in Section \ref{subsubsec:f=1} that \eqref{ineq:Henniart-evrs} is satisfied so long as $n > 2(n - \varphi(n))$, where $\varphi(n) = |(\bbZ/n\bbZ)^\times|$ denotes Euler's totient function. This inequality does not hold for all $n$, but it does for many; for example, it holds for all odd prime $n$, which recovers Henniart's characterization theorem in the ramified setting \cite[Section 8]{Hen93}.
On the other hand, we can check that the assumption \eqref{ineq:toral-Henniart-evrs} always holds.
\end{rem}

\begin{rem}\label{rem:strong toral characterization}
    Note that Theorem \ref{thm:toral-characterization} strictly subsumes \cite[Theorem 9.1]{CO21}. Indeed, Theorem \ref{thm:toral-characterization} says that there is a unique supercuspidal representation whose character on $S_{\evrs}$ is equal to $\pi_{(\bfS,\theta)}^{\FKS}$ for $\theta$ toral, whereas \cite[Theorem 9.1]{CO21} only says that there is a unique \textit{regular} supercuspidal representation whose character on $S_{\evrs}$ is equal to $\pi_{(\bfS,\theta)}^{\FKS}$ for $\theta$ toral \textit{and $\bfS$ unramified}. 
\end{rem}

\subsection{Unipotent supercuspidal representations}\label{subsec:unipsc-characterization}

Let us proceed with the notations as in Section \ref{subsec:regsc-characterization}.
We next consider the inequality \eqref{ineq:Lusztig-bullet} in Section \ref{subsec:Lusztig-E}:
\begin{equation*}\label{ineq:Lusztig-evrs}
\frac{|[\bbS]^{\star}|}{|[\bbS]^{\star}_{\nevrs}|}
=
\frac{|[\bbS]^{\star}|}{|[\bbS]^{\star} \smallsetminus [\bbS]^{\star}_{\evrs}|}
>
2^{2|W_{\bbG}|\cdot i(\bbS)-1}.
\tag{$\mathfrak{L}_{\evrs}$}
\end{equation*}

\begin{lemma}\label{lem:strong Henniart sat max unram}
    There exists a constant $C$ depending only on the absolute rank of $\bfG$ such that the inequality \eqref{ineq:Lusztig-evrs} is satisfied for every maximally unramified elliptic maximal torus $\bfS$ of $\bfG$ when $q > C$.
\end{lemma}

\begin{proof}
If we can show the right-hand side of the inequality \eqref{ineq:Lusztig-evrs} can be bounded above by a constant depending only on the absolute rank of $\bfG$ (especially, independent of $q$), then the proof of \cite[Lemma 5.6]{CO21} holds even after relaxing the unramifiedness condition on $\bfS$ to maximally unramifiedness because Kaletha's results (especially \cite[Lemma 3.4.12 (1)]{Kal19}) still hold.
Applying then \cite[Lemma 5.7]{CO21} and the proof strategy of Proposition 5.8 of \textit{op.\ cit.}, we can find a constant $C$ satisfying the desired condition.

We first note that $i(\bbS)\, (:=[\bbS(\F_{q}):\bbS^{\circ}(\F_{q})\bbZ_{\bbG}^{\star}(\F_{q})])$ is not greater than $[\bbG(\F_{q}):\bbG^{\circ}(\F_{q})\bbZ_{\bbG}^{\star}(\F_{q})]$.
By putting $\bbG^{\star}:=\bbG^{\circ}\bbZ_{\bbG}^{\star}$, we have
\begin{align*}
[\bbG(\F_{q}):\bbG^{\circ}(\F_{q})\bbZ_{\bbG}^{\star}(\F_{q})]
&=
[\bbG(\F_{q}):\bbG^{\star}(\F_{q})]\cdot[\bbG^{\star}(\F_{q}):\bbG^{\circ}(\F_{q})\bbZ_{\bbG}^{\star}(\F_{q})]\\
&\leq
[\bbG:\bbG^{\star}]\cdot[\bbG^{\star}(\F_{q}):\bbG^{\circ}(\F_{q})\bbZ_{\bbG}^{\star}(\F_{q})].
\end{align*}
Hence it is enough to evaluate $[\bbG^{\star}(\F_{q}):\bbG^{\circ}(\F_{q})\bbZ_{\bbG}^{\star}(\F_{q})]$.
For any element $g^{\star}\in\bbG^{\star}(\F_{q})$, we write $g^{\star}=gz$ with $g\in\bbG^{\circ}(\F_{q}), z\in\bbZ_{\bbG}^{\star}(\F_{q})$ and define a $1$-cocycle $\Gal(\overline{\F}_{q}/\F_{q})\rightarrow \bbG^{\circ}\cap\bbZ_{\bbG}^{\star}$ by $\sigma\mapsto z^{-1}\sigma(z)$.
We claim that this defines an injective map
\[
\bbG^{\star}(\F_{q})/\bbG^{\circ}(\F_{q})\bbZ_{\bbG}^{\star}(\F_{q})
\hookrightarrow
H^{1}(\F_{q},\bbG^{\circ}\cap\bbZ_{\bbG}^{\star}).
\]
Indeed, suppose that two elements $g^{\star}_{1}=g_{1}z_{1}$ and $g^{\star}_{2}=g_{2}z_{2}$ give rise to $1$-cocycles which are $1$-cohomologous, i.e., there exists $z\in \bbG^{\circ}\cap\bbZ_{\bbG}^{\star}$ satisfying $z_{1}^{-1}\sigma(z_{1})=z^{-1}z_{2}^{-1}\sigma(z_{2})\sigma(z)$.
Then $z_{1}^{-1}zz_{2}\in\bbZ_{\bbG}^{\star}$ is $\F_{q}$-rational.
By noting that $g^{\star}_{1}$ and $g^{\star}_{2}$ are $\F_{q}$-rational, we see that $g_{1}^{-1}g_{2}z^{-1}\in\bbG^{\circ}$ is also $\F_{q}$-rational.
Hence we get
\[
g_{2}^{\star}
=g_{2}z_{2}
=g_{1}z_{1}\cdot g_{1}^{-1}g_{2}z^{-1} \cdot z_{1}^{-1}zz_{2}
\equiv g_{1}z_{1}
=g_{1}^{\star}
\mod{\bbG^{\circ}(\F_{q})\bbZ_{\bbG}^{\star}(\F_{q})}.
\]
Since $|H^{1}(\F_{q},\bbG^{\circ}\cap\bbZ_{\bbG}^{\star})|$ is determined only by $\bbG$ and $\bbZ_{\bbG}^{\star}$, we get the assertion.
\end{proof}

\begin{theorem}\label{thm:p-adic unipotent}
    Assume $q$ is larger than the constant $C$ as in Lemma \ref{lem:strong Henniart sat max unram}. Then an irreducible supercuspidal representation $\pi$ of $G$ is unipotent in the sense of Definition \ref{def:ss and unip} if and only if the following two conditions hold:
    \begin{enumerate}[label=(\roman*)]
        \item $\Theta_\pi|_{S_{0,\evrs}}$ is constant for every maximally unramified elliptic maximal torus $\bfS$ of $\bfG$, and
        \item $\Theta_\pi|_{S_{0,\evrs}} \neq 0$ for some maximally unramified elliptic maximal torus $\bfS$ of $\bfG$.
    \end{enumerate}
\end{theorem}

\begin{proof}
Let $\pi$ be an irreducible unipotent supercuspidal representation which corresponds to a Yu-datum $(\vec{\bfG}=(\bfG^{0}),\vec{\phi}=(\phi_{0}=\mathbbm{1}),\vec{r}=(r_{0}=0),\x,\rho_{0})$ under the modified construction of Fintzen--Kaletha--Spice. 
Let $\bfS$ be a maximally unramified elliptic maximal torus of $\bfG$.
Then, by Theorem \ref{thm:CF}, we have
\[
\Theta_{\pi}(\gamma)
=
\sum_{g \in G_{\bar{\x}}\backslash N_{G}(S,G_{\bar{\x}})}\Theta_{\rho_{0}}({}^{g}\gamma)
\]
for any $\gamma\in S_{0,\evrs}$.
Note that the index set is at most singleton by Lemma \ref{lem:Weyl}.
Thus the constancy and the non-vanishing (for some $\bfS$) follow from Corollary \ref{cor:ff unipotent cuspidal}.

Let $\pi$ be an irreducible supercuspidal representation satisfying the two conditions (i) and (ii).
By Fintzen's exhaustion theorem (Theorem \ref{thm:Fintzen}), our assumption that $p\nmid|W_{\bfG}|$ implies that $\pi$ is a tame supercuspidal representation.
Let $\pi=\pi_{\Psi}^{\FKS}$ for a Yu-datum $\Psi=(\vec{\bfG},\vec{\phi},\vec{r},\x,\rho_{0})$.
By the condition (ii), there exists a maximally unramified elliptic maximal torus $\bfS$ of $\bfG$ such that $\Theta_\pi(\gamma) \neq 0$ for some $\gamma \in S_{0,\evrs}$.
For any $\gamma_{0+}\in S_{0+}$, by applying Theorem \ref{thm:CF} to $\gamma$ and $\gamma\gamma_{0+}$, the condition (i) implies that
    \begin{align*}
        \sum_{g \in G_{\bar \x}^0 \backslash N_G(S,G_{\bar \x}^0)} &\Theta_{\rho_0}({}^g \gamma) \cdot \frac{\Delta_{\II}^{\abs,\bfG}[a_\Psi,\chi_\Psi'']({}^g \gamma)}{\Delta_{\II}^{\abs,\bfG^0}[a_\Psi,\chi_\Psi'']({}^g \gamma)} \cdot \phi_{\geq 0}({}^g \gamma) \cdot \phi_{\geq 0}^g(\gamma_{0+}) \\
        &= 
        \sum_{g \in G_{\bar \x}^0 \backslash N_G(S,G_{\bar \x}^0)} \Theta_{\rho_0}({}^g \gamma) \cdot \frac{\Delta_{\II}^{\abs,\bfG}[a_\Psi,\chi_\Psi'']({}^g \gamma)}{\Delta_{\II}^{\abs,\bfG^0}[a_\Psi,\chi_\Psi'']({}^g \gamma)} \cdot \phi_{\geq 0}({}^g \gamma) 
    \end{align*}
    for all $\gamma_{0+} \in S_{0+}$. We may regard this equality as a linear relation amongst the $S_{0+}$-characters $\{\phi_{\geq 0}^g|_{S_{0+}}\}_{g \in G_{\bar \x}^0 \backslash N_G(S,G_{\bar \x}^0)} \cup \{\mathbbm{1}\}$. Applying $\langle -, \mathbbm{1}\rangle_{S_{0+}}$ to this identity, we see that the right-hand side must be nonzero by assumption, which necessarily means that one of the terms $\langle \phi_{\geq 0}^g|_{S_{0+}}, \mathbbm{1}\rangle_{S_{0+}}$ must also be nonzero.
By replacing $\bfS$ with ${}^{g}\bfS$, we may suppose that $g=1$.
But now this implies that $\phi_{\geq 0}|_{S_{0+}}$ is the trivial character. 
By Proposition \ref{prop:refacto}, we see that $(\vec{\bfG},\vec{\phi}, \vec{r},\x)$ is a refactorization of $((\bfG^{0}),(\mathbbm{1}), (0),\x)$.
Thus, by replacing $\Psi$ with a $\bfG$-equivalent Yu-datum, we may assume that $\Psi=((\bfG^{0}=\bfG), (\phi_{0}=\mathbbm{1}), (r_{0}=0), \x, \rho_{0})$.
Now our remaining task is to show that $\rho_{0}$ is a unipotent cuspidal representation (recall Definition \ref{def:ss and unip}).

As discussed in the first paragraph of this proof, for each elliptic maximal torus $\bbS^{\circ}$ of $\bbG^{\circ}$ corresponding to a maximally unramified elliptic maximal torus $\bfS$ of $\bfG$ (Proposition \ref{prop:DeBacker-Kaletha}), we have
\[
\Theta_{\pi}(\gamma)
=
\Theta_{\rho_{0}}(\gamma)
\]
for any $\gamma\in S_{0,\evrs}$ (note that $g\in G_{\bar{\x}}\backslash N_{G}(S,G_{\bar{\x}})$ can be taken to be $1$ in this case).
Hence the condition (i) implies that $\Theta_{\rho_{0}}(\gamma)$ is constant for any $\gamma\in\bbS^{\circ}(\F_{q})_{\evrs}$.
By Corollary \ref{cor:ff unipotent cuspidal} (applied to the setting that $\bullet = \evrs$), we conclude that $\rho_{0}$ is unipotent.
\end{proof}

\subsection{Families of tame supercuspidal representations}

We prove the following analogue of Corollary \ref{cor:tilde-Z} for supercuspidal representations of $G$.

\begin{theorem}\label{thm:p-adic E}
Let $\Psi=(\vec{\bfG},\vec{\phi},\vec{r},\x,\rho_{0})$ be a Yu-datum and $\pi=\pi_{\Psi}^{\FKS}$ be the associated tame supercuspidal representation.
Assume that the inequality \eqref{ineq:Lusztig-evrs} is satisfied for every maximally unramified elliptic maximal torus of $\bfG^{0}$ with associated point $\x$.
For any irreducible supercuspidal representation $\pi'$ of $G$, 
\begin{equation*}
\Theta_\pi(\gamma) = \Theta_{\pi'}(\gamma) \qquad \text{for all $\gamma\in G^{0}_{\bar{\x},\evrs}$}
\end{equation*}
if and only if $\pi'\cong\pi_{\Psi'}^{\FKS}$ for a Yu-datum $\Psi'=(\vec{\bfG},\vec{\phi},\vec{r},\x,\rho'_{0})$ with $\tilde E(\rho_0) = \tilde E(\rho_0')$.
\end{theorem}

\begin{proof}
Let $\pi'$ be an irreducible supercuspidal representation of $G$ satisfying the assumption as in the statement.
Let $\Psi'=(\vec{\bfG}',\vec{\phi}',\vec{r}',\x',\rho'_{0})$ be a Yu-datum such that $\pi'\cong\pi_{\Psi'}^{\FKS}$ (we can always take such a $\Psi'$ by Fintzen's exhaustion theorem and our assumption on $p$).

We first show, using the assumption \eqref{ineq:Lusztig-evrs} that there exists an element $\overline{\gamma_0} \in \bbG'(\FF_q)_{\evrs}$ such that $\Theta_{\rho_0}(\overline{\gamma_0}) \neq 0$.
Indeed, if this were not true, then of course $\Theta_{\rho_0}|_{\bbS(\FF_q)_{\evrs}}$ would be constant for every elliptic maximal torus $\bbS^\circ$ of $\bbG^\circ$. 
Since $\rho_{0}$ is cuspidal by assumption, by Corollary \ref{cor:ff unipotent cuspidal}, we have that $\rho_{0}$ is unipotent. 
Let $(\bbS,\theta)$ be arbitrary such that $\langle \rho_0, R_{\bbS}^{\bbG}(\theta) \rangle \neq 0$ (such a pair exists by Corollary \ref{cor:exhaustion}). Then necessarily $\theta^{\circ} = \mathbbm{1}$ and $\bbS^\circ$ is elliptic in $\bbG^\circ$ (Lemmas \ref{lem:not-geom-conj} and \ref{lem:cuspidal}). But then there exists $s \in \bbG'(\FF_q)_{\evrs}$ such that $\Theta_{\rho_0}(s) \neq 0$ by Proposition \ref{prop:rho ss}, a contradiction.

We may now apply Proposition \ref{prop:clipped characterization} to obtain that $\dashover{\Psi}$ and $\dashover{\Psi}'$ are $\bfG$-equivalent. 
    By Proposition \ref{prop:reduction depth 0}, we have moreover that 
    \begin{equation*}
        \Theta_{\rho_0}(\gamma) = \Theta_{\rho_0'}(\gamma) \qquad \text{for any $\gamma \in \bbG'(\FF_q)_{\evrs}$}. 
    \end{equation*}
    Applying Corollary \ref{cor:tilde-Z}, we conclude that $\tilde E(\rho_0) = \tilde E(\rho_0')$.
\end{proof}

One takeaway from Theorem \ref{thm:p-adic E} is that even if $\bfG$ satisfies the strong hypothesis of Theorem \ref{thm:p-adic unipotent}, not every irreducible supercuspidal representation is distinguished by its character values on elliptic very regular elements. Moreover, in this setting, the reason for this failure is a depth zero phenomenon: Lusztig's map $\tilde E$ (Section \ref{subsec:Lusztig-E}) is not injective.

We end on a corollary of the work in this paper that particularly exemplifies this failure.

\begin{cor}\label{cor:non-singular failure}
Let $(\bfS,\theta)$ be a tame elliptic $k_{F}$-non-singular pair of $\bfG$ whose $\bfS$ is unramified.
Then for any irreducible representations $\pi, \pi' \in [\pi^{\FKS}_{(\bfS,\theta)}]$, 
\begin{equation*}
    \Theta_{\pi}(\gamma) = \Theta_{\pi'}(\gamma)
\end{equation*}
for any $\gamma\in G^{0}_{\bar{\x},\evrs}$.
In particular, such $\pi$ can be uniquely determined by its character values on $G^{0}_{\bar{\x},\evrs}$ \textit{only if} $\pi$ is regular supercuspidal.
\end{cor}

\begin{proof}
Let $\Psi = (\vec \bfG, \vec \phi, \vec r, \x, \rho_0)$ be a Yu-datum such that $\pi \cong \pi_{\Psi}^{\FKS}$; by assumption, $\pi'\cong\pi_{\Psi'}^{\FKS}$ for $\Psi' = (\vec \bfG, \vec \phi, \vec r, \x, \rho_0')$ where $\rho_0$ and $\rho_0'$ are both irreducible summands of $(-1)^{d(\bfS)}R_{\bbS}^{\bbG}(\phi_{-1})$ for some $k_{F}$-non-singular character $\phi_{-1}$ of $S$ of depth zero. 
Since $\bfS$ is unramified, we have $G^{0}_{\bar{\x}} = Z_{\bfG^{0}}G^{0}_{\x,0}$. 
In particular, this implies $\bbG=\bbG'$.
It follows from Lusztig \cite{Lus88} (see also \cite[Theorem 2.3.1]{Kal19-sc}) that the irreducible decomposition of $(-1)^{d(\bfS)}R_{\bbS^{\circ}}^{\bbG^{\circ}}(\phi_{-1})$ is multiplicity-free, i.e., $(-1)^{d(\bfS)}R_{\bbS^{\circ}}^{\bbG^{\circ}}(\phi_{-1})=\bigoplus_{i=1}^{r}\rho_{i}$, where $\rho_{i}$'s are pairwise distinct irreducible representations of $\bbG^{\circ}(\F_{q})$.
Since $\bbG=\bbG'$, $(-1)^{d(\bfS)}R_{\bbS}^{\bbG}(\phi_{-1})=(-1)^{d(\bfS)}R_{\bbS}^{\bbG'}(\phi_{-1})$ is an extension of $(-1)^{d(\bfS)}R_{\bbS^{\circ}}^{\bbG^{\circ}}(\phi_{-1})$.
Thus $(-1)^{d(\bfS)}R_{\bbS}^{\bbG}(\phi_{-1})$ is also multiplicity-free.
In particular, we have $\langle \rho_0, R_{\bbS}^{\bbG}(\phi_{-1}) \rangle = \langle \rho_0', R_{\bbS}^{\bbG}(\phi_{-1}) \rangle$. 

By Theorem \ref{thm:CF} and Proposition \ref{prop:rho ss}, we see that to prove $\Theta_\pi|_{G_{\bar \x, \evrs}^0} = \Theta_{\pi'}|_{G_{\bar \x, \evrs}^0},$ it suffices to show that for any pair $(\bbS',\phi_{-1}')$, we have 
\begin{equation}\label{eq:rho rho'}
	\langle \rho_0, R_{\bbS'}^\bG(\phi_{-1}') \rangle = \langle \rho_0', R_{\bbS'}^\bbG(\phi_{-1}') \rangle.
\end{equation} 
To do this, by symmetry, it is enough to prove that if $(\bbS',\phi_{-1}')$ is such that $\langle \rho_0, R_{\bbS'}^{\bbG}(\phi_{-1}') \rangle \neq 0$, then \eqref{eq:rho rho'} holds; for the rest of the proof, let $(\bbS',\phi_{-1}')$ be any such pair.
The condition that $\langle \rho_0, R_{\bbS'}^\bbG(\phi_{-1}') \rangle \neq 0$ implies that $R_{\bbS}^\bbG(\phi_{-1})$ and $R_{\bbS'}^\bbG(\phi_{-1}')$ have the common irreducible constituent $\rho_{0}$.
Then Lemma \ref{lem:not-geom-conj} implies that $(\bbS,\phi_{-1})$ and $(\bbS',\phi'_{-1})$ are geometrically conjugate.
In particular, $\phi'_{-1}$ is also non-singular, hence $(-1)^{d(\bfS')} R_{\bbS'}^{\bbG}(\phi'_{-1})$ is a genuine representation.
Therefore, again noting that $R_{\bbS}^{\bbG}(\phi_{-1})$ and $R_{\bbS'}^{\bbG}(\phi'_{-1})$ have the common irreducible constituent $\rho_{0}$, we get $\langle R_{\bbS}^{\bbG}(\phi_{-1}),R_{\bbS'}^{\bbG}(\phi'_{-1})\rangle\neq0$.
By the scalar product formula (Corollary \ref{cor:scalar-prod}), this implies that $(\bbS,\phi_{-1}), (\bbS',\phi_{-1}')$ are $\bbG(\FF_q)$-conjugate and furthermore that $\langle R_{\bbS}^\bbG(\phi_{-1}), R_{\bbS}^\bbG(\phi_{-1}) \rangle = \langle R_{\bbS}^\bbG(\phi_{-1}), R_{\bbS'}^\bbG(\phi_{-1}') \rangle = \langle R_{\bbS'}^\bbG(\phi_{-1}'), R_{\bbS'}^\bbG(\phi_{-1}') \rangle$. Hence $R_{\bbS}^\bbG(\phi_{-1}) \cong R_{\bbS}^{\bbG}(\phi_{-1}')$, and \eqref{eq:rho rho'} follows.

In particular, the above argument shows that if $\pi$ is uniquely determined by its character values on $G_{\bar \x, \evrs}^0$, then necessarily $R_{\bbS}^\bbG(\phi_{-1})$ must be irreducible.
\end{proof}

\section{Application to the explicit local Jacquet--Langlands correspondence}\label{sec:LJLC}

In this section, we obtain an explicit characterization of the local Jacquet--Langlands correspondence for $L$-packets consisting of a single regular supercuspidal representation (Theorems \ref{thm:regsc-characterization-evsrs} and \ref{thm:LJLC}) in the flavor of our characterization theorem (Theorem \ref{thm:regsc-characterization}).
As proof of concept, we then show that this establishes a new instance of local Jacquet--Langlands transfers in the case of depth zero supercuspidal representations of $\SO_{2n+1}$ (Theorem \ref{thm:LJLC-SO}).

\subsection{Local Jacquet--Langlands correspondence}\label{subsec:LJLC}
Let us first briefly review the conjectural local Langlands correspondence.
For a connected reductive group $\bfG$ over $F$, we let $\Pi(\bfG)$ denote the set of equivalence classes of irreducible admissible representations of $G$ and $\Phi(\bfG)$ denote the set of $\hat{\bfG}$-conjugacy classes of $L$-parameters of $\bfG$, where $\hat{\bfG}$ is the Langlands dual group of $\bfG$ taken over $\C$.
\textit{The local Langlands correspondence for $\bfG$}, which is still conjectural in general, asserts that there exists a finite-to-one map
\[
\LLC_{\bfG}\colon\Pi(\bfG)\rightarrow\Phi(\bfG).
\]
In other words, it is conjectured that the set $\Pi(\bfG)$ is partitioned into the disjoint union of finite sets $\Pi^{\bfG}_{\phi}\colonequals \LLC_{\bfG}^{-1}(\phi)$ (called \textit{$L$-packets}) labelled by $L$-parameters $\phi\in\Phi(\bfG)$:
\[
\Pi(\bfG)
=
\bigsqcup_{\phi\in\Phi(\bfG)}\Pi_{\phi}^{\bfG}.
\]
Furthermore, each $L$-packet $\Pi_{\phi}^{\bfG}$ is expected to be equipped with a map $\iota\colon \Pi_{\phi}^{\bfG} \rightarrow\mathrm{Irr}(\mathcal{S}_{\phi})$ to the set of irreducible representations of a certain finite group $\mathcal{S}_{\phi}$ determined by the $L$-parameter $\phi$.
(We refer the reader to \cite{Art06} and \cite{Kal16-LLC} for details.)
The map $\LLC_{\bfG}$ is believed to be ``natural'' in the sense that it satisfies various nice properties.
Among such properties, we are especially interested in \textit{the standard endoscopic character relation} between inner forms.

To explain it, let us consider the quasi-split inner form $\bfG^{\ast}$ of $\bfG$ over $F$ realized by an inner twist $\psi\colon\bfG\rightarrow\bfG^{\ast}$.
Since the $L$-groups of $\bfG^{\ast}$ and $\bfG$ are the same, we may identify $\Phi(\bfG)$ with $\Phi(\bfG^{\ast})$.
Hence, if the local Langlands correspondence exists for both groups $\bfG$ and $\bfG^{\ast}$, we can associate an $L$-packet $\Pi^{\bfG}_{\phi}$ of $\bfG$ to any $L$-packet $\Pi^{\bfG^{\ast}}_{\phi}$ of $\bfG^{\ast}$.
We call this association $\Pi^{\bfG^{\ast}}_{\phi}\mapsto\Pi^{\bfG}_{\phi}$ \textit{the local Jacquet--Langlands correspondence between $\bfG$ and $\bfG^{\ast}$}.
In this situation, it is expected that the $L$-packets $\Pi^{\bfG}_{\phi}$ and $\Pi^{\bfG^{\ast}}_{\phi}$ satisfy the following identity:
\[\label{eq:JLCR}
e(\bfG)\sum_{\pi\in\Pi^{\bfG}_{\phi}}\dim\iota(\pi)\cdot\Theta_{\pi}(\gamma)
=
\sum_{\pi^{\ast}\in\Pi^{\bfG^{\ast}}_{\phi}}\dim\iota^{\ast}(\pi^{\ast})\cdot \Theta_{\pi^{\ast}}(\gamma^{\ast})
\tag{JLCR}
\]
for any strongly regular semisimple element $\gamma$ of $G$, where 
\begin{itemize}
\item
$e(\bfG)$ is the Kottwitz sign of $\bfG$, 
\item
$\iota$ (resp.\ $\iota^{\ast}$) is the map $\Pi_{\phi}^{\bfG}\rightarrow\mathrm{Irr}(\mathcal{S}_{\phi})$ (resp.\ $\Pi_{\phi}^{\bfG^{\ast}}\rightarrow\mathrm{Irr}(\mathcal{S}_{\phi})$) mentioned above,
\item
$\gamma^{\ast}$ is a(ny) strongly regular semisimple element of $G^{\ast}$ \textit{related to $\gamma$} in the sense that $\psi(\gamma)$ is conjugate to $\gamma^{\ast}$ in $\bfG^{\ast}$.
\end{itemize}
This is a special case of the standard endoscopic character relation.
(The quasi-split form $\bfG^{\ast}$ is the most special case of \textit{a standard endoscopic group of $\bfG$}.)
Let us call the identity \eqref{eq:JLCR} \textit{the Jacquet--Langlands character relation}.
Note that, by linear independence of the Harish-Chandra characters of irreducible admissible representations, the $L$-packet $\Pi^{\bfG}_{\phi}$ is characterized by the identity \eqref{eq:JLCR} as long as the $L$-packet $\Pi^{\bfG^{\ast}}_{\phi}$ is given.

\begin{rem}
\begin{enumerate}
\item
Recall that we say that strongly regular semisimple elements of $G$ (resp.\ $G^{\ast}$) are \textit{stably conjugate} if they are conjugate in $\bfG$ (resp.\ $\bfG^{\ast}$).
The identity \eqref{eq:JLCR} presupposes that the both sides are invariant under the stable conjugacy of $\gamma$ and $\gamma^{\ast}$.
This property is called \textit{the stability of $L$-packets}.
\item
The maps $\iota\colon\Pi^{\bfG}_{\phi}\rightarrow\mathrm{Irr}(\mathcal{S}_{\phi})$ and $\iota^{\ast}\colon\Pi^{\bfG^{\ast}}_{\phi}\rightarrow\mathrm{Irr}(\mathcal{S}_{\phi})$ are supposed to depend on the choice of \textit{a Whittaker datum of $\bfG^{\ast}$}.
Thus, precisely speaking, we must fix a Whittaker datum of $\bfG^{\ast}$ at the beginning.
However, later we will focus only on the case where the structure of an $L$-packet is trivial.
Thus we do not go into this matter in depth.
See \cite{Kal16-LLC} for the details.
\end{enumerate}
\end{rem}

In the following, we assume that we have the local Langlands correspondence for $\bfG$ and $\bfG^{\ast}$ satisfying the Jacquet--Langlands character relation \eqref{eq:JLCR} between $\bfG$ and $\bfG^{\ast}$.

\begin{example}\label{ex:LLC}
As explained in the beginning, our aim is to describe the local Jacquet--Langlands correspondence for regular supercuspidal representations explicitly.
In \cite{Kal19,Kal19-sc}, Kaletha established the local Langlands correspondence for non-singular supercuspidal representations.
In fact, it is essentially tautological to describe the local Jacquet--Langlands correspondence in the sense of Kaletha by looking at his construction.
Hence the problem we are interested in is only meaningful outside Kaletha's framework. 
We keep in mind the following cases, noting also that it is not obvious whether Kaletha's correspondence coincides with the ones listed below:
\begin{enumerate}
\item
When $\bfG^{\ast}=\GL_{n}$ and $\bfG$ is an inner form of $\bfG^{\ast}$, the local Langlands correspondence for $\bfG$ and $\bfG^{\ast}$ has been established by Harris--Taylor \cite{HT01} and Henniart \cite{Hen00}.
(The local Jacquet--Langlands correspondence had been established before the works \cite{HT01,Hen00} by Deligne--Kazhdan--Vign\'eras \cite{DKV84}.)
\item
When $\bfG^{\ast}$ is a quasi-split special orthogonal group or symplectic group, the local Langlands correspondence has been established by Arthur \cite{Art13}.
The case of inner forms of special orthogonal groups is treated in \cite{MR18}.
\item
When $\bfG^{\ast}$ is a quasi-split unitary group, the local Langlands correspondence has been established by Mok \cite{Mok15} based on the same method as in \cite{Art13}.
The case of inner forms is treated in \cite{KMSW14}.
\end{enumerate}
\end{example}

\subsection{Transfer of tame elliptic regular pairs}\label{subsec:transfer}
Let $\bfG$, $\bfG^{\ast}$, and $\psi$ be as in Section \ref{subsec:LJLC}.
Suppose that $(\bfS^{\ast},\theta^{\ast})$ is a tame elliptic extra regular pair of $\bfG^{\ast}$.
Since $\bfS^{\ast}$ is elliptic in $\bfG^{\ast}$, by replacing the inner twist $\psi\colon\bfG\rightarrow\bfG^{\ast}$ if necessary, we may assume that $\psi$ induces an $F$-rational isomorphism $\psi|_{\bfS}\colon\bfS\rightarrow\bfS^{\ast}$ between an elliptic maximal torus $\bfS$ of $\bfG$ and $\bfS^{\ast}$ (see \cite[Section 10]{Kot86} and also \cite[Section 3.2]{Kal19}).
We define a character $\theta\colon S\rightarrow\C^{\times}$ by $\theta\colonequals \theta^{\ast}\circ\psi|_{\bfS}$.

\begin{lem}\label{lem:TERP-transfer}
The pair $(\bfS,\theta)$ is a tame elliptic extra regular pair of $\bfG$.
\end{lem}

\begin{proof}
Since $\psi|_{\bfS}$ is $F$-rational, $\psi|_{\bfS}$ induces a Galois-equivariant bijection between $R(\bfG,\bfS)$ and $R(\bfG^{\ast},\bfS^{\ast})$.
This means that $\psi$ also induces a bijection between $R_{0+}$ and $R^{\ast}_{0+}$, which are subsets of $R(\bfG,\bfS)$ and $R(\bfG^{\ast},\bfS^{\ast})$ defined as in Definition \ref{defn:non-sing-pair} (1), respectively.
If we let $\bfG^{0}$ and $\bfG^{\ast0}$ be the connected reductive subgroups of $\bfG$ and $\bfG^{\ast}$ defined as in Definition \ref{defn:non-sing-pair} (1), respectively, then $\psi$ gives an inner twist $\psi|_{\bfG^{0}}\colon\bfG^{0}\rightarrow\bfG^{\ast0}$.

Now let us check the condition (1) of Definition \ref{defn:non-sing-pair} for $(\bfS,\theta)$.
By \cite[Fact 3.4.1]{Kal19}, $\bfS$ (resp.\ $\bfS^{\ast}$) is maximally unramified in $\bfG^{0}$ (resp.\ $\bfG^{0\ast}$) if and only if the action of the inertia subgroup $I_{F}$ on $R_{0+}$ (resp.\ $R^{\ast}_{0+}$) preserves a set of positive roots.
Thus, since the identification between $R_{0+}$ and $R^{\ast}_{0+}$ is Galois-equivariant, the maximally unramifiedness of $\bfS^{\ast}$ in $\bfG^{0\ast}$ (guaranteed by the assumption that $(\bfS^{\ast},\theta^{\ast})$ is tame elliptic regular) implies that of $\bfS$ in $\bfG^{0}$.

The condition (2) of Definition \ref{defn:non-sing-pair} for $(\bfS,\theta)$ follows from that for $(\bfS^{\ast},\theta^{\ast})$ by noting that the inner twist $\psi|_{\bfG^{0}}\colon\bfG^{0}\rightarrow\bfG^{\ast0}$ naturally induces a Galois-equivariant identification $W_{\bfG^{0}}(\bfS)\cong W_{\bfG^{\ast0}}(\bfS^{\ast})$, hence $W_{\bfG^{0}}(\bfS)(F)\cong W_{\bfG^{\ast0}}(\bfS^{\ast})(F)$.
\end{proof}

When tame elliptic extra regular pairs $(\bfS^{\ast},\theta^{\ast})$ of $\bfG^{\ast}$ and $(\bfS,\theta)$ of $\bfG$ are related in this way, we say that $(\bfS,\theta)$ is a \textit{transfer} of $(\bfS^{\ast},\theta^{\ast})$.
Note that, when $(\bfS,\theta)$ is a transfer of $(\bfS^{\ast},\theta^{\ast})$ through an inner twist $\psi$, any strongly regular semisimple element $\gamma\in S$ is related to $\psi(\gamma)\in S^{\ast}$.
Furthermore, again by noting that $\psi$ induces a Galois-equivariant bijection between $R(\bfG^{\ast},\bfS^{\ast})$ and $R(\bfG,\bfS)$, we see that $\psi(\gamma)$ is very regular in $\bfG^{\ast}$ if and only if $\gamma$ is very regular in $\bfG$.

\begin{lem}\label{lem:tran-transfer}
Suppose that $(\bfS^{\ast},\theta^{\ast})$ is a tame elliptic regular pair of $\bfG^{\ast}$ and $(\bfS,\theta)$ is its transfer to $\bfG$ given by an inner twist $\psi$.
Let $\Psi^{\ast}$ (resp.\ $\Psi$) be a regular Yu-datum of $\bfG^{\ast}$ (resp.\ $\bfG$) corresponding to $(\bfS^{\ast},\theta^{\ast})$ (resp.\ $(\bfS,\theta)$).
Then, for any very regular element $\gamma\in S$, we have
\[
\Delta_{\II}^{\abs,\bfG^{\ast}}[a_{\Psi^{\ast}},\chi''_{\Psi^{\ast}}](\psi(\gamma))
=
\Delta_{\II}^{\abs,\bfG}[a_{\Psi},\chi''_{\Psi}](\gamma).
\]
\end{lem}

\begin{proof}
This directly follows from the definitions of $\Delta_{\II}^{\abs}$, Kaletha's $a$-data, and $\chi$-data (recall that $\theta$ is defined to be $\theta^{\ast}\circ\psi$ and again note that $\psi$ induces a Galois-equivariant bijection between $R(\bfG^{\ast},\bfS^{\ast})$ and $R(\bfG,\bfS)$).
\end{proof}

\subsection{Local Jacquet--Langlands correspondence for singleton $L$-packets}\label{subsec:LJLC-1}

\begin{lem}\label{lem:Weyl-eq}
Suppose that $\bfS$ is an elliptic maximal torus of $\bfG$ such that the natural map $H^{1}(F,\bfS)\rightarrow H^{1}(F,\bfG)$ is injective.
Then, for any $F$-rational maximal torus $\bfT$ of $\bfG$, we have $W_{G}(\bfT,\bfS)=W_{\bfG}(\bfT,\bfS)(F)$ and $W_{G}(\bfS,\bfT)=W_{\bfG}(\bfS,\bfT)(F)$.
\end{lem}

\begin{proof}
Since $W_{\bfG}(\bfT,\bfS)$ is isomorphic to $W_{\bfG}(\bfS,\bfT)$ by $w\mapsto w^{-1}$, it is enough to only show $W_{G}(\bfS,\bfT)=W_{\bfG}(\bfS,\bfT)(F)$.
Since the inclusion $W_{G}(\bfS,\bfT)\subset W_{\bfG}(\bfS,\bfT)(F)$ is obvious, our task is to show the converse.
Let $w\in W_{\bfG}(\bfS,\bfT)(F)$.
We take a representative $n\in N_{\bfG}(\bfS,\bfT)(\overline{F})$ of $w$.
Then, as $w$ is $F$-rational, we have $n^{-1}\sigma(n)\in\bfS(\overline{F})$ for any $\sigma\in\Gamma_{F}$.
If we put $s_{\sigma}\colonequals n^{-1}\sigma(n)$, then we get a $1$-cocycle $s_{\sigma}\in Z^{1}(F,\bfS)$.
Moreover, by construction, its image in $H^{1}(F,\bfS)$ belongs to the kernel of the natural map $H^{1}(F,\bfS)\rightarrow H^{1}(F,\bfG)$.
Thus the assumption implies that the image of $s_{\sigma}$ in $H^{1}(F,\bfS)$ is trivial.
In other words, there exists an element $s\in\bfS(\overline{F})$ satisfying $s_{\sigma}=s^{-1}\sigma(s)$ for any $\sigma\in\Gamma_{F}$.
This means that $ns^{-1}\in N_{\bfG}(\bfS,\bfT)(\overline{F})$ is an $F$-rational element which represents $w$.
\end{proof}

In the following, let us assume that 
\[\label{hyp:injectivity}
\text{
the maps $H^{1}(F,\bfS^{\ast})\rightarrow H^{1}(F,\bfG^{\ast})$ and $H^{1}(F,\bfS)\rightarrow H^{1}(F,\bfG)$ are injective.
}
\tag{$\mathfrak{inj}$}
\]
Let $(\bfS^{\ast},\theta^{\ast})$ be a tame elliptic regular pair of $\bfG^{\ast}$ and $(\bfS,\theta)$ its transfer to $\bfG$.
(Note that the regularity of of a tame elliptic pair is equivalent to the extra regularity by Lemma \ref{lem:Weyl-eq} under the assumption \eqref{hyp:injectivity}.)
Let $\Psi=(\vec{\bfG},\vec{\phi},\vec{r},\x,\rho_{0})$ and $\Psi^{\ast}=(\vec{\bfG}^{\ast},\vec{\phi}^{\ast},\vec{r}^{\ast},\x^{\ast},\rho^{\ast}_{0})$ be regular Yu-data associated to $(\bfS,\theta)$ and $(\bfS^{\ast},\theta^{\ast})$, respectively.

\begin{rem}\label{rem:injectivity}
Since $\bfS^{\ast}$ is elliptic, the map $H^{1}(F,\bfS^{\ast})\rightarrow H^{1}(F,\bfG^{\ast})$ is surjective by \cite[10.2 Lemma]{Kot86}.
Hence, the above condition is equivalent to the condition that $H^{1}(F,\bfS^{\ast})$ and $H^{1}(F,\bfG^{\ast})$ have the same order.
It is not so difficult to check if the latter condition holds when the pair $(\bfG^{\ast},\bfS^{\ast})$ is given explicitly.
Also note that we have $|H^{1}(F,\bfS^{\ast})|=|H^{1}(F,\bfS)|$ and $|H^{1}(F,\bfG^{\ast})|=|H^{1}(F,\bfG)|$.
(For example, this can be seen by using the Kottwitz isomorphism \cite[1.2 Theorem]{Kot86}).
Hence the above injectivity holds for $(\bfG^{\ast},\bfS^{\ast})$ if and only if it holds for $(\bfG,\bfS)$.
\end{rem}

We must be careful about the discrepancy between the strongly regular semisimple elements and regular semisimple elements.
Since the character relation \eqref{eq:JLCR} holds only for strongly regular semisimple elements, we have to be able to recover a regular supercuspidal representation from its character on elliptic very regular elements which are strongly regular semisimple (let us say ``elliptic very strongly regular elements'').
Thus let us introduce further variants of \eqref{ineq:Henniart-evrs} and \eqref{ineq:toral-Henniart-evrs}:
\[\label{ineq:Henniart-evsrs}
\frac{|[\bbS]^{\star}|}{|[\bbS]^{\star} \smallsetminus [\bbS]^{\star}_{\evsrs}|} 
> 2\cdot|W_{\bbG(\F_{q})}(\bbS)|,
\tag{$\mathfrak{H}_{\evsrs}$}
\]
\[\label{ineq:toral-Henniart-evsrs}
\text{$[\bbS]^{\star}_{\evsrs}$ generates $[\bbS]^{\star}$ as a group.}
\tag{$\mathfrak{T}_{\evsrs}$}
\]
Here, $[\bbS]^{\star}_{\evsrs}$ denotes the quotient $\bbS(\F_{q})_{\evsrs}/\bbZ_{\bbG}^{\star}(\F_{q})$, where $\bbS(\F_{q})_{\evsrs}$ is the image of elliptic very strongly regular elements of $S$ under the reduction map $S\twoheadrightarrow\bbS(\F_{q})$.
Then the exactly same proof as in Theorems \ref{thm:regsc-characterization} and \ref{thm:toral-characterization} implies the following:

\begin{theorem}\label{thm:regsc-characterization-evsrs}
Assume that \eqref{ineq:Henniart-evsrs} is satisfied for $(\bfG,\bfS)$. 
Then there exists a unique irreducible supercuspidal representation $\pi$ of $G$ such that there exists a constant $c \in \bbC^1$ for which
\[
\Theta_{\pi}(\gamma)
=
c\cdot
\sum_{w\in W_{G}(T_{\gamma},S)}\Delta_{\II}^{\abs,\bfG}[a_{\Psi},\chi''_{\Psi}]({}^{w}\gamma)\cdot\theta({}^{w}\gamma)
\]
for any elliptic very strongly regular element $\gamma\in G$ contained in $G^{0}_{\bar{\x}}$.
Moreover, such $\pi$ is given by the regular supercuspidal representation $\pi_{(\bfS, \theta)}^{\FKS}$.
Furthermore, when $(\bfS,\theta)$ is toral, only \eqref{ineq:toral-Henniart-evsrs} is enough.
\end{theorem}

In the following, let us furthermore assume that \eqref{ineq:Henniart-evsrs} is satisfied for $(\bfG,\bfS)$.
(When $(\bfG,\bfS)$ is toral, we only assume \eqref{ineq:toral-Henniart-evsrs}.)

\begin{thm}\label{thm:LJLC}
Suppose that the following conditions are satisfied:
\begin{enumerate}
\item
the $L$-packet $\Pi^{\bfG^{\ast}}_{\phi}$ of $\bfG^{\ast}$ containing $\pi_{(\bfS^{\ast},\theta^{\ast})}$ is a singleton;
\item
its Jacquet--Langlands transfer $\Pi^{\bfG}_{\phi}$ to $\bfG$ is a singleton;
\item
$\dim\iota(\pi^{\FKS}_{(\bfS,\theta)})=1$ and $\dim\iota^{\ast}(\pi^{\FKS}_{(\bfS^{\ast},\theta^{\ast})})=1$.
\end{enumerate}
Then the unique element of the $L$-packet $\Pi^{\bfG}_{\phi}$ is given by $\pi_{(\bfS,\theta)}$.
\end{thm}

\begin{proof}
Let $\pi$ be the unique element of $\Pi^{\bfG}_{\phi}$.
Let us take an elliptic very strongly regular element $\gamma$ of $G^{0}_{\bar{\x}}$.
Then there exists $g_{\gamma}\in \bfG^{\ast}$ such that $\gamma^{\ast}\colonequals {}^{g_{\gamma}}\psi(\gamma)$ is $F$-rational and $\Int(g_{\gamma})\circ\psi$ defines an $F$-rational isomorphism $\bfT_{\gamma}\rightarrow\bfT_{\gamma^{\ast}}$ (see \cite[Section 3.2]{Kal19}).
In particular, $\gamma^{\ast}$ is an elliptic very strongly regular element of $G^{\ast}$ related to $\gamma$.
Thus, by the assumptions and the Jacquet--Langlands character relation \eqref{eq:JLCR}, we have 
\[
e(\bfG)\cdot\Theta_{\pi}(\gamma)
=
\Theta_{\pi^{\FKS}_{(\bfS^{\ast},\theta^{\ast})}}(\gamma^{\ast}).
\]
By applying Corollary \ref{cor:regsc} to the right-hand side, we get
\[
e(\bfG)\cdot\Theta_{\pi}(\gamma)
=
\varepsilon_{L}(\bfT_{\bfG^{\ast}}-\bfS^{\ast})\cdot
\sum_{w^{\ast}\in W_{G^{\ast}}(T_{\gamma^{\ast}},S^{\ast})}\Delta_{\II}^{\abs,\bfG^{\ast}}[a_{\Psi^{\ast}},\chi''_{\Psi^{\ast}}]({}^{w^{\ast}}\gamma^{\ast})\cdot\theta^{\ast}({}^{w^{\ast}}\gamma^{\ast}).
\]
Note that we have an $F$-rational isomorphism
\[
W_{\bfG}(\bfT_{\gamma},\bfS)
\rightarrow
W_{\bfG^{\ast}}(\bfT_{\gamma^{\ast}},\bfS^{\ast})
\colon w\mapsto \psi(w)g_{\gamma}^{-1}.
\]
Hence, by applying Lemma \ref{lem:Weyl-eq} to both $W_{\bfG}(\bfT_{\gamma},\bfS)$ and $W_{\bfG^{\ast}}(\bfT_{\gamma^{\ast}},\bfS^{\ast})$, we see that the map $w\mapsto \psi(w)g_{\gamma}^{-1}$ induces a bijection $W_{G}(T_{\gamma},S)\rightarrow W_{G^{\ast}}(T_{\gamma^{\ast}},S^{\ast})$.
If we put $w^{\ast}\colonequals \psi(w)g_{\gamma}^{-1}$ for $w\in W_{G}(T_{\gamma},S)$, then we have 
\[
{}^{w^{\ast}}\gamma^{\ast}
={}^{\psi(w)g_{\gamma}^{-1}}{}^{g_{\gamma}}\psi(\gamma)
=\psi({}^{w}\gamma).
\]
By recalling that $\theta\colonequals \theta^{\ast}\circ\psi$, we have
\[
\theta^{\ast}({}^{w^{\ast}}\gamma^{\ast})
=\theta\circ\psi^{-1}(\psi({}^{w}\gamma))
=\theta({}^{w}\gamma).
\]
Moreover, by Lemma \ref{lem:tran-transfer}, we have
\[
\Delta_{\II}^{\abs,\bfG^{\ast}}[a_{\Psi^{\ast}},\chi''_{\Psi^{\ast}}]({}^{w^{\ast}}\gamma^{\ast})
=\Delta_{\II}^{\abs,\bfG^{\ast}}[a_{\Psi^{\ast}},\chi''_{\Psi^{\ast}}](\psi({}^{w}\gamma))
=\Delta_{\II}^{\abs,\bfG}[a_{\Psi},\chi''_{\Psi}]({}^{w}\gamma).
\]
Therefore, by noting that $\varepsilon_{L}(\bfT_{\bfG^{\ast}}-\bfS^{\ast})=\varepsilon_{L}(\bfT_{\bfG^{\ast}}-\bfS)$ (the tori $\bfS$ and $\bfS^{\ast}$ are $F$-rationally isomorphic), we get
\[
\Theta_{\pi}(\gamma)
=
e(\bfG)\cdot\varepsilon_{L}(\bfT_{\bfG^{\ast}}-\bfS)\cdot
\sum_{w\in W_{G}(T_{\gamma},S)}\Delta_{\II}^{\abs,\bfG}[a_{\Psi},\chi''_{\Psi}]({}^{w}\gamma)\cdot\theta({}^{w}\gamma).
\]
Now Theorem \ref{thm:regsc-characterization-evsrs} implies that $\pi$ is isomorphic to $\pi^{\FKS}_{(\bfS,\theta)}$.
\end{proof}

\begin{rem}
In Kaletha's construction of the local Langlands correspondence for regular supercuspidal representations \cite{Kal19}, the members of each $L$-packet are parametrized by the set of rational conjugacy class within a stable conjugacy class of $F$-rational embeddings of a tame elliptic maximal torus (see \cite[Section 5.3]{Kal19} and also \cite[Section 8]{CO21}).
In fact, this parametrizing set is bijective to the kernel of the map $H^{1}(F,\bfS^{\ast})\rightarrow H^{1}(F,\bfG^{\ast})$.
Hence the assumption \eqref{hyp:injectivity} amounts to supposing that the $L$-packet of $\bfG^{\ast}$ (resp.\ $\bfG$) containing the representation $\pi^{\FKS}_{(\bfS^{\ast},\theta^{\ast})}$ (resp.\ $\pi^{\FKS}_{(\bfS,\theta)}$) in the sense of Kaletha is a singleton.
Moreover, the third condition that $\dim\iota(\pi^{\FKS}_{(\bfS,\theta)})=1$ and $\dim\iota^{\ast}(\pi^{\FKS}_{(\bfS^{\ast},\theta^{\ast})})=1$ is automatically satisfied in Kaletha's construction since the group $\mathcal{S}_{\phi}$ is abelian (\cite[Section 5.3]{Kal19}).
Hence the assumptions of Theorem \ref{thm:LJLC} are expected to be implied by the assumption \eqref{hyp:injectivity}.
\end{rem}

The assumptions on $(\bfG^{\ast},\bfG,\bfS^{\ast},\bfS)$ we made so far are 
\begin{itemize}
\item
the condition \eqref{hyp:injectivity}, 
\item
the inequality \eqref{ineq:Henniart-evsrs} for $(\bfG,\bfS)$ (or $(\bfG^{\ast},\bfS^{\ast})$), and 
\item
the assumptions of Theorem \ref{thm:LJLC}.
\end{itemize}
We present some examples of such $(\bfG^{\ast},\bfG,\bfS^{\ast},\bfS)$ in the following sections.

\subsection{Examples of singleton $L$-packets}

\subsubsection{The case of $\GL_{n}$}
We first consider the case where $\bfG^{\ast}=\GL_{n}$ and $\bfG$ is an inner form of $\GL_{n}$.
In this case, since any elliptic maximal torus $\bfS^{\ast}$ of $\bfG^{\ast}=\GL_{n}$ is given by $\Res_{E/F}\Gm$ for a degree $n$ field extension $E$ of $F$, its first cohomology $H^{1}(F,\bfS^{\ast})$ is trivial by Shapiro's lemma and Hilbert's 90th theorem.
In particular, the condition \eqref{hyp:injectivity} is always satisfied.
Moreover, also the assumptions of Theorem \ref{thm:LJLC} are always satisfied.

Let us consider the inequality \eqref{ineq:Henniart-evsrs} for $(\bfG^{\ast},\bfS^{\ast})$.
Since the derived group of $\bfG^{\ast}$ is given by $\SL_{n}$, which is simply-connected, the strong regularity is equivalent to the regularity for semisimple elements of $\bfG^{\ast}$ (see, e.g., the final paragraph of \cite[Section 3, 788 page]{Kot82}).
Thus the inequality \eqref{ineq:Henniart-evsrs} for $(\bfG^{\ast},\bfS^{\ast})$ is the same as the inequality \eqref{ineq:Henniart-evrs} for $(\bfG^{\ast},\bfS^{\ast})$, which is investigated in Section \ref{subsec:Henn-GL}.
By the computation in Section \ref{subsec:Henn-GL},  we obtain the following from Theorem \ref{thm:LJLC}:

\begin{thm}\label{thm:LJLC-GL}
The regular supercuspidal representations $\pi^{\FKS}_{(\bfS^{\ast},\bfG^{\ast})}$ of $G^{\ast}$ and $\pi^{\FKS}_{(\bfS,\theta)}$ of $G$ correspond under the local Jacquet--Langlands correspondence in the following cases:
\begin{itemize}
\item
$E$ is unramified, $n$ is a prime such that $(n,q)\neq (1,\mathrm{any}), (2,2),(2,3)$;
\item
$E$ is totally ramified;
\item
the residue degree $f$ of $E/F$ is a prime satisfying $\frac{e}{e-\varphi(e)}>2f$, where $e$ denotes the ramification index of $E/F$, and $q$ is sufficiently large.
\end{itemize}
\end{thm}

The above result is not new in any case.
Indeed, for $\GL_{n}$, the regular supercuspidal representations are nothing but so-called \textit{essentially tame supercuspidal representations}, which has been thoroughly studied by Bushnell--Henniart (see \cite[Lemma 3.7.7]{Kal19} and also \cite[Section 4.1]{OT21}).
In \cite{BH11}, Bushnell--Henniart gave a description of the local Jacquet--Langlands correspondence for essentially tame supercuspidal representations, so it includes Theorem \ref{thm:LJLC-GL} completely.
However, we still emphasize that we obtained the above result (Theorem \ref{thm:LJLC-GL}) according to Henniart's original method used in his work \cite{Hen92,Hen93}.

\subsubsection{The case of $\SO_{2n+1}$}
We next consider the case where $\bfG^{\ast}=\SO_{2n+1}$ and $\bfG$ is an inner form of $\SO_{2n+1}$.
Let $\x^{\ast}\in\mcB(\bfG^{\ast},F)$ be a hyperspecial point.
Then the associated $\F_{q}$-group $\bbG^{\circ}$ defined as in Section \ref{subsec:sc classes} is given by $\SO_{2n+1,\F_{q}}$.

Let us construct an unramified elliptic maximal torus $\bfS^{\ast}$ with associated point $\x^{\ast}$ as follows.
Let $E$ be an unramified degree $2n$ extension of $F$ and $E_{\pm}$ a degree $n$ unramified extension of $F$ contained in $E$.
Let $\tau$ be the nontrivial Galois conjugation of $E/E_{\pm}$, i.e., $\Gal(E/E_{\pm})=\langle\tau\rangle$.
We define a quadratic form $q_{E}$ on $E$ (as an $F$-vector space) by
\[
q_{E}\colon E\times E\rightarrow F:\quad
(x,y)
\mapsto
\Tr_{E/F}(\tau(x)y).
\]
Then we can find a quadratic form $q_{F}$ on $F$ such that the quadratic space $(E',q_{E'})\colonequals (E\oplus F, q_{E}\oplus q_{F})$ is split, hence, the associated special orthogonal group $\SO(E', q_{E'})$ is isomorphic to $\bfG^{\ast}$.
We let $\bfS^{\ast}$ be the subgroup of $\bfG^{\ast}$ consisting of elements of $E^{\times}$ preserves the quadratic form $q_{E'}$.
More precisely, $\bfS^{\ast}$ is an $F$-rational subgroup of $\bfG^{\ast}$ whose set of $R$-valued points for any $F$-algebra $R$ is given by
\[
\{z\in (E\otimes_{F} R)^{\times} \mid q_{E'}(z\cdot x,z\cdot y)=q_{E'}(x,y),\, \forall x,y\in E'\otimes_{F}R\}
\]
(the action of $z\in(E\otimes_{F}R)^{\times}$ on $x=(x_{1},x_{2})\in E'\otimes_{F}R=E\otimes_{F}R\oplus F\otimes_{F}R$ is given by $z\cdot x=(zx_{1},x_{2})$).
Then $\bfS^{\ast}$ is an unramified elliptic maximal torus of $\bfG^{\ast}$, which is isomorphic to
\[
\Ker(\Nr_{E/E_{\pm}}\colon \Res_{E/F}\Gm \rightarrow \Res_{E_{\pm}/F}\Gm).
\]
Note that the reduction $\bbS^{\circ}$ of $\bfS^{\ast}$ can be described in a similar way; it is isomorphic to 
\[
\Ker(\Nr_{\F_{q^{2n}}/\F_{q^{n}}}\colon \Res_{\F_{q^{2n}}/\F_{q}}\Gm \rightarrow \Res_{\F_{q^{n}}/\F_{q}}\Gm).
\]
This is a maximal torus of $\SO_{2n+1,\F_{q}}$ of Coxeter type, i.e., its Frobenius structure is given by a Coxeter element of the absolute Weyl group of $\SO_{2n+1,\F_{q}}$.

Let us compute $H^{1}(F,\bfS^{\ast})$.
We consider the long exact sequence associated to
\[
1
\rightarrow
\bfS^{\ast}
\rightarrow
\Res_{E/F}\Gm
\xrightarrow{\Nr_{E/E_{\pm}}}
\Res_{E_{\pm}/F}\Gm
\rightarrow
1.
\]
Then, since $H^{1}(F,\Res_{E/F}\Gm)$ vanishes, we get an exact sequence
\[
E^{\times}
\xrightarrow{\Nr_{E/E_{\pm}}}
E_{\pm}^{\times}
\rightarrow
H^{1}(F,\bfS^{\ast})
\rightarrow
1.
\]
This implies that the order of $H^{1}(F,\bfS^{\ast})$ equals $2$.
On the other hand, by the Kottwitz isomorphism (\cite[1.2 Theorem]{Kot86}), we have $H^{1}(F,\SO_{2n+1})\cong\pi_{0}(\bfZ_{\Sp_{2n}(\C)}^{\Gamma_{F}})^{\wedge}$.
Thus the order of $H^{1}(F,\SO_{2n+1})$ is also equal to $2$.
Hence, by the surjectivity of the map $H^{1}(F,\bfS^{\ast})\rightarrow H^{1}(F,\bfG^{\ast})$, the assumption \eqref{hyp:injectivity} is satisfied (see Remark \ref{rem:injectivity}).

We next consider the inequality \eqref{ineq:Henniart-evsrs} for $(\bfG^{\ast},\bfS^{\ast})$.
Since $\bfG^{\ast}$ is not simply-connected, there might be a difference between \eqref{ineq:Henniart-evrs} and \eqref{ineq:Henniart-evsrs}.
However, at least we can show that \eqref{ineq:Henniart-evsrs} is satisfied if $q$ is larger than a constant determined by $n$ (see the discussion in Section \ref{subsec:Henn-ur}).
Thus let us just assume that $q$ is sufficiently large so that \eqref{ineq:Henniart-evsrs} is satisfied.

Finally, we investigate the assumptions of Theorem \ref{thm:LJLC} for regular characters of $S^{\ast}$.
Since the group $\mathcal{S}_{\phi}$ is abelian for any $L$-parameter of $\bfG$ and $\bfG^{\ast}$, the condition (3) is always satisfied.
Moreover, the condition (2) is implied by the condition (1).
So the problem is when the condition (1) is satisfied.
Although we expect that there are many examples of $\theta^{\ast}$ satisfying (1), we focus only on the depth zero case in the following.

We first briefly review \textit{Lusztig series} based on \cite[Section 2.5]{GM20}.
Let $\hat{\bbG}^{\circ}$ be the Langlands dual group of $\bbG^{\circ}$ taken over $\overline{\F}_{q}$.
Then the $\F_{q}$-rational structure on $\bbG^{\circ}$ induces an $\F_{q}$-rational structure on $\hat{\bbG}^{\circ}$.
We have a natural bijection (let us write $\bbT^{\circ}\leftrightarrow\hat{\bbT}^{\circ}$) between
\begin{itemize}
\item
the set of $\bbG^{\circ}(\F_{q})$-conjugacy classes of $\F_{q}$-rational maximal tori of $\bbG^{\circ}$ and
\item
the set of $\hat{\bbG}^{\circ}(\F_{q})$-conjugacy classes of $\F_{q}$-rational maximal tori of $\hat{\bbG}^{\circ}$.
\end{itemize}
Based on this, we also have a bijective correspondence between
\begin{itemize}
\item
the set $\bbG^{\circ}(\F_{q})$-conjugacy classes of pairs $(\bbT^{\circ},\theta^{\circ})$ of an $\F_{q}$-rational maximal torus $\bbT^{\circ}$ of $\bbG^{\circ}$ and a character $\theta^{\circ}$ of $\bbT^{\circ}(\F_{q})$ and 
\item
the set of $\hat{\bbG}^{\circ}(\F_{q})$-conjugacy classes of pairs $(\hat{\bbT}^{\circ},s)$ of an $\F_{q}$-rational maximal torus $\hat{\bbT}^{\circ}$ of $\hat{\bbG}^{\circ}$ and a semisimple element $s\in\hat{\bbT}^{\circ}(\F_{q})$
\end{itemize} 
(see \cite[Definition 2.5.17]{GM20}).
For any semisimple element $s\in\hat{\bbG}^{\circ}(\F_{q})$, we define $\mathcal{E}(\bbG^{\circ},s)$ to be the set of irreducible representations $\rho$ of $\bbG^{\circ}(\F_{q})$ satisfying $\langle R_{\bbT^{\circ}}^{\bbG^{\circ}}(\theta^{\circ}),\rho\rangle\neq0$ for some $(\bbT^{\circ},\theta^{\circ})$ satisfying $s\in\hat{\bbT}^{\circ}(\F_{q})$ (this is so-called \textit{Lusztig series}).
Then Lusztig's theorem \cite[7.6]{Lus77} asserts that we have a partition
\[
\Irr(\bbG^{\circ}(\F_{q}))
=
\bigsqcup_{s}\mathcal{E}(\bbG^{\circ},s),
\]
where the disjoint union is over the $\hat{\bbG}^{\circ}(\F_{q})$-conjugacy classes of semisimple elements of $\hat{\bbG}^{\circ}(\F_{q})$ (see also \cite[Theorem 2.6.2]{GM20}).

\begin{rem}
We note that a character $\theta^{\circ}$ of $\bbT^{\circ}(\F_{q})$ is in general position if the pair $(\bbT^{\circ},\theta^{\circ})$ corresponds to a pair $(\hat{\bbT}^{\circ},s)$ whose $s$ is regular semisimple (see \cite[Example 2.6.7]{GM20}).
\end{rem}

\begin{prop}\label{prop:SO singleton}
Let $\theta^{\ast}$ be a regular character of $S^{\ast}$ of depth zero such that its restriction to $S^{\ast}_{0}$ induces a character $\theta^{\circ}$ of $\bbS^{\circ}(\F_{q})$ corresponding to a pair $(\hat{\bbS}^{\circ},s)$ whose $s$ is regular semisimple.
Then the $L$-packet (in the sense of Arthur) containing $\pi^{\FKS}_{(\bfS^{\ast},\theta^{\ast})}$ is a singleton.
\end{prop}

\begin{proof}
Let $\phi$ be the $L$-parameter of $\pi^{\FKS}_{(\bfS^{\ast},\theta^{\ast})}$.
Then, since the dual group of $\bfG^{\ast}$ is given by $\Sp_{2n}(\C)$, we may think of $\phi$ as a $2n$-dimensional symplectic representation of $W_{F}\times\SL_{2}(\C)$.
As $\pi^{\FKS}_{(\bfS^{\ast},\theta^{\ast})}$ is supercuspidal, in particular, discrete series, $\phi$ is also discrete.
In this case, the discreteness of $\phi$ is equivalent to that $\phi$ is the direct sum of pairwise inequivalent symplectic representations of $W_{F}\times\SL_{2}(\C)$.
Let us write the irreducible decomposition of $\phi$ as
\[
\phi=\bigoplus_{i=1}^{r} \rho_{i}\boxtimes S_{n_{i}},
\]
where $\rho_{i}$ is an irreducible representation of $W_{F}$ and $S_{n_{i}}$ is the unique $n_{i}$-dimensional irreducible representation of $\SL_{2}(\C)$ for an integer $n_{i}\in\Z_{>0}$.
Then we can easily see that the group $\mathcal{S}_{\phi}$, which is defined to be $\Cent_{\Sp_{2n}(\C)}(\mathrm{Im}(\phi))/Z_{\Sp_{2n}(\C)}$ in this case, is isomorphic to $(\Z/2\Z)^{\oplus r-1}$.
As the map $\iota^{\ast}\colon\Pi^{\bfG^{\ast}}_{\phi}\rightarrow \mathrm{Irr}(\mathcal{S}_{\phi})$ is bijective in this setting, the order of the $L$-packet $\Pi^{\bfG^{\ast}}_{\phi}$ is given by $2^{r-1}$.
(See, e.g., \cite[Section 2]{Xu17-MM} for the details of all the arguments in this paragraph.)

Thus our task is to show that $r$ is given by $1$, i.e., $\phi$ is irreducible as a representation of $\SL_{2}(\C)\times W_{F}$.
For this, we utilize a result of Lust--Stevens \cite{LS20}, which enables us to describe the set $\{(\rho_{1},n_{1}),\ldots,(\rho_{r},n_{r})\}$ up to unramified twists of $\rho_{i}$'s by looking at the supercuspidal representation $\pi^{\FKS}_{(\bfS^{\ast},\theta^{\ast})}$.
Let $\Psi^{\ast}=(\vec{\bfG}^{\ast},\vec{\phi}^{\ast},\vec{r}^{\ast},\x^{\ast},\rho^{\ast}_{0})$ be a regular depth zero Yu-datum associated to $(\bfS^{\ast},\theta^{\ast})$.
Recall that $\rho^{\ast}_{0}$ is an extension of the inflation of an irreducible cuspidal representation of $\bbG^{\circ}(\F_{q})$ (say $\rho$).
(Since $\theta^{\ast}$ is regular of depth zero, $\rho$ is given by $\pm R_{\bbS^{\circ}}^{\bbG^{\circ}}(\theta^{\circ})$).

Let $s$ be a semisimple element of $\hat{\bbG}^{\circ}(\F_{q})$ such that $\rho\in\mathcal{E}(\bbG^{\circ},s)$.
We put $P_{s}(T)\in \F_{q}[X]$ to be the characteristic polynomial of $s\in \hat{\bbG}^{\circ}(\F_{q})=\Sp_{2n}(\F_{q})$ and write
\[
P_{s}(T)
=
\prod_{\begin{subarray}{c}P(T)\in \F_{q}[T], \\ \text{irreducible} \end{subarray}} P(T)^{a_{P}},
\]
where $a_{P}\in\Z_{\geq0}$.
In general, Lust--Stevens' description of $\{(\rho_{1},n_{1}),\ldots,(\rho_{r},n_{r})\}$ is given based on the information of $P(T)$ with $a_{P}>0$.
(See \cite[Section 9]{LS20} for the details.)
In our case, we can easily check that $P_{s}(T)$ is irreducible.
Indeed, by assumption, $s$ is a regular semisimple element of $\hat{\bbS}^{\circ}(\F_{q})$.
Since $\bbS^{\circ}$ is a maximal torus of $\bbG^{\circ}=\SO_{2n+1,\F_{q}}$ of Coxeter type, $\hat{\bbS}^{\circ}(\F_{q})$ is also a maximal torus of $\hat{\bbG}^{\circ}=\Sp_{2n,\F_{q}}$ of Coxeter type.
Similarly to the case of $\SO_{2n+1,\F_{q}}$, such a torus of $\Sp_{2n,\F_{q}}$ is given by the kernel of the norm map from $\F_{q^{2n}}^{\times}$ to $\F_{q^{n}}^{\times}$.
Thus the regular semisimplicity of $s$ is equivalent to the condition that $s$ does not belong to any proper subextension of $\F_{q^{2n}}/\F_{q}$, which implies that $P_{s}(T)$ is irreducible over $\F_{q}$.
Then the result of Lust--Stevens implies that, in particular, at least one $\rho_{i}$ must have the dimension $2n$.
However, as the dimension of $\phi$ is $2n$, this implies that $r=1$ (and furthermore $n_{1}=1$).
\end{proof}

Now we obtain the following from Theorem \ref{thm:LJLC}:

\begin{thm}\label{thm:LJLC-SO}
The depth zero regular supercuspidal representations $\pi^{\FKS}_{(\bfS^{\ast},\theta^{\ast})}$ of $G^{\ast}$ and $\pi^{\FKS}_{(\bfS,\theta)}$ of $G$ correspond under the local Jacquet--Langlands correspondence as long as $q$ is sufficiently large.
\end{thm}

\begin{rem}
It is proved that Kaletha's local Langlands correspondence indeed satisfies the standard endoscopic character relation under an additional assumption on $p$.
The toral case was firstly treated in \cite[Theorem 6.3.4]{Kal19}, and then the general non-singular case was treated by Fintzen--Kaletha--Spice in \cite[Theorem 4.4.4]{FKS21}.
Their method is based on the character formula of Adler--DeBacker--Spice (\cite{AS09,DS18,Spi18,Spi21}), which describes the Harish-Chandra character of a tame supercuspidal representation at any elliptic regular semisimple elements completely.
The point is that the shape of the character formula becomes more complicated if an elliptic regular semisimple element is not very regular.
Especially, we have to use a nice logarithm map from the $p$-adic group to its Lie algebra in order to express the contribution of the positive depth part of such an element to the Harish-Chandra character.
In general, the existence of such a logarithm map requires a stronger assumption on $p$ than the assumptions needed in Kaletha's construction of the local Langlands correspondence.
For example, a sufficient condition is that $p\geq (2+e)n$, where $e$ is the ramification index of $F/\bbQ_{p}$ and $n$ is the dimension of the smallest faithful representation of $\bfG$ (see \cite[Section 4.3]{Kal19-sc}).

Therefore, when $F$ has characteristic zero and $p$ is sufficiently large so that the result of \cite{FKS21} is available, Theorem \ref{thm:LJLC-SO} is not new.
(Nevertheless, it is worth noting that our approach provides a new proof.)
On the other hand, when $F$ has positive characteristic or when $F$ has characteristic zero and $p$ satisfies our basic assumptions but is not sufficiently large, Theorem \ref{thm:LJLC-SO} is \textit{not} covered by \cite{FKS21} as long as $q$ is sufficiently large.
\end{rem}

\newpage

\appendix

\section{Henniart inequality}\label{sec:henniart}

Let $\bfG$, $\bfG^{0}$, and $\bfS$ be as in Section \ref{subsec:ur-vreg}.
In this section, we present several examples of pairs $(\bfG,\bfS)$ satisfying the inequality \eqref{ineq:Henniart-evrs}:
\[
\frac{|[\bbS]^{\star}|}{|[\bbS]^{\star} \smallsetminus [\bbS]^{\star}_\evrs|} 
>
2\cdot|W_{\bbG(\F_{q})}(\bbS)|.
\]

\subsection{The unramified case}\label{subsec:Henn-ur}
We first consider the case where $\bfS$ is unramified.
In this case, we have $S=Z_{\bfG}S_{0}$, which implies that $\bbS(\F_{q})=\bbZ^{\star}_{\bbG}(\F_{q})\bbS^{\circ}(\F_{q})$ (see \cite[1095 page]{Kal19}).
Hence we get
\[
\bbS^{\circ}(\F_{q})/\bbZ_{\bbG^{\circ}}(\F_{q})
\xrightarrow{\cong}
\bbS(\F_{q})/\bbZ^{\star}_{\bbG}(\F_{q})=[\bbS]^{\star}.
\]
This implies that
\[
\frac{|[\bbS]^{\star}|}{|[\bbS]^{\star} \smallsetminus [\bbS]^{\star}_\evrs|} 
=\frac{|\bbS^{\circ}(\F_{q})|}{|\bbS^{\circ}(\F_{q}) \smallsetminus \bbS^{\circ}(\F_{q})_{\evrs}|},
\]
where $\bbS^{\circ}(\F_{q})_{\evrs}\colonequals \bbS^{\circ}(\F_{q})\cap\bbS(\F_{q})_{\evrs}$.
In fact, it can be checked that this quantity can be bounded below by a polynomial in $q$ with a positive coefficient on its highest degree.
See \cite[Section 5.1]{CO21} (especially, the proof of \cite[Proposition 5.8]{CO21}) for the details.
We caution that we adopted a different usage of notations in \cite[Section 5.1]{CO21}; ``$\bbS$'' in \cite[Section 5.1]{CO21} is the connected torus $\bbS^{\circ}$ and ``$\bbS(\F_{q})_{\nvreg}$'' in \cite[Section 5.1]{CO21} is $\bbS^{\circ}(\F_{q}) \smallsetminus \bbS^{\circ}(\F_{q})_\evrs$.

On the other hand, we have $|W_{\bbG(\F_{q})}(\bbS)|=|W_{\bbG^{\circ}(\F_{q})}(\bbS^{\circ})|$ and this quantity is determined only by the Lie type of $\bbG^{\circ}$ and the rational structure on $\bbS^{\circ}$.
At least, we see that $|W_{\bbG^{\circ}(\F_{q})}(\bbS^{\circ})|$ is not greater than the order of the absolute Weyl group of $\bbG^{\circ}$, which is independent of $q$.
Therefore, we conclude that the inequality \eqref{ineq:Henniart-evrs} holds whenever $q$ is sufficiently large compared to the absolute rank of $\bfG^{\circ}$.

We note that it is possible to explicate the precise bound of $q$ so that the Henniart inequality holds as long as the data $(\bfS\subset\bfG^{0}\subset\bfG)$ is given explicitly.
For example, any connected reductive group over a finite field has a particular elliptic maximal torus called ``Coxeter type''.
Thus it is natural to attempt to compute the inequality \eqref{ineq:Henniart-evrs} by choosing $\bfG$ to be a split connected reductive group over $F$ and $\bfS$ to be an unramified elliptic maximal torus of $\bfG$ such that $\bbS^{\circ}$ is a maximal torus of $\bbG^{\circ}$ of Coxeter type.
In this case, we have $\bbS^{\circ}(\F_{q})_{\evrs}=\bbS^{\circ}(\F_{q})_{\rs}$.

For example, when $\bfG=\GL_{n}$ with prime $n$, such an $\bfS$ is given by an unramified maximal torus of $\GL_{n}$ corresponding to an unramified extension of $F$ of degree $n$.
Thus the left-hand side of the inequality \eqref{ineq:Henniart-evrs} is given by $\frac{q^{n}-1}{q-1}$.
On the other hand, the left-hand side depends on $\bfG^{0}$; $|W_{\bbG(\F_{q})}(\bbS)|=1$ when $\bfG^{0}=\bfS$ and $|W_{\bbG(\F_{q})}(\bbS)|=n$ when $\bfG^{0}=\bfG$.
We can easily see that there are very few pairs of $(n,q)$ not satisfying the inequality in the ``worst'' case
\[
\frac{q^{n}-1}{q-1}>2n
\]
 (see Section \ref{subsubsec:e=1}).
In fact, this inequality is nothing but the one considered in the work of Henniart \cite{Hen92}.
(It is still possible to explicate the inequality even when $n$ is not a prime, but the computation is more complicated; see \cite[Section 2.7]{Hen92}.)
For this reason, we call the inequality \eqref{ineq:Henniart-bullet} (and its variant \eqref{ineq:Lusztig-bullet}) \textit{the Henniart inequality}.

Since an explicit formula of the number of regular semisimple elements in a Coxeter torus is known (see, e.g., \cite{FJ93}), similar computation should be able to be done also for other classical groups.
For exceptional groups of adjoint type, the numbers $|\bbS^{\circ}(\F_{q})|$ and $|\bbS^{\circ}(\F_{q})\smallsetminus\bbS^{\circ}(\F_{q})_{\rs}|$ can be computed case-by-case (for example using Carter's classification of conjugacy classes in Weyl groups \cite{Car72}).
We provide the conclusion of these calculations:
\begin{tabular}{|c|c|c|c|c|} \hline
type & $|\bbS^{\circ}(\F_{q})|$ & $|\bbS^{\circ}(\F_{q})\smallsetminus\bbS^{\circ}(\F_{q})_{\evrs}|$ & $\bfG^{0}=\bfS$ & $\bfG^{0}=\bfG$ \\ \hline
$E_{6}$ & $(q^{4}-q^{2}+1)(q^{2}+q+1)$ & $q^{2}+q+1$ & $q$: any & $q>2$ \\\hline
$E_{7}$ & $(q^{6}-q^{3}+1)(q+1)$ & $\begin{cases}3(q+1)&\text{$q\equiv-1\mod3$}\\q+1&\text{$q\not\equiv-1\mod3$}\end{cases}$ &  $q$: any & $\begin{cases}q>2\\ \text{$q$: any}\end{cases}$\\ \hline
$E_{8}$ & $q^{8}+q^{7}-q^{5}-q^{4}-q^{3}+q+1$ & 1  & $q$: any & $q$: any \\ \hline
$F_{4}$ & $q^{4}-q^{2}+1$ & $1$ & $q$: any & $q>2$ \\ \hline
$G_{2}$ & $q^{2}-q+1$ & $\begin{cases}3&\text{$q\equiv-1\mod3$}\\1&\text{$q\not\equiv-1\mod3$}\end{cases}$ & $\begin{cases}q>2\\ \text{$q$: any}\end{cases}$ & $\begin{cases}q>6\\ q>3 \end{cases}$ \\ \hline
\end{tabular}

\subsection{The case of $\GL_{n}$, $\SL_{n}$, and $\PGL_{n}$}\label{subsec:Henn-GL}
We next consider the case where $\bfG=\GL_{n}$.
Let $\bfS$ be a tame elliptic maximal torus $\bfS$ in $\bfG$.
Then, it is well-known that there exists a tamely ramified extension $E/F$ of degree $n$ such that $\bfS$ is isomorphic to $\Res_{E/F}\Gm$.

The set $\Gamma_{F}\backslash R(\bfG,\bfS)$ of $\Gamma_{F}$-orbits of absolute roots of $\bfS$ in $\bfG$ can described in the following manner (see \cite[Section 3.1]{OT21} and also \cite[Section 3.2]{Tam16}).
We fix a set $\{g_{1},\ldots,g_{n}\}$ of representatives of the quotient $\Gamma_{F}/\Gamma_{E}$ such that $g_{1}=\id$.
Then we get an isomorphism $\bfS(\overline{F})\cong\prod_{i=1}^{n} \overline{F}^{\times}$ which maps $x\in E^{\times}\cong\bfS(F)$ to $(g_{1}(x),\ldots,g_{n}(x))$.
The projections 
\[
\delta_{i}
\colon 
\bfS(\overline{F})\xrightarrow{\cong}\prod_{i=1}^{n} \overline{F}^{\times}\rightarrow \overline{F}^{\times}
;\quad
(x_{1},\ldots,x_{n})\mapsto x_{i}
\]
form a $\Z$-basis of the group $X^{\ast}(\bfS)$ of absolute characters of $\bfS$.
The set $R(\bfG,\bfS)$ of absolute roots of $\bfS$ in $\bfG$ is given by
\[
\biggl\{
\begin{bmatrix}g_{i}\\g_{j}\end{bmatrix}\colonequals \delta_{i}-\delta_{j}
\,\bigg\vert\,
1\leq i \neq j \leq n
\biggr\}
\]
and we have
\[
\Gamma_{F}\backslash R(\bfG,\bfS)
=\biggl\{
\Gamma_{F}\cdot\begin{bmatrix}1\\g_{i}\end{bmatrix}
\,\bigg\vert\,
i=1,\ldots,n
\biggr\}\smallsetminus \Gamma_{F}\cdot\begin{bmatrix}1\\1\end{bmatrix}.
\]
(Note that $\begin{bmatrix}1\\g_{i}\end{bmatrix}$ is the character of $\bfS$ such that $\begin{bmatrix}1\\g_{i}\end{bmatrix}(x)=x/g_{i}(x)$ for $x\in E^{\times}\cong\bfS(F)$.)

We write $e$ (resp.\ $f$) for the ramification index (resp.\ residue degree) of the extension $E/F$.
Let us recall an explicit choice of a set of representatives of 
$\Gamma_F/\Gamma_E$, following \cite[Section 3.2]{OT21} (see also \cite[Section 5.1]{Tam16}).
We fix uniformizers $\varpi_{E}$ and $\varpi_{F}$ of $E$ and $F$, respectively, so that
\[
\varpi_{E}^e=\zeta _{E/F}\varpi_{F}
\]
for some root of unity $\zeta_{E/F} \in E^{\times}$.
We fix a primitive $e$-th root $\zeta_{e}$ of unity 
and an $e$-th root $\zeta_{E/F, e}$ of $\zeta_{E/F}$,
and put $\zeta_{\phi}\colonequals \zeta_{E/F, e}^{q-1}$.
Then $L\colonequals E[\zeta_{e}, \zeta_{E/F, e}]$ is a tamely ramified extension of $F$
which contains the Galois closure of $E/F$ and is unramified over $E$.
The Galois group $\Gamma_{L/F}$ of the extension $L/F$ is given by the semi-direct product
$\langle \sigma \rangle \rtimes \langle \phi \rangle$, where
\begin{align*}
&\sigma \colon \zeta \mapsto \zeta \quad (\zeta \in \mu_{L}), \quad \varpi_{E} \mapsto \zeta_e \varpi_{E} \\
&\phi \colon \zeta \mapsto \zeta ^q \quad (\zeta \in \mu_{L}), \quad \varpi_{E} \mapsto \zeta_{\phi}\varpi_{E}
\end{align*}
and $\phi \sigma \phi ^{-1}=\sigma ^q$.
Here $\mu_{L}$ denotes the set of roots of unity in $L^{\times}$.
As explained in \cite[Proposition 3.3 (i)]{OT21}, we can take a set of representatives of $\Gamma_{F}/\Gamma_{E}$ to be
\[
\{
\sigma^{i}\phi^{j}
\mid
0\leq i \leq e-1,\, 0\leq j \leq f-1
\}.
\]
(Here we implicitly regard each $\sigma^{i}\phi^{j}\in\Gamma_{L/F}$ as an element of $\Gamma_{F}$ by taking its extension to $\overline{F}$ from $L$.)

Now, based on this description of $R(\bfG,\bfS)$, let us investigate the very regular elements of $S$.
We fix an isomorphism $\bfS\cong\Res_{E/F}\Gm$.
As $\bfG^{0}$ is a tame twisted Levi subgroup of $\bfG=\GL_{n}$ such that $\bfS$ is maximally unramified in $\bfG^{0}$, there exists a subextension $K/F$ of degree $m$ in $E$ such that
\begin{itemize}
\item
$\bfG^{0}\cong\Res_{K/F}\GL_{n/m}$ and
\item
$E/K$ is unramified.
\end{itemize}
We fix a uniformizer of $\varpi_{K}$ of $K$.
Then we have
\begin{itemize}
\item
$\bbS(\F_{q})\cong E^{\times}/(1+\mfp_{E})\cong\langle\varpi_{K}\rangle\times k_{E}^{\times}$,
\item
$\bbS^{\circ}(\F_{q})\cong \mcO_{E}^{\times}/(1+\mfp_{E})\cong k_{E}^{\times}$, 
\item
$\bbZ_{\bbG}(\F_{q})\cong K^{\times}/(1+\mfp_{K})\cong\langle\varpi_{K}\rangle\times k_{K}^{\times}$, and 
\item
$\bbZ^{\star}_{\bbG}(\F_{q})\cong F^{\times}/(1+\mfp_{F})\cong\langle\varpi_{F}\rangle\times k_{F}^{\times}$.
\end{itemize}
In particular, we have
\[
[\bbS]^{\star}
=\bbS(\F_{q})/\bbZ^{\star}_{\bbG}(\F_{q})
\cong (\langle\varpi_{K}\rangle/\langle\varpi_{F}\rangle)\times(k_{E}^{\times}/k_{F}^{\times}).
\]

We regard $k_{E}^{\times}$ as a subgroup of $E^{\times}$ via the Teichm\"uller lift.
Note that, for any $l\in\Z$ and any $y\in k_{E}^{\times}$, the element $\varpi_{E}^{l}y\in E^{\times}=S$ is a topologically semisimple element (i.e., of finite prime-to-$p$ order modulo $Z_{\bfG}$).

\begin{lem}\label{lem:A1}
For any $l\in\Z$ and $y\in k_{E}^{\times}$, the element $\varpi_{E}^{l}y\in E^{\times}=S$ is regular semisimple in $\GL_{n}$ if and only if
\[\label{A1}
y^{q^{j}-1}\neq (\zeta_{E/F,e}^{q^{j}-1}\cdot\zeta_{e}^{i})^{-l}.
\tag{$\ast$}
\]
for any $0\leq i<e$ and $0\leq j<f$ satisfying $(i,j)\neq(0,0)$.
In particular, for such $l$ and $y\in k_{E}^{\times}$, the element $\varpi_{E}^{l}y\in E^{\times}=S$ is shallow.
\end{lem}

\begin{proof}
By the above description of the $\Gamma_{F}$-orbits of roots $\Gamma_{F}\backslash R(\bfG,\bfS)$, the element $\varpi_{E}y$ is regular semisimple if and only if 
\[
\begin{bmatrix}1\\\sigma^{i}\phi^{j}\end{bmatrix}(\varpi_{E}y)=\varpi_{E}y/\sigma^{i}\phi^{j}(\varpi_{E}y)\neq1
\]
for any $0\leq i\leq e-1$ and $0\leq j\leq f-1$ satisfying $(i,j)\neq(0,0)$.
Since we have
\begin{itemize}
\item
$\sigma^{i}\phi^{j}(y)=y^{q^{j}}$ and
\item
$\sigma^{i}\phi^{j}(\varpi_{E})=\sigma^{i}(\zeta_{\phi}^{1+q+\cdots+q^{j-1}}\varpi_{E})=\zeta_{\phi}^{\frac{q^{j}-1}{q-1}}\zeta_{e}^{i}\varpi_{E}=\zeta_{E/F,e}^{q^{j}-1}\zeta_{e}^{i}\varpi_{E}$,
\end{itemize}
we get 
\[
\sigma^{i}\phi^{j}(\varpi_{E}^{l}y)
=(\zeta_{E/F,e}^{q^{j}-1}\zeta_{e}^{i}\varpi_{E})^{l}y^{q^{j}}.
\]
Thus $\varpi_{E}^{l}y/\sigma^{i}\phi^{j}(\varpi_{E}^{l}y)\neq1$ if and only if $y^{q^{j}-1}\neq (\zeta_{E/F,e}^{q^{j}-1}\cdot\zeta_{e}^{i})^{-l}$.
\end{proof}

\subsubsection{The case where $e=1$}\label{subsubsec:e=1}
Let us first consider the case where $e=1$ (thus $n=f$).
In this case, $\bfS$ is an unramified torus.
Moreover, for simplicity, we also assume that $f$ is a prime and $\bfS\subsetneq\bfG^{0}$.

The group $[\bbS]^{\star}$ is isomorphic to $k_{E}^{\times}/k_{F}^{\times}\cong\F_{q^{n}}^{\times}/\F_{q}^{\times}$.
By Lemma \ref{lem:A1}, an element $y\in \F_{q^{n}}^{\times}\subset S$ is shallow if and only if $y^{q^{j}-1}\neq1$ for any $0\leq j<f$.
In other words, $y\in \F_{q^{n}}^{\times}$ is shallow if and only if $y$ does not belong to $\F_{q}^{\times}$.
Thus $[\bbS]^{\star}_{\evrs}$ is given by $(\F_{q^{n}}^{\times}\smallsetminus\F_{q}^{\times})/\F_{q}^{\times}$.
Since $|W_{\bbG(\F_{q})}(\bbS)|=|W_{\bbG^{\circ}(\F_{q})}(\bbS^{\circ})|=n$, the inequality \eqref{ineq:Henniart-evrs} is given by 
\[
\frac{q^{n}-1}{q-1}>2n.
\]

This inequality is not satisfied when $n=1$, so let us suppose that $n\geq2$.
When $q=2$, we can check that only $n=2$ does not satisfy this inequality.
When $q\geq3$, we have
\[
\frac{q^{n}-1}{q-1}-2n
=(q^{n-1}-2)+\cdots+(q-2)-1\geq0,
\]
where the equality holds only if $n=2$ and $q=3$.
In summary, the Henniart inequality holds when $(n,q)\neq(1,\text{any}), (2,2), (2,3)$.

See \cite[Section 2.7]{Hen92} for a computation for general (not necessarily a prime) $n$.
In fact, the inequality holds for any $(n,q)\neq(1,\text{any}), (2,2), (2,3), (4,2), (6,2)$.

\subsubsection{The case where $f=1$}\label{subsubsec:f=1}
We next consider the case where $f=1$ (thus $n=e$).
In this case, $\bfS$ is a totally ramified torus, hence we necessarily have $\bfS=\bfG^{0}$ (see Remark \ref{rem:toral-characterization}).

The group $[\bbS]^{\star}$ is isomorphic to $\langle\varpi_{E}\rangle/\langle\varpi_{F}\rangle\cong \Z/n\Z$.
By Lemma \ref{lem:A1}, for any $l\in\Z$, the element $\varpi_{E}^{l}\in S$ is shallow if and only if $\zeta_{e}^{il}\neq1$ for any $0\leq i<e$.
In other words, $\varpi_{E}^{l}\in S$ is shallow if and only if $l$ is prime to $e=n$.
Thus $[\bbS]^{\star}_{\evrs}$ can be identified with the subset of units $(\Z/n\Z)^{\times}$ in $\Z/n\Z$.
Hence the Henniart inequality becomes
\[
\frac{n}{n-\varphi(n)}>2,
\]
where $\varphi(n)$ denotes the Euler's totient function ($\varphi(n)=|(\Z/n\Z)^{\times}|$).

We have a lot of examples of $n$ satisfying this inequality; for example, any prime $n$ greater than $2$ satisfies this inequality.
However, we also have a lot of examples of $n$ not satisfying this inequality; for example, any $n$ divisible by $2$ cannot satisfy this inequality.

On the other hand, by the above description of $[\bbS]^{\star}$ and $[\bbS]^{\star}_{\evrs}$, we immediately see that $[\bbS]^{\star}_{\evrs}$ generates $[\bbS]^{\star}$ as a group.

\subsubsection{The case where $e>1$ and $f>1$}\label{subsubsec:remaining}
Let us next consider the case where $e>1$ and $f>1$.
For simplicity, let us again assume that the residue degree $f$ of $E/F$ is a prime.

\begin{prop}\label{prop:A1}
For any $l\in\Z$ and $y\in k_{E}^{\times}$, the element $\varpi_{E}^{l}y\in E^{\times}=S$ can be shallow only if $l$ is prime to $e$.
Moreover, for $l\in\Z$ prime to $e$, the number of elements $y\in k_{E}^{\times}$ such that $\varpi_{E}^{l}y$ is shallow is bounded below by
\[
(q^{f}-1)-(q-1)e(f-1).
\]
\end{prop}

\begin{proof}
Let us investigate when the condition \eqref{A1} of Lemma \ref{lem:A1} holds for any $0\leq i<e$ and $0\leq j<f$ satisfying $(i,j)\neq(0,0)$.

We first consider the condition \eqref{A1} for $j=0$ (hence $0<i<e$).
In this case, \eqref{A1} is equivalent to $il\not\equiv0 \pmod{e}$ as $\zeta_{e}$ is a primitive $e$-th root of unity.
Thus \eqref{A1} holds for any $0<i<e$ if and only if $l$ is prime to $e$.

We next consider the condition \eqref{A1} for $0<j<f$.
Let us fix $0\leq i<e$ and $0<j<f$ and count the number of $y\in k_{E}^{\times}$ which does not satisfy \eqref{A1}.
If $y_{1}\in k_{E}^{\times}$ and $y_{2}\in k_{E}^{\times}$ do not satisfy \eqref{A1}, then we must have $(y_{1}/y_{2})^{q^{j}-1}=1$.
This means that $y_{1}/y_{2}$ belongs to the degree $j$ extension $\F_{q^{j}}$ of $k_{F}=\F_{q}$.
Since we assume that $f$ is a prime number, this implies that $y_{1}/y_{2}\in k_{F}^{\times}$.
Hence, we see that at most $(q-1)$ elements of $k_{E}^{\times}$ can fail to satisfy the condition \eqref{A1} for fixed $0\leq i<e$ and $0<j<f$.
Therefore, in total, at most $(q-1)e(f-1)$ elements of $k_{E}^{\times}$ can be non-shallow.
In other words, the number of elements $y\in k_{E}^{\times}$ with shallow $\varpi_{E}^{l}y$ is bounded below by $(q^{f}-1)-(q-1)e(f-1)$.
\end{proof}

As we assume that $f$ is a prime, we necessarily have $[E:K]=1$ or $[E:K]=f$.
If $[E:K]=1$, then $\bfG^{0}=\bfS$.
Since the Henniart inequality is more restrictive in the non-toral situation, let us suppose that $[E:K]=f$ in the following.
Recall that $\bbG$ is the reduction of the subgroup $G^{0}_{\bar{\x}}$.
Thus the Weyl group $W_{\bbG(\F_{q})}(\bbS)$ is contained in the Weyl group $W_{G^{0}}(\bfS)$, which is isomorphic to the Galois group of the extension $E/K$, hence is of order $f$.
In particular, we have $f\geq |W_{\bbG(\F_{q})}(\bbS)|$.

\begin{cor}\label{cor:A1}
If we have
\[
\frac{e}{e-\varphi(e)}>2f,
\]
then the inequality \eqref{ineq:Henniart-evrs} is satisfied for sufficiently large $q$.
\end{cor}

\begin{proof}
We have $|[\bbS]^{\star}|=|(\langle\varpi_{K}\rangle/\langle\varpi_{F}\rangle)\times(k_{E}^{\times}/k_{F}^{\times})|=e(q^{f}-1)(q-1)^{-1}$.
If we put $g(q,f)\colonequals q^{f-1}+\cdots+q+1$, then we have $|[\bbS]^{\star}|=e\cdot g(q,f)$.

Let us evaluate $|[\bbS]^{\star}_{\evrs}|=|\bbS(\F_{q})^{\star}_{\evrs}/\bbZ^{\star}_{\bbG}(\F_{q})|$.
By Proposition \ref{prop:A1}, for any $l\in\Z$ prime to $e$, the number of elements $y\in k_{E}^{\times}$ such that $\varpi_{K}^{l}y$ is shallow is bounded below by $(q^{f}-1)-(q-1)e(f-1)$.
Hence we have
\[
|[\bbS]^{\star}_{\evrs}|
\geq
\varphi(e)\cdot\frac{(q^{f}-1)-(q-1)e(f-1)}{q-1}
=\varphi(e)\cdot g(q,f)-\varphi(e)\cdot e(f-1),
\]
where $\varphi(-)$ denotes Euler's totient function.
Thus we get
\begin{align*}
\frac{|[\bbS]^{\star}|}{|[\bbS]^{\star}\smallsetminus [\bbS]^{\star}_{\evrs}|} 
&\geq
\frac{e\cdot g(q,f)}{e\cdot g(q,f)-\varphi(e)\cdot g(q,f)+\varphi(e)e(f-1)}\\
&=\frac{e}{e-\varphi(e)+\varphi(e)e(f-1)\cdot g(q,f)^{-1}}.
\end{align*}
The right-hand side tends to $\frac{e}{e-\varphi(e)}$ when $q$ tends to infinity.
Therefore, if we have $\frac{e}{e-\varphi(e)}>2f$, then the inequality \eqref{ineq:Henniart-evrs} is satisfied for sufficiently large $q$.
\end{proof}

Similarly to the case in Section \ref{subsubsec:f=1}, we can find many examples of $n$ satisfying this inequality; for example, any prime $e$ greater than $2f$ satisfies this inequality.
However, we can also find many examples of $n$ not satisfying this inequality.

\subsubsection{The case of $\SL_{n}$ and $\PGL_{n}$}\label{subsubsec:Henn-SL-PGL}

We put $\bfG\colonequals \GL_{n}$, $\bfG_{\sc}\colonequals \SL_{n}$, and $\bfG_{\ad}\colonequals \PGL_{n}$.
Let $\bfS\cong\Res_{E/F}\Gm$ be an elliptic maximal torus of $\bfG$.
Let $\bfS_{\sc}$ be the preimage of $\bfS$ in $\bfG_{\sc}$ and $\bfS_{\ad}$ the image of $\bfS$ in $\bfG_{\ad}$.

Via the identification $\bfS\cong\Res_{E/F}\Gm$, $\bfS_{\sc}$ is mapped to the subgroup given by the kernel of the norm map $\Nr\colon \Res_{E/F}\Gm\rightarrow\Gm$.
Thus we can see that
\begin{itemize}
\item
$\bbS_{\sc}(\F_{q})=\Ker(\Nr\colon k_{E}^{\times}\rightarrow k_{F}^{\times})$.
\end{itemize}
In particular, this implies that no element of $S$ can be shallow if $e>1$.
Hence the Henniart inequality can hold only when $\bfS$ is unramified.

On the other hand, we see that
\begin{itemize}
\item
$\bbS_{\ad}(\F_{q})\cong (\langle\varpi_{E}\rangle/\langle\varpi_{F}\rangle)\times (k_{E}^{\times}/k_{F}^{\times})$,
\item
$\bbS_{\ad}^{\circ}(\F_{q})\cong k_{E}^{\times}/k_{F}^{\times}$, and
\item
$\Z^{\star}_{\overline{\bbG}}(\F_{q})=\{1\}$.
\end{itemize}
Thus exactly the same estimate as in the $\GL_{n}$ case is available for $\bfG_{\ad}=\PGL_{n}$.

\end{document}